\documentclass[a4]{amsart}

\usepackage{newtxtext}
\usepackage{amsmath,amsthm,amssymb,amscd,stmaryrd}
\usepackage{mathrsfs}
\usepackage{tikz-cd}
\usepackage{tikz}
\usetikzlibrary{decorations.markings,positioning}
\usepackage{xcolor}
\usepackage[linktocpage]{hyperref}
\usepackage{enumitem}
\setlist[enumerate,1]{label=\arabic*.,ref=\arabic*}

\newtheorem{thm}{Theorem}
\newtheorem{lem}[thm]{Lemma}
\newtheorem{prop}[thm]{Proposition}
\newtheorem{cor}[thm]{Corollary}

\newtheorem*{claim}{Claim}

\newtheorem{intthm}{Theorem}
\newtheorem{intprop}[intthm]{Proposition}

\theoremstyle{definition}
\newtheorem{dfn}[thm]{Definition}
\newtheorem{rem}[thm]{Remark}
\newtheorem{example}[thm]{Example}
\newtheorem{notation}[thm]{Notation}

\newcommand\Hom{\mathop{\mathrm{Hom}}\nolimits}
\newcommand\Nat{\mathop{\mathrm{Nat}}\nolimits}
\newcommand\End{\mathop{\mathrm{End}}\nolimits}
\newcommand\Aut{\mathop{\mathrm{Aut}}\nolimits}
\newcommand\Out{\mathop{\mathrm{Out}}\nolimits}

\newcommand\Ad{\mathop{\mathrm{Ad}}\nolimits}

\newcommand\id{\mathop{\mathrm{id}}\nolimits}
\newcommand\im{\mathop{\mathrm{im}}\nolimits}

\newcommand\GL{\mathop{\mathrm{GL}}\nolimits}
\newcommand\PGL{\mathop{\mathrm{PGL}}\nolimits}
\newcommand\SL{\mathop{\mathrm{SL}}\nolimits}
\newcommand\PSL{\mathop{\mathrm{PSL}}\nolimits}

\newcommand\PSO{\mathop{\mathrm{PSO}}\nolimits}
\newcommand\U{\mathop{\mathrm{U}}\nolimits}
\newcommand\SU{\mathop{\mathrm{SU}}\nolimits}

\newcommand\Homeo{\mathop{\mathrm{Homeo}}\nolimits}
\newcommand\Diff{\mathop{\mathrm{Diff}}\nolimits}
\newcommand\Isom{\mathop{\mathrm{Isom}}\nolimits}

\newcommand\Stab{\mathop{\mathrm{Stab}}\nolimits}
\newcommand\Mon{\mathop{\mathrm{Mon}}\nolimits}
\newcommand\Pair{\mathop{\mathrm{Pair}}\nolimits}
\newcommand\Lie{\mathop{\mathrm{Lie}}\nolimits}

\newcommand\MAG{\mathop{\mathrm{MAG}}\nolimits}
\newcommand\MHS{\mathop{\mathrm{MHS}}\nolimits}

\newcommand\MCG{\mathop{\mathrm{MCG}}\nolimits}

\newcommand\hol{\mathop{\mathrm{hol}}\nolimits}

\newcommand{\bb}[1]{{\mathbb{#1}}}
\newcommand{\mca}[1]{{\mathcal{#1}}}
\newcommand{\mr}[1]{{\mathrm{#1}}}
\newcommand{\ms}[1]{{\mathscr{#1}}}
\newcommand{\mt}[1]{{\mathtt{#1}}}
\newcommand{\ol}[1]{{\overline{#1}}}
\newcommand{\ul}[1]{{\underline{#1}}}
\newcommand{\wt}[1]{{\widetilde{#1}}}
\newcommand{\wh}[1]{{\widehat{#1}}}

\newcommand{\bs}{{\backslash}}

\allowdisplaybreaks

\begin{document}
\title[The orbit-fixing deformation spaces of an action of a Lie groupoid]{The orbit-fixing deformation spaces\\
of an action of a Lie groupoid}
\author[H. Maruhashi]{Hirokazu Maruhashi}
\address{Department of Mathematics and Informatics, Faculty of Science, Chiba University, 1-33 Yayoi-cho, Inage, Chiba, 263-8522, Japan}
\email{maruhashihirokazu@gmail.com}

\begin{abstract}
The orbit-fixing deformation spaces of $C^\infty$ locally free actions of simply connected Lie groups on closed $C^\infty$ manifolds have been studied by several authors. In this paper we reformulate the deformation space by imitating the Teichm\"{u}ller space of a surface. The new formulation seems to be more appropriate for actions of Lie groups which are not simply connected. We also consider actions which may not be locally free, and generalize the deformation spaces for actions of Lie groupoids. Furthermore by using bornologies on Lie groupoids, we make the definition of the deformation space more suitable to deal with actions on noncompact manifolds. In this generality we prove that ``cocycle rigidity'' implies the deformation space is a point. We compute the deformation space of the action of $\PSL(2,\bb{R})$ on $\Gamma\bs\PSL(2,\bb{R})$ by right multiplication for a torsion free cocompact lattice $\Gamma$ in $\PSL(2,\bb{R})$. 
\end{abstract}

\maketitle
\tableofcontents

\section{Introduction}
The orbit-fixing deformation spaces of $C^\infty$ locally free actions of simply connected Lie groups on closed $C^\infty$ manifolds have been studied in papers such as \cite{KS}, \cite{MM}, \cite{dSa}, \cite{A2}, \cite{Ma}, \cite{Ma2} and \cite{Ma3}. 

For a $C^\infty$ locally free right action $\rho_0$ of a simply connected Lie group $G$ on a closed $C^\infty$ manifold $N_0$, the orbit-fixing deformation space of $\rho_0$ is 
\begin{equation*}
\mca{A}_{LF}(\mca{O}_{\rho_0},G)/(\Diff(\mca{O}_{\rho_0})_1\times\Aut(G)), 
\end{equation*}
ie the quotient of the set of locally free $C^\infty$ right actions of $G$ on $N_0$ which have the same orbit foliation $\mca{O}_{\rho_0}$ as $\rho_0$ by $C^\infty$ conjugacies given by diffeomorphisms of $N_0$ with certain properties and an automorphism of $G$ (Definition \ref{pralfdiffaut}). 

In this paper we introduce a different formulation of the deformation space, which uses Lie groupoids and is defined analogously to the Teichm\"{u}ller spaces of surfaces. 

Recall that the Teichm\"{u}ller space $T(\Sigma_0)$ of a surface $\Sigma_0$ is defined as the set of all homotopy (isotopy) classes of marked hyperbolic surfaces $\varphi\colon\Sigma_0\to\Sigma$ (Definition \ref{mhssmhsfxx}). Imitating this we define the Teichm\"{u}ller spaces of a $C^\infty$ right action $\rho_0$ of a Lie groupoid $\mca{G}\rightrightarrows M$ on a surjective submersion $\nu_0\colon N_0\to M$ as the sets of certain homotopy classes of ``$\rho_0$-marked action groupoids'' $\varphi\colon N_0\rtimes_{\rho_0}\mca{G}\to N\rtimes_\rho\mca{G}$ (Definition \ref{gmngprmagl} and Definition \ref{Teistarsim}). Here we consider three types of homotopies: 
\begin{itemize}
\setlength\itemsep{0em}
\item orbitwise homotopy $\mathrel{\ul{\sim_o}}$ (Definition \ref{orbitwiseh}, Definition \ref{lgpcccffsimo} and Definition \ref{gmngprmagl})
\item $nt$-homotopy $\mathrel{\ul{\sim_{nt}}}$ (Definition \ref{nthomptppp}, Definition \ref{raragpppp} and Definition \ref{gmngprmagl})
\item bounded $nt$-homotopy $\mathrel{\ul{\sim_{bnt}}}$ (Definition \ref{hnhnpnib}, Definition \ref{simbntbthe} and Definition \ref{gmngprmagl}), 
\end{itemize}
resulting in three types of Teichm\"{u}ller spaces $T_o(\rho_0)$, $T_{nt}(\rho_0)$ and $T_{bnt}(\rho_0)$. (To be more precise we use only ``bibounded $\rho_0$-marked action groupoids'' for $T_{bnt}(\rho_0)$.) $nt$ stands for natural transformations, and $b$ in $bnt$ comes from boundedness, so in the case of $bnt$ we require certain boundedness for homotopies. We use ``bornologies'' to define boundedness. A bornology is a structure on a set which gives a concept of bounded subsets (Definition \ref{asetbbor}). 

We have maps 
\begin{equation*}
T_{bnt}(\rho_0)\to T_{nt}(\rho_0)\twoheadrightarrow T_o(\rho_0)
\end{equation*}
and we will see $T_{bnt}(\rho_0)\simeq T_{nt}(\rho_0)$ if $N_0$ is compact (Proposition \ref{gmncptbnt}) since compactness implies boundedness, and $T_{nt}(\rho_0)\simeq T_o(\rho_0)$ if $\rho_0$ is locally free (Corollary \ref{tnteqto}). So all these deformation spaces coincide for a locally free action on a compact manifold. 

We will prove that if simply connectedness and locally freeness are assumed, the new and old deformation spaces coincide (Proposition \ref{gmsimiso}). 

\begin{intprop}
Let $\mca{G}\rightrightarrows M$ be an $s$-simply connected Lie groupoid, $N_0$ be a $C^\infty$ manifold, $\nu_0\colon N_0\to M$ be a surjective submersion and $\rho_0$ be a $C^\infty$ locally free right action of $\mca{G}$ on $\nu_0$. Then we have 
\begin{equation*}
T_o(\rho_0)\simeq\mca{A}_{LF}(\mca{O}_{\rho_0},\mca{G})/(\ul{\Diff}(\mca{O}_{\rho_0})_1\times\ul{\Aut}(\mca{G})). 
\end{equation*}
\end{intprop}

In the studies of the orbit-fixing deformation spaces, we have been assuming the locally freeness of the actions, the simply connectedness of the Lie groups and the compactness of the manifolds. These come mainly from the following proposition in \cite{A} (not stated explicitly) which generalize Proposition 2.3 of \cite{MM}. 

\begin{intprop}\label{simpoint}
Let $G$ be a simply connected Lie group, $N_0$ be a closed $C^\infty$ manifold and $\rho_0$ be a $C^\infty$ locally free right action of $G$ on $N_0$. Assume that for any $G$-valued $C^\infty$ cocycle $c\colon N_0\times G\to G$ over $\rho_0$, there exists $\Phi\in\End(G)$ such that $c$ is $C^\infty$ cohomologous to the constant cocycle $\Phi$. Then the deformation space $\mca{A}_{LF}(\mca{O}_{\rho_0},G)/(\Diff(\mca{O}_{\rho_0})_1\times\Aut(G))$ is a point. 
\end{intprop}

One purpose of this paper is generalizing Proposition \ref{simpoint} to the following theorem (Theorem \ref{cosp}). 

\begin{intthm}\label{gmlnchombbt}
Let $\mca{G}\rightrightarrows M$ be a Lie groupoid, $N_0$ be a $C^\infty$ manifold, $\nu_0\colon N_0\to M$ be a surjective submersion and $\rho_0$ be a $C^\infty$ right action of $\mca{G}$ on $\nu_0$. Assume that for any ``bibounded invertible $\mca{G}$-valued $C^\infty$ cocycle'' $c\in\ul{\Hom}^{bb}(N_0\rtimes_{\rho_0}\mca{G},\mca{G})^\times$, there exists $\Phi\in\ul{\End}(\mca{G})$ such that $c$ is boundedly $nt$-homotopic to the constant cocycle $\Phi\varphi_{\rho_0}$ with respect to the ``locally compact bornology'' $\mca{B}_\mca{G}$, ie $c\mathrel{\ul{\sim_{bnt}}}\Phi\varphi_{\rho_0}$. Then $T_{bnt}(\rho_0)$ is a point. 
\end{intthm}

We do not explain the terminology here but the assumptions corresponding to simply connectedness and locally freeness in Proposition \ref{simpoint} are removed, and the one corresponding to compactness is weakened to certain boundedness property expressed by $\mathrel{\ul{\sim_{bnt}}}$ and $\mca{B}_\mca{G}$ (Definition \ref{simbntbthe} and Definition \ref{loccombor}). So the generalization in this paper is focusing on the following four points: 
\begin{itemize}
\setlength\itemsep{0em}
\item actions of Lie groups which may not be simply connected
\item actions which may not be locally free
\item actions on possibly noncompact manifolds
\item actions of Lie groupoids instead of Lie groups (a Lie group is a Lie groupoid over a point). 
\end{itemize}

The first (and partially the second) is made possible by the use of the Teichm\"{u}ller space like formulation. The third is done by introducing bornologies on Lie groupoids. 

We deduce Theorem \ref{gmlnchombbt} from the following ``main theorem'' of this paper (Theorem \ref{ctop}). It says a semiconjugacy is in fact a conjugacy under certain conditions. 

\begin{intthm}[Semiconjugacy-to-Conjugacy Theorem for $bnt$]
Let $\mca{G}\rightrightarrows M$, $\mca{G}^\prime\rightrightarrows M^\prime$ be Lie groupoids, $N$, $N^\prime$ be $C^\infty$ manifolds, $\nu\colon N\to M$, $\nu^\prime\colon N^\prime\to M^\prime$ be $C^\infty$ maps, $\rho$ (resp. $\rho^\prime$) be a $C^\infty$ right action of $\mca{G}$ (resp. $\mca{G}^\prime$) on $\nu$ (resp. $\nu^\prime$) and 
\begin{equation*}
N\rtimes_\rho\mca{G}\overset{\varphi}{\underset{\varphi^\prime}{\rightrightarrows}}N^\prime\rtimes_{\rho^\prime}\mca{G}^\prime
\end{equation*}
be morphisms between the action groupoids. Assume $\varphi$ is an isomorphism and $\varphi^\prime$ is a semiconjugacy, ie there exists a morphism $\Phi\colon\mca{G}\to\mca{G}^\prime$ such that 
\begin{equation*}
\begin{tikzcd}
N\rtimes_\rho\mca{G}\ar[r,"\varphi^\prime"]\ar[d]\ar[rd,phantom,"\circlearrowright"{xshift=-5}]&N^\prime\rtimes_{\rho^\prime}\mca{G}^\prime\ar[d]\\
\mca{G}\ar[r,"\Phi"']&\mca{G}^\prime, 
\end{tikzcd}
\end{equation*}
where the vertical arrows are the second projections. Let $F_\Phi\colon M\to M^\prime$ be the base map of $\Phi$. If: 
\begin{itemize}
\setlength\itemsep{0em}
\item $\varphi\sim_{bnt}\varphi^\prime$ with respect to the locally compact bornology $\mca{B}_{N^\prime\rtimes_{\rho^\prime}\mca{G}^\prime}$
\item $F_\Phi\colon\nu(N)\to M^\prime$ is injective
\item $F_\Phi\colon\ul{O}\to M^\prime$ is an immersion for any orbit $\ul{O}$ of $\mca{G}$ such that $\ul{O}\subset\nu(N)$
\item $\nu^\prime$ is a surjective submersion, 
\end{itemize}
then $\varphi^\prime$ is a conjugacy. 
\end{intthm}

\subsubsection*{Outline of the paper}
In Section \ref{222} we review the definitions of parameter rigidity in \cite{MM} (Definition \ref{prmmgna}) and \cite{Ma} (Definition \ref{pralfdiffaut}). This is the case when the orbit-fixing deformation space is a point. So roughly speaking, a $C^\infty$ locally free action of a connected Lie group on a compact $C^\infty$ manifold is parameter rigid if any $C^\infty$ action of the Lie group with the same orbit foliation is $C^\infty$ conjugate to the original action. The definition in \cite{MM} requires the conjugacy is isotopic to the identity map through leaf-preserving diffeomorphisms, whereas the definition in \cite{Ma} requires the conjugacy is only homotopic to the identity map through leaf-preserving diffeomorphisms. Generally speaking, isotopies are more difficult to deal with than homotopies. And in fact with the definition in \cite{MM} there is a gap in the proof of Proposition 2.3 in \cite{MM}, which is essential to the theory. Namely we do not know how to show the conjugacy is isotopic to the identity, while the existence of a homotopy is obvious. For this reason I changed the definition of parameter rigidity in \cite{Ma}, which uses a homotopy instead of an isotopy. 

In Section \ref{333} we review basic things about Lie groupoids (Definition \ref{grptsgm1g2} and Definition \ref{liegrpgmtx}) and morphisms between them (Definition \ref{morgmgmpfpfp}). 

In Section \ref{444} we review the definitions of an action of a Lie groupoid (Definition \ref{actleftnrgyt}), an action morphism (Definition \ref{hnggnmactmor}) and an action groupoid (Definition \ref{actiongroupoid}). Conjugacies and semiconjugacies between actions are formulated as certain morphisms between corresponding action groupoids (Definition \ref{semiconconpng}). 

In Section \ref{999} we review the fact that the target fibers of a Lie groupoid have the structures of principal bundles (Proposition \ref{isoorb}). Here we deal with the case when the base manifold of a Lie groupoid may not be Hausdorff or second countable, though the proof of the fact remains the same. 

In Section \ref{555} we define orbitwise homotopies (Definition \ref{orbitwiseh} and Definition \ref{lgpcccffsimo}) and $nt$-homotopies (Definition \ref{nthomptppp} and Definition \ref{raragpppp}) between morphisms of Lie groupoids, which are used to define two of the three Teichm\"{u}ller spaces. To prove certain properties (Lemma \ref{coocooc} and Lemma \ref{hhscconor}) we review basic things on Stefan singular foliations (Definition \ref{stelmixlf}) and singular distributions (Definition \ref{sigdistmdm}). 

In Section \ref{666} after reviewing the definition of a bornology (Definition \ref{asetbbor}), we introduce the locally compact bornology on a topological groupoid with Hausdorff base space (Proposition \ref{loccombor}) and the action groupoid bornology of an action of a Lie groupoid (Definition \ref{acgrbong}). We define bounded $nt$-homotopies here (Definition \ref{hnhnpnib} and Definition \ref{simbntbthe}). 

In Section \ref{777} after reviewing coarse structures (Definition \ref{coarsestrdee}), we define a fiberwise coarse structure (Definition \ref{fibwcoaseee}) and show that there is a one-to-one correspondence between base-bounded compatible bornologies and fiberwise coarse structures on a topological groupoid (Proposition \ref{glfcsbbcb}). 

In Section \ref{888} we define three types of Teichm\"{u}ller spaces $T_o(\rho_0)$, $T_{nt}(\rho_0)$ and $T_{bnt}(\rho_0)$ of an action $\rho_0$ of a Lie groupoid (Definition \ref{gmngprmagl} and Definition \ref{Teistarsim}). We use a bornology to define $T_{bnt}(\rho_0)$. We introduce the mapping class group $\MCG_a(\rho_0)$, which acts on $T_a(\rho_0)$, for each $a\in\left\{o,nt,bnt\right\}$. 

In Section \ref{1010} we review the monodromy groupoids (Definition \ref{c0xfmon} and Proposition \ref{monts1lgm}) and foliation groupoids (Definition \ref{folgrgxxx}). Since the action groupoid of an action is a foliation groupoid if and only if the action is locally free (Example \ref{gmnmnrgyy}), this amounts to deal with locally freeness in terms of action groupoids. 

In Section \ref{1111} we define a comparison map between the Teichm\"{u}ller space $T_o(\rho_0)$ and the set $\mca{A}_{LF}(\mca{O}_{\rho_0},\mca{G})/(\ul{\Diff}(\mca{O}_{\rho_0})_1\times\ul{\Aut}(\mca{G}))$. We show the map is bijective if the Lie groupoid is $s$-simply connected (Proposition \ref{gmsimiso}). 

In Section \ref{1313} we prove that under certain conditions a semiconjugacy which is $nt$-homotopic to an isomorphism is in fact a conjugacy (Proposition \ref{btop}). As a corollary we see the Teichm\"{u}ller spaces $T_{bnt}(\rho_0)$ and $T_{nt}(\rho_0)$ are identified with certain quotient sets of some sets of ``invertible $\mca{G}$-valued cocycles'' (Proposition \ref{ampre}): 
\begin{align*}
T_{nt}(\rho_0)&\simeq\ul{\Aut}(\mca{G})\bs(\ul{\Hom}(N_0\rtimes_{\rho_0}\mca{G},\mca{G})^\times/\ul{\sim_{nt}}), \\
T_{bnt}(\rho_0)&\simeq\ul{\Aut}(\mca{G})\bs(\ul{\Hom}^{bb}(N_0\rtimes_{\rho_0}\mca{G},\mca{G})^\times/\ul{\sim_{bnt}}). 
\end{align*}

In Section \ref{1212} we prove that under certain conditions a semiconjugacy which is boundedly $nt$-homotopic to an isomorphism is in fact a conjugacy (Theorem \ref{ctop}). This is the main result of the paper. As a corollary we see that ``cocycle rigidity'' implies ``parameter rigidity'' in this general setting (Theorem \ref{cosp}). 

In Section \ref{1414} we compute the Teichm\"{u}ller space $T_o(\rho_0)$ of a transitive action of a Lie groupoid (Proposition \ref{trans}). 

In Section \ref{1717} we compute the Teichm\"{u}ller space $T_o(\rho_0)$ for the action $G\curvearrowleft G$ by right multiplication and for the action $\Gamma\bs S\curvearrowleft S$ by right multiplication, where $S$ is a simply connected solvable Lie group and $\Gamma$ is a lattice in $S$ (Proposition \ref{torautshgs} and its corollaries). 

In Section \ref{1818} we compute $T_o(\rho_\Gamma)$ for the action $\Gamma\bs\PSL(2,\bb{R})\stackrel{\rho_\Gamma}{\curvearrowleft}\PSL(2,\bb{R})$ by right multiplication, where $\Gamma$ is a torsion free cocompact lattice in $\PSL(2,\bb{R})$. As a result we obtain 
\begin{equation*}
T_o(\rho_\Gamma)\simeq T(\Sigma)\times\Hom(\wt{\Gamma},Z), 
\end{equation*}
where $T(\Sigma)$ is the Teichm\"{u}ller space of the surface $\Sigma=\Gamma\bs\PSL(2,\bb{R})/\PSO(2)$ and 
\begin{equation*}
1\to Z\to\wt{\Gamma}\to\Gamma\to1
\end{equation*}
is the central extension obtained from the universal covering $\wt{\PSL}(2,\bb{R})\to\PSL(2,\bb{R})$ (Theorem \ref{traghg}, Theorem \ref{authauthhom} and Proposition \ref{bijtsigauth}).

\subsubsection*{Notations and conventions}
\begin{itemize}
\setlength\itemsep{0em}
\item A topological space $X$ is simply connected if $X$ is path connected and $\pi_1(X,x)$ is trivial for some point $x\in X$. 
\item $C^\infty$ manifolds are assumed to be pure, Hausdorff, second countable and without boundary unless otherwise stated. 
\item For $C^\infty$ manifolds $M$ and $M^\prime$, $C^\infty(M,M^\prime)$ denotes the set of $C^\infty$ maps from $M$ to $M^\prime$, and $\Diff(M)$ denotes the group of $C^\infty$ diffeomorphisms of $M$. 
\item For a $C^\infty$ real vector bundle $E\to M$, $\Gamma(E)$ denotes the real vector space of $C^\infty$ sections of $E$. 
\item $\id$ denotes the identity map. 
\item $I$ often denotes the closed unit interval $[0,1]$. 
\item For Lie groupoids $\mca{G}\rightrightarrows M$, $\mca{G}^\prime\rightrightarrows M^\prime$ and a morphism $\varphi\colon\mca{G}\to\mca{G}^\prime$, $F_\varphi\colon M\to M^\prime$ denotes the base map of $\varphi$. 
\item For a Lie groupoid $\mca{G}\rightrightarrows M$, $\mca{O}_\mca{G}$ is the set of orbits of $\mca{G}$ and $\mca{F}_\mca{G}$ is the set of connected components of orbits of $\mca{G}$. 
\item For a Lie groupoid $\mca{G}\rightrightarrows M$, a $C^\infty$ manifold $N$, a $C^\infty$ map $\nu\colon N\to M$ and an action $\rho\in\mca{A}(\nu,\mca{G})$, $\varphi_\rho\colon N\rtimes_\rho\mca{G}\to\mca{G}$ denotes the second projection, which we call the associated action morphism of $\rho$. We write $\mca{O}_\rho=\mca{O}_{N\rtimes_\rho\mca{G}}$ and $\mca{F}_\rho=\mca{F}_{N\rtimes_\rho\mca{G}}$. 
\item See Definition \ref{gxgxpxx} for $\iota_x\colon\pi_1(\Gamma\bs X,\Gamma x)\to\Gamma$. 
\item See Notation \ref{gh2agh2yb} for $\ol{y}$, $\ol{i}$ and so on. 
\end{itemize}

\subsubsection*{Acknowledgement}
In 2022, I was asked to give informal lectures on parameter rigidity at Chiba University by Hiroki Matui. During the preparation, I tried using Lie groupoids, and it finally resulted in this paper. I would like to thank Hiroki Matui and those who attended my talks, including Masahiro Futaki, Akihiko Arai, Azuna Nishida, Takehiko Mori, Kaoru Kusano and Hiroshi Ando. 

I also thank Mitsunobu Tsutaya and Yu Nishimura for the discussion on this paper in online seminars held almost every two weeks during the development of this paper. The idea of comparing the diffeomorphism group with the homeomorphism group in the proof of Lemma \ref{sigconorihthe} is due to Mitsunobu Tsutaya. 

I thank Masayuki Asaoka for pointing out an error in the proof of a lemma about $\Diff(\mca{F})$, which is not contained in the current version of the paper.

\section{Two different definitions of parameter rigidity}\label{222}
\subsection{The definition in \cite{MM}}
\begin{dfn}
Let $G$ be a Lie group and $N$ be a $C^\infty$ manifold. Let 
\begin{equation*}
\mca{A}(N,G)=\left\{\text{$C^\infty$ right actions of $G$ on $N$}\right\}. 
\end{equation*}
For $\rho\in\mca{A}(N,G)$, $F\in\Diff(N)$ and $\Phi\in\Aut(G)$, define $\rho^{(F,\Phi)}\in\mca{A}(N,G)$ by $\rho^{(F,\Phi)}(y,g)=F^{-1}\rho(F(y),\Phi(g))$ for $y\in N$ and $g\in G$. Then we have an action $\mca{A}(N,G)\curvearrowleft\Diff(N)\times\Aut(G)$ defined by $(\rho,(F,\Phi))\mapsto\rho^{(F,\Phi)}$. 
\end{dfn}

\begin{dfn}
Let $G$ be a Lie group and $N$ be a $C^\infty$ manifold. Let 
\begin{equation*}
\mca{A}_{LF}(N,G)=\left\{\rho\in\mca{A}(N,G)\ \middle|\ \text{$\rho$ is locally free}\right\}. 
\end{equation*}
When $G$ is connected, let $\mca{F}_\rho$ be the orbit foliation of $\rho\in\mca{A}_{LF}(N,G)$. 
\end{dfn}

\begin{dfn}
Let $N$ be a $C^\infty$ manifold and $\mca{F}$ be a $C^\infty$ foliation of $N$. Let 
\begin{equation*}
\Diff(\mca{F})=\left\{F\in\Diff(N)\ \middle|\ F(L)\subset L\text{ for any }L\in\mca{F}\right\}. 
\end{equation*}
It is known that $\Diff(N)$, equipped with the weak $C^\infty$ topology (aka the $C^\infty$ compact-open topology or the weak Whitney topology) of $C^\infty(N,N)$, is a topological group. So $\Diff(\mca{F})$ is also a topological group, and let $\Diff(\mca{F})_0$ be its identity path component. 
\end{dfn}

\begin{rem}
We do not know whether the path components of $\Diff(\mca{F})$ coincide with the connected components of $\Diff(\mca{F})$. (The path components of $\Diff(N)$ coincide with the connected components of $\Diff(N)$.) 
\end{rem}

\begin{dfn}
Let $G$ be a connected Lie group, $N$ be a $C^\infty$ manifold and $\mca{F}$ be a $C^\infty$ foliation of $N$. Let 
\begin{equation*}
\mca{A}_{LF}(\mca{F},G)=\left\{\rho\in\mca{A}_{LF}(N,G)\ \middle|\ \mca{F}_\rho=\mca{F}\right\}. 
\end{equation*}
We have an action $\mca{A}_{LF}(\mca{F},G)\curvearrowleft\Diff(\mca{F})\times\Aut(G)$. 
\end{dfn}

Let $I=[0,1]$. 

\begin{lem}
Let $N$ be a compact $C^\infty$ manifold, $\mca{F}$ be a $C^\infty$ foliation of $N$ and $F$, $F^\prime\in\Diff(\mca{F})$. Then the following are equivalent: 
\begin{enumerate}
\item There exists a continuous map $c\colon I\to\Diff(\mca{F})$ such that $c(0)=F$ and $c(1)=F^\prime$, where $\Diff(\mca{F})$ is equipped with the weak $C^\infty$ topology of $C^\infty(N,N)$. 
\item There exists a continuous map $H\colon N\times I\to N$ such that: 
\begin{itemize}
\item $H(\cdot,0)=F$ and $H(\cdot,1)=F^\prime$
\item $H(\cdot,t)\in\Diff(\mca{F})$ for all $t\in I$. 
\end{itemize}
\item There exists a $C^\infty$ map $H\colon N\times I\to N$ such that: 
\begin{itemize}
\item $H(\cdot,0)=F$ and $H(\cdot,1)=F^\prime$
\item $H(\cdot,t)\in\Diff(\mca{F})$ for all $t\in I$. 
\end{itemize}
\end{enumerate}
\end{lem}

\begin{proof}
$3\Rightarrow1$ and $1\Rightarrow2$ are OK. We prove $2\Rightarrow3$. Let $H\colon N\times I\to N$ be a continuous map such that: 
\begin{itemize}
\setlength\itemsep{0em}
\item $H(\cdot,0)=F$ and $H(\cdot,1)=F^\prime$
\item $H(\cdot,t)\in\Diff(\mca{F})$ for all $t\in I$. 
\end{itemize}
Let 
\begin{equation*}
N\times_\mca{F}N=\left\{(y,y^\prime)\in N\times N\ \middle|\ \text{$y$ and $y^\prime$ are in the same leaf of $\mca{F}$}\right\}. 
\end{equation*}
Take a $C^\infty$ Riemannian metric $g$ of $p\colon T\mca{F}\to N$. Let $\exp\colon T\mca{F}\to N$ be the exponential map 
and $d_\mca{F}\colon N\times_\mca{F}N\to\bb{R}$ be the metric defined by $g$. For $\epsilon>0$, let 
\begin{equation*}
U_\epsilon=\left\{v\in T\mca{F}\ \middle|\ \lVert v\rVert<\epsilon\right\}
\end{equation*}
and 
\begin{equation*}
V_\epsilon=\left\{(y,y^\prime)\in N\times_\mca{F}N\ \middle|\ d_\mca{F}(y,y^\prime)<\epsilon\right\}. 
\end{equation*}
Since $N$ is compact, there exists $\epsilon>0$ such that: 
\begin{itemize}
\setlength\itemsep{0em}
\item $\exp\xi\in\Diff(\mca{F})$ for any $\xi\in\Gamma(T\mca{F})$ such that $\lVert\xi\rVert<\epsilon$
\item $V_\epsilon$ is an embedded submanifold of $N\times N$ and 
\begin{align*}
E\colon U_\epsilon&\to V_\epsilon\\
v&\mapsto(pv,\exp v)
\end{align*}
is a diffeomorphism. 
\end{itemize}
Since $N$ is compact, there exist $0=t_0<t_1<\cdots<t_k=1$ such that 
\begin{equation*}
d_\mca{F}(H(\cdot,t_i),H(\cdot,t_{i+1}))<\epsilon
\end{equation*}
for all $i=0,\ldots,k-1$. Since the map 
\begin{align*}
N&\to V_\epsilon\\
y&\mapsto(H(y,t_i),H(y,t_{i+1}))
\end{align*}
is $C^\infty$, $\xi_i=E^{-1}(H(\cdot,t_i),H(\cdot,t_{i+1}))\colon N\to T\mca{F}$ is $C^\infty$. (Note that $\xi_i(y)\in T_{H(y,t_i)}\mca{F}$, so $\xi_i$ may not be a section of $T\mca{F}$.) We have 
\begin{equation*}
\exp_{H(y,t_i)}(\xi_i(y))=H(y,t_{i+1}). 
\end{equation*}
For each $i=0,\ldots,k-1$, take a $C^\infty$ function $\theta_i\colon[t_i,t_{i+1}]\to\bb{R}$ such that: 
\begin{itemize}
\setlength\itemsep{0em}
\item $0\leq\theta_i\leq1$
\item $\theta_i=0$ on a neighborhood of $t_i$
\item $\theta_i=1$ on a neighborhood of $t_{i+1}$. 
\end{itemize}
Define $H^\prime\colon N\times I\to N$ by 
\begin{equation*}
H^\prime(y,t)=\exp_{H(y,t_i)}\left(\theta_i(t)\xi_i(y)\right)
\end{equation*}
for $t_i\leq t\leq t_{i+1}$. $H^\prime$ is $C^\infty$. We have 
\begin{equation*}
H^\prime(\cdot,t)=\left(\exp_{H(\cdot,t_i)}\left(\theta_i(t)\xi_i\circ H(\cdot,t_i)^{-1}\right)\right)\circ H(\cdot,t_i)
\end{equation*}
for $t_i\leq t\leq t_{i+1}$. Since $\theta_i(t)\xi_i\circ H(\cdot,t_i)^{-1}\in\Gamma(T\mca{F})$ and $\lVert\theta_i(t)\xi_i\circ H(\cdot,t_i)^{-1}\rVert<\epsilon$, $H^\prime(\cdot,t)\in\Diff(\mca{F})$ for all $t\in I$. 
\end{proof}

\begin{rem}
We do not know whether compactness of $N$ can be dropped in the above lemma. 
\end{rem}

\begin{cor}
Let $N$ be a compact $C^\infty$ manifold and $\mca{F}$ be a $C^\infty$ foliation of $N$. Then 
\begin{align*}
\Diff(\mca{F})_0&=\left\{F\in\Diff(\mca{F})\ \middle|\ \text{$F$ is $C^0$ homotopic to $\id$ through maps in $\Diff(\mca{F})$}\right\}\\
&=\left\{F\in\Diff(\mca{F})\ \middle|\ \text{$F$ is $C^\infty$ homotopic to $\id$ through maps in $\Diff(\mca{F})$}\right\}. 
\end{align*}
\end{cor}

\begin{dfn}[The definition of parameter rigidity in \cite{MM} (and \cite{A})]\label{prmmgna}
Let $G$ be a connected Lie group, $N$ be a compact $C^\infty$ manifold and $\rho\in\mca{A}_{LF}(N,G)$. We say that $\rho$ is \emph{parameter rigid} if $\mca{A}_{LF}(\mca{F}_\rho,G)/(\Diff(\mca{F}_\rho)_0\times\Aut(G))$ is a singleton. 
\end{dfn}

With this definition of parameter rigidity, there is a gap in the proof of Proposition 2.3 in \cite{MM}, which I expect cannot be fixed. Proposition 2.3 in \cite{MM} says cocycle rigidity implies parameter rigidity, which is a fundamental result. To make Proposition 2.3 in \cite{MM} hold, I changed the definition of parameter rigidity in \cite{Ma}, \cite{Ma2}, \cite{Ma3}, to what we explain in the following subsection.

\subsection{The definition in \cite{Ma}}
\begin{dfn}
Let $N$ be a $C^\infty$ manifold and $\mca{F}$ be a $C^\infty$ foliation of $N$. Let 
\begin{equation*}
C^\infty(\mca{F})=\left\{F\in C^\infty(N,N)\ \middle|\ \text{$F(L)\subset L$ for all $L\in\mca{F}$}\right\}. 
\end{equation*}
\end{dfn}

\begin{dfn}
Let $M$, $M^\prime$ be $C^\infty$ manifolds, $\mca{F}$, $\mca{F}^\prime$ be $C^\infty$ foliations of $M$, $M^\prime$ and $r\in\bb{Z}_{\geq0}\cup\{\infty\}$. Let 
\begin{equation*}
C^r(\mca{F},\mca{F}^\prime)=\left\{F\in C^r(M,M^\prime)\ \middle|\ 
\begin{gathered}
\text{for all $L\in\mca{F}$, there exists}\\
\text{$L^\prime\in\mca{F}^\prime$ such that $F(L)\subset L^\prime$}
\end{gathered}
\right\}
\end{equation*}
and 
\begin{equation*}
C^{0,\infty}(\mca{F},\mca{F}^\prime)=\left\{F\in C^0(\mca{F},\mca{F}^\prime)\ \middle|\ 
\begin{gathered}
\text{for any foliation charts $U$, $U^\prime=L^\prime\times T^\prime$ of $\mca{F}$, $\mca{F}^\prime$}\\
\text{such that $F(U)\subset U^\prime$, $p_{T^\prime}F\colon U\to T^\prime$ is $C^\infty$}
\end{gathered}
\right\}. 
\end{equation*}
\end{dfn}

\begin{prop}[A consequence of leafwise approximation theorem (Theorem 4.3) in \cite{ALK}]\label{latffmm}
Let $M$, $M^\prime$ be $C^\infty$ manifolds and $\mca{F}$ (resp. $\mca{F}^\prime$) be a $C^\infty$ foliation of $M$ (resp. $M^\prime$). Let $E$ be a closed subset of $M$ and $F\in C^{0,\infty}(\mca{F},\mca{F}^\prime)$ be such that $F$ is $C^\infty$ on an open neighborhood of $E$ in $M$. Then there exists $F^\prime\in C^\infty(\mca{F},\mca{F}^\prime)$ such that $F=F^\prime$ on an open neighborhood of $E$ in $M$. 
\end{prop}

\begin{rem}
In the above proposition, $F^\prime$ can be chosen arbitrarily close to $F$ in the strong plaquewise topology of $C^{0,\infty}(\mca{F},\mca{F}^\prime)$ defined in \cite{ALK}. 
\end{rem}

\begin{rem}
$C^{0,\infty}(\mca{F},\mca{F}^\prime)$ is denoted by $C^{\infty,0}(\mca{F},\mca{F}^\prime)$ in \cite{ALK}. (We changed the notation because our notation looks more natural to us.) 
\end{rem}

\begin{lem}
Let $N$ be a $C^\infty$ manifold, $\mca{F}$ be a $C^\infty$ foliation of $N$ and $F$, $F^\prime\in C^\infty(\mca{F})$. Then the following are equivalent: 
\begin{enumerate}
\item There exists a continuous map $c\colon I\to C^\infty(\mca{F})$ such that $c(0)=F$ and $c(1)=F^\prime$, where $C^\infty(\mca{F})$ is equipped with the weak $C^\infty$ topology of $C^\infty(N,N)$. 
\item $F$ is $C^0$ homotopic to $F^\prime$ through maps in $C^\infty(\mca{F})$. 
\item $F$ is $C^\infty$ homotopic to $F^\prime$ through maps in $C^\infty(\mca{F})$. 
\end{enumerate}
\end{lem}

\begin{proof}
$3\Rightarrow1$ and $1\Rightarrow2$ are OK. We prove $2\Rightarrow3$. There exists a continuous map $H\colon N\times\bb{R}\to N$ such that: 
\begin{itemize}
\setlength\itemsep{0em}
\item $H(\cdot,t)=F$ for all $t<0.1$
\item $H(\cdot,t)=F^\prime$ for all $t>0.9$
\item $H(\cdot,t)\in C^\infty(\mca{F})$ for all $t\in\bb{R}$. 
\end{itemize}
Let $\mca{F}\times\bb{R}=\left\{L\times\bb{R}\ \middle|\ L\in\mca{F}\right\}$. This is a $C^\infty$ foliation of $N\times\bb{R}$ and $H\in C^{0,\infty}(\mca{F}\times\bb{R},\mca{F})$. Applying Proposition \ref{latffmm} for $E=(N\times(-\infty,0])\cup(N\times[1,\infty))$, there exists $H^\prime\in C^\infty(\mca{F}\times\bb{R},\mca{F})$ such that $H^\prime(\cdot,t)=H(\cdot,t)$ for $t\leq0$ or $t\geq1$. So $H^\prime\colon N\times\bb{R}\to N$ is a $C^\infty$ map such that: 
\begin{itemize}
\setlength\itemsep{0em}
\item $H^\prime(L\times\bb{R})\subset L$ for all $L\in\mca{F}$
\item $H^\prime(\cdot,t)=F$ for all $t\leq0$
\item $H^\prime(\cdot,t)=F^\prime$ for all $t\geq1$
\item $H^\prime(\cdot,t)\in C^\infty(\mca{F})$ for all $t\in\bb{R}$. \qedhere
\end{itemize}
\end{proof}

\begin{dfn}
Let $N$ be a $C^\infty$ manifold and $\mca{F}$ be a $C^\infty$ foliation of $N$. Define 
\begin{align*}
\Diff(\mca{F})_1&=\left\{F\in\Diff(\mca{F})\ \middle|\ \text{$F$ is $C^0$ homotopic to $\id$ through maps in $C^\infty(\mca{F})$}\right\}\\
&=\left\{F\in\Diff(\mca{F})\ \middle|\ \text{$F$ is $C^\infty$ homotopic to $\id$ through maps in $C^\infty(\mca{F})$}\right\}. 
\end{align*}
Let $C^\infty(\mca{F})_0$ be the path component of $C^\infty(\mca{F})$ equipped with the weak $C^\infty$ topology of $C^\infty(N,N)$ containing $\id$. Then $\Diff(\mca{F})_1=C^\infty(\mca{F})_0\cap\Diff(\mca{F})$ and $\Diff(\mca{F})_0\subset\Diff(\mca{F})_1$. 
\end{dfn}

\begin{dfn}[The definition of parameter rigidity in \cite{Ma}, \cite{Ma2}, \cite{Ma3}]\label{pralfdiffaut}
Let $G$ be a connected Lie group, $N$ be a $C^\infty$ manifold and $\rho\in\mca{A}_{LF}(N,G)$. We say that $\rho$ is \emph{parameter rigid} if $\mca{A}_{LF}(\mca{F}_\rho,G)/(\Diff(\mca{F}_\rho)_1\times\Aut(G))$ is a singleton. 
\end{dfn}

With this definition of parameter rigidity, Proposition 2.3 in \cite{MM} is correct. We will use this definition. Note that we have a map 
\begin{equation*}
\mca{A}_{LF}(\mca{F}_\rho,G)/(\Diff(\mca{F}_\rho)_0\times\Aut(G))\twoheadrightarrow\mca{A}_{LF}(\mca{F}_\rho,G)/(\Diff(\mca{F}_\rho)_1\times\Aut(G)). 
\end{equation*}

\section{Lie groupoids}\label{333}
\subsection{Lie groupoids}
\begin{lem}\label{fiberproduct}
Let $N_1$, $N_2$, $M$ be $C^\infty$ manifolds (perhaps non Hausdorff, non second countable) and $f_i\colon N_i\to M$ be a $C^\infty$ map for each $i=1$, $2$. Let 
\begin{equation*}
N_1\times_{f_1,f_2}N_2=\left\{(y_1,y_2)\in N_1\times N_2\ \middle|\ f_1(y_1)=f_2(y_2)\right\}. 
\end{equation*}
We have 
\begin{equation*}
\begin{tikzcd}
N_1\times_{f_1,f_2}N_2\ar[r,"p_2"]\ar[d,"p_1"']\ar[rd,phantom,"\circlearrowright"]&N_2\ar[d,"f_2"]\\
N_1\ar[r,"f_1"']&M, 
\end{tikzcd}
\end{equation*}
where $p_1$, $p_2$ are the projections. Then the following statements hold: 
\begin{enumerate}
\item If $M$ is Hausdorff, then $N_1\times_{f_1,f_2}N_2$ is closed in $N_1\times N_2$. 
\item If $f_2$ is surjective, then $p_1$ is surjective. 
\item If $f_2$ is a submersion, then: 
\begin{enumerate}
\item $N_1\times_{f_1,f_2}N_2$ is an embedded submanifold of $N_1\times N_2$. 
\item $T_{(y_1,y_2)}(N_1\times_{f_1,f_2}N_2)=\left\{(Y_1,Y_2)\in T_{y_1}N_1\times T_{y_2}N_2\ \middle|\ (f_1)_*(Y_1)=(f_2)_*(Y_2)\right\}$ for any $(y_1,y_2)\in N_1\times_{f_1,f_2}N_2$. 
\item $p_1$ is a submersion. 
\end{enumerate}
\end{enumerate}
\end{lem}

\begin{proof}
1. The map $f_1\times f_2\colon N_1\times N_2\to M\times M$ is continuous and $N_1\times_{f_1,f_2}N_2=(f_1\times f_2)^{-1}(\Delta_M)$, where $\Delta_M$ is the diagonal set of $M\times M$. Recall that $\Delta_M$ is closed in $M\times M$ if and only if $M$ is Hausdorff. So $N_1\times_{f_1,f_2}N_2$ is closed in $N_1\times N_2$. 

2. OK. 

3. (a) For $(y_1,y_2)\in N_1\times_{f_1,f_2}N_2$, the composition 
\begin{equation*}
T_{y_1}N_1\times T_{y_2}N_2\xrightarrow{(f_1\times f_2)_*}T_{f_1(y_1)}M\times T_{f_2(y_2)}M\to(T_{f_1(y_1)}M\times T_{f_2(y_2)}M)/\Delta_{T_{f_1(y_1)}M}
\end{equation*}
is surjective since $f_2$ is a submersion. $f_1\times f_2$ is transverse to $\Delta_M$. Hence $N_1\times_{f_1,f_2}N_2$ is an embedded submanifold of $N_1\times N_2$. 

(b) The inclusion $\subset$ is OK. This is $=$ by dimension counting. Both spaces have dimension $\dim N_1+\dim N_2-\dim M$. 

(c) We have 
\begin{align*}
(p_1)_*\colon T_{(y_1,y_2)}(N_1\times_{f_1,f_2}N_2)&\to T_{y_1}N_1\\
(Y_1,Y_2)&\mapsto Y_1. 
\end{align*}
For any $Y_1\in T_{y_1}N_1$, there exists $Y_2\in T_{y_2}N_2$ such that $(f_1)_*(Y_1)=(f_2)_*(Y_2)$. Hence $(p_1)_*$ is surjective. 
\end{proof}

\begin{dfn}\label{grptsgm1g2}
We say that $\mca{G}\rightrightarrows M$ is a \emph{groupoid} if $\mca{G}$, $M$ are sets and there are maps 
\begin{equation}\label{structmaps}
t,s\colon\mca{G}\to M,\qquad
\begin{aligned}
1\colon M&\to\mca{G}\qquad&\mca{G}&\to\mca{G}\\
x&\mapsto1_x, &g&\mapsto g^{-1}, 
\end{aligned}
\end{equation}
\begin{align}\label{composition}
\begin{split}
\mca{G}^{(2)}&\to\mca{G}\\
(g_1,g_2)&\mapsto g_1g_2, 
\end{split}
\end{align}
where 
\begin{equation*}
\mca{G}^{(2)}=\mca{G}\times_{s,t}\mca{G}=\left\{(g_1,g_2)\in\mca{G}^2\ \middle|\ s(g_1)=t(g_2)\right\}, 
\end{equation*}
such that: 
\begin{itemize}
\setlength\itemsep{0em}
\item $t(g_1g_2)=t(g_1)$, $s(g_1g_2)=s(g_2)$ for all $(g_1,g_2)\in\mca{G}^{(2)}$
\item $(g_1g_2)g_3=g_1(g_2g_3)$ for all $(g_1,g_2,g_3)\in\mca{G}^{(3)}$, where 
\begin{equation*}
\mca{G}^{(3)}=\left\{(g_1,g_2,g_3)\in\mca{G}^3\ \middle|\ (g_1,g_2),(g_2,g_3)\in\mca{G}^{(2)}\right\}
\end{equation*}
\item $t(1_x)=x=s(1_x)$ for all $x\in M$
\item $1_{t(g)}g=g=g1_{s(g)}$ for all $g\in\mca{G}$
\item $t(g^{-1})=s(g)$, $s(g^{-1})=t(g)$ for all $g\in\mca{G}$
\item $g^{-1}g=1_{s(g)}$, $gg^{-1}=1_{t(g)}$ for all $g\in\mca{G}$. 
\end{itemize}

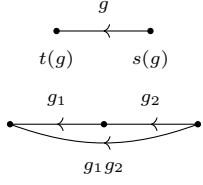
\begin{equation*}
\begin{tikzpicture}[every label/.append style={font=\scriptsize},matrix of math nodes,decoration={markings,mark=at position0.5with{\arrow{>}}}]
\node(1)[label=below:t(g)]{};
\node(2)[right=of 1,label=below:s(g)]{};
\draw[postaction={decorate}](2.center)--node[label=above:g]{}(1.center);
\foreach\x in{1,2}\filldraw(\x)circle(1pt);
\end{tikzpicture}
\end{equation*}
\begin{equation*}
\begin{tikzpicture}[every label/.append style={font=\scriptsize},matrix of math nodes,decoration={markings,mark=at position0.5with{\arrow{>}}}]
\node(1){};
\node(2)[right=of 1]{};
\node(3)[right=of 2]{};
\draw[postaction={decorate}](3.center)--node[label=above:g_2]{}(2.center);
\draw[postaction={decorate}](2.center)--node[label=above:g_1]{}(1.center);
\draw[postaction={decorate}](3.center)to[out=200,in=340]node[label=below:g_1g_2]{}(1.center);
\foreach\x in{1,2,3}\filldraw(\x)circle(1pt);
\end{tikzpicture}
\end{equation*}
\end{dfn}

\begin{notation}
Let $\mca{G}\rightrightarrows M$ be a groupoid. We put 
\begin{equation*}
\mca{G}^{x^\prime}=t^{-1}(x^\prime),\quad\mca{G}_x=s^{-1}(x),\quad\mca{G}^{x^\prime}_x=\mca{G}^{x^\prime}\cap\mca{G}_x
\end{equation*}
for $x$, $x^\prime\in M$, 
\begin{equation*}
\mca{G}^{S^\prime}=t^{-1}(S^\prime),\quad\mca{G}_S=s^{-1}(S),\quad\mca{G}^{S^\prime}_S=\mca{G}^{S^\prime}\cap\mca{G}_S,\quad1_S=\left\{1_x\ \middle|\ x\in S\right\}
\end{equation*}
for $S$, $S^\prime\subset M$. 
\end{notation}

\begin{dfn}
A groupoid $\mca{G}\rightrightarrows M$ is a \emph{topological groupoid} if $\mca{G}$, $M$ are topological spaces and the maps \eqref{structmaps}, \eqref{composition} are continuous, where $\mca{G}^{(2)}$ is equipped with the subspace topology of the product topology of $\mca{G}^2$. 
\end{dfn}

\begin{dfn}\label{liegrpgmtx}
A groupoid $\mca{G}\rightrightarrows M$ is a \emph{Lie groupoid} if: 
\begin{itemize}
\setlength\itemsep{0em}
\item $\mca{G}$ is a $C^\infty$ manifold (perhaps non Hausdorff, non second countable)
\item $M$ is a Hausdorff, second countable $C^\infty$ manifold
\item the maps \eqref{structmaps} are $C^\infty$
\item $t$ is a submersion (hence $\mca{G}^{(2)}$ is an embedded submanifold of $\mca{G}^2$ by Lemma \ref{fiberproduct})
\item the map \eqref{composition} is $C^\infty$
\item $\mca{G}^x$ is Hausdorff for any $x\in M$. 
\end{itemize}
\end{dfn}

\begin{lem}
For a Lie groupoid $\mca{G}\rightrightarrows M$ the following statements hold: 
\begin{enumerate}
\item The map 
\begin{align*}
1\colon M&\to\mca{G}\\
x&\mapsto1_x
\end{align*}
is an embedding (but the image is not necessarily closed). 
\item If $\mca{G}$ is Hausdorff, then $1_M$ is closed in $\mca{G}$. 
\item $\mca{G}\to\mca{G}$, $g\mapsto g^{-1}$ is an involutive diffeomorphism. 
\item $s\colon\mca{G}\to M$ is a submersion. 
\item $\mca{G}_x$, $\mca{G}^x$ are Hausdorff closed embedded submanifolds of $\mca{G}$ for any $x\in M$. 
\end{enumerate}
\end{lem}

\begin{proof}
1. $1\colon M\to\mca{G}$ is an injective immersion since $t1=\id$. $1\colon M\to1_M$ is a homeomorphism since $t\colon1_M\to M$ is the continuous inverse. 

2. Let $g\in\ol{1_M}$. There exists a net $(x_a)_{a\in A}$ in $M$ such that $1_{x_a}\to g$. Take a compact neighborhood $K$ of $g$ in $\mca{G}$. There exists $a_0\in A$ such that $1_{x_a}\in K$ for $a\geq a_0$. Hence $x_a\in t(K)$ for $a\geq a_0$. The subset $1_{t(K)}$ is closed in $\mca{G}$ since $1_{t(K)}$ is compact and $\mca{G}$ is Hausdorff. Hence $g=\lim_{\substack{a\to\infty\\a\geq a_0}}1_{x_a}\in1_{t(K)}\subset1_M$. So $1_M$ is closed in $\mca{G}$. 

3. 4. and 5. are easy. 
\end{proof}

\begin{dfn}
Let $\mca{G}\rightrightarrows M$ be a Lie groupoid. For $x$, $x^\prime\in M$, write $x\sim x^\prime$ if $\mca{G}^x_{x^\prime}\neq\emptyset$. Then $\sim$ is an equivalence relation on $M$ and an equivalence class modulo $\sim$ is called an \emph{orbit} of $\mca{G}$. The set of all orbits of $\mca{G}$ is denoted by $\mca{O}_\mca{G}$. 
\end{dfn}

\begin{example}
Let $\mca{G}\rightrightarrows M$ be a groupoid and $S$ be a subset of $M$ which is a union of orbits of $\mca{G}$. Let $\mca{G}|_S=t^{-1}(S)=s^{-1}(S)$. Then $\mca{G}|_S\rightrightarrows S$ is a groupoid. 
\end{example}

\begin{example}
A Lie groupoid $G\rightrightarrows\mr{pt}$ over a point is a (perhaps non second countable) Lie group $G$. 
\end{example}

\begin{example}
For a $C^\infty$ manifold $M$, $M\rightrightarrows M$ is a Lie groupoid called the \emph{base groupoid}, where $t,s,1=\id\colon M\to M$. 
\end{example}

\begin{example}
For a $C^\infty$ manifold $M$, let $\Pair(M)=M\times M$. Then $\Pair(M)\rightrightarrows M$ is a Lie groupoid, which is called the \emph{pair groupoid} of $M$, with the following structure: 
\begin{align*}
t\colon\Pair(M)&\to M&s\colon\Pair(M)&\to M\\
(x,x^\prime)&\mapsto x,&(x,x^\prime)&\mapsto x^\prime, 
\end{align*}
\begin{align*}
\Pair(M)\times_{s,t}\Pair(M)&\to\Pair(M)\\
((x,x^\prime),(x^\prime,x^{\prime\prime}))&\mapsto(x,x^{\prime\prime}), 
\end{align*}
\begin{align*}
1\colon M&\to\Pair(M)&\Pair(M)&\to\Pair(M)\\
x&\mapsto(x,x),&(x,x^\prime)&\mapsto(x^\prime,x). 
\end{align*}
\end{example}

\begin{example}\label{mapgroupoid}
Let $\mca{H}^\prime\rightrightarrows N^\prime$ be a Lie groupoid and $N$ be a $C^\infty$ manifold. Then $C^\infty(N,\mca{H}^\prime)\rightrightarrows C^\infty(N,N^\prime)$ is a groupoid with the following structure: 
\begin{align*}
t^\prime\colon C^\infty(N,\mca{H}^\prime)&\to C^\infty(N,N^\prime)&s^\prime\colon C^\infty(N,\mca{H}^\prime)&\to C^\infty(N,N^\prime)\\
P&\mapsto t^\prime P,&P&\mapsto s^\prime P,
\end{align*}
\begin{align*}
C^\infty(N,\mca{H}^\prime)\times_{s^\prime,t^\prime}C^\infty(N,\mca{H}^\prime)&\to C^\infty(N,\mca{H}^\prime)\\
(P,P^\prime)&\mapsto PP^\prime, 
\end{align*}
\begin{align*}
1\colon C^\infty(N,N^\prime)&\to C^\infty(N,\mca{H}^\prime)&C^\infty(N,\mca{H}^\prime)&\to C^\infty(N,\mca{H}^\prime)\\
F&\mapsto 1_F,&P&\mapsto P^{-1}, 
\end{align*}
where $(PP^\prime)(y)=P(y)P^\prime(y)$, $1_F(y)=1_{F(y)}$, $P^{-1}(y)=P(y)^{-1}$ for $y\in N$. 
\end{example}

\begin{example}
For Lie groupoids $\mca{G}_1\rightrightarrows M_1$, $\mca{G}_2\rightrightarrows M_2$, $\mca{G}_1\times\mca{G}_2\rightrightarrows M_1\times M_2$ is a Lie groupoid with the following structure: $t=t_1\times t_2$, $s=s_1\times s_2$, $(g_1,g_2)(g_1^\prime,g_2^\prime)=(g_1g_1^\prime,g_2g_2^\prime)$, $1_{(x_1,x_2)}=(1_{x_1},1_{x_2})$, $(g_1,g_2)^{-1}=(g_1^{-1},g_2^{-1})$. 
\end{example}

\subsection{Morphisms of Lie groupoids}
\begin{dfn}\label{morgmgmpfpfp}
Let $\mca{G}\rightrightarrows M$ and $\mca{G}^\prime\rightrightarrows M^\prime$ be Lie groupoids. A $C^\infty$ map $\varphi\colon\mca{G}\to\mca{G}^\prime$ is a \emph{morphism} if there exists a $C^\infty$ map $F_\varphi\colon M\to M^\prime$ such that 
\begin{equation*}
\begin{tikzcd}
\mca{G}\ar[r,"\varphi"]\ar[d,"t"']\ar[rd,phantom,"\circlearrowright"{xshift=2,yshift=-2}]&\mca{G}^\prime\ar[d,"t"]\\
M\ar[r,"F_\varphi"']&M^\prime, 
\end{tikzcd}\qquad
\begin{tikzcd}
\mca{G}\ar[r,"\varphi"]\ar[d,"s"']\ar[rd,phantom,"\circlearrowright"{xshift=1,yshift=-2}]&\mca{G}^\prime\ar[d,"s"]\\
M\ar[r,"F_\varphi"']&M^\prime
\end{tikzcd}
\end{equation*}
and $\varphi(g_1g_2)=\varphi(g_1)\varphi(g_2)$ for all $(g_1,g_2)\in\mca{G}^{(2)}$. 
\begin{equation*}
\begin{tikzpicture}[every label/.append style={font=\scriptsize},matrix of math nodes,decoration={markings,mark=at position0.5with{\arrow{>}}}]
\node(1)[label=below:t(g)]{};
\node(2)[right=of 1,label=below:s(g)]{};
\draw[postaction={decorate}](2.center)--node[label=above:g]{}(1.center);
\node(3)[right=6em of 2,label=below:F_\varphi(t(g))]{};
\node(4)[right=of 3,label=below:F_\varphi(s(g))]{};
\draw[postaction={decorate}](4.center)--node[label=above:\varphi(g)]{}(3.center);
\foreach\x in{1,...,4}\filldraw(\x)circle(1pt);
\end{tikzpicture}
\end{equation*}
The map $F_\varphi$ is unique and is called the \emph{base map} of $\varphi$. The set of morphisms from $\mca{G}$ to $\mca{G}^\prime$ is denoted by $\Hom(\mca{G},\mca{G}^\prime)$. 

For a morphism $\varphi\colon\mca{G}\to\mca{G}^\prime$, we have $\varphi(1_x)=1_{F_\varphi(x)}$ for any $x\in M$ and $\varphi(g^{-1})=\varphi(g)^{-1}$ for any $g\in\mca{G}$. 

A morphism $\varphi\colon\mca{G}\to\mca{G}^\prime$ is called an \emph{isomorphism} if there exists a morphism $\varphi^\prime\colon\mca{G}^\prime\to\mca{G}$ such that $\varphi\circ\varphi^\prime=\id$ and $\varphi^\prime\circ\varphi=\id$. 

If $\varphi$ is an isomorphism, then $\varphi$ and $F_\varphi$ are diffeomorphisms and $F_{\varphi^{-1}}=F_\varphi^{-1}$. 
\end{dfn}

\begin{prop}\label{isodiff}
Let $\mca{G}\rightrightarrows M$ and $\mca{G}^\prime\rightrightarrows M^\prime$ be Lie groupoids and $\varphi\in\Hom(\mca{G},\mca{G}^\prime)$. Then $\varphi$ is an isomorphism if and only if $\varphi\colon\mca{G}\to\mca{G}^\prime$ is a diffeomorphism. 
\end{prop}

\begin{proof}
The implication $\Rightarrow$ is obvious. We prove $\Leftarrow$. Assume $\varphi$ is a diffeomorphism. We have 
\begin{equation*}
\begin{tikzcd}
\mca{G}\ar[r,"\varphi"]\ar[d,"s"']\ar[rd,phantom,"\circlearrowright"{xshift=2,yshift=-2}]&\mca{G}^\prime\ar[d,"s"]\\
M\ar[r,"F_\varphi"']&M^\prime, 
\end{tikzcd}
\end{equation*}
hence $F_\varphi$ is a surjective submersion. On the other hand, 
\begin{equation*}
\begin{tikzcd}
\mca{G}\ar[r,"\varphi"]\ar[rd,phantom,"\circlearrowright"{xshift=2,yshift=-2}]&\mca{G}^\prime\\
M\ar[r,"F_\varphi"']\ar[u,"1"]&M^\prime\ar[u,"1"']
\end{tikzcd}
\end{equation*}
shows that $F_\varphi$ is an injective immersion. So $F_\varphi\colon M\to M^\prime$ is a diffeomorphism. We have 
\begin{equation*}
\begin{tikzcd}
\mca{G}\ar[d,shift left]\ar[d,shift right]\ar[rd,phantom,"\circlearrowright"{xshift=2,yshift=-2}]&\mca{G}^\prime\ar[l,"\varphi^{-1}"']\ar[d,shift left]\ar[d,shift right]\\
M&M^\prime. \ar[l,"F_\varphi^{-1}"]
\end{tikzcd}
\end{equation*}
For $(g^\prime_1,g^\prime_2)\in(\mca{G}^\prime)^{(2)}$, put $g_i=\varphi^{-1}(g^\prime_i)$. We have 
\begin{equation*}
s(g_1)=s\varphi^{-1}(g^\prime_1)=F_\varphi^{-1}s^\prime(g^\prime_1)=F_\varphi^{-1}t^\prime(g^\prime_2)=t\varphi^{-1}(g^\prime_2)=t(g_2), 
\end{equation*}
so $(g_1,g_2)\in\mca{G}^{(2)}$. Since $\varphi(g_1g_2)=\varphi(g_1)\varphi(g_2)=g^\prime_1g^\prime_2$, we have 
\begin{equation*}
\varphi^{-1}(g^\prime_1)\varphi^{-1}(g^\prime_2)=g_1g_2=\varphi^{-1}(g^\prime_1g^\prime_2). 
\end{equation*}
Hence $\varphi^{-1}\colon\mca{G}^\prime\to\mca{G}$ is a morphism with base map $F_\varphi^{-1}$. Thus $\varphi$ is an isomorphism. 
\end{proof}

\begin{dfn}
Let $\mca{G}\rightrightarrows M$ be a Lie groupoid. A \emph{Lie subgroupoid} of $\mca{G}\rightrightarrows M$ is a pair $(\mca{G}^\prime\rightrightarrows M^\prime,\varphi)$ of a Lie groupoid $\mca{G}^\prime\rightrightarrows M^\prime$ and $\varphi\in\Hom(\mca{G}^\prime,\mca{G})$ such that $\varphi$ and $F_\varphi$ are injective immersions. 

An \emph{embedded Lie subgroupoid} is a Lie subgroupoid $(\mca{G}^\prime\rightrightarrows M^\prime,\varphi)$ such that $\varphi$ and $F_\varphi$ are embeddings. 

A Lie subgroupoid $(\mca{G}^\prime\rightrightarrows M^\prime,\varphi)$ of $\mca{G}\rightrightarrows M$ is \emph{wide} if $M^\prime=M$ and $F_\varphi=\id$. 
\end{dfn}

\subsection{The identity component of a Lie groupoid}
\begin{lem}\label{identitycomp}
Let $\mca{G}\rightrightarrows M$ be a Lie groupoid. For $x\in M$, let $(\mca{G}_0)^x$ be the connected component of $\mca{G}^x$ containing $1_x$, and let $\mca{G}_0=\bigcup_{x\in M}(\mca{G}_0)^x$. Then $\mca{G}_0$ is a wide embedded Lie subgroupoid of $\mca{G}$, called the \emph{identity component} of $\mca{G}$. For each $x\in M$, $(\mca{G}_0)_x$ coincides with the connected component of $\mca{G}_x$ containing $1_x$. 
\end{lem}

\begin{proof}
See Proposition 1.5.1 of \cite{Mackenzie}. 
\end{proof}

\begin{lem}\label{connorb}
Let $\mca{G}\rightrightarrows M$ be a Lie groupoid and $O\in\mca{O}_\mca{G}$. If $O_0$ is a connected component of $O$, then $O_0$ is an orbit of $\mca{G}_0$. 
\end{lem}

\begin{proof}
Let $x\in O_0$. Since $s\colon\mca{G}^x\to O$ is a submersion, so is $s\colon(\mca{G}_0)^x\to O_0$. Hence $O_0$ is a disjoint union of orbits of $\mca{G}_0$. Since each of these orbits is open in $O_0$, $O_0$ must coincide with an orbit of $\mca{G}_0$ by connectedness. 
\end{proof}

\section{Actions of Lie groupoids}\label{444}
\subsection{Actions of Lie groupoids}
\begin{dfn}\label{actleftnrgyt}
Let $\mca{G}\rightrightarrows M$ be a Lie groupoid, $N$ be a $C^\infty$ manifold and $\nu\colon N\to M$ be a $C^\infty$ map. Then $\mca{G}\times_{s,\nu}N$ is a closed embedded submanifold of $\mca{G}\times N$ since $s$ is a submersion and $M$ is Hausdorff (Lemma \ref{fiberproduct}). A $C^\infty$ \emph{left action} of $\mca{G}$ on $\nu$ is a $C^\infty$ map $\rho\colon\mca{G}\times_{s,\nu}N\to N$ such that: 
\begin{itemize}
\setlength\itemsep{0em}
\item $\nu(\rho(g,y))=t(g)$ for all $(g,y)\in\mca{G}\times_{s,\nu}N$
\item $\rho(g,\rho(g^\prime,y))=\rho(gg^\prime,y)$ for all $(g,g^\prime)\in\mca{G}^{(2)}$ and $(g^\prime,y)\in\mca{G}\times_{s,\nu}N$
\item $\rho(1_{\nu(y)},y)=y$ for all $y\in N$. 
\end{itemize}
A $C^\infty$ \emph{right action} of $\mca{G}$ on $\nu$ is a $C^\infty$ map $\rho\colon N\times_{\nu,t}\mca{G}\to N$ such that: 
\begin{itemize}
\setlength\itemsep{0em}
\item $\nu(\rho(y,g))=s(g)$ for all $(y,g)\in N\times_{\nu,t}\mca{G}$
\item $\rho(\rho(y,g),g^\prime)=\rho(y,gg^\prime)$ for all $(y,g)\in N\times_{\nu,t}\mca{G}$ and $(g,g^\prime)\in\mca{G}^{(2)}$
\item $\rho(y,1_{\nu(y)})=y$ for all $y\in N$. 
\end{itemize}
Let 
\begin{equation*}
\mca{A}(\nu,\mca{G})=\left\{\text{$C^\infty$ right actions of $\mca{G}$ on $\nu$}\right\}. 
\end{equation*}
If $\mca{A}(\nu,\mca{G})\neq\emptyset$, then $\nu(N)$ is a union of orbits of $\mca{G}$. We have $\mca{G}|_{\nu(N)}=t^{-1}(\nu(N))$ and $\mca{G}|_{\nu(N)}\rightrightarrows\nu(N)$ is a groupoid (but not necessarily a Lie subgroupoid of $\mca{G}\rightrightarrows M$). 
\end{dfn}

\begin{example}
The subset $\nu(N)$ may not be a submanifold of $M$. Let $N$, $M$ be $C^\infty$ manifolds and $\nu\colon N\to M$ be any $C^\infty$ map. Consider the base groupoid $M\rightrightarrows M$. Then $\mca{A}(\nu,M)=\{\rho\}$, where $\rho(y,1_{\nu(y)})=y$ for all $y\in N$. 
\end{example}

\begin{example}\label{canaction}
Let $\mca{G}\rightrightarrows M$ be a Lie groupoid. Consider the map $(t,s)\colon\mca{G}\to M\times M$ and the product groupoid $\mca{G}\times\mca{G}\rightrightarrows M\times M$. Define $\gamma\in\mca{A}((t,s),\mca{G}\times\mca{G})$ by $\gamma(g,(g_1,g_2))=g_1^{-1}gg_2$ for $(g,(g_1,g_2))\in\mca{G}\times_{(t,s),t\times t}(\mca{G}\times\mca{G})$, ie $(t(g),s(g))=(t(g_1),t(g_2))$. Then the action groupoid $\mca{G}\rtimes_\gamma(\mca{G}\times\mca{G})\rightrightarrows\mca{G}$ is a Lie groupoid where the manifold $\mca{G}$ of objects might be non Hausdorff, non second countable. 
\end{example}

\subsection{Action morphisms and action groupoids}
\begin{dfn}[Definition 1.6.15 in \cite{Mackenzie}]\label{hnggnmactmor}
Let $\mca{H}\rightrightarrows N$, $\mca{G}\rightrightarrows M$ be Lie groupoids and $\varphi\in\Hom(\mca{H},\mca{G})$. Let $\varphi^!=(t,\varphi)\colon\mca{H}\to N\times_{F_\varphi,t}\mca{G}$. We have a commutative diagram 
\begin{equation*}
\begin{tikzcd}
\mca{H}\ar[rd,"\varphi^!"]\ar[rrd,bend left,"\varphi"]\ar[rdd,bend right,"t"']\\
&N\times_{F_\varphi,t}\mca{G}\ar[r]\ar[d]&\mca{G}\ar[d,"t"]\\
&N\ar[r,"F_\varphi"']&M. 
\end{tikzcd}
\end{equation*}
We say that $\varphi$ is an \emph{action morphism} if $\varphi^!$ is a diffeomorphism. Let 
\begin{equation*}
\Hom(\mca{H},\mca{G})^\times=\left\{\varphi\in\Hom(\mca{H},\mca{G})\ \middle|\ \text{$\varphi$ is an action morphism}\right\}. 
\end{equation*}
If $\varphi$ is an action morphism and $y\in N$, then $\varphi\colon\mca{H}^y\to\mca{G}^{F_\varphi(y)}$ is a diffeomorphism. 
\end{dfn}

\begin{prop}\label{actiongroupoid}
Let $\mca{G}\rightrightarrows M$ be a Lie groupoid, $N$ be a Hausdorff second countable $C^\infty$ manifold, $\nu\colon N\to M$ be a $C^\infty$ map and $\rho\in\mca{A}(\nu,\mca{G})$. Then we have a Lie groupoid $N\rtimes_\rho\mca{G}\rightrightarrows N$, called the \emph{action groupoid} of $\rho$, defined by $N\rtimes_\rho\mca{G}=N\times_{\nu,t}\mca{G}$ and 
\begin{align*}
t\colon N\rtimes_\rho\mca{G}&\to N&s\colon N\rtimes_\rho\mca{G}&\to N\\
(y,g)&\mapsto y, &(y,g)&\mapsto\rho(y,g), 
\end{align*}
\begin{align*}
(N\rtimes_\rho\mca{G})^{(2)}&\to N\rtimes_\rho\mca{G}\\
((y,g),(\rho(y,g),g^\prime))&\mapsto(y,gg^\prime), 
\end{align*}
\begin{align*}
N&\to N\rtimes_\rho\mca{G}&N\rtimes_\rho\mca{G}&\to N\rtimes_\rho\mca{G}\\
y&\mapsto(y,1_{\nu(y)}), &(y,g)&\mapsto(\rho(y,g),g^{-1}). 
\end{align*}
(Even if $N$ is not assumed to be Hausdorff or second countable, $N\rtimes_\rho\mca{G}\rightrightarrows N$ satisfies the conditions to be a Lie groupoid except $N$ being Hausdorff and second countable.) 
\begin{align*}
\varphi_\rho\colon N\rtimes_\rho\mca{G}&\to\mca{G}\\
(y,g)&\mapsto g
\end{align*}
is an action morphism with base map $\nu$. We call $\varphi_\rho$ the \emph{associated action morphism} of $\rho$. If $\nu$ is surjective, then $\varphi_\rho$ is surjective. If $\nu$ is a submersion, then $\varphi_\nu$ is a submersion. 
\end{prop}

\begin{proof}
We have 
\begin{equation*}
\begin{tikzcd}
N\rtimes_\rho\mca{G}\ar[r,"\varphi_\rho"]\ar[d,"t"']\ar[rd,phantom,"\circlearrowright"]&\mca{G}\ar[d,"t"]\\
N\ar[r,"\nu"']&M. 
\end{tikzcd}
\end{equation*}
$N\rtimes_\rho\mca{G}$ is an embedded submanifold of $N\times\mca{G}$ and $t\colon N\rtimes_\rho\mca{G}\to N$ is a submersion since $t\colon\mca{G}\to M$ is a submersion (Lemma \ref{fiberproduct}). For $y\in N$, $t^{-1}(y)=\left\{y\right\}\times\mca{G}^{\nu(y)}$ is Hausdorff. Since $s\varphi_\rho(y,g)=s(g)=\nu\rho(y,g)=\nu s(y,g)$ for $(y,g)\in N\rtimes_\rho\mca{G}$, $\varphi_\rho$ is a morphism with base map $\nu$. Since $\varphi_\rho^!\colon N\rtimes_\rho\mca{G}\to N\times_{\nu,t}\mca{G}$, $\varphi_\rho^!(y,g)=(y,g)$ is the identity map, $\varphi_\rho$ is an action morphism. If $\nu$ is surjective (resp. a submersion), $\varphi_\nu$ is surjective (resp. a submersion) by Lemma \ref{fiberproduct}. 
\end{proof}

\begin{prop}[Theorem 1.6.16 and Proposition 1.6.17 (i) in \cite{Mackenzie}]
Let $\mca{H}\rightrightarrows N$, $\mca{G}\rightrightarrows M$ be Lie groupoids and $\varphi\in\Hom(\mca{H},\mca{G})^\times$. Let 
\begin{equation*}
\rho_\varphi\colon N\times_{F_\varphi,t}\mca{G}\xrightarrow{(\varphi^!)^{-1}}\mca{H}\xrightarrow{s}N. 
\end{equation*}
Then $\rho_\varphi\in\mca{A}(F_\varphi,\mca{G})$, $\varphi^!\colon\mca{H}\to N\rtimes_{\rho_\varphi}\mca{G}$ is an isomorphism with base map $\id$ and 
\begin{equation*}
\begin{tikzcd}[column sep=tiny]
\mca{H}\ar[rr,"\sim"{name=U},"\varphi^!"{yshift=7}]\ar[rd,"\varphi"']&[2]&N\rtimes_{\rho_\varphi}\mca{G}\ar[ld,"\varphi_{\rho_\varphi}"]\\
&\mca{G}. \ar[to=U,phantom,"\circlearrowright"{xshift=2}]
\end{tikzcd}
\end{equation*}
\end{prop}

\begin{proof}
For $(y,g)\in N\times_{F_\varphi,t}\mca{G}$, let $h=(\varphi^!)^{-1}(y,g)\in\mca{H}$. Then $(t(h),\varphi(h))=(y,g)$ and 
\begin{equation}\label{fvarfs}
F_\varphi(\rho_\varphi(y,g))=F_\varphi(s(h))=s\varphi(h)=s(g). 
\end{equation}

For $(y,g)\in N\times_{F_\varphi,t}\mca{G}$ and $(g,g^\prime)\in\mca{G}^{(2)}$, we have $(\rho_\varphi(y,g),g^\prime)\in N\times_{F_\varphi,t}\mca{G}$ by \eqref{fvarfs}. Let $h=(\varphi^!)^{-1}(y,g)$, $h^\prime=(\varphi^!)^{-1}(\rho_\varphi(y,g),g^\prime)$. Then $(t(h),\varphi(h))=(y,g)$ and $(t(h^\prime),\varphi(h^\prime))=(\rho_\varphi(y,g),g^\prime)=(s(h),g^\prime)$. In particular $s(h)=t(h^\prime)$, hence $hh^\prime$ is defined, and $\varphi(hh^\prime)=\varphi(h)\varphi(h^\prime)=gg^\prime$. So $\varphi^!(hh^\prime)=(t(hh^\prime),\varphi(hh^\prime))=(y,gg^\prime)$. We have 
\begin{gather*}
\rho_\varphi(\rho_\varphi(y,g),g^\prime)=s(\varphi^!)^{-1}(\rho_\varphi(y,g),g^\prime)=s(h^\prime), \\
\rho_\varphi(y,gg^\prime)=s(\varphi^!)^{-1}(y,gg^\prime)=s(hh^\prime)=s(h^\prime). 
\end{gather*}
For $y\in N$, we have $\varphi^!(1_y)=(y,1_{F_\varphi(y)})$, hence 
\begin{equation*}
\rho_\varphi(y,1_{F_\varphi(y)})=s(\varphi^!)^{-1}(y,1_{F_\varphi(y)})=s(1_y)=y. 
\end{equation*}
Therefore $\rho_\varphi\in\mca{A}(F_\varphi,\mca{G})$. 

For $h\in\mca{H}$, we have 
\begin{gather*}
t\varphi^!(h)=t(t(h),\varphi(h))=t(h), \\
s\varphi^!(h)=\rho_\varphi\varphi^!(h)=s(\varphi^!)^{-1}\varphi^!(h)=s(h), \\
\varphi^!(hh^\prime)=(t(hh^\prime),\varphi(hh^\prime))=(t(h),\varphi(h)\varphi(h^\prime)), \\
\varphi^!(h)\varphi^!(h^\prime)=(t(h),\varphi(h))(t(h^\prime),\varphi(h^\prime))=(t(h),\varphi(h)\varphi(h^\prime)), 
\end{gather*}
hence $\varphi^!$ is a morphism with base map $\id$. It is an isomorphism by Proposition \ref{isodiff}. 
\end{proof}

\begin{lem}\label{compacmor}
Let $\mca{G}\rightrightarrows M$, $\mca{G}^\prime\rightrightarrows M^\prime$, $\mca{G}^{\prime\prime}\rightrightarrows M^{\prime\prime}$ be Lie groupoids and $\varphi\in\Hom(\mca{G},\mca{G}^\prime)^\times$, $\varphi^\prime\in\Hom(\mca{G}^\prime,\mca{G}^{\prime\prime})^\times$. Then $\varphi^\prime\varphi\in\Hom(\mca{G},\mca{G}^{\prime\prime})^\times$. 
\end{lem}

\begin{proof}
We have 
\begin{equation*}
\begin{tikzcd}
M\times_{F_\varphi,p_1}(M^\prime\times_{F_{\varphi^\prime},t^{\prime\prime}}\mca{G}^{\prime\prime})\ar[r]\ar[d]&M^\prime\times_{F_{\varphi^\prime},t^{\prime\prime}}\mca{G}^{\prime\prime}\ar[r]\ar[d,"p_1"]&\mca{G}^{\prime\prime}\ar[d,"t^{\prime\prime}"]\\
M\ar[r,"F_\varphi"']&M^\prime\ar[r,"F_{\varphi^\prime}"']&M^{\prime\prime}. 
\end{tikzcd}
\end{equation*}
By Lemma \ref{fiberproduct}, $M\times_{F_\varphi,p_1}(M^\prime\times_{F_{\varphi^\prime},t^{\prime\prime}}\mca{G}^{\prime\prime})$, $M\times_{F_{\varphi^\prime}F_\varphi,t^{\prime\prime}}\mca{G}^{\prime\prime}$ are $C^\infty$ manifolds and 
\begin{align*}
M\times_{F_\varphi,p_1}(M^\prime\times_{F_{\varphi^\prime},t^{\prime\prime}}\mca{G}^{\prime\prime})&\simeq M\times_{F_{\varphi^\prime}F_\varphi,t^{\prime\prime}}\mca{G}^{\prime\prime}\\
(x,x^\prime,g^{\prime\prime})&\mapsto(x,g^{\prime\prime})\\
(x,F_\varphi(x),g^{\prime\prime})&\mapsfrom(x,g^{\prime\prime})
\end{align*}
is a diffeomorphism. We have 
\begin{equation*}
\begin{tikzcd}
\mca{G}\ar[r,"(\varphi^\prime\varphi)^!"]\ar[d,"\sim"{sloped},"\varphi^!"']\ar[rd,phantom,"\circlearrowright"{xshift=5,yshift=-3}]&M\times_{F_{\varphi^\prime}F_\varphi,t^{\prime\prime}}\mca{G}^{\prime\prime}\ar[d,dash,sloped,"\sim"]\\
M\times_{F_\varphi,t^\prime}\mca{G}^\prime\ar[r,"\sim","\id\times(\varphi^\prime)^!"']&M\times_{F_\varphi,p_1}(M^\prime\times_{F_{\varphi^\prime},t^{\prime\prime}}\mca{G}^{\prime\prime}), 
\end{tikzcd}
\end{equation*}
where the bottom map comes from 
\begin{equation*}
\begin{tikzcd}[column sep=tiny]
\mca{G}^\prime\ar[rr,"\sim","(\varphi^\prime)^!"{yshift=7},""{name=U}]\ar[rd,"t^\prime"']&&M^\prime\times_{F_{\varphi^\prime},t^{\prime\prime}}\mca{G}^{\prime\prime}\ar[ld,"p_1"]\\
&M^\prime\ar[to=U,phantom,"\circlearrowright"{xshift=4,yshift=-1}]
\end{tikzcd}
\end{equation*}
and commutativity of the square is checked as 
\begin{equation*}
g\mapsto(t(g),\varphi(g))\mapsto(t(g),t^\prime\varphi(g),\varphi^\prime\varphi(g))\mapsto(t(g),\varphi^\prime\varphi(g)). 
\end{equation*}
Hence $(\varphi^\prime\varphi)^!$ is a diffeomorphism. 
\end{proof}

\begin{prop}\label{isoactdiff}
Let $\mca{G}\rightrightarrows M$, $\mca{G}^\prime\rightrightarrows M^\prime$ be Lie groupoids and $\varphi\in\Hom(\mca{G},\mca{G}^\prime)$. Then the following statements hold: 
\begin{enumerate}
\setlength\itemsep{0em}
\item If $\varphi$ is an isomorphism, then $\varphi$ is an action morphism. 
\item If $\varphi$ is an action morphism and $F_\varphi$ is a diffeomorphism, then $\varphi$ is an isomorphism. 
\end{enumerate}
\end{prop}

\begin{proof}
1. The map 
\begin{align*}
M\times_{F_\varphi,t^\prime}\mca{G}^\prime&\to\mca{G}\\
(x,g^\prime)&\mapsto\varphi^{-1}(g^\prime)
\end{align*}
is the inverse of $\varphi^!$. 

2. We have a diffeomorphism 
\begin{align*}
M\times_{F_\varphi,t^\prime}\mca{G}^\prime&\simeq\mca{G}^\prime\\
(x,g^\prime)&\mapsto g^\prime\\
(F_\varphi^{-1}t^\prime(g^\prime),g^\prime)&\mapsfrom g^\prime
\end{align*}
and a diagram 
\begin{equation*}
\begin{tikzcd}[column sep=tiny]
\mca{G}\ar[rr,"\varphi"{name=U}]\ar[rd,"\sim"{sloped},"\varphi^!"']&&\mca{G}^\prime\\
&M\times_{F_\varphi,t^\prime}\mca{G}^\prime. \ar[ru,sloped,"\sim"]
\ar[from=2-2,to=U,phantom,"\circlearrowright"]
\end{tikzcd}
\end{equation*}
So $\varphi$ is a diffeomorphism, hence an isomorphism by Proposition \ref{isodiff}. 
\end{proof}

\subsection{(Semi)conjugacies and (constant) cocycles}
\begin{dfn}\label{semiconconpng}
Let $\mca{G}\rightrightarrows M$, $\mca{G}^\prime\rightrightarrows M^\prime$ be Lie groupoids, $N$, $N^\prime$ be $C^\infty$ manifolds, $\nu\colon N\to M$, $\nu^\prime\colon N^\prime\to M^\prime$ be $C^\infty$ maps, $\rho\in\mca{A}(\nu,\mca{G})$, $\rho^\prime\in\mca{A}(\nu^\prime,\mca{G}^\prime)$ and $\varphi\in\Hom(N\rtimes_\rho\mca{G},N^\prime\rtimes_{\rho^\prime}\mca{G}^\prime)$. 
\begin{itemize}
\item $\varphi$ is a \emph{semiconjugacy} if there exists $\Phi\in\Hom(\mca{G},\mca{G}^\prime)$ such that 
\begin{equation*}
\begin{tikzcd}
N\rtimes_\rho\mca{G}\ar[r,"\varphi"]\ar[d,"\varphi_\rho"']\ar[rd,phantom,"\circlearrowright"{xshift=-5}]&N^\prime\rtimes_{\rho^\prime}\mca{G}^\prime\ar[d,"\varphi_{\rho^\prime}"]\\
\mca{G}\ar[r,"\Phi"']&\mca{G}^\prime. 
\end{tikzcd}
\end{equation*}
We write $\varphi=F_\varphi\rtimes\Phi$ and call $\Phi$ a \emph{constant part} of $\varphi$. We have $\varphi_\rho(N\rtimes_\rho\mca{G})=\mca{G}|_{\nu(N)}$ and $\Phi$ is unique on $\mca{G}|_{\nu(N)}$. 
\item $\varphi$ is a \emph{conjugacy} if $\varphi$ is a semiconjugacy and there exists a semiconjugacy $\varphi^\prime\colon N^\prime\rtimes_{\rho^\prime}\mca{G}^\prime\to N\rtimes_\rho\mca{G}$ such that $\varphi^\prime\varphi=\id$, $\varphi\varphi^\prime=\id$. 
\item We say that $\rho$ is \emph{conjugate} to $\rho^\prime$ if there exists a conjugacy $N\rtimes_\rho\mca{G}\to N^\prime\rtimes_{\rho^\prime}\mca{G}^\prime$. 
\end{itemize}
\end{dfn}

\begin{prop}\label{conjdiffinj}
Let $\mca{G}\rightrightarrows M$, $\mca{G}^\prime\rightrightarrows M^\prime$ be Lie groupoids, $N$, $N^\prime$ be $C^\infty$ manifolds, $\nu\colon N\to M$, $\nu^\prime\colon N^\prime\to M^\prime$ be $C^\infty$ maps, $\rho\in\mca{A}(\nu,\mca{G})$, $\rho^\prime\in\mca{A}(\nu^\prime,\mca{G}^\prime)$ and $\varphi=F_\varphi\rtimes\Phi\in\Hom(N\rtimes_\rho\mca{G},N^\prime\rtimes_{\rho^\prime}\mca{G}^\prime)$ be a semiconjugacy. Assume $\nu^\prime$ is a surjective submersion. Then $\varphi$ is a conjugacy if and only if $\varphi$ is a diffeomorphism and $F_\Phi\colon\nu(N)\to M^\prime$ is injective. 
\end{prop}

\begin{proof}
First we prove $\Rightarrow$. Assume $\varphi$ is a conjugacy. Then $\varphi$ is a diffeomorphism. Since $\varphi^{-1}$ is a semiconjugacy, $\varphi^{-1}=F_{\varphi^{-1}}\rtimes\Phi^\prime$ for some $\Phi^\prime\in\Hom(\mca{G}^\prime,\mca{G})$. We have 
\begin{equation*}
\begin{tikzcd}
N\rtimes_\rho\mca{G}\ar[r,"\varphi"]\ar[d,"\varphi_\rho"']\ar[rd,phantom,"\circlearrowright"{xshift=-5}]&N^\prime\rtimes_{\rho^\prime}\mca{G}^\prime\ar[r,"\varphi^{-1}"]\ar[d,"\varphi_{\rho^\prime}"]\ar[rd,phantom,"\circlearrowright"{xshift=-5,yshift=1}]&N\rtimes_\rho\mca{G}\ar[d,"\varphi_\rho"]\\
\mca{G}\ar[r,"\Phi"']\ar[d,"t"']\ar[rd,phantom,"\circlearrowright"{yshift=-2}]&\mca{G}^\prime\ar[r,"\Phi^\prime"']\ar[d,"t^\prime"]\ar[rd,phantom,"\circlearrowright"{yshift=-2}]&\mca{G}\ar[d,"t"]\\
M\ar[r,"F_\Phi"']&M^\prime\ar[r,"F_{\Phi^\prime}"']&M, 
\end{tikzcd}
\end{equation*}
hence 
\begin{equation*}
\begin{tikzcd}
N\rtimes_\rho\mca{G}\ar[r,"\id"]\ar[d,"t\varphi_\rho"']\ar[rd,phantom,"\circlearrowright"{xshift=-4}]&N\rtimes_\rho\mca{G}\ar[d,"t\varphi_\rho"]\\
M\ar[r,"F_{\Phi^\prime}F_\Phi"']&M. 
\end{tikzcd}
\end{equation*}
Since $t\varphi_\rho(N\rtimes_\rho\mca{G})=\nu(N)$, $F_{\Phi^\prime}F_\Phi=\id$ on $\nu(N)$. So $F_\Phi\colon\nu(N)\to M^\prime$ is injective. 

Next we prove $\Leftarrow$. We have 
\begin{equation*}
\begin{tikzcd}
N\rtimes_\rho\mca{G}\ar[r,"\varphi"]\ar[d,"\varphi_\rho"']\ar[rd,phantom,"\circlearrowright"{xshift=-5}]&N^\prime\rtimes_{\rho^\prime}\mca{G}^\prime\ar[d,"\varphi_{\rho^\prime}"]\\
\mca{G}|_{\nu(N)}\ar[r,"\Phi"']&\mca{G}^\prime. 
\end{tikzcd}
\end{equation*}
$\varphi_{\rho^\prime}$ is surjective since $\nu^\prime$ is surjective. $\Phi\colon\mca{G}|_{\nu(N)}\to\mca{G}^\prime$ is surjective since $\varphi$, $\varphi_{\rho^\prime}$ are surjective. 

\begin{claim}
$\Phi\colon\mca{G}|_{\nu(N)}\to\mca{G}^\prime$ is injective, hence bijective. 
\end{claim}

\begin{proof}
Let $g$, $g^\prime\in\mca{G}|_{\nu(N)}$ be such that $\Phi(g)=\Phi(g^\prime)$. We have 
\begin{equation*}
F_\Phi t(g)=t^\prime\Phi(g)=t^\prime\Phi(g^\prime)=F_\Phi t(g^\prime). 
\end{equation*}
Then $t(g)=t(g^\prime)$ by the injectivity of $F_\Phi\colon\nu(N)\to M^\prime$. There exists $y\in N$ such that $\nu(y)=t(g)$. Hence $(y,g)$, $(y,g^\prime)\in N\rtimes_\rho\mca{G}$ and 
\begin{equation*}
\varphi(y,g)=(F_\varphi(y),\Phi(g))=(F_\varphi(y),\Phi(g^\prime))=\varphi(y,g^\prime). 
\end{equation*}
We have $(y,g)=(y,g^\prime)$ since $\varphi$ is injective. So $g=g^\prime$. 
\end{proof}

Thus we have 
\begin{equation}\label{ngphng}
\begin{tikzcd}
N\rtimes_\rho\mca{G}\ar[d,"\varphi_\rho"']\ar[rd,phantom,"\circlearrowright"{xshift=-5}]&N^\prime\rtimes_{\rho^\prime}\mca{G}^\prime\ar[l,"\varphi^{-1}"']\ar[d,"\varphi_{\rho^\prime}"]\\
\mca{G}&\mca{G}^\prime. \ar[l,"(\Phi|_{\mca{G}|_{\nu(N)}})^{-1}"]
\end{tikzcd}
\end{equation}
$\varphi_{\rho^\prime}$ is a surjective submersion since $\nu^\prime$ is a surjective submersion (Proposition \ref{actiongroupoid}). So there exists a $C^\infty$ local section of $\varphi_{\rho^\prime}$ at any point of $\mca{G}^\prime$. Therefore $(\Phi|_{\mca{G}|_{\nu(N)}})^{-1}\colon\mca{G}^\prime\to\mca{G}$ is $C^\infty$. Since the horizontal maps in 
\begin{equation*}
\begin{tikzcd}[column sep=large]
\mca{G}|_{\nu(N)}\ar[r,"\Phi|_{\mca{G}|_{\nu(N)}}"]\ar[d,shift left]\ar[d,shift right]\ar[rd,phantom,"\circlearrowright"{xshift=-2,yshift=1}]&\mca{G}^\prime\ar[d,shift left]\ar[d,shift right]\\
\nu(N)\ar[r,"F_\Phi|_{\nu(N)}"']&M^\prime
\end{tikzcd}
\end{equation*}
are bijective, we get 
\begin{equation*}
\begin{tikzcd}[column sep=huge]
\mca{G}\ar[d,shift left]\ar[d,shift right]\ar[rd,phantom,"\circlearrowright"{xshift=2,yshift=-2}]&\mca{G}^\prime\ar[l,"(\Phi|_{\mca{G}|_{\nu(N)}})^{-1}"']\ar[d,shift left]\ar[d,shift right]\\
M&M^\prime. \ar[l,"(F_\Phi|_{\nu(N)})^{-1}"]
\end{tikzcd}
\end{equation*}
$(F_\Phi|_{\nu(N)})^{-1}\colon M^\prime\to M$ is $C^\infty$ since $t^\prime\colon\mca{G}^\prime\to M^\prime$ is a surjective submersion. 

Let $(g^\prime_1,g^\prime_2)\in(\mca{G}^\prime)^{(2)}$. There exists $g_i\in\mca{G}|_{\nu(N)}$ such that $\Phi|_{\mca{G}|_{\nu(N)}}(g_i)=g^\prime_i$. We have $(g_1,g_2)\in\mca{G}^{(2)}$ and $\Phi|_{\mca{G}|_{\nu(N)}}(g_1g_2)=\Phi|_{\mca{G}|_{\nu(N)}}(g_1)\Phi|_{\mca{G}|_{\nu(N)}}(g_2)=g^\prime_1g^\prime_2$. Hence $(\Phi|_{\mca{G}|_{\nu(N)}})^{-1}(g^\prime_1g^\prime_2)=g_1g_2=(\Phi|_{\mca{G}|_{\nu(N)}})^{-1}(g^\prime_1)(\Phi|_{\mca{G}|_{\nu(N)}})^{-1}(g^\prime_2)$. So $(\Phi|_{\mca{G}|_{\nu(N)}})^{-1}\in\Hom(\mca{G}^\prime,\mca{G})$. Thus $\varphi^{-1}$ is a semiconjugacy by \eqref{ngphng}. This proves $\varphi$ is a conjugacy. 
\end{proof}

\begin{dfn}
Let $\mca{G}\rightrightarrows M$, $\mca{G}^\prime\rightrightarrows M^\prime$ be Lie groupoids, $N$ be a $C^\infty$ manifold, $\nu\colon N\to M$ be a $C^\infty$ map and $\rho\in\mca{A}(\nu,\mca{G})$. 
\begin{itemize}
\item A \emph{$\mca{G}^\prime$-valued $C^\infty$ cocycle} over $\rho$ is an element of $\Hom(N\rtimes_\rho\mca{G},\mca{G}^\prime)$. 
\item A cocycle $c\in\Hom(N\rtimes_\rho\mca{G},\mca{G}^\prime)$ is \emph{constant} if it is a semiconjugacy, where $\mca{G}^\prime$ is identified with $M^\prime\rtimes_{\tau^\prime}\mca{G}^\prime$ by $\varphi_{\tau^\prime}$ for $\mca{A}(\id,\mca{G}^\prime)=\{\tau^\prime\}$, ie there exists $\Phi\in\Hom(\mca{G},\mca{G}^\prime)$ such that 
\begin{equation*}
\begin{tikzcd}
N\rtimes_\rho\mca{G}\ar[r,"c"]\ar[d,"\varphi_\rho"']&\mca{G}^\prime. \\
\mca{G}\ar[ru,"\Phi"'{name=U}]
\ar[from=1-1,to=U,phantom,"\circlearrowright"{pos=.4}]
\end{tikzcd}
\end{equation*}
This means that $c(y,g)$ does not depend on $y$. 

If $\nu$ is surjective, then 
\begin{align*}
\Hom(\mca{G},\mca{G}^\prime)&\to\Hom(N\rtimes_\rho\mca{G},\mca{G}^\prime)\\
\Phi&\mapsto\Phi\varphi_\rho
\end{align*}
is injective. So constant cocycles can be identified with morphisms $\mca{G}\to\mca{G}^\prime$. 
\end{itemize}
\end{dfn}

\subsection{Action of an action groupoid}
\begin{prop}\label{nqmaphic}
Let $N$, $Q$, $M$ be $C^\infty$ manifolds, $\nu\colon N\to Q$, $\mu\colon Q\to M$ be $C^\infty$ maps: 
\begin{equation}\label{qnmgmnu}
\begin{tikzcd}[column sep=tiny]
N\ar[rr,"\nu"{name=U}]\ar[rd,"\mu\nu"']&&Q\ar[ld,"\mu"]\\
&M. \ar[to=U,phantom,"\circlearrowright"]
\end{tikzcd}
\end{equation}
Let $\mca{G}\rightrightarrows M$ be a Lie groupoid and $\tau\in\mca{A}(\mu,\mca{G})$. We have a Lie groupoid $Q\rtimes_\tau\mca{G}\rightrightarrows Q$. Then there is a bijection 
\begin{align*}
\mca{A}(\nu,Q\rtimes_\tau\mca{G})&\simeq\left\{\rho^\prime\in\mca{A}(\mu\nu,\mca{G})\ \middle|\ \text{$\nu$ is $(\rho^\prime,\tau)$-equivariant}\right\}\\
\rho&\leftrightarrow\rho^\prime
\end{align*}
such that $\rho^\prime(y,g)=\rho(y,(\nu(y),g))$ for all $(y,g)\in N\rtimes_{\rho^\prime}\mca{G}$. We have an isomorphism 
\begin{align*}
\varphi\colon N\rtimes_\rho(Q\rtimes_\tau\mca{G})&\xrightarrow{\sim}N\rtimes_{\rho^\prime}\mca{G}\\
(y,(\nu(y),g))&\mapsto(y,g)
\end{align*}
and a commutative diagram 
\begin{equation}\label{nqfngqgg}
\begin{tikzcd}
N\rtimes_\rho(Q\rtimes_\tau\mca{G})\ar[r,"\sim","\varphi"{yshift=7}]\ar[d,"\varphi_\rho"']&N\rtimes_{\rho^\prime}\mca{G}\ar[d,"\varphi_{\rho^\prime}"]\ar[ld,"\nu\rtimes\id"]\\
Q\rtimes_\tau\mca{G}\ar[r,"\varphi_\tau"']&\mca{G}. 
\end{tikzcd}
\end{equation}
In particular $\varphi=\id\rtimes\varphi_\tau$ is a semiconjugacy. 

Assume that $\mu\nu\colon N\to M$ is a surjective submersion. Then $\varphi$ is a conjugacy if and only if $\mu\colon\nu(N)\to M$ is injective. 
\end{prop}

\begin{rem}
The upper left triangle of Diagram \eqref{nqfngqgg} says the action groupoid of an action of the action groupoid of an action of $\mca{G}$ is identified with the action groupoid of an action of $\mca{G}$, and the associated action morphism of the action of an action groupoid is identified with the quotient map of the action groupoids of the actions of $\mca{G}$. 
\end{rem}

\begin{proof}[Proof of Proposition \ref{nqmaphic}]
Let $\rho\in\mca{A}(\nu,Q\rtimes_\tau\mca{G})$. Consider the morphism $\varphi_\tau\varphi_\rho\colon N\rtimes_\rho(Q\rtimes_\tau\mca{G})\to Q\rtimes_\tau\mca{G}\to\mca{G}$ with base map $F_{\varphi_\tau\varphi_\rho}=F_{\varphi_\tau}F_{\varphi_\rho}=\mu\nu$. Since 
\begin{align*}
\varphi:=(\varphi_\tau\varphi_\rho)^!\colon N\rtimes_\rho(Q\rtimes_\tau\mca{G})&\to N\times_{\mu\nu,t}\mca{G}\\
(y,(\nu(y),g))&\mapsto(y,g)
\end{align*}
is a diffeomorphism, $\varphi_\tau\varphi_\rho$ is an action morphism. Let $\rho^\prime=\rho_{\varphi_\tau\varphi_\rho}\in\mca{A}(\mu\nu,\mca{G})$. Then 
\begin{equation*}
\rho^\prime(y,g)=\rho_{\varphi_\tau\varphi_\rho}(y,g)=s((\varphi_\tau\varphi_\rho)^!)^{-1}(y,g)=\rho(y,(\nu(y),g))
\end{equation*}
for $(y,g)\in N\rtimes_{\rho^\prime}\mca{G}$. Since 
\begin{equation*}
\nu\rho^\prime(y,g)=\nu\rho(y,(\nu(y),g))=s(\nu(y),g)=\tau(\nu(y),g)
\end{equation*}
for $(y,g)\in N\rtimes_{\rho^\prime}\mca{G}$, $\nu\colon N\to Q$ is $\mca{G}$-equivariant. The commutativity of Diagram \eqref{nqfngqgg} is obvious. 

For the opposite direction, let $\rho^\prime\in\mca{A}(\mu\nu,\mca{G})$ be such that $\nu$ is $(\rho^\prime,\tau)$-equivariant. Define 
\begin{align*}
\rho\colon N\times_{\nu,t}(Q\rtimes_\tau\mca{G})&\to N\\
(y,(\nu(y),g))&\mapsto\rho^\prime(y,g). 
\end{align*}

\begin{claim}
$\rho\in\mca{A}(\nu,Q\rtimes_\tau\mca{G})$. 
\end{claim}

\begin{proof}
We have 
\begin{equation*}
\nu\rho(y,(\nu(y),g))=\nu\rho^\prime(y,g)=\tau(\nu(y),g)=s(\nu(y),g)
\end{equation*}
for $(y,g)\in N\rtimes_{\rho^\prime}\mca{G}$, 
\begin{align*}
\rho(y,(\nu(y),g)(\tau(\nu(y),g),g^\prime))&=\rho(y,(\nu(y),gg^\prime))=\rho^\prime(y,gg^\prime)\\
&=\rho^\prime(\rho(y,(\nu(y),g)),g^\prime)\\
&=\rho(\rho(y,(\nu(y),g)),(\tau(\nu(y),g),g^\prime))
\end{align*}
for $(y,g)\in N\rtimes_{\rho^\prime}\mca{G}$ and $(g,g^\prime)\in\mca{G}^{(2)}$, 
\begin{equation*}
\rho(y,(\nu(y),1_{\mu\nu(y)}))=\rho^\prime(y,1_{\mu\nu(y)})=y
\end{equation*}
for any $y\in N$. 
\end{proof}

It is obvious that the two maps $\rho\mapsto\rho^\prime$ and $\rho^\prime\mapsto\rho$ are inverse to each other. The final assertion follows from Proposition \ref{conjdiffinj}. 
\end{proof}

\begin{example}
Specializing Proposition \ref{nqmaphic}, we see that there exists a semiconjugacy which is an isomorphism but not a conjugacy. Let $N_0$ be a nonempty $C^\infty$ manifold, $G$ be a Lie group and $\rho_0\in\mca{A}(N_0,G)$. Let $\nu_0\colon N_0\to\mr{pt}$ be the unique map and $\nu=\nu_0\sqcup\nu_0\colon N_0\sqcup N_0\to\mr{pt}\sqcup\mr{pt}$. Diagram \eqref{qnmgmnu} in Proposition \ref{nqmaphic} is now 
\begin{equation*}
\begin{tikzcd}[column sep=tiny]
N_0\sqcup N_0\ar[rr,"\nu"{name=U}]\ar[rd]&&\mr{pt}\sqcup\mr{pt}\ar[ld]\\
&\mr{pt}. \ar[to=U,phantom,"\circlearrowright"]
\end{tikzcd}
\end{equation*}
Let $\tau\in\mca{A}(\mr{pt}\sqcup\mr{pt},G)$ be the trivial action and $\rho=\rho_0\sqcup\rho_0\in\mca{A}(\nu,(\mr{pt}\sqcup\mr{pt})\rtimes_\tau G)$. Then $\rho^\prime=\rho_0\sqcup\rho_0\in\mca{A}(N_0\sqcup N_0,G)$. 
Consider $\varphi=\id\rtimes\varphi_\tau$ as in 
\begin{equation*}
\begin{tikzcd}
(N_0\sqcup N_0)\rtimes_\rho((\mr{pt}\sqcup\mr{pt})\rtimes_\tau G)\ar[r,"\sim","\varphi"{yshift=7}]\ar[d,"\varphi_\rho"']\ar[rd,phantom,"\circlearrowright"{xshift=-7,yshift=2}]&(N_0\sqcup N_0)\rtimes_{\rho^\prime}G\ar[d,"\varphi_{\rho^\prime}"]\\
(\mr{pt}\sqcup\mr{pt})\rtimes_\tau G\ar[r,"\varphi_\tau"']&G. 
\end{tikzcd}
\end{equation*}
Then $\varphi$ is a semiconjugacy and an isomorphism. But $\mr{pt}\sqcup\mr{pt}\to\mr{pt}$ is not injective, hence $\varphi$ is not a conjugacy. (It is easy to see directly that $\varphi^{-1}$ is not a semiconjugacy.) 
\end{example}

\begin{example}
The constant parts of a conjugacy may not be an isomorphism. Let $\nu\colon\{0\}\to\bb{R}$ be the inclusion. We specialize Diagram \eqref{qnmgmnu} in Proposition \ref{nqmaphic} to 
\begin{equation*}
\begin{tikzcd}[column sep=tiny]
\{0\}\ar[rr,"\nu"{name=U}]\ar[rd]&&\bb{R}\ar[ld]\\
&\mr{pt}. \ar[to=U,phantom,"\circlearrowright"]
\end{tikzcd}
\end{equation*}
Define $\tau\in\mca{A}(\bb{R},\bb{R})$ by $\tau(x,t)=xe^t$. Then $\mca{A}(\nu,\bb{R}\rtimes_\tau\bb{R})=\{\rho\}$ and $\mca{A}(\{0\},\bb{R})=\{\rho^\prime\}$. We have 
\begin{equation*}
\begin{tikzcd}
\{0\}\rtimes_\rho(\bb{R}\rtimes_\tau\bb{R})\ar[r,"\sim","\varphi"{yshift=7}]\ar[d,"\varphi_\rho"']&\{0\}\rtimes_{\rho^\prime}\bb{R}\ar[d,"\sim"{sloped},"\varphi_{\rho^\prime}"{xshift=7}]\ar[ld,"\nu\rtimes\id"]\\
\bb{R}\rtimes_\tau\bb{R}\ar[r,"\varphi_\tau"']&\bb{R}. 
\end{tikzcd}
\end{equation*}
Since $\nu(\{0\})\to\mr{pt}$ is injective, $\varphi=\id\rtimes\varphi_\tau$ is a conjugacy. Indeed let 
\begin{align*}
\theta=(\nu\rtimes\id)(\varphi_{\rho^\prime})^{-1}\colon\bb{R}&\to\bb{R}\rtimes_\tau\bb{R}\\
t&\mapsto(0,t). 
\end{align*}
Then $\theta$ is a morphism and 
\begin{equation*}
\begin{tikzcd}
\{0\}\rtimes_\rho(\bb{R}\rtimes_\tau\bb{R})\ar[d,"\varphi_\rho"']\ar[rd,phantom,"\circlearrowright"{xshift=-5,yshift=1}]&\{0\}\rtimes_{\rho^\prime}\bb{R}\ar[l,"\varphi^{-1}"']\ar[d,"\varphi_{\rho^\prime}"]\\
\bb{R}\rtimes_\tau\bb{R}&\bb{R}. \ar[l,"\theta"]
\end{tikzcd}
\end{equation*}
Hence $\varphi^{-1}=\id\rtimes\theta$ is a semiconjugacy. Note that the constant parts $\varphi_\tau$ and $\theta$ are not isomorphisms. 
\end{example}

\subsection{Action groupoid morphisms}
\begin{dfn}\label{acgrmois}
Let $\mca{G}\rightrightarrows M$ be a Lie groupoid, $N$, $N^\prime$ be $C^\infty$ manifolds, $\nu\colon N\to M$, $\nu^\prime\colon N^\prime\to M$ be $C^\infty$ maps and $\rho\in\mca{A}(\nu,\mca{G})$, $\rho^\prime\in\mca{A}(\nu^\prime,\mca{G})$. 
\begin{itemize}
\item Put 
\begin{equation*}
\ul{C}^\infty(N,N^\prime)=\left\{F\in C^\infty(N,N^\prime)\ \middle|\ 
\begin{tikzcd}[column sep=tiny]
N\ar[rr,"F"{name=U}]\ar[rd,"\nu"']&&N^\prime\ar[ld,"\nu^\prime"]\\
&M\ar[to=U,phantom,"\circlearrowright"]
\end{tikzcd}
\right\}
\end{equation*}
and 
\begin{equation*}
\ul{\Diff}(N)=\left\{F\in\Diff(N)\ \middle|\ \nu F=\nu\right\}. 
\end{equation*}
\item Let $\varphi\in\Hom(N\rtimes_\rho\mca{G},N^\prime\rtimes_{\rho^\prime}\mca{G})$. We say that $\varphi$ is an \emph{action groupoid morphism} if $F_\varphi\in\ul{C}^\infty(N,N^\prime)$. Let 
\begin{align*}
\ul{\Hom}(N\rtimes_\rho\mca{G},N^\prime\rtimes_{\rho^\prime}\mca{G})&=\left\{\varphi\in\Hom(N\rtimes_\rho\mca{G},N^\prime\rtimes_{\rho^\prime}\mca{G})\ \middle|\ F_\varphi\in\ul{C}^\infty(N,N^\prime)\right\}, \\
\ul{\End}(N\rtimes_\rho\mca{G})&=\ul{\Hom}(N\rtimes_\rho\mca{G},N\rtimes_\rho\mca{G}). 
\end{align*}
\item Let $\varphi\in\ul{\Hom}(N\rtimes_\rho\mca{G},N^\prime\rtimes_{\rho^\prime}\mca{G})$. We say that $\varphi$ is an \emph{action groupoid isomorphism} if there exists $\varphi^\prime\in\ul{\Hom}(N^\prime\rtimes_{\rho^\prime}\mca{G},N\rtimes_\rho\mca{G})$ such that $\varphi\varphi^\prime=\id$ and $\varphi^\prime\varphi=\id$. Put 
\begin{equation*}
\ul{\Aut}(N\rtimes_\rho\mca{G})=\left\{\varphi\in\Aut(N\rtimes_\rho\mca{G})\ \middle|\ F_\varphi\in\ul{C}^\infty(N,N)\right\}. 
\end{equation*}
\end{itemize}
\end{dfn}

\begin{example}
We have $\nu\rtimes\id\in\ul{\Hom}(N\rtimes_{\rho^\prime}\mca{G},Q\rtimes_\tau\mca{G})$ in Proposition \ref{nqmaphic}. 
\end{example}

\begin{lem}\label{ulid}
Let $\mca{G}\rightrightarrows M$, $\mca{G}^\prime\rightrightarrows M^\prime$ be Lie groupoids, $N$, $N^\prime$ be $C^\infty$ manifolds, $\nu\colon N\to M$, $\nu^\prime\colon N^\prime\to M^\prime$ be $C^\infty$ maps, $\rho\in\mca{A}(\nu,\mca{G})$, $\rho^\prime\in\mca{A}(\nu^\prime,\mca{G}^\prime)$ and $\varphi=F_\varphi\rtimes\Phi\in\Hom(N\rtimes_\rho\mca{G},N^\prime\rtimes_{\rho^\prime}\mca{G}^\prime)$ be a semiconjugacy. Assume $\nu$ is surjective. Then $\varphi\in\ul{\Hom}(N\rtimes_\rho\mca{G},N^\prime\rtimes_{\rho^\prime}\mca{G}^\prime)$ if and only if $F_\Phi=\id$. 
\end{lem}

\begin{proof}
We have 
\begin{equation*}
\begin{tikzcd}
N\ar[r,"F_\varphi"]\ar[d,"\nu"']\ar[rd,phantom,"\circlearrowright"{xshift=1,yshift=-1}]&N^\prime\ar[d,"\nu^\prime"]\\
M\ar[r,"F_\Phi"']&M^\prime. 
\end{tikzcd}
\end{equation*}
If $F_\Phi=\id$, then $\varphi$ is an action groupoid morphism. If $\varphi$ is an action groupoid morphism, then $F_\Phi\nu=\nu^\prime F_\varphi=\nu$. $F_\Phi=\id$ since $\nu$ is surjective. 
\end{proof}

\begin{example}
Let $\mca{G}\rightrightarrows M$ be a Lie groupoid. Then $\mca{A}(\id,\mca{G})=\left\{\tau\right\}$. We have $\tau(x,g)=s(g)$ for all $(x,g)\in M\times_{\id,t}\mca{G}$ and $\varphi_\tau\colon M\rtimes_\tau\mca{G}\to\mca{G}$ is an isomorphism. By regarding $\mca{G}$ as the action groupoid $M\rtimes_\tau\mca{G}$, we have 
\begin{equation*}
\ul{\End}(\mca{G})=\left\{\Phi\in\End(\mca{G})\ \middle|\ F_\Phi=\id\right\}
\end{equation*}
and 
\begin{equation*}
\ul{\Aut}(\mca{G})=\left\{\Phi\in\Aut(\mca{G})\ \middle|\ F_\Phi=\id\right\}. 
\end{equation*}
\end{example}

\begin{lem}\label{ngconj}
Let $\mca{G}\rightrightarrows M$ be a Lie groupoid, $N$, $N^\prime$ be $C^\infty$ manifolds, $\nu^\prime\colon N^\prime\to M$ be a $C^\infty$ map, $\rho^\prime\in\mca{A}(\nu^\prime,\mca{G})$, $F\colon N\to N^\prime$ be a diffeomorphism and $\Phi\in\ul{\Aut}(\mca{G})$. Let 
\begin{align*}
(\rho^\prime)^{(F,\Phi)}\colon N\times_{\nu^\prime F,t}\mca{G}&\to N\\
(y,g)&\mapsto F^{-1}\rho^\prime(F(y),\Phi(g)). 
\end{align*}
Then $(\rho^\prime)^{(F,\Phi)}\in\mca{A}(\nu^\prime F,\mca{G})$ and $F\rtimes\Phi\in\ul{\Hom}(N\rtimes_{(\rho^\prime)^{(F,\Phi)}}\mca{G},N^\prime\rtimes_{\rho^\prime}\mca{G})$ is a conjugacy. 
\end{lem}

\begin{proof}
Note that the assumption $\Phi\in\ul{\Aut}(\mca{G})$ instead of $\Phi\in\Aut(\mca{G})$ is used for example for the well definedness of $\rho^\prime(F(y),\Phi(g))$ and the computation 
\begin{equation*}
\nu^\prime F(\rho^\prime)^{(F,\Phi)}(y,g)=\nu^\prime\rho^\prime(F(y),\Phi(g))=s\Phi(g)=s(g). 
\end{equation*}
We obviously have $F(\rho^\prime)^{(F,\Phi)}(y,g)=\rho^\prime(F(y),\Phi(g))$ for $(y,g)\in N\rtimes_{(\rho^\prime)^{(F,\Phi)}}\mca{G}$. 
\end{proof}

\begin{lem}
Let $\mca{G}\rightrightarrows M$ be a Lie groupoid, $N$, $N^\prime$, $N^{\prime\prime}$ be $C^\infty$ manifolds, $\nu^{\prime\prime}\colon N^{\prime\prime}\to M$ be a $C^\infty$ map, $\rho^{\prime\prime}\in\mca{A}(\nu^{\prime\prime},\mca{G})$, $F\colon N\to N^\prime$, $F^\prime\colon N^\prime\to N^{\prime\prime}$ be diffeomorphisms and $\Phi$, $\Phi^\prime\in\ul{\Aut}(\mca{G})$. Then $(\rho^{\prime\prime})^{(F^\prime F,\Phi^\prime\Phi)}=((\rho^{\prime\prime})^{(F^\prime,\Phi^\prime)})^{(F,\Phi)}$. 
\end{lem}

\begin{prop}\label{gmnnmfnnfn}
Let $\mca{G}\rightrightarrows M$ be a Lie groupoid, $N$ be a $C^\infty$ manifold, $\nu\colon N\to M$ be a $C^\infty$ map, $\rho\in\mca{A}(\nu,\mca{G})$, $F\in\ul{\Diff}(N)$ and $\Phi\in\ul{\Aut}(\mca{G})$. Then $\rho^{(F,\Phi)}\in\mca{A}(\nu,\mca{G})$, $F\rtimes\Phi\in\ul{\Hom}(N\rtimes_{\rho^{(F,\Phi)}}\mca{G},N\rtimes_\rho\mca{G})$ is a conjugacy and 
\begin{align*}
\mca{A}(\nu,\mca{G})\times(\ul{\Diff}(N)\times\ul{\Aut}(\mca{G}))&\to\mca{A}(\nu,\mca{G})\\
(\rho,(F,\Phi))&\mapsto\rho^{(F,\Phi)}
\end{align*}
is a right action. 
\end{prop}

\section{The principal bundle structures on the target fibers}\label{999}
The results in this section are taken from Section 5.1 of \cite{MoerMr}. The difference is that we state the results for a Lie groupoid $\mca{G}\rightrightarrows M$, where $M$ is not assumed to be Hausdorff or second countable. But the proofs do not change. Our exposition is much more detailed to make it clear that Hausdorff and second countability assumptions are not necessary. We will apply these results in Lemma \ref{tsiso} to the action groupoid of the action $\mca{G}\stackrel{\gamma}{\curvearrowleft}\mca{G}\times\mca{G}$ by multiplication (Example \ref{canaction}) for a Lie groupoid $\mca{G}\rightrightarrows M$, where the base manifold $\mca{G}$ may not be Hausdorff or second countable. This in turn will be used to prove the main results of this paper (Proposition \ref{btop} and Theorem \ref{ctop}). 

\begin{dfn}\label{localbisection}
Let $\mca{G}\rightrightarrows M$ be a Lie groupoid ($M$ can be non Hausdorff, non second countable). A \emph{local bisection} of $\mca{G}$ is a $C^\infty$ map $\sigma\colon V\to\mca{G}$ on an open set $V$ of $M$ such that: 
\begin{itemize}
\setlength\itemsep{0em}
\item $t\sigma=\id$
\item $U=s\sigma(V)$ is an open set of $M$
\item $s\sigma\colon V\to U$ is a diffeomorphism. 
\end{itemize}
In this situation, define 
\begin{align*}
\tau\colon U&\to\mca{G}\\
x&\mapsto\sigma((s\sigma)^{-1}(x))^{-1}. 
\end{align*}
Then $\tau$ is $C^\infty$, $t\tau=\id$, $s\tau=(s\sigma)^{-1}$ and $s\tau(U)=V$. Hence $\tau$ is a local bisection of $\mca{G}$. 
\end{dfn}

\begin{lem}
Let $\mca{G}\rightrightarrows M$ be a Lie groupoid ($M$ can be non Hausdorff, non second countable) and $g\in\mca{G}$. Then there exists a local bisection $\sigma$ of $\mca{G}$ such that $\sigma(t(g))=g$. 
\end{lem}

\begin{proof}
Since $\dim\mca{G}^{t(g)}=\dim\mca{G}_{s(g)}$, there exists a subspace $L$ of $T_g\mca{G}$ such that $T_g\mca{G}=T_g\mca{G}^{t(g)}\oplus L$ and $T_g\mca{G}=T_g\mca{G}_{s(g)}\oplus L$. There exist an open neighborhood $V_1$ of $t(g)$ in $M$ and a $C^\infty$ map $\sigma\colon V_1\to\mca{G}$ such that $t\sigma=\id$, $\sigma(t(g))=g$ and $\sigma_*T_{t(g)}M=L$. Since $T_g\mca{G}_{s(g)}=\ker s_*$, $(s\sigma)_*\colon T_{t(g)}M\to T_{s(g)}M$ is an isomorphism. By the inverse function theorem there exists an open neighborhood $V\subset V_1$ of $t(g)$ such that $s\sigma$ is a diffeomorphism from $V$ to an open subset $U$ of $M$. Then $\sigma\colon V\to\mca{G}$ satisfies the required conditions. 
\end{proof}

\begin{lem}\label{constantrank}
Let $\mca{G}\rightrightarrows M$ be a Lie groupoid ($M$ can be non Hausdorff, non second countable) and $x\in M$. Then $s\colon\mca{G}^x\to M$ is of constant rank. 
\end{lem}

\begin{proof}
Let $g$, $g^\prime\in\mca{G}^x$. We have $g^\prime=g(g^{-1}g^\prime)$. There exist an open neighborhood $V$ of $s(g)$ in $M$ and a local bisection $\sigma\colon V\to\mca{G}$ such that $\sigma(s(g))=g^{-1}g^\prime$. Let $U=s\sigma(V)$ and define $\tau\colon U\to\mca{G}$ as in Definition \ref{localbisection}. Consider $C^\infty$ maps 
\begin{align*}
R_\sigma\colon\mca{G}^x_V&\to\mca{G}^x_U&R_\tau\colon\mca{G}^x_U&\to\mca{G}^x_V\\
g_1&\mapsto g_1\sigma(s(g_1)), &g_2&\mapsto g_2\tau(s(g_2)). 
\end{align*}
Then $R_\sigma R_\tau=\id$, $R_\tau R_\sigma=\id$, hence $R_\sigma$ is a diffeomorphism. We have 
\begin{equation*}
\begin{tikzcd}
\mca{G}^x_V\ar[r,"\sim","R_\sigma"{yshift=7}]\ar[d,"s"']\ar[rd,phantom,"\circlearrowright"{xshift=-2,yshift=1}]&\mca{G}^x_U\ar[d,"s"]\\
V\ar[r,"\sim","s\sigma"']&U
\end{tikzcd}
\end{equation*}
and $R_\sigma(g)=g^\prime$. Hence 
\begin{equation*}
\begin{tikzcd}
T_g\mca{G}^x\ar[r,"\sim","(R_\sigma)_*"{yshift=6}]\ar[d,"s_*"']\ar[rd,phantom,"\circlearrowright"{xshift=-1}]&T_{g^\prime}\mca{G}^x\ar[d,"s_*"]\\
T_{s(g)}M\ar[r,"\sim","(s\sigma)_*"']&T_{s(g^\prime)}M. 
\end{tikzcd}
\end{equation*}
Therefore $s\colon\mca{G}^x\to M$ is of constant rank. 
\end{proof}

\begin{lem}\label{fttwhen}
Let $G$ be a Lie group (perhaps non second countable), $P$ be a Hausdorff $C^\infty$ manifold (perhaps non second countable), $\mu\colon G\times P\to P$ be a free $C^\infty$ action with quotient map $\pi\colon P\to G\bs P$. Then the following are equivalent: 
\begin{enumerate}
\item For any $x\in P$, there exists an embedded submanifold $S$ of $P$ such that: 
\begin{itemize}
\setlength\itemsep{0em}
\item $x\in S$
\item $\mu\colon G\times S\to P$ is an open embedding. 
\end{itemize}
\item There exists a $C^\infty$ differential structure on $G\bs P$ (perhaps non Hausdorff, non second countable) such that $\pi\colon P\to G\bs P$ is a left principal $G$ bundle. 
\item There exist a $C^\infty$ manifold $M$ (perhaps non Hausdorff, non second countable), a $C^\infty$ map $f\colon P\to M$ and a map $\ol{f}\colon G\bs P\to M$ such that: 
\begin{itemize}
\setlength\itemsep{0em}
\item \qquad$\begin{aligned}
\begin{tikzcd}
P\ar[r,"f"]\ar[d,"\pi"']&M\\
G\bs P\ar[ru,"\ol{f}"'{name=U}]
\ar[from=1-1,to=U,phantom,"\circlearrowright"{pos=.4}]
\end{tikzcd}
\end{aligned}$
\item For any $x\in P$, the image of $T_1G\times\{0\}$ by $\mu_*\colon T_1G\times T_xP\to T_xP$ coincides with $\ker f_*$. (Note that we always have $\mu_*(T_1G\times\{0\})\subset\ker f_*$.) 
\end{itemize}
\end{enumerate}
When one of the above three conditions holds, the map $\ol{f}\colon G\bs P\to M$ in $3$ is an immersion. 
\end{lem}

\begin{proof}
$1\Rightarrow2$. Equip $G\bs P$ with the quotient topology. Let $S$ be an embedded submanifold of $P$ such that: 
\begin{itemize}
\setlength\itemsep{0em}
\item $\mu\colon G\times S\to P$ is an open embedding
\item $S$ is diffeomorphic to an open subset of $\bb{R}^{\dim S}$. 
\end{itemize}
Let 
\begin{equation*}
\begin{tikzcd}
&P\ar[d,"\pi"]\\
S\ar[r,"\iota_S"']\ar[ru,hookrightarrow,""{name=U}]&G\bs P. \ar[to=U,phantom,"\circlearrowright"{pos=.3}]
\end{tikzcd}
\end{equation*}
$\iota_S$ is continuous. 

\begin{claim}
$\iota_S$ is injective. 
\end{claim}

\begin{proof}
Let $y$, $y^\prime\in S$ be such that $\iota_S(y)=\iota_S(y^\prime)$. There exists $g\in G$ such that $y=\mu(g,y^\prime)$. Since $\mu(1,y)=\mu(g,y^\prime)$, we have $(1,y)=(g,y^\prime)$, hence $y=y^\prime$. 
\end{proof}

\begin{claim}
$\iota_S$ is an open map. 
\end{claim}

\begin{proof}
Let $U$ be an open subset of $S$. Then $\pi^{-1}\iota_S(U)=\mu(G\times U)$ is open in $P$ since $\mu\colon G\times S\to P$ is an open map. This means that $\iota_S(U)$ is open in $G\bs P$. 
\end{proof}

Hence $\iota_S$ is a topological embedding and $\iota_S(S)$ is open in $G\bs P$. Let $S^\prime$ be another embedded submanifold of $P$ satisfying the same two conditions as $S$ such that $\iota_S(S)\cap\iota_{S^\prime}(S^\prime)\neq\emptyset$. We have 
\begin{equation*}
\begin{tikzcd}
&P\ar[d,"\pi"]&G\times S^\prime\ar[l,"\mu"']\ar[d,"p_{S^\prime}"]\\
S\ar[r,"\iota_S"']\ar[ru]&G\bs P&S^\prime, \ar[l,"\iota_{S^\prime}"]\ar[lu]
\end{tikzcd}
\end{equation*}
where $p_{S^\prime}$ is the projection. By restriction we get 
\begin{equation*}
\begin{tikzcd}
&\pi^{-1}(\iota_S(S)\cap\iota_{S^\prime}(S^\prime))\ar[d,"\pi"]&G\times(\iota_{S^\prime}^{-1}(\iota_S(S)\cap\iota_{S^\prime}(S^\prime)))\ar[l,"\sim"',"\mu"'{yshift=7}]\ar[d,"p_{S^\prime}"]\\
\iota_S^{-1}(\iota_S(S)\cap\iota_{S^\prime}(S^\prime))\ar[r,"\iota_S"']\ar[ru]&\iota_S(S)\cap\iota_{S^\prime}(S^\prime)&\iota_{S^\prime}^{-1}(\iota_S(S)\cap\iota_{S^\prime}(S^\prime)). \ar[l,"\iota_{S^\prime}"]
\end{tikzcd}
\end{equation*}
This shows $\iota_{S^\prime}^{-1}\iota_S\colon\iota_S^{-1}(\iota_S(S)\cap\iota_{S^\prime}(S^\prime))\to\iota_{S^\prime}^{-1}(\iota_S(S)\cap\iota_{S^\prime}(S^\prime))$ is $C^\infty$. Hence $G\bs P$ is a $C^\infty$ manifold and $\pi\colon P\to G\bs P$ is a $C^\infty$ map. $\pi\colon P\to G\bs P$ is a principal $G$ bundle, where the local triviality comes from 
\begin{equation*}
\begin{tikzcd}
G\times S\ar[r,"\mu"]\ar[d,"p_S"']\ar[rd,phantom,"\circlearrowright"]&P\ar[d,"\pi"]\\
S\ar[r,"\iota_S"']&G\bs P. 
\end{tikzcd}
\end{equation*}

$2\Rightarrow3$. Let $M=G\bs P$, $f=\pi$ and $\ol{f}=\id$. 

$3\Rightarrow1$. Let $x\in P$. We have $f_*T_xP\subset T_{f(x)}M$ and let $k=\dim f_*T_xP$. There exist an open neighborhood $V$ of $f(x)$ in $M$ and a submersion $h\colon V\to\bb{R}^k$ such that $T_{f(x)}M=f_*T_xP\oplus\ker h_*$. There exists  an embedded submanifold $S$ of $P$ such that: 
\begin{itemize}
\setlength\itemsep{0em}
\item $x\in S$
\item $T_yP=\mu_*(T_1G\times\{0\})\oplus T_yS=\ker f_*\oplus T_yS$ for all $y\in S$
\item $f(S)\subset V$. 
\end{itemize}
Consider $hf\colon S\to\bb{R}^k$. Note that 
\begin{equation*}
\dim S=\dim T_xP-\dim\ker f_*=\dim f_*T_xP=k. 
\end{equation*}
$(hf)_*\colon T_xS\to T_{hf(x)}\bb{R}^k$ is injective, hence an isomorphism. Take $S$ small enough so that $hf\colon S\to\bb{R}^k$ is an open embedding. 

\begin{claim}
$\mu\colon G\times S\to P$ is injective. 
\end{claim}

\begin{proof}
Let $(g,y)$, $(g^\prime,y^\prime)\in G\times S$ be such that $\mu(g,y)=\mu(g^\prime,y^\prime)$. Applying $f$, we get $f(y)=f(y^\prime)\in V$. Hence $hf(y)=hf(y^\prime)$. So $y=y^\prime$ and $g=g^\prime$. 
\end{proof}

\begin{claim}
For any $(g,y)\in G\times S$, $\mu_*\colon T_{(g,y)}(G\times S)\to T_{\mu(g,y)}P$ is an isomorphism. 
\end{claim}

\begin{proof}
$\mu_*\colon T_{(1,y)}(G\times S)\to T_yP$ is surjective, hence an isomorphism. We have 
\begin{equation*}
\begin{tikzcd}
G\times S\ar[r,"\mu"]\ar[d,"L_{g^{-1}}\times\id"']\ar[rd,phantom,"\circlearrowright"]&P\ar[d,"\mu(g^{-1}{,}\cdot)"]\\
G\times S\ar[r,"\mu"']&P, 
\end{tikzcd}
\end{equation*}
hence 
\begin{equation*}
\begin{tikzcd}
T_{(g,y)}(G\times S)\ar[r,"\mu_*"]\ar[d,"\sim"{sloped},"(L_{g^{-1}})_*\times\id"']\ar[rd,phantom,"\circlearrowright"]&T_{\mu(g,y)}P\ar[d,"\sim"{sloped},"\mu(g^{-1}{,}\cdot)_*"{xshift=7}]\\
T_{(1,y)}(G\times S)\ar[r,"\sim","\mu_*"']&T_yP. 
\end{tikzcd}
\end{equation*}
\end{proof}

$\mu\colon G\times S\to P$ is an open embedding by the inverse function theorem. 

We prove the final assertion of the lemma. Let $x\in P$ and $Y\in T_{\pi(x)}(G\bs P)$ be such that $\ol{f}_*Y=0$. There exists $X\in T_xP$ such that $\pi_*X=Y$. Then $f_*X=\ol{f}_*\pi_*X=0$. Hence $X\in\ker f_*=\mu_*(T_1G\times\{0\})$ and $Y=\pi_*X=0$. So $\ol{f}$ is an immersion. 
\end{proof}

\begin{prop}\label{isoorb}
Let $\mca{G}\rightrightarrows M$ be a Lie groupoid ($M$ can be non Hausdorff, non second countable). Then the following statements hold: 
\begin{enumerate}
\item For any $x$, $x^\prime\in M$, $\mca{G}^{x^\prime}_x$ is a Hausdorff closed embedded submanifold of $\mca{G}$ (perhaps non second countable) and for any $g\in\mca{G}^{x^\prime}_x$, we have 
\begin{equation*}
T_g\mca{G}^{x^\prime}_x=\ker\left((t,s)_*\colon T_g\mca{G}\to T_{x^\prime}M\times T_xM\right). 
\end{equation*}
\item For any $x\in M$, $\mca{G}^x_x$ is a Lie group (perhaps non second countable). 
\item For any $x\in M$, there exists a $C^\infty$ differential structure on the quotient topological space $\mca{G}^x_x\bs\mca{G}^x$ (perhaps non Hausdorff, non second countable) with respect to the left action 
\begin{align*}
\mca{G}^x_x\times\mca{G}^x&\to\mca{G}^x\\
(g,g^\prime)&\mapsto gg^\prime
\end{align*}
such that the quotient map $\pi\colon\mca{G}^x\to\mca{G}^x_x\bs\mca{G}^x$ is a left principal $\mca{G}^x_x$ bundle. 
\item Let $O\in\mca{O}_\mca{G}$ and take $x\in O$. We put a $C^\infty$ manifold structure on $O$ through the bijection $\ol{s}\colon\mca{G}^x_x\bs\mca{G}^x\to O$ given by 
\begin{equation*}
\begin{tikzcd}
\mca{G}^x\ar[r,"s"]\ar[d,"\pi"']&O. \\
\mca{G}^x_x\bs\mca{G}^x\ar[ru,"\ol{s}"'{name=U}]
\ar[from=1-1,to=U,phantom,"\circlearrowright"{xshift=-3}]
\end{tikzcd}
\end{equation*}
The $C^\infty$ manifold structure does not depend on the choice of $x\in O$. $O$ is an injectively immersed submanifold of $M$ (perhaps non Hausdorff, non second countable) and $s\colon\mca{G}^x\to O$ is a left principal $\mca{G}^x_x$ bundle. 

Similarly $t\colon\mca{G}_x\to O$ is a right principal $\mca{G}^x_x$ bundle, where $O$ is equipped with the previous $C^\infty$ manifold structure. 
\end{enumerate}
\end{prop}

\begin{proof}
1. $s\colon\mca{G}^{x^\prime}\to M$ is of constant rank by Lemma \ref{constantrank}. Hence $\mca{G}^{x^\prime}_x=(s|_{\mca{G}^{x^\prime}})^{-1}(x)$ is an embedded submanifold of $\mca{G}^{x^\prime}$ by the constant rank theorem. $\mca{G}^{x^\prime}_x=(s|_{\mca{G}^{x^\prime}})^{-1}(x)$ is closed in $\mca{G}^{x^\prime}$ since $\{x\}$ is closed in $M$. $\mca{G}^{x^\prime}_x$ is a closed embedded submanifold of $\mca{G}$. We have 
\begin{align*}
\ker\left((t,s)_*\colon T_g\mca{G}\to T_{x^\prime}M\times T_xM\right)&=\ker(t_*\colon T_g\mca{G}\to T_{x^\prime}M)\cap\ker s_*\\
&=T_g\mca{G}^{x^\prime}\cap\ker s_*\\
&=T_g\mca{G}^{x^\prime}_x. 
\end{align*}

2. $\mca{G}^x_x$ is a Hausdorff closed embedded submanifold of $\mca{G}$ by 1. $\mca{G}^x_x\times\mca{G}^x_x$ is an embedded submanifold of $\mca{G}\times\mca{G}$, hence an embedded submanifold of $\mca{G}\times_{s,t}\mca{G}$. Since the multiplication $\mca{G}\times_{s,t}\mca{G}\to\mca{G}$ and the inverse $\mca{G}\to\mca{G}$ are $C^\infty$, the restrictions $\mca{G}^x_x\times\mca{G}^x_x\to\mca{G}^x_x$ and $\mca{G}^x_x\to\mca{G}^x_x$ are $C^\infty$. 

3, 4. We will check the condition 3 in Lemma \ref{fttwhen} is satisfied. $s\colon\mca{G}^x\to M$ is a $C^\infty$ map, the multiplication $\mu\colon\mca{G}^x_x\times\mca{G}^x\to\mca{G}^x$ is a free $C^\infty$ action and there exists a map $\ol{s}$ such that 
\begin{equation*}
\begin{tikzcd}
\mca{G}^x\ar[r,"s"]\ar[d,"\pi"']&M. \\
\mca{G}^x_x\bs\mca{G}^x\ar[ru,"\ol{s}"'{name=U}]
\ar[from=1-1,to=U,phantom,"\circlearrowright"{xshift=-3}]
\end{tikzcd}
\end{equation*}
Let $g\in\mca{G}^x$. We have $\mu_*\colon T_{(1_x,g)}(\mca{G}^x_x\times\mca{G}^x)\to T_g\mca{G}^x$ and 
\begin{equation*}
\mu_*(T_{1_x}\mca{G}^x_x\times\{0\})\subset\ker(s_*\colon T_g\mca{G}^x\to T_{s(g)}M). 
\end{equation*}
Since 
\begin{equation*}
\dim\mu_*(T_{1_x}\mca{G}^x_x\times\{0\})=\dim\mca{G}^x_x=\dim\mca{G}^x_{s(g)}=\dim\ker(s_*\colon T_g\mca{G}^x\to T_{s(g)}M), 
\end{equation*}
we get $\mu_*(T_{1_x}\mca{G}^x_x\times\{0\})=\ker(s_*\colon T_g\mca{G}^x\to T_{s(g)}M)$. 

By Lemma \ref{fttwhen}, $\pi\colon\mca{G}^x\to\mca{G}^x_x\bs\mca{G}^x$ is a left principal $\mca{G}^x_x$ bundle and $\ol{s}\colon\mca{G}^x_x\bs\mca{G}^x\to M$ is an injective immersion with image $O$. Hence $O$, with the $C^\infty$ manifold structure which makes $\ol{s}\colon\mca{G}^x_x\bs\mca{G}^x\to O$ a diffeomorphism, is an injectively immersed submanifold of $M$. 

To check the independence of the $C^\infty$ manifold structure, let $x^\prime\in O$ and take $g\in\mca{G}^{x^\prime}_x$. Then $L_g\colon\mca{G}^x\to\mca{G}^{x^\prime}$ is a diffeomorphism which is $(g\cdot g^{-1}\colon\mca{G}^x_x\simeq\mca{G}^{x^\prime}_{x^\prime})$-equivariant. Hence there exists a bijection $\ol{L_g}\colon\mca{G}^x_x\bs\mca{G}^x\to\mca{G}^{x^\prime}_{x^\prime}\bs\mca{G}^{x^\prime}$ such that 
\begin{equation*}
\begin{tikzcd}
\mca{G}^x\ar[r,"L_g"]\ar[d,"\pi"']\ar[rd,phantom,"\circlearrowright"{xshift=4,yshift=-3}]&\mca{G}^{x^\prime}\ar[d,"\pi^\prime"]\\
\mca{G}^x_x\bs\mca{G}^x\ar[r,"\ol{L_g}"']&\mca{G}^{x^\prime}_{x^\prime}\bs\mca{G}^{x^\prime}. 
\end{tikzcd}
\end{equation*}
Since $L_g$ is a diffeomorphism and $\pi$, $\pi^\prime$ are surjective submersions, $\ol{L_g}$ is a diffeomorphism (Take local $C^\infty$ sections of $\pi$, $\pi^\prime$). We have 
\begin{equation*}
\begin{tikzcd}
\mca{G}^x_x\bs\mca{G}^x\ar[d,"\sim"{sloped},"\ol{s}"']\ar[r,"\sim","\ol{L_g}"{yshift=7}]\ar[rd,phantom,"\circlearrowright"{xshift=-5}]&\mca{G}^{x^\prime}_{x^\prime}\bs\mca{G}^{x^\prime}\ar[d,"\sim"{sloped},"\ol{s}"{xshift=7}]\\
O\ar[r,"\id"']&O, 
\end{tikzcd}
\end{equation*}
which shows the independence. 

Since 
\begin{equation*}
\begin{tikzcd}[column sep=tiny]
\mca{G}^x\ar[rd,"s"']\ar[rr,"\cdot^{-1}"{name=U}]&&\mca{G}_x\ar[ld,"t"]\\
&O\ar[to=U,phantom,"\circlearrowright"]
\end{tikzcd}
\end{equation*}
and $\cdot^{-1}$ transforms the left action of $\mca{G}^x_x$ on $\mca{G}^x$ to the right action of $\mca{G}^x_x$ on $\mca{G}_x$ by right multiplication, $t\colon\mca{G}_x\to O$ is a right principal $\mca{G}^x_x$ bundle. 
\end{proof}

\begin{example}\label{discs1s1}
Let $S^1_d$ be $S^1$ with the discrete topology. Then $S^1_d$ is paracompact but not second countable. Let $\rho_d\in\mca{A}(S^1,S^1_d)$ be the action defined by $\rho_d(x,y)=xy$ for $x\in S^1$ and $y\in S^1_d$. Consider $S^1\rtimes_{\rho_d}S^1_d$. Then $\mca{O}_{S^1\rtimes_{\rho_d}S^1_d}=\left\{O_d\right\}$. For $x\in O_d$, $s\colon(S^1\rtimes_{\rho_d}S^1_d)^x\to O_d$ is a bijection. Since $(S^1\rtimes_{\rho_d}S^1_d)^x$ is diffeomorphic to $S^1_d$, $O_d$ is diffeomorphic to $S^1_d$. So $O_d\to S^1$ is not a diffeomorphism. 
\end{example}

\section{Orbitwise homotopies and $nt$-homotopies}\label{555}
\subsection{Natural transformations}
\begin{dfn}\label{lsphif}
Let $\mca{H}\rightrightarrows N$, $\mca{H}^\prime\rightrightarrows N^\prime$ be Lie groupoids and $\varphi$, $\varphi^\prime\in\Hom(\mca{H},\mca{H}^\prime)$. We say that $P$ is a \emph{natural transformation} from $\varphi$ to $\varphi^\prime$ if $P\colon N\to\mca{H}^\prime$ is a $C^\infty$ map, $t^\prime P=F_\varphi$, $s^\prime P=F_{\varphi^\prime}$ and 
\begin{equation*}
\varphi(h)P(s(h))=P(t(h))\varphi^\prime(h)
\end{equation*}
for all $h\in\mca{H}$. 
\begin{equation*}
\begin{tikzpicture}[every label/.append style={font=\scriptsize},matrix of math nodes,decoration={markings,mark=at position0.5with{\arrow{>}}}]
\def\a{3};
\node(1)[label=left:F_\varphi(t(h))]{};
\node(2)[right=\a em of 1,label=right:F_\varphi(s(h))]{};
\node(3)[below=\a em of 2,label=right:F_{\varphi^\prime}(s(h))]{};
\node(4)[left=\a em of 3,label=left:F_{\varphi^\prime}(t(h))]{};
\draw[postaction={decorate}](2.center)--node[label=above:\varphi(h)]{}(1.center);
\draw[postaction={decorate}](3.center)--node[label=right:P(s(h))]{}(2.center);
\draw[postaction={decorate}](3.center)--node[label=below:\varphi^\prime(h)]{}(4.center);
\draw[postaction={decorate}](4.center)--node(5)[label=left:P(t(h))]{}(1.center);
\node(6)[left=7em of 5,label=below:s(h)]{};
\node(7)[left=of 6,label=below:t(h)]{};
\draw[postaction={decorate}](6.center)--node[label=above:h]{}(7.center);
\path(1)--node{\circlearrowright}(3);
\foreach\x in{1,...,4,6,7}\filldraw(\x)circle(1pt);
\end{tikzpicture}
\end{equation*}
Let $\Nat(\varphi,\varphi^\prime)$ denote the set of all natural transformations from $\varphi$ to $\varphi^\prime$. 
\end{dfn}

\begin{lem}\label{hnhncnnb}
Let $\mca{H}\rightrightarrows N$, $\mca{H}^\prime\rightrightarrows N^\prime$ be Lie groupoids. Recall that we have a groupoid $C^\infty(N,\mca{H}^\prime)\rightrightarrows C^\infty(N,N^\prime)$ from Example \ref{mapgroupoid}. Let 
\begin{align*}
b\colon\Hom(\mca{H},\mca{H}^\prime)&\to C^\infty(N,N^\prime)\\
\varphi&\mapsto F_\varphi. 
\end{align*}
Define 
\begin{align*}
\varphi P\colon\mca{H}&\to\mca{H}^\prime\\
h&\mapsto P(t(h))^{-1}\varphi(h)P(s(h))=\gamma^\prime(\varphi(h),(P(t(h)),P(s(h))))
\end{align*}
for $(\varphi,P)\in\Hom(\mca{H},\mca{H}^\prime)\times_{b,t^\prime}C^\infty(N,\mca{H}^\prime)$, where $\gamma^\prime\in\mca{A}((t^\prime,s^\prime),\mca{H}^\prime\times\mca{H}^\prime)$ is the action by multiplication (Example \ref{canaction}) and 
\begin{equation*}
FP=s^\prime P\in C^\infty(N,N^\prime)
\end{equation*}
for $(F,P)\in C^\infty(N,N^\prime)\times_{\id,t^\prime}C^\infty(N,\mca{H}^\prime)$. 
\begin{equation*}
\begin{tikzpicture}[every label/.append style={font=\scriptsize},matrix of math nodes,decoration={markings,mark=at position0.5with{\arrow{>}}}]
\def\a{3};
\node(1)[label=left:F_\varphi(t(h))]{};
\node(2)[right=\a em of 1,label=right:F_\varphi(s(h))]{};
\node(3)[below=\a em of 2,label=right:(F_\varphi P)(s(h))]{};
\node(4)[left=\a em of 3,label=left:(F_\varphi P)(t(h))]{};
\draw[postaction={decorate}](2.center)--node[label=above:\varphi(h)]{}(1.center);
\draw[postaction={decorate}](3.center)--node[label=right:P(s(h))]{}(2.center);
\draw[postaction={decorate}](3.center)--node[label=below:(\varphi P)(h)]{}(4.center);
\draw[postaction={decorate}](4.center)--node(5)[label=left:P(t(h))]{}(1.center);
\node(6)[left=7em of 5,label=below:s(h)]{};
\node(7)[left=of 6,label=below:t(h)]{};
\draw[postaction={decorate}](6.center)--node[label=above:h]{}(7.center);
\path(1)--node{\circlearrowright}(3);
\foreach\x in{1,...,4,6,7}\filldraw(\x)circle(1pt);
\end{tikzpicture}
\end{equation*}
Then $\varphi P\in\Hom(\mca{H},\mca{H}^\prime)$ and $F_{\varphi P}=F_\varphi P$. 
\end{lem}

\begin{prop}
Let $\mca{H}\rightrightarrows N$, $\mca{H}^\prime\rightrightarrows N^\prime$ be Lie groupoids. Then 
\begin{align*}
\Hom(\mca{H},\mca{H}^\prime)\times_{b,t^\prime}C^\infty(N,\mca{H}^\prime)&\to\Hom(\mca{H},\mca{H}^\prime)\\
(\varphi,P)&\mapsto\varphi P, 
\end{align*}
\begin{align*}
C^\infty(N,N^\prime)\times_{\id,t^\prime}C^\infty(N,\mca{H}^\prime)&\to C^\infty(N,N^\prime)\\
(F,P)&\mapsto FP
\end{align*}
are right actions of $C^\infty(N,\mca{H}^\prime)\rightrightarrows C^\infty(N,N^\prime)$ on $b$ and $\id$, and $b\colon\Hom(\mca{H},\mca{H}^\prime)\to C^\infty(N,N^\prime)$ is an equivariant map. 
\end{prop}

\begin{rem}
For the action groupoid $\Hom(\mca{H},\mca{H}^\prime)\rtimes C^\infty(N,\mca{H}^\prime)\rightrightarrows\Hom(\mca{H},\mca{H}^\prime)$, we have $(\Hom(\mca{H},\mca{H}^\prime)\rtimes C^\infty(N,\mca{H}^\prime))^\varphi_{\varphi^\prime}\simeq\Nat(\varphi,\varphi^\prime)$. 
\end{rem}

\begin{prop}\label{natmor}
Let $\mca{H}\rightrightarrows N$, $\mca{H}^\prime\rightrightarrows N^\prime$ be Lie groupoids and $\varphi$, $\varphi^\prime\in\Hom(\mca{H},\mca{H}^\prime)$. Then there is a bijection 
\begin{equation*}
\Nat(\varphi,\varphi^\prime)\simeq\left\{\psi\in\Hom(\mca{H},\mca{H}^\prime\rtimes_{\gamma^\prime}(\mca{H}^\prime\times\mca{H}^\prime))\ \middle|\ 
\begin{tikzcd}
&\mca{H}^\prime\rtimes_{\gamma^\prime}(\mca{H}^\prime\times\mca{H}^\prime)\ar[d,"\varphi_{\gamma^\prime}"]\\
\mca{H}\ar[r,"(\varphi{,}\varphi^\prime)"']\ar[ru,"\psi",""{name=U,very near end}]&\mca{H}^\prime\times\mca{H}^\prime\ar[to=U,phantom,"\circlearrowright"{pos=.4,xshift=-2,yshift=-1}]
\end{tikzcd}
\right\}
\end{equation*}
by 
\begin{align*}
P&\mapsto(Pt,(\varphi,\varphi^\prime))\\
F_\psi&\mapsfrom\psi, 
\end{align*}
ie 
\begin{equation*}
\begin{tikzcd}
\mca{H}\ar[r,"\psi"]\ar[d,shift left]\ar[d,shift right]\ar[rd,phantom,"\circlearrowright"{yshift=-2}]&\mca{H}^\prime\rtimes_{\gamma^\prime}(\mca{H}^\prime\times\mca{H}^\prime)\ar[d,shift left]\ar[d,shift right]\\
N\ar[r,"P"']&\mca{H}^\prime. 
\end{tikzcd}
\end{equation*}
\end{prop}

\begin{proof}
Let $P\in\Nat(\varphi,\varphi^\prime)$. We have $(t^\prime,s^\prime)P(t(h))=t^\prime(\varphi(h),\varphi^\prime(h)))$ for $h\in\mca{H}$. (See the diagram in Definition \ref{lsphif}.) Hence 
\begin{align*}
(Pt,(\varphi,\varphi^\prime))\colon\mca{H}&\to\mca{H}^\prime\rtimes_{\gamma^\prime}(\mca{H}^\prime\times\mca{H}^\prime)\\
h&\mapsto(P(t(h)),(\varphi(h),\varphi^\prime(h)))
\end{align*}
is defined and $C^\infty$. We have 
\begin{equation*}
\begin{tikzcd}[column sep=huge]
\mca{H}\ar[r,"(Pt{,}(\varphi{,}\varphi^\prime))"]\ar[d,shift left]\ar[d,shift right]\ar[rd,phantom,"\circlearrowright"]&\mca{H}^\prime\rtimes_{\gamma^\prime}(\mca{H}^\prime\times\mca{H}^\prime)\ar[d,shift left]\ar[d,shift right]\\
N\ar[r,"P"']&\mca{H}^\prime
\end{tikzcd}
\end{equation*}
since 
\begin{align*}
s^\prime(P(t(h)),(\varphi(h),\varphi^\prime(h)))&=\varphi(h)^{-1}P(t(h))\varphi^\prime(h)\\
&=\varphi(h)^{-1}P(t(h))P(t(h))^{-1}\varphi(h)P(s(h))\\
&=P(s(h))
\end{align*}
for all $h\in\mca{H}$. And $(Pt,(\varphi,\varphi^\prime))\in\Hom(\mca{H},\mca{H}^\prime\rtimes_{\gamma^\prime}(\mca{H}^\prime\times\mca{H}^\prime))$ since 
\begin{align*}
(Pt,(\varphi,\varphi^\prime))(hh^\prime)&=(P(t(h)),(\varphi(h)\varphi(h^\prime),\varphi^\prime(h)\varphi^\prime(h^\prime)))\\
&=(P(t(h)),(\varphi(h),\varphi^\prime(h)))(P(t(h^\prime)),(\varphi(h^\prime),\varphi^\prime(h^\prime)))\\
&=(Pt,(\varphi,\varphi^\prime))(h)(Pt,(\varphi,\varphi^\prime))(h^\prime)
\end{align*}
for all $h$, $h^\prime\in\mca{H}$. Hence the map $P\mapsto(Pt,(\varphi,\varphi^\prime))$ is defined. 

For the opposite direction let $\psi\in\Hom(\mca{H},\mca{H}^\prime\rtimes_{\gamma^\prime}(\mca{H}^\prime\times\mca{H}^\prime))$ be such that $\varphi_{\gamma^\prime}\psi=(\varphi,\varphi^\prime)$. Then $F_\psi\colon N\to\mca{H}^\prime$ and $\psi=(F_\psi t,(\varphi,\varphi^\prime))$. Since $(F_\psi(t(h)),(\varphi(h),\varphi^\prime(h)))\in\mca{H}^\prime\rtimes_{\gamma^\prime}(\mca{H}^\prime\times\mca{H}^\prime)$ for $h\in\mca{H}$, we have $(t^\prime,s^\prime)F_\psi(t(h))=t^\prime(\varphi(h),\varphi^\prime(h))$, hence $t^\prime F_\psi(t(h))=t^\prime\varphi(h)$, $s^\prime F_\psi(t(h))=t^\prime\varphi^\prime(h)$. By setting $h=1_y$ for $y\in N$, we get $t^\prime F_\psi(y)=F_\varphi(y)$, $s^\prime F_\psi(y)=F_{\varphi^\prime}(y)$. By the commutativity with source maps, we have 
\begin{equation*}
F_\psi(s(h))=s^\prime(F_\psi(t(h)),(\varphi(h),\varphi^\prime(h)))=\varphi(h)^{-1}F_\psi(t(h))\varphi^\prime(h), 
\end{equation*}
hence $\varphi^\prime(h)=F_\psi(t(h))^{-1}\varphi(h)F_\psi(s(h))$. Therefore $F_\psi\in\Nat(\varphi,\varphi^\prime)$. 
\end{proof}

\subsection{Orbitwise homotopies}
Recall that $I=[0,1]$. 

\begin{dfn}\label{orbitwiseh}
Let $\mca{H}\rightrightarrows N$, $\mca{H}^\prime\rightrightarrows N^\prime$ be Lie groupoids. 
\begin{itemize}
\item Let 
\begin{equation*}
C^\infty(\mca{O}_\mca{H},\mca{O}_{\mca{H}^\prime})=\left\{F\in C^\infty(N,N^\prime)\ \middle|\ 
\begin{gathered}
\text{for any $O\in\mca{O}_\mca{H}$, there exists}\\
\text{$O^\prime\in\mca{O}_{\mca{H}^\prime}$ such that $F(O)\subset O^\prime$}\\
\text{and $F\colon O\to O^\prime$ is $C^\infty$}
\end{gathered}
\right\}. 
\end{equation*}
\item Let $F$, $F^\prime\in C^\infty(\mca{O}_\mca{H},\mca{O}_{\mca{H}^\prime})$. An \emph{orbitwise homotopy} between $F$ and $F^\prime$ is a continuous map $H\colon N\times I\to N^\prime$ satisfying the following conditions: 
\begin{itemize}
\setlength\itemsep{0em}
\item $H(\cdot,t)\in C^\infty(N,N^\prime)$ for all $t\in I$
\item $H(\cdot,0)=F$, $H(\cdot,1)=F^\prime$
\item For any $O\in\mca{O}_\mca{H}$, there exists $O^\prime\in\mca{O}_{\mca{H}^\prime}$ such that $H(O\times I)\subset O^\prime$ and $H\colon O\times I\to O^\prime$ is continuous. 
\end{itemize}
We say that $F$ is \emph{orbitwise homotopic} to $F^\prime$ (written $F\sim_oF^\prime$) if there exists an orbitwise homotopy between $F$ and $F^\prime$. $\sim_o$ is an equivalence relation on $C^\infty(\mca{O}_\mca{H},\mca{O}_{\mca{H}^\prime})$. 
\item Let $\varphi$, $\varphi^\prime\in\Hom(\mca{H},\mca{H}^\prime)$. A continuous map $\eta\colon\mca{H}\times I\to\mca{H}^\prime$ is an \emph{orbitwise homotopy} between $\varphi$ and $\varphi^\prime$ if: 
\begin{itemize}
\setlength\itemsep{0em}
\item $\eta(\cdot,t)\in\Hom(\mca{H},\mca{H}^\prime)$ for all $t\in I$
\item $\eta(\cdot,0)=\varphi$, $\eta(\cdot,1)=\varphi^\prime$
\item $F_{\eta(\cdot,t)}$ is an orbitwise homotopy. 
\end{itemize}
We say that $\varphi$ is \emph{orbitwise homotopic} to $\varphi^\prime$ (written $\varphi\sim_o\varphi^\prime$) if there exists an orbitwise homotopy between $\varphi$ and $\varphi^\prime$. Then $\sim_o$ is an equivalence relation on $\Hom(\mca{H},\mca{H}^\prime)$. 
\end{itemize}
\end{dfn}

\begin{lem}\label{oosmhs}
Let $\mca{G}\rightrightarrows M$, $\mca{G}^\prime\rightrightarrows M^\prime$ be Lie groupoids ($M$, $M^\prime$ can be non Hausdorff, non second countable) and $\varphi\in\Hom(\mca{G},\mca{G}^\prime)$. Then $F_\varphi\in C^\infty(\mca{O}_\mca{G},\mca{O}_{\mca{G}^\prime})$. 
\end{lem}

\begin{proof}
Let $O\in\mca{O}_\mca{G}$. Take $O^\prime\in\mca{O}_{\mca{G}^\prime}$ such that $F_\varphi(O)\subset O^\prime$. It suffices to prove $F_\varphi\colon O\to O^\prime$ is $C^\infty$. Take $x\in O$ and put $x^\prime=F_\varphi(x)\in O^\prime$. We have 
\begin{equation*}
\begin{tikzcd}
\mca{G}^x\ar[r,"\varphi"]\ar[d,"s"']\ar[rd,phantom,"\circlearrowright"]&(\mca{G}^\prime)^{x^\prime}\ar[d,"s"]\\
O\ar[r,"F_\varphi"']&O^\prime. 
\end{tikzcd}
\end{equation*}
$\varphi\colon\mca{G}^x\to(\mca{G}^\prime)^{x^\prime}$ is $C^\infty$ since $(\mca{G}^\prime)^{x^\prime}$ is an embedded submanifold of $\mca{G}^\prime$. $s\colon\mca{G}^x\to O$ is a surjective submersion by Proposition \ref{isoorb}. Hence there exists a local $C^\infty$ sections of $s$ at any point of $O$. Therefore $F_\varphi\colon O\to O^\prime$ is $C^\infty$. 
\end{proof}

The proofs of the following two lemmas will be given in the next subsection. 

\begin{lem}\label{coocooc}
Let $\mca{H}\rightrightarrows N$, $\mca{H}^\prime\rightrightarrows N^\prime$ be Lie groupoids. Then the following statements hold: 
\begin{enumerate}
\item We have 
\begin{equation*}
C^\infty(\mca{O}_\mca{H},\mca{O}_{\mca{H}^\prime})=\left\{F\in C^\infty(N,N^\prime)\ \middle|\ 
\begin{gathered}
\text{for any $O\in\mca{O}_\mca{H}$, there exists}\\
\text{$O^\prime\in\mca{O}_{\mca{H}^\prime}$ such that $F(O)\subset O^\prime$}\\
\text{and $F\colon O\to O^\prime$ is continuous}
\end{gathered}
\right\}. 
\end{equation*}
\item If $\mca{H}^\prime$ is second countable, then 
\begin{equation*}
C^\infty(\mca{O}_\mca{H},\mca{O}_{\mca{H}^\prime})=\left\{F\in C^\infty(N,N^\prime)\ \middle|\ 
\begin{gathered}
\text{for any $O\in\mca{O}_\mca{H}$, there exists}\\
\text{$O^\prime\in\mca{O}_{\mca{H}^\prime}$ such that $F(O)\subset O^\prime$}
\end{gathered}
\right\}. 
\end{equation*}
\end{enumerate}
\end{lem}

\begin{lem}\label{hhscconor}
Let $\mca{H}\rightrightarrows N$, $\mca{H}^\prime\rightrightarrows N^\prime$ be Lie groupoids and $F$, $F^\prime\in C^\infty(\mca{O}_\mca{H},\mca{O}_{\mca{H}^\prime})$. Let $H\colon N\times I\to N^\prime$ be a continuous map satisfying the following conditions: 
\begin{itemize}
\setlength\itemsep{0em}
\item $H(\cdot,t)\in C^\infty(N,N^\prime)$ for all $t\in I$
\item $H(\cdot,0)=F$, $H(\cdot,1)=F^\prime$
\item For any $O\in\mca{O}_\mca{H}$, there exists $O^\prime\in\mca{O}_{\mca{H}^\prime}$ such that $H(O\times I)\subset O^\prime$. 
\end{itemize}
If $\mca{H}^\prime$ is second countable, then $H\colon O\times I\to O^\prime$ is continuous, hence $H$ is an orbitwise homotopy between $F$ and $F^\prime$. 
\end{lem}

\begin{example}
Let $S^1_d$ be $S^1$ with the discrete topology. Let $\rho\in\mca{A}(S^1,S^1)$, $\rho_d\in\mca{A}(S^1,S^1_d)$ be the actions defined by $\rho(x,y)=xy$ and $\rho_d(x,y)=xy$. Then $S^1\rtimes_{\rho_d}S^1_d$ is not second countable and $\mca{O}_{S^1\rtimes_\rho S^1}=\left\{O=S^1\right\}$, $\mca{O}_{S^1\rtimes_{\rho_d}S^1_d}=\left\{O_d=S^1_d\right\}$. 

For $\id\colon S^1\to S^1$, $\id\colon O\to O_d$ is not continuous, so $\id\not\in C^\infty(\mca{O}_{S^1\rtimes_\rho S^1},\mca{O}_{S^1\rtimes_{\rho_d}S^1_d})$. This shows the continuity (smoothness) condition of $F\colon O\to O^\prime$ in the definition of $C^\infty(\mca{O}_\mca{H},\mca{O}_{\mca{H}^\prime})$ is not superfluous. 

For $z\in S^1$, define $\varphi_z\in\Aut(S^1\rtimes_{\rho_d}S^1_d)$ by $\varphi_z(x,y)=(xz,y)$. Then $F_{\varphi_z}(x)=xz$. Take a nonconstant continuous map $c\colon I\to S^1$ and consider $\varphi_{c(t)}$. We have 
\begin{align*}
F_{\varphi_{c(t)}}\colon O_d\times I&\to O_d\\
(x,t)&\mapsto xc(t). 
\end{align*}
Since $O_d$ has the discrete topology, this is not continuous. So the continuity assumption on $H\colon O\times I\to O^\prime$ in Definition \ref{orbitwiseh} is not superfluous as well. 
\end{example}

\subsection{Proofs of Lemma \ref{coocooc} and Lemma \ref{hhscconor}}
\subsubsection{Stefan singular foliations}
\begin{dfn}[\cite{Stefan}]\label{stelmixlf}
Let $M$ be an $n$-dimensional $C^\infty$ manifold and $\mca{F}$ be a partition of $M$. We say that $\mca{F}$ is a \emph{$C^\infty$ Stefan singular foliation} of $M$ if: 
\begin{itemize}
\setlength\itemsep{0em}
\item for any $L\in\mca{F}$, there exists a $C^\infty$ manifold structure of $L$ such that: 
\begin{itemize}
\setlength\itemsep{0em}
\item $L$ is connected
\item the inclusion $\iota_L\colon L\to M$ is an immersion 
\item for a locally connected topological space $X$ and a map $f\colon X\to L$, $f\colon X\to L$ is continuous if and only if $\iota_Lf\colon X\to M$ is continuous
\end{itemize}
\item let $x_0\in L_0\in\mca{F}$ and $k=\dim L_0$. Then there exist an open neighborhood $U$ of $\iota_{L_0}(x_0)$ in $M$, an open neighborhood $U_1$ of $0$ in $\bb{R}^k$, an open neighborhood $U_2$ of $0$ in $\bb{R}^{n-k}$ and a diffeomorphism $\theta\colon U\to U_1\times U_2$ such that: 
\begin{itemize}
\setlength\itemsep{0em}
\item $\theta\iota_{L_0}(x_0)=(0,0)$
\item for any $L\in\mca{F}$, $\theta(\iota_L(L)\cap U)=U_1\times S_L$, where 
\begin{equation*}
S_L=\left\{u_2\in U_2\ \middle|\ (0,u_2)\in\theta(\iota_L(L)\cap U)\right\}. 
\end{equation*}
\end{itemize}
\end{itemize}


In \cite{Stefan} a $C^\infty$ Stefan singular foliation of $M$ is called a $C^\infty$ foliation of $M$ with singularities. 
\end{dfn}

\begin{lem}\label{nlnmcinf}
Let $M$ be a $C^\infty$ manifold, $\mca{F}$ be a $C^\infty$ Stefan singular foliation of $M$ and $L\in\mca{F}$. Then for a $C^\infty$ manifold $N$ and a map $f\colon N\to L$, $f\colon N\to L$ is $C^\infty$ if and only if $\iota_Lf\colon N\to M$ is $C^\infty$. 
\end{lem}

\begin{proof}
Let $k=\dim L$. Assume $\iota_Lf\colon N\to M$ is $C^\infty$. Let $y_0\in N$. Since $\iota_L\colon L\to M$ is an immersion, there exist a chart $(U,\varphi)$ of $L$ and a chart $(V,\psi)$ of $M$ such that $f(y_0)\in U$, $\iota_Lf(y_0)\in V$, $\iota_L(U)\subset V$ and 
\begin{equation}
(\psi\iota_L\varphi^{-1})(x_1,\ldots,x_k)=(x_1,\ldots,x_k,0,\ldots,0)^\top\label{xxxx0}
\end{equation}
for all $(x_1,\ldots,x_k)\in\varphi(U)$. Since $\iota_Lf$ is continuous and $\mca{F}$ is a $C^\infty$ Stefan singular foliation, $f$ is continuous. So there exists a chart $(W,\theta)$ of $N$ such that $y_0\in W$ and $f(W)\subset U$. $(\psi\iota_L\varphi^{-1})(\varphi f\theta^{-1})=\psi\iota_Lf\theta^{-1}$ is $C^\infty$. By \eqref{xxxx0}, $\varphi f\theta^{-1}$ is $C^\infty$. So $f$ is $C^\infty$. 
\end{proof}

\begin{lem}\label{slcount}
Let $M$ be an $n$-dimensional $C^\infty$ manifold and $\mca{F}$ be a $C^\infty$ Stefan singular foliation of $M$. Let $x_0\in L_0\in\mca{F}$, $k=\dim L_0$, $U$ be an open neighborhood of $\iota_{L_0}(x_0)$ in $M$, $U_1$ be an open neighborhood of $0$ in $\bb{R}^k$, $U_2$ be an open neighborhood of $0$ in $\bb{R}^{n-k}$ and $\theta\colon U\to U_1\times U_2$ be a diffeomorphism such that: 
\begin{itemize}
\setlength\itemsep{0em}
\item $\theta\iota_{L_0}(x_0)=(0,0)$
\item for any $L\in\mca{F}$, $\theta(\iota_L(L)\cap U)=U_1\times S_L$, where 
\begin{equation*}
S_L=\left\{u_2\in U_2\ \middle|\ (0,u_2)\in\theta(\iota_L(L)\cap U)\right\}. 
\end{equation*}
\end{itemize}
Then the following statements hold: 
\begin{enumerate}
\setlength\itemsep{0em}
\item For any $L\in\mca{F}$ and $x\in\iota_L^{-1}(U)$, let $\theta\iota_L(x)=(a_1,a_2)$. Then 
\begin{equation*}
(\iota_L)_*T_xL\supset T_{\iota_L(x)}\theta^{-1}(U_1\times\left\{a_2\right\}). 
\end{equation*}
So we have $(\iota_L)_*T_xL=T_{\iota_L(x)}\theta^{-1}(U_1\times\left\{a_2\right\})$ if $\dim L=k$. 
\item Let $L\in\mca{F}$ be such that $\dim L=k$. 
\begin{enumerate}
\item If $V$ is a connected open subset of $\iota_L^{-1}(U)$, then $p_2\theta\iota_L(V)$ is a point, where $p_2\colon U_1\times U_2\to U_2$ is the projection. 
\item $S_L$ is countable. 
\end{enumerate}
\end{enumerate}
\end{lem}

\begin{proof}
1. Let $v\in T_{\iota_L(x)}\theta^{-1}(U_1\times\left\{a_2\right\})$. There exists a $C^\infty$ map $c\colon\bb{R}\to M$ such that $c(0)=\iota_L(x)$, $\frac{dc}{dt}(0)=v$ and $c(\bb{R})\subset\theta^{-1}(U_1\times\left\{a_2\right\})\subset\iota_L(L)\cap U$. There exists a unique map $c^\prime\colon\bb{R}\to L$ such that $c=\iota_Lc^\prime$. Then $c^\prime$ is $C^\infty$ by Lemma \ref{nlnmcinf}. So $v=(\iota_L)_*\frac{dc^\prime}{dt}(0)\in(\iota_L)_*T_xL$. 

2.(a). $V$ is smoothly path connected. Let $c\colon I\to V$ be any $C^\infty$ map. Then $\frac{d}{dt}(p_2\theta\iota_Lc)(t)=0$ by 1. Hence $p_2\theta\iota_Lc$ is constant. Therefore $p_2\theta\iota_L(V)$ is a point. 

2.(b). Since $p_2\theta\iota_L\colon\iota_L^{-1}(U)\to S_L$ is surjective, there exists a map $\sigma\colon S_L\to\iota_L^{-1}(U)$ such that $p_2\theta\iota\sigma=\id$. Take a connected open neighborhood $V(u_2)$ of $\sigma(u_2)$ in $\iota_L^{-1}(U)$ for each $u_2\in S_L$. Then $p_2\theta\iota_LV(u_2)=\left\{u_2\right\}$ for any $u_2\in S_L$ by the above claim. It follows that $V(u_2)\cap V(u_2^\prime)=\emptyset$ if $u_2$, $u_2^\prime\in S_L$ and $u_2\neq u_2^\prime$. Hence $\left\{V(u_2)\ \middle|\ u_2\in S_L\right\}$ is countable since $\iota_L^{-1}(U)$ is second countable. Therefore $S_L$ is countable. 
\end{proof}

\begin{dfn}\label{sigdistmdm}
Let $M$ be a $C^\infty$ manifold and $\Delta\subset TM$. We say that $\Delta$ is a \emph{singular distribution} on $M$ if for each $x\in M$, $\Delta_x=\Delta\cap T_xM$ is a subspace of $T_xM$. 
\end{dfn}

\begin{dfn}
Let $M$ be a $C^\infty$ manifold and $\Delta$ be a singular distribution on $M$. 
\begin{itemize}
\item A \emph{local $C^\infty$ section} of $\Delta$ is a $C^\infty$ section $X$ of $TM|_U$ for some open subset $U$ of $M$ such that $X_x\in\Delta_x$ for any $x\in U$. 
\item $\Delta$ is $C^\infty$ if for any $x\in M$ and $v\in\Delta_x$, there exists a local $C^\infty$ section $X$ of $\Delta$ such that $X_x=v$. 
\item $L$ is an \emph{integral submanifold} of $\Delta$ if $L$ is a connected immersed submanifold of $M$ such that $T_xL=\Delta_x$ for any $x\in L$. 
\item $L$ is a \emph{maximal integral submanifold} of $\Delta$ if $L$ is an integral submanifold of $\Delta$ and if $L^\prime$ is an integral submanifold of $\Delta$ containing $L$, then $L^\prime=L$. 
\item $\Delta$ is \emph{integrable} if for any $x\in M$, there exists an integral submanifold $L$ of $\Delta$ containing $x$. 
\end{itemize}
\end{dfn}

\begin{lem}\label{ssfisd}
Let $M$ be a $C^\infty$ manifold. 
\begin{enumerate}
\item Let $\Delta$ be an integrable $C^\infty$ singular distribution on $M$. Then the following statements hold: 
\begin{enumerate}
\item For any $x\in M$, there exists a unique maximal integral submanifold of $\Delta$ containing $x$. 
\item Let $\mca{F}$ be the set of all maximal integral submanifolds of $\Delta$. Then $\mca{F}$ is a partition of $M$. 
\end{enumerate}
\item Let $\mca{F}$ be a $C^\infty$ Stefan singular foliation of $M$. Then 
\begin{equation*}
T\mca{F}=\bigcup_{L\in\mca{F}}TL\subset TM
\end{equation*}
is an integrable $C^\infty$ singular distribution on $M$. 
\end{enumerate}
\end{lem}

\begin{proof}
1.(a) See 3.22. Lemma in \cite{Michor}. 

1.(b) This follows from 1.(a). 

2. $T\mca{F}$ is an integrable singular distribution on $M$. We prove that $T\mca{F}$ is $C^\infty$. Let $x_0\in L_0\in\mca{F}$ and $v\in T_{\iota_{L_0}(x_0)}\mca{F}$. Let $k=\dim L_0$. There exist an open neighborhood $U$ of $\iota_{L_0}(x_0)$ in $M$, an open neighborhood $U_1$ of $0$ in $\bb{R}^k$, an open neighborhood $U_2$ of $0$ in $\bb{R}^{n-k}$ and a diffeomorphism $\theta\colon U\to U_1\times U_2$ such that: 
\begin{itemize}
\setlength\itemsep{0em}
\item $\theta\iota_{L_0}(x_0)=(0,0)$
\item for any $L\in\mca{F}$, $\theta(\iota_L(L)\cap U)=U_1\times S_L$, where 
\begin{equation*}
S_L=\left\{u_2\in U_2\ \middle|\ (0,u_2)\in\theta(\iota_L(L)\cap U)\right\}. 
\end{equation*}
\end{itemize}
Then 
\begin{equation*}
v\in T_{\iota_{L_0}(x_0)}\mca{F}=(\iota_{L_0})_*T_{x_0}L_0=T_{\iota_{L_0}(x_0)}\theta^{-1}(U_1\times\left\{0\right\})
\end{equation*}
by Lemma \ref{slcount}. Hence $\theta_*v\in T_0U_1\times\left\{0\right\}\subset T_{(0,0)}(U_1\times U_2)$. Let $X$ be the constant vector field on $U_1\times U_2\subset\bb{R}^n$ with value $\theta_*v$. Let $L\in\mca{F}$ and $x\in\iota_L^{-1}(U)$. Let $\theta\iota_L(x)=(a_1,a_2)$. Then 
\begin{equation*}
T_{\iota_L(x)}\mca{F}=(\iota_L)_*T_xL\supset T_{\iota_L(x)}\theta^{-1}(U_1\times\left\{a_2\right\})\ni(\theta^*X)_{\iota_L(x)}
\end{equation*}
by Lemma \ref{slcount}. $\theta^*X$ is a $C^\infty$ local section of $T\mca{F}$ on $U$ such that $(\theta^*X)_{\iota_{L_0}(x_0)}=v$. Hence $T\mca{F}$ is $C^\infty$. 
\end{proof}

\begin{prop}\label{steinsin}
Let $M$ be a $C^\infty$ manifold. Then we have a bijection 
\begin{align*}
\left\{\text{$C^\infty$ Stefan singular foliations of $M$}\right\}&\simeq\left\{\text{integrable $C^\infty$ singular distributions on $M$}\right\}\\
\mca{F}&\mapsto T\mca{F}\\
\left\{\text{maximal integral submanifolds of $\Delta$}\right\}&\mapsfrom\Delta. 
\end{align*}
\end{prop}

\begin{proof}
The map $\mca{F}\mapsto T\mca{F}$ is defined by 2 in Lemma \ref{ssfisd}. 

We prove the well-definedness of the map from right to left. Let $\Delta$ be an integrable $C^\infty$ singular distribution on $M$. Let $\Gamma$ be the set of all local $C^\infty$ sections of $\Delta$, $A$ be the set of local flows generated by the elements of $\Gamma$ and $\mca{F}$ be the set of accessible sets of $A$ (see \cite{Stefan}). Define singular distributions $\Delta A$, $\ol{\Delta A}$ on $M$ by 
\begin{align*}
(\Delta A)_x&=\left\{\frac{d}{dt}a(t,y)\in T_xM\ \middle|\ a\in A,a(t,y)=x\right\}, \\
(\ol{\Delta A})_x&=\left\{a(t,\cdot)_*v\in T_xM\ \middle|\ a\in A,a(t,y)=x,v\in(\Delta A)_y\right\}
\end{align*}
for $x\in M$. (Note that $\Delta A\subset\ol{\Delta A}$.) Then $\mca{F}$ is a $C^\infty$ Stefan singular foliation of $M$ and $T\mca{F}=\ol{\Delta A}$ by Theorem 1 in \cite{Stefan}. Obviously $\Delta A=\Delta$. We have $\ol{\Delta A}=\Delta$ since $\Delta$ is integrable (see (1) $\Rightarrow$ (2) of 3.24 Lemma in \cite{Michor}). Thus $T\mca{F}=\Delta$. In particular an accessible set of $A$ is an integral submanifold of $\Delta$. Hence an accessible set of $A$ is a maximal integral submanifold of $\Delta$. So $\mca{F}$ is the set of all maximal integral submanifolds of $\Delta$. Therefore $\left\{\text{maximal integral submanifolds of $\Delta$}\right\}\mapsfrom\Delta$ is well-defined. 

It is easy to see two maps are inverse to each other. 
\end{proof}

\begin{cor}\label{mrfmssf}
Let $M$ be a $C^\infty$ manifold and $\mca{F}$ be a regular $C^\infty$ foliation of $M$. Then $\mca{F}$ is a $C^\infty$ Stefan singular foliation of $M$. 
\end{cor}

\begin{proof}
$T\mca{F}$ is an integrable $C^\infty$ singular distribution on $M$. Since $\mca{F}$ is the set of maximal integral submanifolds of $T\mca{F}$, $\mca{F}$ is a $C^\infty$ Stefan singular foliation of $M$ by Proposition \ref{steinsin}. 
\end{proof}

\begin{dfn}
For a Lie groupoid $\mca{G}\rightrightarrows M$, let $\mca{F}_\mca{G}$ be the set of connected components of orbits of $\mca{G}$. 
\end{dfn}

\begin{prop}\label{gmstsif}
Let $\mca{G}\rightrightarrows M$ be a Lie groupoid. Then $\mca{F}_\mca{G}$ is a $C^\infty$ Stefan singular foliation of $M$. 
\end{prop}

\begin{proof}
$\mca{F}_\mca{G}$ is a partition of $M$ and each element of $\mca{F}_\mca{G}$ is a connected immersed submanifold of $M$ by Proposition \ref{isoorb}. So $T\mca{F}_\mca{G}=\bigcup_{L\in\mca{F}_\mca{G}}TL\subset TM$ is an integrable singular distribution on $M$. 

Let $(\Lie\mca{G}\to M,\mt{a},[\cdot,\cdot])$ be the Lie algebroid of $\mca{G}\rightrightarrows M$. By definition 
\begin{equation*}
\Lie\mca{G}=(T\mca{G})|_{1_M}/T1_M,\quad\mt{a}=s_*-t_*\colon\Lie\mca{G}\to TM. 
\end{equation*}
We have $(\ker t_*)|_{1_M}\simeq\Lie\mca{G}$ and $\mt{a}=s_*\colon(\ker t_*)|_{1_M}\to TM$. 

\begin{claim}
$T\mca{F}_\mca{G}=\mt{a}(\Lie\mca{G})$. 
\end{claim}

\begin{proof}
Let $x\in M$ and $x\in O\in\mca{O}_\mca{G}$. Then $s\colon\mca{G}^x\to O$ is a submersion by Proposition \ref{isoorb}. Hence 
\begin{equation*}
T_x\mca{F}_\mca{G}=T_xO=s_*T_{1_x}\mca{G}^x=s_*(\ker t_*)_{1_x}=\mt{a}((\Lie\mca{G})_x). \qedhere
\end{equation*}
\end{proof}

This shows $T\mca{F}_\mca{G}$ is $C^\infty$ since $\Lie\mca{G}$ is a vector bundle on $M$. Each element of $\mca{F}_\mca{G}$ is an integral submanifold of $T\mca{F}_\mca{G}$. Hence each element of $\mca{F}_\mca{G}$ is a maximal integral submanifold of $T\mca{F}_\mca{G}$. Therefore $T\mca{F}_\mca{G}$ corresponds to $\mca{F}_\mca{G}$ by the bijection in Proposition \ref{steinsin}. Hence $\mca{F}_\mca{G}$ is a $C^\infty$ Stefan singular foliation of $M$. 
\end{proof}

\subsubsection{Proofs of Lemma \ref{coocooc} and Lemma \ref{hhscconor}}
\begin{lem}\label{smallvsv}
Let $M$ be an $n$-dimensional $C^\infty$ manifold and $\mca{F}$ be a $C^\infty$ Stefan singular foliation of $M$. Let $x\in L_0\in\mca{F}$, $k=\dim L_0$ and $U$ be an open neighborhood of $x$ in $L_0$. Then there exist an open neighborhood $V$ of $\iota_{L_0}(x)$ in $M$, an open neighborhood $V_1$ of $0$ in $\bb{R}^k$, an open neighborhood $V_2$ of $0$ in $\bb{R}^{n-k}$ and a diffeomorphism $\theta\colon V\to V_1\times V_2$ such that: 
\begin{itemize}
\setlength\itemsep{0em}
\item $\theta\iota_{L_0}(x)=(0,0)$
\item for any $L\in\mca{F}$, $\theta(\iota_L(L)\cap V)=V_1\times S_L$, where 
\begin{equation*}
S_L=\left\{v_2\in V_2\ \middle|\ (0,v_2)\in\theta(\iota_L(L)\cap V)\right\}
\end{equation*}
\item $\theta^{-1}(V_1\times\left\{0\right\})\subset\iota_{L_0}(U)$. 
\end{itemize}
\end{lem}

\begin{proof}
There exist an open neighborhood $V^\prime$ of $\iota_{L_0}(x)$ in $M$, an open neighborhood $V^\prime_1$ of $0$ in $\bb{R}^k$, an open neighborhood $V_2$ of $0$ in $\bb{R}^{n-k}$ and a diffeomorphism $\theta\colon V^\prime\to V^\prime_1\times V_2$ such that: 
\begin{itemize}
\setlength\itemsep{0em}
\item $\theta\iota_{L_0}(x)=(0,0)$
\item for any $L\in\mca{F}$, $\theta(\iota_L(L)\cap V^\prime)=V^\prime_1\times S_L$, where 
\begin{equation*}
S_L=\left\{v_2\in V_2\ \middle|\ (0,v_2)\in\theta(\iota_L(L)\cap V^\prime)\right\}. 
\end{equation*}
\end{itemize}
Since $(\iota_{L_0})_*T_xL_0=T_{\iota_{L_0}(x)}\theta^{-1}(V^\prime_1\times\left\{0\right\})$ by 1 in Lemma \ref{slcount}, $(p_1\theta\iota_{L_0})_*\colon T_xL_0\to T_0V^\prime_1$ is an isomorphism. Hence there exist a connected open neighborhood $U_1$ of $x$ in $L_0$ and an open neighborhood $V_1$ of $0$ in $V^\prime_1$ such that $U_1\subset U$ and $p_1\theta\iota_{L_0}\colon U_1\to V_1$ is a diffeomorphism. Let $V=\theta^{-1}(V_1\times V_2)$. We have $p_2\theta\iota_{L_0}(U_1)=\left\{0\right\}$ by 2.(a) in Lemma \ref{slcount}. Hence $\theta\iota_{L_0}(U_1)=V_1\times\left\{0\right\}$ and so $\theta^{-1}(V_1\times\left\{0\right\})=\iota_{L_0}(U_1)\subset\iota_{L_0}(U)$. 
\end{proof}

\begin{lem}\label{topopsec}
Let $X$, $Y$ be topological spaces and $f\colon X\to Y$ be a continuous open map. If $X$ is second countable, then $f(X)$ is second countable. 
\end{lem}

\begin{lem}\label{locconncon}
Let $\mca{G}\rightrightarrows M$ be a Lie groupoid such that $\mca{G}$ is second countable, $O\in\mca{O}_\mca{G}$, $\iota\colon O\to M$ be the inclusion map, $X$ be a locally connected topological space and $f\colon X\to O$ be a map. Then $f\colon X\to O$ is continuous with respect to the submanifold topology of $O$ if and only if $\iota f\colon X\to M$ is continuous. 
\end{lem}

\begin{proof}
Assume $\iota f\colon X\to M$ is continuous. Let $x_0\in X$ and $V$ be an open neighborhood of $f(x_0)$ in $O$. It suffices to find an open neighborhood $U$ of $x_0$ in $X$ such that $f(U)\subset V$. Let $f(x_0)\in O_0\in\mca{F}_\mca{G}$. $\mca{F}_\mca{G}$ is a $C^\infty$ Stefan singular foliation of $M$ by Proposition \ref{gmstsif}. Let $n=\dim M$ and $k=\dim O$. By Lemma \ref{smallvsv} there exist an open neighborhood $W$ of $\iota f(x_0)$ in $M$, an open neighborhood $W_1$ of $0$ in $\bb{R}^k$, an open neighborhood $W_2$ of $0$ in $\bb{R}^{n-k}$ and a diffeomorphism $\theta\colon W\to W_1\times W_2$ such that: 
\begin{itemize}
\setlength\itemsep{0em}
\item $\theta\iota f(x_0)=(0,0)$
\item if $L\in\mca{F}_\mca{G}$, then there exists $S_L\subset W_2$ such that $\theta(\iota_L(L)\cap W)=W_1\times S_L$
\item $\theta^{-1}(W_1\times\left\{0\right\})=(p_2\theta)^{-1}(0)\subset\iota(V)$, where $p_2\colon W_1\times W_2\to W_2$ is the projection. 
\end{itemize}
Since $\mca{G}$ is second countable, $\mca{G}^{f(x_0)}$ is second countable. Since $s\colon\mca{G}^{f(x_0)}\to O$ is a continuous open surjective map, $O$ is second countable by Lemma \ref{topopsec}. So $O$ has countably many connected components: Let $A$ be a countable set such that $0\in A$, $O=\bigsqcup_{a\in A}O_a$, where $O_a$ is a connected component of $O$ for each $a\in A$. We have $O_a\in\mca{F}_\mca{G}$. Put $S_O=\bigcup_{a\in A}S_{O_a}\subset W_2$. Then $\theta(\iota(O)\cap W)=W_1\times S_O$. Since $S_{O_a}$ is countable by Lemma \ref{slcount}, $S_O$ is countable. $S_O$ is totally disconnected since a connected metrizable space containing two different points is uncountable (the image of a distance function is a connected subset of $\bb{R}$). By the continuity of $\iota f\colon X\to M$, there exists a connected open neighborhood $U$ of $x_0$ in $X$ such that $\iota f(U)\subset W$. Note that $\iota f(U)\subset\iota(O)\cap W$. So $p_2\theta\iota f(U)\subset S_O$ is defined and connected, and $0\in p_2\theta\iota f(U)$. Hence $\left\{0\right\}=p_2\theta\iota f(U)$ by the total disconnectedness. Therefore $f(U)\subset(p_2\theta\iota)^{-1}(0)\subset\iota^{-1}\iota(V)=V$. So $f\colon X\to O$ is continuous at $x_0$ and $f\colon X\to O$ is continuous. 
\end{proof}

\begin{proof}[Proof of Lemma \ref{coocooc}]
1. Take $F$ from the right hand side of the equality. Let $O\in\mca{O}_\mca{H}$ and $O^\prime\in\mca{O}_{\mca{H}^\prime}$ be such that $F(O)\subset O^\prime$. Let $O_0$ be a connected component of $O$. Since $F\colon O_0\to O^\prime$ is continuous, there exists a connected component $O^\prime_0$ of $O^\prime$ such that $F(O_0)\subset O^\prime_0$. By Proposition \ref{gmstsif}, $\mca{F}_{\mca{H}^\prime}$ is a $C^\infty$ Stefan singular foliation of $N^\prime$ and $O^\prime_0\in\mca{F}_{\mca{H}^\prime}$. By Lemma \ref{nlnmcinf}, $F\colon O_0\to O^\prime_0$ is $C^\infty$ since $F\colon O_0\to N^\prime$ is $C^\infty$. So $F\colon O\to O^\prime$ is $C^\infty$. 

2. Take $F$ from the right hand side of the equality. Let $O\in\mca{O}_\mca{H}$ and $O^\prime\in\mca{O}_{\mca{H}^\prime}$ be such that $F(O)\subset O^\prime$. By Lemma \ref{locconncon}, $F\colon O\to O^\prime$ is continuous since $F\colon O\to N^\prime$ is continuous. Hence $F\in C^\infty(\mca{O}_\mca{H},\mca{O}_{\mca{H}^\prime})$ by 1. 
\end{proof}

\begin{proof}[Proof of Lemma \ref{hhscconor}]
$\mca{H}^\prime\rightrightarrows N^\prime$ is a Lie groupoid such that $\mca{H}^\prime$ is second countable, $O^\prime\in\mca{O}_{\mca{H}^\prime}$, $O\times I$ is locally connected and we have a map $H\colon O\times I\to O^\prime$. Since $H\colon O\times I\to N^\prime$ is continuous, $H\colon O\times I\to O^\prime$ is continuous by Lemma \ref{locconncon}. 
\end{proof}

\subsection{$nt$-homotopies}
\begin{dfn}\label{nthomptppp}
Let $\mca{H}\rightrightarrows N$, $\mca{H}^\prime\rightrightarrows N^\prime$ be Lie groupoids and $\varphi$, $\varphi^\prime\in\Hom(\mca{H},\mca{H}^\prime)$. An \emph{$nt$-homotopy} between $\varphi$ and $\varphi^\prime$ is a continuous map $P\colon N\times I\to\mca{H}^\prime$ satisfying: 
\begin{itemize}
\setlength\itemsep{0em}
\item $P(\cdot,t)\in C^\infty(N,\mca{H}^\prime)$ for all $t\in I$
\item $t^\prime P(\cdot,t)=F_\varphi$ for all $t\in I$
\item $P(\cdot,0)=1_{F_\varphi}$
\item $\varphi P(\cdot,1)=\varphi^\prime$. 
\end{itemize}
We say that $\varphi$ is \emph{$nt$-homotopic} to $\varphi^\prime$ (written $\varphi\sim_{nt}\varphi^\prime$) if there exists an $nt$-homotopy between $\varphi$ and $\varphi^\prime$. 
\end{dfn}

\begin{lem}
Let $\mca{H}\rightrightarrows N$, $\mca{H}^\prime\rightrightarrows N^\prime$ be Lie groupoids. Then $\sim_{nt}$ is an equivalence relation on $\Hom(\mca{H},\mca{H}^\prime)$. 
\end{lem}

\begin{proof}
For $\varphi\in\Hom(\mca{H},\mca{H}^\prime)$, 
\begin{align*}
P\colon N\times I&\to\mca{H}^\prime\\
(y,t)&\mapsto1_{F_\varphi(y)}
\end{align*}
is an $nt$-homotopy between $\varphi$ and $\varphi$. 

For $\varphi$, $\varphi^\prime\in\Hom(\mca{H},\mca{H}^\prime)$ such that $\varphi\sim_{nt}\varphi^\prime$, take an $nt$-homotopy $P$ between $\varphi$ and $\varphi^\prime$. Then 
\begin{align*}
P^\prime\colon N\times I&\to\mca{H}^\prime\\
(y,t)&\mapsto P(y,1)^{-1}P(y,1-t)
\end{align*}
is an $nt$-homotopy between $\varphi^\prime$ and $\varphi$. 

Finally assume $\varphi\sim_{nt}\varphi^\prime$ and $\varphi^\prime\sim_{nt}\varphi^{\prime\prime}$. Let $P$ be an $nt$-homotopy between $\varphi$ and $\varphi^\prime$, and $P^\prime$ be an $nt$-homotopy between $\varphi^\prime$ and $\varphi^{\prime\prime}$. Then 
\begin{align*}
P^{\prime\prime}\colon N\times I&\to\mca{H}^\prime\\
(y,t)&\mapsto
\begin{cases}
P(y,2t)&0\leq t\leq\frac{1}{2}\\
P^\prime(y,2t-1)P(y,1)&\frac{1}{2}\leq t\leq1
\end{cases}
\end{align*}
gives an $nt$-homotopy between $\varphi$ and $\varphi^{\prime\prime}$. 
\end{proof}

\begin{lem}\label{hnhnphipoh}
Let $\mca{H}\rightrightarrows N$, $\mca{H}^\prime\rightrightarrows N^\prime$ be Lie groupoids, $\varphi$, $\varphi^\prime\in\Hom(\mca{H},\mca{H}^\prime)$ and $P$ be an $nt$-homotopy between $\varphi$ and $\varphi^\prime$. Then $\varphi P(\cdot,t)\in\Hom(\mca{H},\mca{H}^\prime)$ is an orbitwise homotopy between $\varphi$ and $\varphi^\prime$. 
\end{lem}

\begin{proof}
We have $F_{\varphi P(\cdot,t)}=s^\prime P(\cdot,t)$ for all $t\in I$. Put 
\begin{align*}
H\colon N\times I&\to N^\prime\\
(y,t)&\mapsto s^\prime P(y,t). 
\end{align*}
Let $O\in\mca{O}_\mca{H}$. There exists $O^\prime\in\mca{O}_{\mca{H}^\prime}$ such that $F_\varphi(O)\subset O^\prime$. Then $H(O\times I)\subset O^\prime$. We need to show that $H\colon O\times I\to O^\prime$ is continuous. Let $y_0\in O$. There exists an open neighborhood $U$ of $y_0$ in $O$ and a $C^\infty$ local section $\sigma\colon U\to\mca{H}^{y_0}$ of $s\colon\mca{H}^{y_0}\to O$. Put $y_0^\prime=F_\varphi(y_0)\in O^\prime$. The map 
\begin{align*}
\wt{H}\colon U\times I&\to(\mca{H}^\prime)^{y_0^\prime}\\
(y,t)&\mapsto\varphi(\sigma(y))P(y,t)
\end{align*}
is continuous. We have 
\begin{equation*}
\begin{tikzcd}
&(\mca{H}^\prime)^{y_0^\prime}\ar[d,"s^\prime"]\\
U\times I\ar[ur,"\wt{H}"{name=U}]\ar[r,"H"']&O^\prime. \ar[to=U,phantom,"\circlearrowright"{xshift=3,yshift=-1}]
\end{tikzcd}
\end{equation*}
Hence $H\colon U\times I\to O^\prime$ is continuous. 
\end{proof}

\begin{lem}\label{simsim}
Let $\mca{H}\rightrightarrows N$, $\mca{H}^\prime\rightrightarrows N^\prime$, $\mca{H}^{\prime\prime}\rightrightarrows N^{\prime\prime}$ be Lie groupoids, $\varphi$, $\varphi^\prime\in\Hom(\mca{H},\mca{H}^\prime)$ and $\psi$, $\psi^\prime\in\Hom(\mca{H}^\prime,\mca{H}^{\prime\prime})$. Let $a\in\left\{o,nt\right\}$. If $\varphi\sim_a\varphi^\prime$ and $\psi\sim_a\psi^\prime$, then $\psi\varphi\sim_a\psi^\prime\varphi^\prime$. 
\end{lem}

\begin{proof}
Let $a=o$. Let $\varphi_t\in\Hom(\mca{H},\mca{H}^\prime)$ be an orbitwise homotopy between $\varphi$ and $\varphi^\prime$ and $\psi_t\in\Hom(\mca{H}^\prime,\mca{H}^{\prime\prime})$ be an orbitwise homotopy between $\psi$ and $\psi^\prime$. We have $\psi_t\varphi_t\in\Hom(\mca{H},\mca{H}^{\prime\prime})$ and $F_{\psi_t\varphi_t}=F_{\psi_t}F_{\varphi_t}$. Let $O\in\mca{O}_\mca{H}$. There exists $O^\prime\in\mca{O}_{\mca{H}^\prime}$ such that $F_{\varphi_t}(O)\subset O^\prime$ for all $t\in I$, and there exists $O^{\prime\prime}\in\mca{O}_{\mca{H}^{\prime\prime}}$ such that $F_{\psi_t}(O^\prime)\subset O^{\prime\prime}$ for all $t\in I$. Hence $F_{\psi_t\varphi_t}(O)\subset O^{\prime\prime}$ and $\psi_t\varphi_t$ is an orbitwise homotopy between $\psi\varphi$ and $\psi^\prime\varphi^\prime$. 

Let $a=nt$. Let $P\colon N\times I\to\mca{H}^\prime$ be an $nt$-homotopy between $\varphi$ and $\varphi^\prime$, and $Q\colon N^\prime\times I\to\mca{H}^{\prime\prime}$ be an $nt$-homotopy between $\psi$ and $\psi^\prime$. Then 
\begin{align*}
N\times I&\to\mca{H}^{\prime\prime}\\
(y,t)&\mapsto\psi P(y,t)
\end{align*}
is an $nt$-homotopy between $\psi\varphi$ and $\psi\varphi^\prime$ and 
\begin{align*}
N\times I&\to\mca{H}^{\prime\prime}\\
(y,t)&\mapsto Q(F_{\varphi^\prime}(y),t)
\end{align*}
is an $nt$-homotopy between $\psi\varphi^\prime$ and $\psi^\prime\varphi^\prime$. 
\end{proof}

\subsection{$\mathrel{\ul{\sim_o}}$ and $\mathrel{\ul{\sim_{nt}}}$}
\begin{dfn}
Let $\mca{G}\rightrightarrows M$ be a Lie groupoid, $N$ be a $C^\infty$ manifold, $\nu\colon N\to M$ be a $C^\infty$ map and $\rho\in\mca{A}(\nu,\mca{G})$. An orbit of $N\rtimes_\rho\mca{G}$ is called an orbit of $\rho$, and $\mca{O}_{N\rtimes_\rho\mca{G}}$ is denoted by $\mca{O}_\rho$. 
\end{dfn}

\begin{dfn}\label{lgpcccffsimo}
Let $\mca{G}\rightrightarrows M$ be a Lie groupoid, $N$, $N^\prime$ be $C^\infty$ manifolds, $\nu\colon N\to M$, $\nu^\prime\colon N^\prime\to M$ be $C^\infty$ maps and $\rho\in\mca{A}(\nu,\mca{G})$, $\rho^\prime\in\mca{A}(\nu^\prime,\mca{G})$. 
\begin{itemize}
\item Put 
\begin{equation*}
\ul{C}^\infty(\mca{O}_\rho,\mca{O}_{\rho^\prime})=C^\infty(\mca{O}_\rho,\mca{O}_{\rho^\prime})\cap\ul{C}^\infty(N,N^\prime). 
\end{equation*}
\item For $F$, $F^\prime\in\ul{C}^\infty(\mca{O}_\rho,\mca{O}_{\rho^\prime})$, $F\mathrel{\ul{\sim_o}}F^\prime$ if there exists an orbitwise homotopy $H\colon N\times I\to N^\prime$ between $F$ and $F^\prime$ such that $H(\cdot,t)\in\ul{C}^\infty(N,N^\prime)$ for all $t\in I$. 
\item For $\varphi$, $\varphi^\prime\in\ul{\Hom}(N\rtimes_\rho\mca{G},N^\prime\rtimes_{\rho^\prime}\mca{G})$, $\varphi\mathrel{\ul{\sim_o}}\varphi^\prime$ if there exists an orbitwise homotopy $\eta\colon(N\rtimes_\rho\mca{G})\times I\to N^\prime\rtimes_{\rho^\prime}\mca{G}$ between $\varphi$ and $\varphi^\prime$ such that $\eta(\cdot,t)\in\ul{\Hom}(N\rtimes_\rho\mca{G},N^\prime\rtimes_{\rho^\prime}\mca{G})$ for all $t\in I$. $\mathrel{\ul{\sim_o}}$ is an equivalence relation on $\ul{\Hom}(N\rtimes_\rho\mca{G},N^\prime\rtimes_{\rho^\prime}\mca{G})$. 
\end{itemize}
\end{dfn}

\begin{dfn}\label{raragpppp}
Let $\mca{G}\rightrightarrows M$ be a Lie groupoid, $N$, $N^\prime$ be $C^\infty$ manifolds, $\nu\colon N\to M$, $\nu^\prime\colon N^\prime\to M$ be $C^\infty$ maps, $\rho\in\mca{A}(\nu,\mca{G})$, $\rho^\prime\in\mca{A}(\nu^\prime,\mca{G})$ and $\varphi$, $\varphi^\prime\in\ul{\Hom}(N\rtimes_\rho\mca{G},N^\prime\rtimes_{\rho^\prime}\mca{G})$. We write $\varphi\mathrel{\ul{\sim_{nt}}}\varphi^\prime$ if there exists an $nt$-homotopy $P$ between $\varphi$ and $\varphi^\prime$ such that $\varphi P(\cdot,t)\in\ul{\Hom}(N\rtimes_\rho\mca{G},N^\prime\rtimes_{\rho^\prime}\mca{G})$ for all $t\in I$. 

\begin{claim}
$\varphi P(\cdot,t)\in\ul{\Hom}(N\rtimes_\rho\mca{G},N^\prime\rtimes_{\rho^\prime}\mca{G})$ if and only if $\varphi_{\rho^\prime}P(y,t)\in\mca{G}^{\nu(y)}_{\nu(y)}$ for all $y\in N$. 
\end{claim}

\begin{proof}
Since 
\begin{equation*}
\nu^\prime F_{\varphi P(\cdot,t)}=\nu^\prime s^\prime P(\cdot,t)=\nu^\prime\rho^\prime P(\cdot,t)=s\varphi_{\rho^\prime}P(\cdot,t), 
\end{equation*}
$\varphi P(\cdot,t)\in\ul{\Hom}(N\rtimes_\rho\mca{G},N^\prime\rtimes_{\rho^\prime}\mca{G})$ if and only if $s\varphi_{\rho^\prime}P(\cdot,t)=\nu$. We have 
\begin{equation*}
t\varphi_{\rho^\prime}P(\cdot,t)=\nu^\prime t^\prime P(\cdot,t)=\nu^\prime F_\varphi=\nu, 
\end{equation*}
hence $\varphi_{\rho^\prime}P(y,t)\in\mca{G}^{\nu(y)}$ for all $y\in N$. Therefore $\varphi_{\rho^\prime}P(y,t)\in\mca{G}^{\nu(y)}_{\nu(y)}$ for all $y\in N$ if and only if $s\varphi_{\rho^\prime}P(\cdot,t)=\nu$. 
\end{proof}

$\mathrel{\ul{\sim_{nt}}}$ is an equivalence relation on $\ul{\Hom}(N\rtimes_\rho\mca{G},N^\prime\rtimes_{\rho^\prime}\mca{G})$. 
\end{dfn}

\begin{lem}\label{ulsimsim}
Let $\mca{G}\rightrightarrows M$ be a Lie groupoid, $N$, $N^\prime$, $N^{\prime\prime}$ be $C^\infty$ manifolds, $\nu\colon N\to M$, $\nu^\prime\colon N^\prime\to M$, $\nu^{\prime\prime}\colon N^{\prime\prime}\to M$ be $C^\infty$ maps, $\rho\in\mca{A}(\nu,\mca{G})$, $\rho^\prime\in\mca{A}(\nu^\prime,\mca{G})$, $\rho^{\prime\prime}\in\mca{A}(\nu^{\prime\prime},\mca{G})$, $\varphi$, $\varphi^\prime\in\ul{\Hom}(N\rtimes_\rho\mca{G},N^\prime\rtimes_{\rho^\prime}\mca{G})$ and $\psi$, $\psi^\prime\in\ul{\Hom}(N^\prime\rtimes_{\rho^\prime}\mca{G},N^{\prime\prime}\rtimes_{\rho^{\prime\prime}}\mca{G})$. Let $a\in\left\{o,nt\right\}$. If $\varphi\mathrel{\ul{\sim_a}}\varphi^\prime$ and $\psi\mathrel{\ul{\sim_a}}\psi^\prime$, then $\psi\varphi\mathrel{\ul{\sim_a}}\psi^\prime\varphi^\prime$. 
\end{lem}

\section{Bornologies on groupoids}\label{666}
\subsection{Bornologies}
\begin{dfn}\label{asetbbor}
Let $X$ be a set and $\mca{B}\subset P(X)$, where $P(X)$ is the power set of $X$. We say that $\mca{B}$ is a \emph{bornology} on $X$ if: 
\begin{itemize}
\setlength\itemsep{0em}
\item $B^\prime\in\mca{B}$ if $B\in\mca{B}$ and $B^\prime\subset B$
\item $B\cup B^\prime\in\mca{B}$ if $B$, $B^\prime\in\mca{B}$
\item $X=\bigcup_{B\in\mca{B}}B$. 
\end{itemize}
The pair $(X,\mca{B})$ is called a \emph{bornological set}. An element of $\mca{B}$ is called a \emph{bounded set}. 
\end{dfn}

\begin{dfn}
Let $(X,\mca{B})$, $(X^\prime,\mca{B}^\prime)$ be bornological sets and $f\colon X\to X^\prime$ be a map. 
\begin{itemize}
\item $f$ is \emph{bounded} if $f(B)\in\mca{B}^\prime$ for all $B\in\mca{B}$. 
\item $f$ is \emph{bibounded} if $f$ is bijective and $f$, $f^{-1}$ are bounded. 
\end{itemize}
The composition of two bounded maps is bounded. 
\end{dfn}

\begin{dfn}
Let $\mca{G}\rightrightarrows M$ be a groupoid and $\mca{B}$ be a bornology on $\mca{G}$. We say that $\mca{B}$ is \emph{compatible} if: 
\begin{itemize}
\setlength\itemsep{0em}
\item $B^{-1}\in\mca{B}$ for $B\in\mca{B}$, where $B^{-1}=\left\{g^{-1}\in\mca{G}\ \middle|\ g\in B\right\}$
\item $BB^\prime\in\mca{B}$ for $B$, $B^\prime\in\mca{B}$, where $BB^\prime=\left\{gg^\prime\in\mca{G}\ \middle|\ (g,g^\prime)\in B\times_{s,t}B^\prime\right\}$. 
\end{itemize}
\end{dfn}

\begin{dfn}
Let $\mca{G}\rightrightarrows M$ be a groupoid and $\mca{B}$ be a bornology on $\mca{G}$. We say that $\mca{B}$ is \emph{base-bounded} if $1_M\in\mca{B}$, where $1_M=\left\{1_x\in\mca{G}\ \middle|\ x\in M\right\}$. 
\end{dfn}

\subsection{The locally compact bornology on a topological groupoid}
\begin{prop}\label{loccombor}
Let $\mca{G}\rightrightarrows M$ be a topological groupoid such that $M$ is Hausdorff. Then 
\begin{align*}
\mca{B}_\mca{G}&=\left\{B\subset\mca{G}\ \middle|\ 
\begin{gathered}
\text{for any compact subset $K\subset\mca{G}$, }\\
\text{there exists a compact subset $K^\prime\subset\mca{G}$}\\
\text{such that $KB$, $BK\subset K^\prime$}
\end{gathered}
\right\}\\
&=\left\{B\subset\mca{G}\ \middle|\ 
\begin{gathered}
\text{for any compact subset $K\subset M$, }\\
\text{there exists a compact subset $K^\prime\subset\mca{G}$}\\
\text{such that $B\cap t^{-1}(K)$, $B\cap s^{-1}(K)\subset K^\prime$}
\end{gathered}
\right\}
\end{align*}
is a base-bounded compatible bornology on $\mca{G}$ containing all compact subsets of $\mca{G}$. We call $\mca{B}_\mca{G}$ the \emph{locally compact bornology} on $\mca{G}$. 
\end{prop}

\begin{proof}
\begin{claim}
If $K$, $K^\prime\subset\mca{G}$ are compact, then $KK^\prime$ is compact. 
\end{claim}

\begin{proof}
The map $(s,t)\colon K\times K^\prime\to M\times M$ is continuous and $K\times_{s,t}K^\prime=(s,t)^{-1}(\Delta_M)$. Since $M$ is Hausdorff, $\Delta_M\subset M\times M$ is closed. So $K\times_{s,t}K^\prime\subset K\times K^\prime$ is closed, hence compact. Since $KK^\prime$ is the image of $K\times_{s,t}K^\prime$ by $\mca{G}^{(2)}\to\mca{G}$, $KK^\prime$ is compact. 
\end{proof}

Define $\mca{B}_\mca{G}$ by the first equality in the proposition. Then $\mca{B}_\mca{G}$ contains all compact subsets of $\mca{G}$ by the claim. It is easy to see that $\mca{B}_\mca{G}$ is a base-bounded bornology on $\mca{G}$ and $B^{-1}\in\mca{B}_\mca{G}$ for any $B\in\mca{B}_\mca{G}$. 

\begin{claim}
$BB^\prime\in\mca{B}_\mca{G}$ for any $B$, $B^\prime\in\mca{B}_\mca{G}$. Hence $\mca{B}_\mca{G}$ is compatible. 
\end{claim}

\begin{proof}
For any compact subset $K$ of $\mca{G}$, there exists a compact subset $K^\prime$ of $\mca{G}$ such that $B^\prime K$, $KB\subset K^\prime$. Then there exists a compact subset $K^{\prime\prime}$ of $\mca{G}$ such that $BK^\prime$, $K^\prime B^\prime\subset K^{\prime\prime}$. We have 
\begin{gather*}
(BB^\prime)K=B(B^\prime K)\subset BK^\prime\subset K^{\prime\prime}, \\
K(BB^\prime)=(KB)B^\prime=K^\prime B^\prime\subset K^{\prime\prime}. 
\end{gather*}
Therefore $BB^\prime\in\mca{B}_\mca{G}$. 
\end{proof}

Finally we need to show the second equality in the proposition. To prove the inclusion $\subset$, let $B\in\mca{B}_\mca{G}$. For a compact subset $K$ of $M$, we have 
\begin{equation*}
B\cap t^{-1}(K)=1_KB\quad\text{and}\quad B\cap s^{-1}(K)=B1_K. 
\end{equation*}
So there exists a compact subset $K^\prime\subset\mca{G}$ such that $B\cap t^{-1}(K)$, $B\cap s^{-1}(K)\subset K^\prime$. 

For the inclusion $\supset$, take $B$ from the second set. Let $K$ be a compact subset of $\mca{G}$. There exists a compact subset $K^\prime\subset\mca{G}$ such that $1_{s(K)}B$, $B1_{t(K)}\subset K^\prime$. Then 
\begin{gather*}
KB=(K1_{s(K)})B=K(1_{s(K)}B)\subset KK^\prime, \\
BK=B(1_{t(K)}K)=(B1_{t(K)})K\subset K^\prime K. 
\end{gather*}
Since $KK^\prime\cup K^\prime K$ is compact by the first claim, we have $B\in\mca{B}_\mca{G}$. 
\end{proof}

\begin{prop}\label{properbounded}
Let $\mca{G}\rightrightarrows M$, $\mca{G}^\prime\rightrightarrows M^\prime$ be topological groupoids such that $M$, $M^\prime$ are Hausdorff and $\varphi\colon\mca{G}\to\mca{G}^\prime$ be a continuous morphism. If $F_\varphi$ is proper, then $\varphi$ is bounded with respect to $\mca{B}_\mca{G}$ and $\mca{B}_{\mca{G}^\prime}$. 
\end{prop}

\begin{proof}
Let $B\in\mca{B}_\mca{G}$ and $K^\prime$ be a compact subset of $M^\prime$. Since $F_\varphi^{-1}(K^\prime)\subset M$ is compact by properness of $F_\varphi$, there exists a compact subset $K$ of $\mca{G}$ such that $B1_{F_\varphi^{-1}(K^\prime)}$, $1_{F_\varphi^{-1}(K^\prime)}B\subset K$. 

\begin{claim}
$\varphi(B)1_{K^\prime}=\varphi(B1_{F_\varphi^{-1}(K^\prime)})\subset\varphi(K)$. 
\end{claim}

\begin{proof}
$\supset$ part of the equality and $\subset$ are obvious. For $\subset$ part of the equality, let $g^\prime\in\varphi(B)1_{K^\prime}$. There exist $g\in B$ and $x^\prime\in K^\prime$ such that $g^\prime=\varphi(g)1_{x^\prime}=\varphi(g)$. Then $x^\prime=s^\prime\varphi(g)=F_\varphi(s(g))$. So $s(g)\in F_\varphi^{-1}(K^\prime)$. Hence $g^\prime=\varphi(g)=\varphi(g1_{s(g)})\in\varphi(B1_{F_\varphi^{-1}(K^\prime)})$. 
\end{proof}

Similarly we have $1_{K^\prime}\varphi(B)=\varphi(1_{F_\varphi^{-1}(K^\prime)}B)\subset\varphi(K)$. Hence $\varphi(B)\in\mca{B}_{\mca{G}^\prime}$. 
\end{proof}

\begin{lem}\label{properp}
Let $\mca{G}\rightrightarrows M$ be a topological groupoid such that $M$ is Hausdorff. If $t\colon\mca{G}\to M$ is proper, then $\mca{B}_\mca{G}=P(\mca{G})$. 
\end{lem}

\begin{proof}
Let $B\subset\mca{G}$. For a compact subset $K\subset M$, we have $B\cap t^{-1}(K)\subset t^{-1}(K)$ and $B\cap s^{-1}(K)\subset s^{-1}(K)$. By assumption $t^{-1}(K)$ is compact. Since $s\colon\mca{G}\to M$ is also proper, $s^{-1}(K)$ is compact. Hence $B\in\mca{B}_\mca{G}$. 
\end{proof}

\subsection{The base-bounded compact bornology on a topological groupoid}
\begin{dfn}
Let $X$ be a set and $\mca{S}\subset P(X)$. Since the intersection of bornologies is again a bornology, and $P(X)$ is a bornology, there exists the smallest bornology $\mca{B}$ on $X$ which contains $\mca{S}$. $\mca{B}$ is called the \emph{bornology generated by} $\mca{S}$. It is easy to see 
\begin{equation*}
\mca{B}=\left\{B\subset X\ \middle|\ 
\begin{gathered}
\text{$B\subset B_1\cup\cdots\cup B_n\cup F$ for some}\\
\text{$B_1,\ldots,B_n\in\mca{S}$ and a finite subset $F$ of $X$}
\end{gathered}
\right\}. 
\end{equation*}
\end{dfn}

\begin{lem}\label{borncomp}
Let $\mca{G}\rightrightarrows M$ be a groupoid and $\mca{S}\subset P(\mca{G})$ satisfy: 
\begin{itemize}
\setlength\itemsep{0em}
\item $B^{-1}\in\mca{S}$ for any $B\in\mca{S}$
\item $BB^\prime\in\mca{S}$ for any $B$, $B^\prime\in\mca{S}$
\item $\mca{G}=\bigcup_{B\in\mca{S}}B$. 
\end{itemize}
Then the bornology $\mca{B}$ on $\mca{G}$ generated by $\mca{S}$ is compatible. 
\end{lem}

\begin{proof}
We have 
\begin{align*}
\mca{B}&=\left\{B\subset\mca{G}\ \middle|\ 
\begin{gathered}
\text{$B\subset B_1\cup\cdots\cup B_n\cup F$ for some}\\
\text{$B_1,\ldots,B_n\in\mca{S}$ and a finite subset $F$ of $\mca{G}$}
\end{gathered}
\right\}\\
&=\left\{B\subset\mca{G}\ \middle|\ \text{$B\subset B_1\cup\cdots\cup B_n$ for some $B_1,\ldots,B_n\in\mca{S}$}\right\}. 
\end{align*}
Compatibility follows from this expression. 
\end{proof}

\begin{prop}
Let $\mca{G}\rightrightarrows M$ be a topological groupoid such that $M$ is Hausdorff. Let $\mca{B}^c_\mca{G}$ be the bornology generated by compact subsets of $\mca{G}$ and $1_M$. Then $\mca{B}^c_\mca{G}$ is a base-bounded compatible bornology on $\mca{G}$. We have 
\begin{equation}\label{BK1KG}
\mca{B}^c_\mca{G}=\left\{B\subset\mca{G}\ \middle|\ \text{$B\subset K\cup1_M$ for some compact subset $K\subset\mca{G}$}\right\}. 
\end{equation}
We call $\mca{B}^c_\mca{G}$ the \emph{base-bounded compact bornology} on $\mca{G}$. 
\end{prop}

\begin{proof}
As in the proof of the first claim in the proof of Proposition \ref{loccombor}, $KK^\prime$ is compact for any compact subsets $K$ and $K^\prime$ of $\mca{G}$. Let $\mca{S}=\left\{\text{compact subsets of $\mca{G}$}\right\}\cup\left\{1_M\right\}$. Then $\mca{S}$ satisfies the conditions in Lemma \ref{borncomp}. Hence $\mca{B}^c_\mca{G}$ is compatible. 
\end{proof}

\begin{prop}\label{tgcmbrebb}
Let $\mca{G}\rightrightarrows M$, $\mca{G}^\prime\rightrightarrows M^\prime$ be topological groupoids such that $M$, $M^\prime$ are Hausdorff, and $\varphi\colon\mca{G}\to\mca{G}^\prime$ be a continuous morphism. Then $\varphi$ is bounded with respect to $\mca{B}^c_\mca{G}$ and $\mca{B}^c_{\mca{G}^\prime}$. 
\end{prop}

\begin{proof}
This follows from \eqref{BK1KG}
\end{proof}

\begin{prop}
Let $\mca{G}\rightrightarrows M$ be a topological groupoid such that $M$ is Hausdorff. Then the following statements hold: 
\begin{enumerate}
\item $\mca{B}^c_\mca{G}\subset\mca{B}_\mca{G}$. 
\item If $\mca{B}^c_\mca{G}\neq\mca{B}_\mca{G}$, then there exists a net $g_i\in\mca{G}\setminus1_M$ such that $t(g_i)\to\infty$, $s(g_i)\to\infty$. 
\item If $\mca{B}^c_\mca{G}\neq\mca{B}_\mca{G}$ and $M$ is $\sigma$-compact, then there exists a sequence $g_i\in\mca{G}\setminus1_M$ such that $t(g_i)\to\infty$, $s(g_i)\to\infty$. 
\item If there exists a sequence $g_i\in\mca{G}\setminus1_M$ such that $t(g_i)\to\infty$, $s(g_i)\to\infty$, then $\mca{B}^c_\mca{G}\neq\mca{B}_\mca{G}$. 
\item Let $B\in\mca{B}_\mca{G}$. $B\not\in\mca{B}^c_\mca{G}$ if and only if there exists a net $g_i\in B\setminus1_M$ such that $t(g_i)\to\infty$. 
\item Let $M_{\neq1}=\left\{x\in M\ \middle|\ \mca{G}^x\neq\{1_x\}\right\}=\left\{x\in M\ \middle|\ \mca{G}_x\neq\{1_x\}\right\}$. If $M_{\neq1}$ is relatively compact in $M$ (for example $M$ is compact), then $\mca{B}^c_\mca{G}=\mca{B}_\mca{G}$. 
\end{enumerate}
\end{prop}

\begin{proof}
1. Obvious. 

2. Let $B\in\mca{B}_\mca{G}\setminus\mca{B}^c_\mca{G}$. Let $\ms{A}=\left\{\text{compact subsets of $M$}\right\}$. $\ms{A}$ is a directed set with respect to $\subset$. Let $C\in\ms{A}$. There exists a compact subset $K\subset\mca{G}$ such that $(B\cap t^{-1}(C))\cup(B\cap s^{-1}(C))\subset K$. Since $B\not\in\mca{B}^c_\mca{G}$, there exists $g_C\in B\setminus(1_M\cup K)$. Since $t(g_C)$, $s(g_C)\in M\setminus C$, the net $(g_C)_{C\in\ms{A}}$ satisfies $t(g_C)\to\infty$, $s(g_C)\to\infty$. 

3. Let $B\in\mca{B}_\mca{G}\setminus\mca{B}^c_\mca{G}$. There exist compact subsets $C_i$ ($i\in\bb{N}$) of $M$ such that $C_i\subset C_{i+1}$ for all $i\in\bb{N}$ and $M=\bigcup_{i=1}^\infty C_i$. There exists a compact subset $K_i\subset\mca{G}$ such that $(B\cap t^{-1}(C_i))\cup(B\cap s^{-1}(C_i))\subset K_i$. Since $B\not\in\mca{B}^c_\mca{G}$, there exists $g_i\in B\setminus(1_M\cup K_i)$. We have $t(g_i)$, $s(g_i)\in M\setminus C_i$. Hence $t(g_i)\to\infty$, $s(g_i)\to\infty$. 

4. Let $B=\left\{g_i\ \middle|\ i\in\bb{N}\right\}$. For a compact subset $C\subset M$, $(B\cap t^{-1}(C))\cup(B\cap s^{-1}(C))$ is finite, hence compact. So $B\in\mca{B}_\mca{G}$. 

If $B\in\mca{B}^c_\mca{G}$, there exists a compact subset $K$ of $\mca{G}$ such that $B\subset1_M\cup K$. We have $B\subset K$, hence $t(g_i)\in t(K)$. Since $t(K)$ is compact, this is a contradiction. Thus $B\in\mca{B}_\mca{G}\setminus\mca{B}^c_\mca{G}$. 
%
%
%
%

5. $\Rightarrow$ is proved in 2. We prove $\Leftarrow$. If $B\in\mca{B}^c_\mca{G}$, there exists a compact subset $K$ of $\mca{G}$ such that $B\subset1_M\cup K$. Since $g_i\in B\setminus1_M$, we have $g_i\in K$, hence $t(g_i)\in t(K)$. Since $t(K)$ is compact, this is a contradiction. 

6. If $\mca{B}^c_\mca{G}\neq\mca{B}_\mca{G}$, there exist $B\in\mca{B}_\mca{G}\setminus\mca{B}^c_\mca{G}$ and a net $g_i\in B\setminus1_M$ such that $t(g_i)\to\infty$ by 5. Since $t(g_i)\in M_{\neq1}$, this is a contradiction. 
\end{proof}

\subsection{The action groupoid bornologies}
\begin{dfn}
Let $X$ be a set, $(X^\prime,\mca{B}^\prime)$ be a bornological set and $f\colon X\to X^\prime$ be a map. Let $[f^{-1}(\mca{B}^\prime)]$ be the bornology generated by $\left\{f^{-1}(B^\prime)\subset X\ \middle|\ B^\prime\in\mca{B}^\prime\right\}$. The bornology $[f^{-1}(\mca{B}^\prime)]$ is called the \emph{inverse image bornology} induced by $f$. 
\end{dfn}

\begin{lem}\label{borlarbou}
Let $X$ be a set, $(X^\prime,\mca{B}^\prime)$ be a bornological set and $f\colon X\to X^\prime$ be a map. Then $[f^{-1}(\mca{B}^\prime)]=\left\{B\subset X\ \middle|\ f(B)\in\mca{B}^\prime\right\}$. Hence $[f^{-1}(\mca{B}^\prime)]$ is the largest bornology on $X$ such that $f\colon X\to X^\prime$ is bounded. 
\end{lem}

\begin{proof}
Let $\mca{B}=\left\{B\subset X\ \middle|\ f(B)\in\mca{B}^\prime\right\}$. Then $\mca{B}$ is a bornology on $X$. 

For any $B^\prime\in\mca{B}^\prime$, we have $ff^{-1}(B^\prime)\subset B^\prime$, hence $ff^{-1}(B^\prime)\in\mca{B}^\prime$. Thus $f^{-1}(B^\prime)\in\mca{B}$. This proves $[f^{-1}(\mca{B}^\prime)]\subset\mca{B}$. 

Let $B\in\mca{B}$. Since $f(B)\in\mca{B}^\prime$, we have $B\subset f^{-1}f(B)\in[f^{-1}(\mca{B}^\prime)]$, hence $B\in[f^{-1}(\mca{B}^\prime)]$. Therefore $\mca{B}\subset[f^{-1}(\mca{B}^\prime)]$. 
\end{proof}

\begin{dfn}\label{acgrbong}
Let $\mca{G}\rightrightarrows M$ be a Lie groupoid, $N$ be a $C^\infty$ manifold, $\nu\colon N\to M$ be a $C^\infty$ map and $\rho\in\mca{A}(\nu,\mca{G})$. Let $\mca{B}_\rho$ be the inverse image bornology induced by the projection $\varphi_\rho\colon N\rtimes_\rho\mca{G}\to\mca{G}$, where $\mca{G}$ is equipped with the locally compact bornology $\mca{B}_\mca{G}$. Hence 
\begin{equation*}
\mca{B}_\rho=\left\{B\subset N\rtimes_\rho\mca{G}\ \middle|\ \varphi_\rho(B)\in\mca{B}_\mca{G}\right\}
\end{equation*}
by Lemma \ref{borlarbou} and $\varphi_\rho$ is bounded. $\mca{B}_\rho$ is a base-bounded compatible bornology on $N\rtimes_\rho\mca{G}$ which contains all compact subsets of $N\rtimes_\rho\mca{G}$. We call $\mca{B}_\rho$ the \emph{action groupoid bornology} of $\rho$. 
\end{dfn}

\begin{prop}\label{crhol}
Let $\mca{G}\rightrightarrows M$ be a Lie groupoid, $N$ be a $C^\infty$ manifold, $\nu\colon N\to M$ be a $C^\infty$ map and $\rho\in\mca{A}(\nu,\mca{G})$. Then the following statements hold: 
\begin{enumerate}
\item $\mca{B}^c_{N\rtimes_\rho\mca{G}}\subset\mca{B}_\rho\subset\mca{B}_{N\rtimes_\rho\mca{G}}$. 
\item $\mca{B}^c_{N\rtimes_\rho\mca{G}}=\mca{B}_\rho=\mca{B}_{N\rtimes_\rho\mca{G}}$ if $N$ is compact. 
\item If $t\colon\mca{G}\to M$ is proper, then $\mca{B}_\rho=P(N\rtimes_\rho\mca{G})$. 
\item If $\nu\colon N\to M$ is proper, then $\mca{B}_\rho=\mca{B}_{N\rtimes_\rho\mca{G}}$. 
%
%
\end{enumerate}
\end{prop}

\begin{proof}
1. $\mca{B}^c_{N\rtimes_\rho\mca{G}}\subset\mca{B}_\rho$ is OK. Let $B\in\mca{B}_\rho$. Let $K\subset N$ be a compact subset. Since $\varphi_\rho(B)\in\mca{B}_\mca{G}$ and $\nu(K)$ is a compact subset of $M$, there exists a compact subset $K^\prime\subset\mca{G}$ such that $\varphi_\rho(B)\cap t^{-1}(\nu(K))\subset K^\prime$. 

\begin{claim}
$B\cap t^{-1}(K)\subset K\times_{\nu,t}K^\prime$. 
\end{claim}

\begin{proof}
Let $(y,g)\in B\cap t^{-1}(K)$. We have $y\in K$ and $t(g)=\nu(y)\in\nu(K)$. Hence $g\in\varphi_\rho(B)\cap t^{-1}(\nu(K))\subset K^\prime$. 
\end{proof}

$K\times_{\nu,t}K^\prime$ is compact since $M$ is Hausdorff. Since $B^{-1}\in\mca{B}_\rho$, the same argument shows that there exists a compact subset $K^{\prime\prime}\subset N\rtimes_\rho\mca{G}$ such that $B^{-1}\cap t^{-1}(K)\subset K^{\prime\prime}$. Since $B\cap s^{-1}(K)=(B^{-1}\cap t^{-1}(K))^{-1}\subset(K^{\prime\prime})^{-1}$, we have $B\in\mca{B}_{N\rtimes_\rho\mca{G}}$. 

2. $\mca{B}^c_{N\rtimes_\rho\mca{G}}=\mca{B}_{N\rtimes_\rho\mca{G}}$ since $N$ is compact. 
%
%
%
%
%
%
%
%
%
%
%

3. $\mca{B}_\mca{G}=P(\mca{G})$ by Lemma \ref{properp}. Hence $\mca{B}_\rho=P(N\rtimes_\rho\mca{G})$. 

4. Let $B\in\mca{B}_{N\rtimes_\rho\mca{G}}$. We must show $\varphi_\rho(B)\in\mca{B}_\mca{G}$. Let $K\subset M$ be a compact subset. 

\begin{claim}
$\varphi_\rho(B)\cap t^{-1}(K)=\varphi_\rho(B\cap t^{-1}\nu^{-1}(K))$. 
\end{claim}

\begin{proof}
To prove $\subset$, let $g\in\varphi_\rho(B)\cap t^{-1}(K)$. There exists $(y,g)\in B$. We have $\nu(y)=t(g)\in K$, hence $(y,g)\in B\cap t^{-1}\nu^{-1}(K)$. So $g\in\varphi_\rho(B\cap t^{-1}\nu^{-1}(K))$. 

For $\supset$, let $g\in\varphi_\rho(B\cap t^{-1}\nu^{-1}(K))$. There exists $(y,g)\in B\cap t^{-1}\nu^{-1}(K)$. We have $y\in\nu^{-1}(K)$ and $t(g)=\nu(y)\in K$. Hence $g\in\varphi_\rho(B)\cap t^{-1}(K)$. 
\end{proof}

Since $\nu$ is proper, $\nu^{-1}(K)$ is compact. There exists a compact subset $K^\prime\subset N\rtimes_\rho\mca{G}$ such that $B\cap t^{-1}\nu^{-1}(K)\subset K^\prime$. Applying $\varphi_\rho$ to this inclusion and using the claim, we get $\varphi_\rho(B)\cap t^{-1}(K)\subset\varphi_\rho(K^\prime)$. 

Since $B^{-1}\in\mca{B}_{N\rtimes_\rho\mca{G}}$, the same argument shows that there exists a compact subset $K^{\prime\prime}\subset N\rtimes_\rho\mca{G}$ such that $\varphi_\rho(B^{-1})\cap t^{-1}(K)\subset\varphi_\rho(K^{\prime\prime})$. Taking the inverse, we get $\varphi_\rho(B)\cap s^{-1}(K)\subset\varphi_\rho((K^{\prime\prime})^{-1})$. Therefore $\varphi_\rho(B)\in\mca{B}_\mca{G}$. Hence $B\in\mca{B}_\rho$. 
\end{proof}

\begin{prop}\label{semibounded}
Let $\mca{G}\rightrightarrows M$, $\mca{G}^\prime\rightrightarrows M^\prime$ be Lie groupoids, $N$, $N^\prime$ be $C^\infty$ manifolds, $\nu\colon N\to M$, $\nu^\prime\colon N^\prime\to M^\prime$ be $C^\infty$ maps, $\rho\in\mca{A}(\nu,\mca{G})$, $\rho^\prime\in\mca{A}(\nu^\prime,\mca{G}^\prime)$ and $\varphi\in\Hom(N\rtimes_\rho\mca{G},N^\prime\rtimes_{\rho^\prime}\mca{G}^\prime)$. Consider the action groupoid bornology $\mca{B}_\rho$ (resp. $\mca{B}_{\rho^\prime}$) on $N\rtimes_\rho\mca{G}$ (resp. $N^\prime\rtimes_{\rho^\prime}\mca{G}^\prime$). Then the following statements hold: 
\begin{enumerate}
\item If $\varphi=F_\varphi\rtimes\Phi$ is a semiconjugacy and $F_\Phi\colon M\to M^\prime$ is proper, then $\varphi$ is bounded. 
\item If $\varphi=F_\varphi\rtimes\Phi$ is a semiconjugacy and $F_\Phi\colon\nu(N)\to M^\prime$ is proper, then $\varphi$ is bounded. 
\item If $\varphi$ is a conjugacy, then $\varphi$ is bibounded. 
\item If $N$ is compact, then $\varphi$ is bounded. 
\item If $F_\varphi\colon N\to N^\prime$ and $\nu^\prime\colon N^\prime\to M^\prime$ are proper, then $\varphi$ is bounded. 
\item If $t^\prime\colon\mca{G}^\prime\to M^\prime$ is proper, then $\varphi$ is bounded. 
\end{enumerate}
\end{prop}

\begin{proof}
1. We have 
\begin{equation*}
\begin{tikzcd}
N\rtimes_\rho\mca{G}\ar[r,"\varphi"]\ar[d,"\varphi_\rho"']\ar[rd,phantom,"\circlearrowright"{xshift=-5}]&N^\prime\rtimes_{\rho^\prime}\mca{G}^\prime\ar[d,"\varphi_{\rho^\prime}"]\\
\mca{G}\ar[r,"\Phi"']&\mca{G}^\prime. 
\end{tikzcd}
\end{equation*}
$\varphi_\rho$ is bounded with respect to $\mca{B}_\rho$ and $\mca{B}_\mca{G}$. $\Phi$ is bounded with respect to $\mca{B}_\mca{G}$ and $\mca{B}_{\mca{G}^\prime}$ by Proposition \ref{properbounded}. If $B\in\mca{B}_\rho$, we have $\varphi_{\rho^\prime}\varphi(B)=\Phi\varphi_\rho(B)\in\mca{B}_{\mca{G}^\prime}$, hence $\varphi(B)\in\mca{B}_{\rho^\prime}$. 

2. Let $B\in\mca{B}_\rho$. To prove $\varphi(B)\in\mca{B}_{\rho^\prime}$, we need to show $\varphi_{\rho^\prime}\varphi(B)=\Phi\varphi_\rho(B)\in\mca{B}_{\mca{G}^\prime}$. Let $K^\prime$ be a compact subset of $M^\prime$. Since $\varphi_\rho(B)\subset\mca{G}|_{\nu(N)}$, we have 
\begin{equation}\label{phiphibkkp}
\begin{gathered}
\Phi\varphi_\rho(B)\cap(t^\prime)^{-1}(K^\prime)\subset\Phi(\varphi_\rho(B)\cap t^{-1}(F_\Phi|_{\nu(N)})^{-1}(K^\prime)), \\
\Phi\varphi_\rho(B)\cap(s^\prime)^{-1}(K^\prime)\subset\Phi(\varphi_\rho(B)\cap s^{-1}(F_\Phi|_{\nu(N)})^{-1}(K^\prime)). 
\end{gathered}
\end{equation}
Since $\varphi_\rho(B)\in\mca{B}_\mca{G}$ and $(F_\Phi|_{\nu(N)})^{-1}(K^\prime)$ is a compact subset of $M$, there exists a compact subset $K\subset\mca{G}$ such that 
\begin{equation*}
(\varphi_\rho(B)\cap t^{-1}(F_\Phi|_{\nu(N)})^{-1}(K^\prime))\cup(\varphi_\rho(B)\cap s^{-1}(F_\Phi|_{\nu(N)})^{-1}(K^\prime))\subset K. 
\end{equation*}
Applying $\Phi$ to this inclusion and using \eqref{phiphibkkp}, we get 
\begin{equation*}
(\Phi\varphi_\rho(B)\cap(t^\prime)^{-1}(K^\prime))\cup(\Phi\varphi_\rho(B)\cap(s^\prime)^{-1}(K^\prime))\subset\Phi(K). 
\end{equation*}
Therefore $\Phi\varphi_\rho(B)\in\mca{B}_{\mca{G}^\prime}$. 

3. Let $B\in\mca{B}_\rho$. We prove $\Phi\varphi_\rho(B)\in\mca{B}_{\mca{G}^\prime}$. Let $K^\prime$ be a compact subset of $M^\prime$. There exist $\Phi\in\Hom(\mca{G},\mca{G}^\prime)$ and $\Phi^\prime\in\Hom(\mca{G}^\prime,\mca{G})$ such that $\varphi=F_\varphi\rtimes\Phi$ and $\varphi^{-1}=F_\varphi^{-1}\rtimes\Phi^\prime$. Since 
\begin{equation*}
\begin{tikzcd}
N\ar[r,"F_\varphi"]\ar[d,"\nu"']\ar[rd,phantom,"\circlearrowright"]&N^\prime\ar[d,"\nu^\prime"]\\
M\ar[r,"F_\Phi"']&M^\prime, 
\end{tikzcd}
\qquad
\begin{tikzcd}
N^\prime\ar[r,"F_\varphi^{-1}"]\ar[d,"\nu^\prime"']\ar[rd,phantom,"\circlearrowright"]&N\ar[d,"\nu"]\\
M^\prime\ar[r,"F_{\Phi^\prime}"']&M, 
\end{tikzcd}
\end{equation*}
we have maps 
\begin{equation*}
\begin{tikzcd}
\nu(N)\ar[r,shift left,"F_\Phi"]&\nu^\prime(N^\prime)\ar[l,shift left,"F_{\Phi^\prime}"]
\end{tikzcd}
\end{equation*}
which are inverse to each other. Since $\varphi_\rho(B)\subset\mca{G}|_{\nu(N)}$, we have 
\begin{equation}\label{pptkpppp}
\begin{aligned}
\Phi\varphi_\rho(B)\cap(t^\prime)^{-1}(K^\prime)&\subset\Phi(\varphi_\rho(B)\cap t^{-1}(F_\Phi|_{\nu(N)})^{-1}(K^\prime))\\
&\subset\Phi(\varphi_\rho(B)\cap t^{-1}F_{\Phi^\prime}(K^\prime)), \\
\Phi\varphi_\rho(B)\cap(s^\prime)^{-1}(K^\prime)&\subset\Phi(\varphi_\rho(B)\cap s^{-1}(F_\Phi|_{\nu(N)})^{-1}(K^\prime))\\
&\subset\Phi(\varphi_\rho(B)\cap s^{-1}F_{\Phi^\prime}(K^\prime)). 
\end{aligned}
\end{equation}
Since $\varphi_\rho(B)\in\mca{B}_\mca{G}$ and $F_{\Phi^\prime}(K^\prime)$ is a compact subset of $M$, there exists a compact subset $K\subset\mca{G}$ such that 
\begin{equation*}
(\varphi_\rho(B)\cap t^{-1}F_{\Phi^\prime}(K^\prime))\cup(\varphi_\rho(B)\cap s^{-1}F_{\Phi^\prime}(K^\prime))\subset K. 
\end{equation*}
Applying $\Phi$ to this inclusion and using \eqref{pptkpppp}, we get 
\begin{equation*}
(\Phi\varphi_\rho(B)\cap(t^\prime)^{-1}(K^\prime))\cup(\Phi\varphi_\rho(B)\cap(s^\prime)^{-1}(K^\prime))\subset\Phi(K). 
\end{equation*}
Hence $\Phi\varphi_\rho(B)\in\mca{B}_{\mca{G}^\prime}$. Since $\varphi_{\rho^\prime}\varphi(B)=\Phi\varphi_\rho(B)\in\mca{B}_{\mca{G}^\prime}$, we have $\varphi(B)\in\mca{B}_{\rho^\prime}$, hence $\varphi$ is bounded. $\varphi^{-1}$ is bounded by the same argument. 

4. $\mca{B}_\rho=\mca{B}^c_{N\rtimes_\rho\mca{G}}$ since $N$ is compact. Since $\varphi(\mca{B}^c_{N\rtimes_\rho\mca{G}})\subset\mca{B}^c_{N^\prime\rtimes_{\rho^\prime}\mca{G}^\prime}$ by Proposition \ref{tgcmbrebb}, and $\mca{B}^c_{N^\prime\rtimes_{\rho^\prime}\mca{G}^\prime}\subset\mca{B}_{\rho^\prime}$, we get $\varphi(\mca{B}_\rho)\subset\mca{B}_{\rho^\prime}$. 

5. We have $\mca{B}_\rho\subset\mca{B}_{N\rtimes_\rho\mca{G}}$ by Proposition \ref{crhol}. Since $F_\varphi$ is proper, $\varphi(\mca{B}_{N\rtimes_\rho\mca{G}})\subset\mca{B}_{N^\prime\rtimes_{\rho^\prime}\mca{G}^\prime}$ by Proposition \ref{properbounded}. Since $\nu^\prime$ is proper, $\mca{B}_{N^\prime\rtimes_{\rho^\prime}\mca{G}^\prime}=\mca{B}_{\rho^\prime}$ by Proposition \ref{crhol}. Hence $\varphi(\mca{B}_\rho)\subset\varphi(\mca{B}_{N\rtimes_\rho\mca{G}})\subset\mca{B}_{N^\prime\rtimes_{\rho^\prime}\mca{G}^\prime}=\mca{B}_{\rho^\prime}$. 

6. Since $t^\prime\colon\mca{G}^\prime\to M^\prime$ is proper, we have $\mca{B}_{\rho^\prime}=P(N^\prime\rtimes_{\rho^\prime}\mca{G}^\prime)$ by Proposition \ref{crhol}. Hence $\varphi$ is bounded. 
\end{proof}

\subsection{Bounded $nt$-homotopies}
\begin{dfn}\label{hnhnpnib}
Let $\mca{H}\rightrightarrows N$, $\mca{H}^\prime\rightrightarrows N^\prime$ be Lie groupoids, $\mca{B}^\prime$ be a base-bounded compatible bornology on $\mca{H}^\prime$ and $\varphi$, $\varphi^\prime\in\Hom(\mca{H},\mca{H}^\prime)$. An $nt$-homotopy $P\colon N\times I\to\mca{H}^\prime$ between $\varphi$ and $\varphi^\prime$ is \emph{bounded} if $P(N\times I)\in\mca{B}^\prime$. 

We say that $\varphi$ is \emph{boundedly $nt$-homotopic} to $\varphi^\prime$ (written as $\varphi\sim_{bnt}\varphi^\prime$) if there exists a bounded $nt$-homotopy between $\varphi$ and $\varphi^\prime$. 
\end{dfn}

\begin{lem}
$\sim_{bnt}$ is an equivalence relation on $\Hom(\mca{H},\mca{H}^\prime)$. 
\end{lem}

\begin{proof}
For $\varphi\in\Hom(\mca{H},\mca{H}^\prime)$ and the $nt$-homotopy 
\begin{align*}
P\colon N\times I&\to\mca{H}^\prime\\
(y,t)&\mapsto1_{F_\varphi(y)}
\end{align*}
between $\varphi$ and $\varphi$, we have $P(N\times I)\subset1_{N^\prime}\in\mca{B}^\prime$, hence $P(N\times I)\in\mca{B}^\prime$. Thus $\varphi\sim_{bnt}\varphi$. 

For $\varphi$, $\varphi^\prime\in\Hom(\mca{H},\mca{H}^\prime)$ such that $\varphi\sim_{bnt}\varphi^\prime$, take a bounded $nt$-homotopy $P$ between $\varphi$ and $\varphi^\prime$ and consider the $nt$-homotopy 
\begin{align*}
P^\prime\colon N\times I&\to\mca{H}^\prime\\
(y,t)&\mapsto P(y,1)^{-1}P(y,1-t)
\end{align*}
from $\varphi^\prime$ and $\varphi$. Then $P^\prime(N\times I)\subset P(N\times I)^{-1}P(N\times I)\in\mca{B}^\prime$, hence $P^\prime(N\times I)\in\mca{B}^\prime$. Therefore $\varphi^\prime\sim_{bnt}\varphi$. 

Assume $\varphi\sim_{bnt}\varphi^\prime$, $\varphi^\prime\sim_{bnt}\varphi^{\prime\prime}$ and take a bounded $nt$-homotopy $P$ between $\varphi$ and $\varphi^\prime$, and a bounded $nt$-homotopy $P^\prime$ between $\varphi^\prime$ and $\varphi^{\prime\prime}$. 
Then the $nt$-homotopy 
\begin{align*}
P^{\prime\prime}\colon N\times I&\to\mca{H}^\prime\\
(y,t)&\mapsto
\begin{cases}
P(y,2t)&0\leq t\leq\frac{1}{2}\\
P^\prime(y,2t-1)P(y,1)&\frac{1}{2}\leq t\leq1
\end{cases}
\end{align*}
between $\varphi$ and $\varphi^{\prime\prime}$ satisfies $P^{\prime\prime}(N\times I)\subset P(N\times I)\cup P^\prime(N\times I)P(N\times I)\in\mca{B}^\prime$, hence $P^{\prime\prime}(N\times I)\in\mca{B}^\prime$. Thus $\varphi\sim_{bnt}\varphi^{\prime\prime}$. 
\end{proof}

Note that $\varphi\sim_{bnt}\varphi^\prime\ \Rightarrow\ \varphi\sim_{nt}\varphi^\prime\ \Rightarrow\ \varphi\sim_o\varphi^\prime$. 

\begin{lem}\label{simsim2}
Let $\mca{H}\rightrightarrows N$, $\mca{H}^\prime\rightrightarrows N^\prime$, $\mca{H}^{\prime\prime}\rightrightarrows N^{\prime\prime}$ be Lie groupoids, $\mca{B}^\prime$ (resp. $\mca{B}^{\prime\prime}$) be a base-bounded compatible bornology on $\mca{H}^\prime$ (resp. $\mca{H}^{\prime\prime}$), $\varphi$, $\varphi^\prime\in\Hom(\mca{H},\mca{H}^\prime)$ and $\psi$, $\psi^\prime\in\Hom(\mca{H}^\prime,\mca{H}^{\prime\prime})$. Then the following statements hold: 
\begin{enumerate}
\setlength\itemsep{0em}
\item If $\psi\sim_{bnt}\psi^\prime$, then $\psi\varphi\sim_{bnt}\psi^\prime\varphi$. 
\item If $\varphi\sim_{bnt}\varphi^\prime$ and $\psi$ is bounded, then $\psi\varphi\sim_{bnt}\psi\varphi^\prime$. 
\end{enumerate}
\end{lem}

\begin{proof}
The $nt$-homotopies constructed in the proof of Lemma \ref{simsim} are bounded. 
\end{proof}

\begin{lem}\label{bounboun}
Let $\mca{H}\rightrightarrows N$, $\mca{H}^\prime\rightrightarrows N^\prime$ be groupoids, $\mca{B}$ be a bornology on $\mca{H}$, $\mca{B}^\prime$ be a compatible bornology on $\mca{H}^\prime$ and $\varphi\in\Hom(\mca{H},\mca{H}^\prime)$ be bounded with respect to $\mca{B}$ and $\mca{B}^\prime$. Let $P\colon N\to\mca{H}^\prime$ be a map such that $t^\prime P=F_\varphi$ and $P(N)\in\mca{B}^\prime$. Then $\varphi P\in\Hom(\mca{H},\mca{H}^\prime)$ is bounded with respect to $\mca{B}$ and $\mca{B}^\prime$. 
\end{lem}

\begin{proof}
Let $B\in\mca{B}$. For $h\in\mca{H}$, we have $(\varphi P)(h)=P(t(h))^{-1}\varphi(h)P(s(h))$, hence $(\varphi P)(B)\subset P(N)^{-1}\varphi(B)P(N)\in\mca{B}^\prime$. Thus $(\varphi P)(B)\in\mca{B}^\prime$. 
\end{proof}

\begin{example}
In the situation of Definition \ref{hnhnpnib}, let $\mca{H}^\prime=N^\prime\rtimes_{\rho^\prime}\mca{G}^\prime$ and $\mca{B}^\prime=\mca{B}_{\rho^\prime}$, where $\mca{G}^\prime\rightrightarrows M^\prime$ be a Lie groupoid, $N^\prime$ be a $C^\infty$ manifold, $\nu^\prime\colon N^\prime\to M^\prime$ be a $C^\infty$ map and $\rho^\prime\in\mca{A}(\nu^\prime,\mca{G}^\prime)$. If $N$ is compact or $t^\prime\colon\mca{G}^\prime\to M^\prime$ is proper, then any $nt$-homotopy is bounded by Proposition \ref{crhol}. 
\end{example}

\begin{dfn}\label{simbntbthe}
Let $\mca{G}\rightrightarrows M$ be a Lie groupoid, $N$, $N^\prime$ be $C^\infty$ manifolds, $\nu\colon N\to M$, $\nu^\prime\colon N^\prime\to M$ be $C^\infty$ maps, $\rho\in\mca{A}(\nu,\mca{G})$, $\rho^\prime\in\mca{A}(\nu^\prime,\mca{G})$ and $\varphi$, $\varphi^\prime\in\ul{\Hom}(N\rtimes_\rho\mca{G},N^\prime\rtimes_{\rho^\prime}\mca{G})$. We write $\varphi\mathrel{\ul{\sim_{bnt}}}\varphi^\prime$ if there exists a bounded $nt$-homotopy $P$ between $\varphi$ and $\varphi^\prime$ such that $\varphi P(\cdot,t)\in\ul{\Hom}(N\rtimes_\rho\mca{G},N^\prime\rtimes_{\rho^\prime}\mca{G})$ for all $t\in I$. The binary relation $\mathrel{\ul{\sim_{bnt}}}$ is an equivalence relation on $\ul{\Hom}(N\rtimes_\rho\mca{G},N^\prime\rtimes_{\rho^\prime}\mca{G})$. 
\end{dfn}

\begin{lem}\label{ulsimsim2}
Let $\mca{G}\rightrightarrows M$ be a Lie groupoid, $N$, $N^\prime$, $N^{\prime\prime}$ be $C^\infty$ manifolds, $\nu\colon N\to M$, $\nu^\prime\colon N^\prime\to M$, $\nu^{\prime\prime}\colon N^{\prime\prime}\to M$ be $C^\infty$ maps, $\rho\in\mca{A}(\nu,\mca{G})$, $\rho^\prime\in\mca{A}(\nu^\prime,\mca{G})$, $\rho^{\prime\prime}\in\mca{A}(\nu^{\prime\prime},\mca{G})$, $\varphi$, $\varphi^\prime\in\ul{\Hom}(N\rtimes_\rho\mca{G},N^\prime\rtimes_{\rho^\prime}\mca{G})$ and $\psi$, $\psi^\prime\in\ul{\Hom}(N^\prime\rtimes_{\rho^\prime}\mca{G},N^{\prime\prime}\rtimes_{\rho^{\prime\prime}}\mca{G})$. Then the following statements hold: 
\begin{enumerate}
\setlength\itemsep{0em}
\item If $\psi\mathrel{\ul{\sim_{bnt}}}\psi^\prime$, then $\psi\varphi\mathrel{\ul{\sim_{bnt}}}\psi^\prime\varphi$. 
\item If $\varphi\mathrel{\ul{\sim_{bnt}}}\varphi^\prime$ and $\psi$ is bounded, then $\psi\varphi\mathrel{\ul{\sim_{bnt}}}\psi\varphi^\prime$. 
\end{enumerate}
\end{lem}

\begin{proof}
This follows from Lemma \ref{simsim2}. 
\end{proof}

\section{Fiberwise coarse structures}\label{777}
We will not use the results in this section later in this paper, but 2 in Corollary \ref{gmgmngngde} is a ``generalization'' of 4.1.3 Lemma in \cite{Ma2}, which is used to show parameter rigidity of certain actions. 

\subsection{Coarse structures}
\begin{dfn}[Definition 2.3 in \cite{Roe}]\label{coarsestrdee}
Let $X$ be a set and $\mca{E}\subset P(X\times X)$. We say that $\mca{E}$ is a \emph{coarse structure} on $X$ if: 
\begin{itemize}
\setlength\itemsep{0em}
\item $\Delta_X\in\mca{E}$, where $\Delta_X=\left\{(x,x)\in X\times X\ \middle|\ x\in X\right\}$
\item $E\cup E^\prime\in\mca{E}$ for any $E$, $E^\prime\in\mca{E}$
\item $E^\prime\in\mca{E}$ if $E\in\mca{E}$ and $E^\prime\subset E$
\item $E^{-1}\in\mca{E}$ for any $E\in\mca{E}$, where $E^{-1}=\left\{(x^\prime,x)\in X\times X\ \middle|\ (x,x^\prime)\in E\right\}$
\item $E\circ E^\prime\in\mca{E}$ for any $E$, $E^\prime\in\mca{E}$, where 
\begin{equation*}
E\circ E^\prime=\left\{(x,x^\prime)\in X\times X\ \middle|\ \text{$(x,x^{\prime\prime})\in E$, $(x^{\prime\prime},x^\prime)\in E^\prime$ for some $x^{\prime\prime}\in X$}\right\}. 
\end{equation*}
\end{itemize}
An element of $\mca{E}$ is called an \emph{entourage} (or a \emph{controlled set}). The pair $(X,\mca{E})$ is called a \emph{coarse space}. 
\end{dfn}

\begin{dfn}
Let $(X,\mca{E})$ be a coarse space. We say that $(X,\mca{E})$ is \emph{connected} if $\left\{(x,x^\prime)\right\}\in\mca{E}$ for any $(x,x^\prime)\in X\times X$. 
\end{dfn}

\begin{dfn}
Let $X$ be a set and $(X^\prime,\mca{E}^\prime)$ be a coarse space and $f$, $f^\prime\colon X\to X^\prime$ be maps. We say that $f$ and $f^\prime$ are \emph{close} if $\left\{(f(x),f^\prime(x))\in X^\prime\times X^\prime\ \middle|\ x\in X\right\}\in\mca{E}^\prime$. 
\end{dfn}

\begin{dfn}
Let $(X,\mca{E})$, $(X^\prime,\mca{E}^\prime)$ be coarse spaces and $f\colon X\to X^\prime$. We say that $f$ is \emph{bornologous} (or \emph{coarsely uniform}) if $(f\times f)(E)\in\mca{E}^\prime$ for $E\in\mca{E}$. $f$ is a \emph{coarse equivalence} if $f$ is coarsely uniform and there exists a coarsely uniform map $f^\prime\colon X^\prime\to X$ such that $ff^\prime$, $f^\prime f$ are close to the identity maps. 
\end{dfn}

\begin{dfn}[\cite{NR}]\label{gccsegggg}
Let $G$ be a topological group. Let 
\begin{align*}
q\colon G\times G&\to G\\
(g,g^\prime)&\mapsto g^{-1}g^\prime. 
\end{align*}
Define 
\begin{align*}
\mca{E}_G&=\left\{E\subset G\times G\ \middle|\ \text{$E\subset G(K\times K)$ for some compact set $K\subset G$}\right\}\\
&=\left\{E\subset G\times G\ \middle|\ \text{$q(E)$ is relatively compact in $G$}\right\}. 
\end{align*}
Then $\mca{E}_G$ is a coarse structure on $G$, called the \emph{group-compact coarse structure} on $G$. 
\end{dfn}

\begin{dfn}
Let $G$ be a group and $d$ be a left invariant metric on $G$. Define 
\begin{equation*}
\mca{E}_d=\left\{E\subset G\times G\ \middle|\ \sup_{(g,g^\prime)\in E}d(g,g^\prime)<\infty\right\}. 
\end{equation*}
Then $\mca{E}_d$ is a coarse structure on $G$, called the \emph{bounded coarse structure} associated to $d$. 
\end{dfn}

\begin{rem}
A base-bounded compatible bornology on a groupoid can be seen as a generalization of a connected coarse structure. Indeed a base-bounded compatible bornology on a pair groupoid $X\times X\rightrightarrows X$ is a connected coarse structure on $X$. 

For a Hausdorff space $X$, the locally compact bornology on the pair groupoid $X\times X\rightrightarrows X$ is the indiscrete coarse structure on $X$ (see Example 2.8 of \cite{Roe}). 
\end{rem}

\subsection{Fiberwise coarse structures}
\begin{dfn}\label{fibwcoaseee}
Let $X$, $Y$ be sets and $f\colon X\to Y$ be a surjective map. Let $\mca{E}\subset P(X\times_{f,f}X)$. We say that $\mca{E}$ is a \emph{fiberwise coarse structure} of $f$ if: 
\begin{itemize}
\setlength\itemsep{0em}
\item $\Delta_X\in\mca{E}$, where $\Delta_X=\left\{(x,x)\in X\times_{f,f}X\ \middle|\ x\in X\right\}$
\item $E\cup E^\prime\in\mca{E}$ for any $E$, $E^\prime\in\mca{E}$
\item $E^\prime\in\mca{E}$ for any $E\in\mca{E}$ and $E^\prime\subset E$
\item $E^{-1}\in\mca{E}$ if $E\in\mca{E}$, where $E^{-1}=\left\{(x^\prime,x)\in X\times_{f,f}X\ \middle|\ (x,x^\prime)\in E\right\}$
\item $E\circ E^\prime\in\mca{E}$ for $E$, $E^\prime\in\mca{E}$, where 
\begin{equation*}
E\circ E^\prime=\left\{(x,x^\prime)\in X\times_{f,f}X\ \middle|\ \text{$(x,x^{\prime\prime})\in E$, $(x^{\prime\prime},x^\prime)\in E^\prime$ for some $x^{\prime\prime}\in X$}\right\}. 
\end{equation*}
\end{itemize}
When $Y=\mr{pt}$, we recover the definition of a coarse structure. 
\end{dfn}

\begin{dfn}
Let $X$, $Y$ be sets, $f\colon X\to Y$ be a surjective map and $\mca{E}$ be a fiberwise coarse structure of $f$. We say that $\mca{E}$ is \emph{connected} if $\left\{(x,x^\prime)\right\}\in\mca{E}$ for any $(x,x^\prime)\in X\times_{f,f}X$. 
\end{dfn}

\begin{dfn}
Let $X$, $Y$, $X^\prime$, $Y^\prime$ be sets, $f\colon X\to Y$, $f^\prime\colon X^\prime\to Y^\prime$ be surjective maps and $\mca{E}$ (resp. $\mca{E}^\prime$) be a fiberwise coarse structure of $f$ (resp. $f^\prime$). Let $\varphi\colon X\to X^\prime$ and $F\colon Y\to Y^\prime$ be such that 
\begin{equation*}
\begin{tikzcd}
X\ar[r,"\varphi"]\ar[d,"f"']\ar[rd,phantom,"\circlearrowright"{xshift=1,yshift=-1}]&X^\prime\ar[d,"f^\prime"]\\
Y\ar[r,"F"']&Y^\prime. 
\end{tikzcd}
\end{equation*}
We say that $\varphi$ is \emph{bornologous} if $(\varphi\times\varphi)(E)\in\mca{E}^\prime$ for any $E\in\mca{E}$. 
\end{dfn}

\begin{prop}\label{xyfeyefy}
Let $X$, $Y$ be sets, $f\colon X\to Y$ be a surjective map and $\mca{E}$ be a fiberwise coarse structure of $f$. Then for any $y\in Y$, 
\begin{align*}
\mca{E}_y&=\left\{E\cap(f^{-1}(y)\times f^{-1}(y))\ \middle|\ E\in\mca{E}\right\}\\
&=\left\{E\in\mca{E}\ \middle|\ E\subset f^{-1}(y)\times f^{-1}(y)\right\}
\end{align*}
is a coarse structure on $f^{-1}(y)$. 
\end{prop}

\begin{proof}
Easy. 
\end{proof}

\begin{prop}\label{wywyffphiyyfy}
Let $X$, $Y$, $X^\prime$, $Y^\prime$ be sets, $f\colon X\to Y$, $f^\prime\colon X^\prime\to Y^\prime$ be surjective maps and $\mca{E}$ (resp. $\mca{E}^\prime$) be a fiberwise coarse structure of $f$ (resp. $f^\prime$). Let $\varphi\colon X\to X^\prime$ and $F\colon Y\to Y^\prime$ be such that 
\begin{equation*}
\begin{tikzcd}
X\ar[r,"\varphi"]\ar[d,"f"']\ar[rd,phantom,"\circlearrowright"{xshift=1,yshift=-1}]&X^\prime\ar[d,"f^\prime"]\\
Y\ar[r,"F"']&Y^\prime
\end{tikzcd}
\end{equation*}
and $\varphi$ is bornologous. Then for any $y\in Y$, $\varphi\colon f^{-1}(y)\to(f^\prime)^{-1}(F(y))$ is bornologous. 
\end{prop}

\begin{proof}
This follows from the second expression of $\mca{E}_y$ in Proposition \ref{xyfeyefy}. 
\end{proof}

\begin{dfn}
Let $\mca{G}\rightrightarrows M$ be a groupoid. Let 
\begin{align*}
t\colon\mca{G}\times_{t,t}\mca{G}&\to M\\
(g,g^\prime)&\mapsto t(g)
\end{align*}
and define $\mca{G}\curvearrowright\mca{G}\times_{t,t}\mca{G}$ by 
\begin{align*}
\mca{G}\times_{s,t}(\mca{G}\times_{t,t}\mca{G})&\to\mca{G}\times_{t,t}\mca{G}\\
(g,(g^\prime,g^{\prime\prime}))&\mapsto(gg^\prime,gg^{\prime\prime}). 
\end{align*}
Let $\mca{E}\subset P(\mca{G}\times_{t,t}\mca{G})$ be a fiberwise coarse structure of $t\colon\mca{G}\to M$. We say that $\mca{E}$ is \emph{left invariant} if $\mca{G}E\in\mca{E}$ for any $E\in\mca{E}$. Note that $\mca{E}$ is left invariant if and only if any $E\in\mca{E}$ is contained in a $\mca{G}$-invariant $E^\prime\in\mca{E}$. 
\end{dfn}

\subsection{The correspondence between bornologies and fiberwise coarse structures}
\begin{prop}\label{glfcsbbcb}
Let $\mca{G}\rightrightarrows M$ be a groupoid and let 
\begin{align*}
q\colon\mca{G}\times_{t,t}\mca{G}&\to\mca{G}\\
(g,g^\prime)&\mapsto g^{-1}g^\prime. 
\end{align*}
Then the following statements hold: 
\begin{enumerate}
\item For a connected left invariant fiberwise coarse structure $\mca{E}$ of $t\colon\mca{G}\to M$, 
\begin{equation*}
\mca{B}_\mca{E}=\left\{q(E)\subset\mca{G}\ \middle|\ E\in\mca{E}\right\}
\end{equation*}
is a base-bounded compatible bornology on $\mca{G}$. 
\item For a base-bounded compatible bornology $\mca{B}$ on $\mca{G}$, 
\begin{equation}\label{ebeggegbb}
\begin{aligned}
\mca{E}_\mca{B}&=\left\{E\subset\mca{G}\times_{t,t}\mca{G}\ \middle|\ \text{$E\subset\mca{G}(B\times_{t,t}B)$ for some $B\in\mca{B}$}\right\}\\
&=\left\{E\subset\mca{G}\times_{t,t}\mca{G}\ \middle|\ q(E)\in\mca{B}\right\}
\end{aligned}
\end{equation}
is a connected left invariant fiberwise coarse structure of $t\colon\mca{G}\to M$. 
\item We have a bijection 
\begin{align*}
\left\{
\begin{gathered}
\text{connected left invariant fiberwise}\\
\text{coarse structures of $t\colon\mca{G}\to M$}
\end{gathered}
\right\}&\simeq\left\{
\begin{gathered}
\text{base-bounded compatible}\\
\text{bornologies on $\mca{G}$}
\end{gathered}
\right\}\\
\mca{E}&\mapsto\mca{B}_\mca{E}\\
\mca{E}_\mca{B}&\mapsfrom\mca{B}. 
\end{align*}
\end{enumerate}
\end{prop}

\begin{proof}
1. We have $B$, $B^\prime\in\mca{B}_\mca{E}$ $\Rightarrow$ $B\cup B^\prime\in\mca{B}_\mca{E}$. $B\in\mca{B}_\mca{E}$, $B^\prime\subset B$ $\Rightarrow$ $B^\prime\in\mca{B}_\mca{E}$ by the axiom of choice. For $g\in\mca{G}$, $\left\{(1_{t(g)},g)\right\}\in\mca{E}$ by the connectedness of $\mca{E}$, hence $\left\{g\right\}\in\mca{B}_\mca{E}$. We have $1_M=q(\Delta_\mca{G})\in\mca{B}_\mca{E}$. Therefore $\mca{B}_\mca{E}$ is a base-bounded bornology on $\mca{G}$. 

We prove the compatibility. Let $B\in\mca{B}_\mca{E}$. Then $B=q(E)$ for some $E\in\mca{E}$. We have $E^{-1}\in\mca{E}$ and $B^{-1}=q(E)^{-1}=q(E^{-1})\in\mca{B}_\mca{E}$. 

Let $B$, $B^\prime\in\mca{B}_\mca{E}$. $B=q(E)$, $B^\prime=q(E^\prime)$ for some $E$, $E^\prime\in\mca{E}$. Since $\mca{E}$ is left invariant, there exists $\wt{E^\prime}\in\mca{E}$ such that $E^\prime\subset\wt{E^\prime}$ and $\wt{E^\prime}$ is $\mca{G}$-invariant. We have $BB^\prime\subset q(E)q(\wt{E^\prime})$. 

\begin{claim}
$q(E)q(\wt{E^\prime})=q(E\circ\wt{E^\prime})$. 
\end{claim}

\begin{proof}
$\supset$ is OK. We prove $\subset$. Let $(g,h)\in E$, $(g^\prime,h^\prime)\in\wt{E^\prime}$ be such that $s(h)=s(g^\prime)$. Since $\wt{E^\prime}$ is $\mca{G}$-invariant, we have $\wt{E^\prime}\ni h(g^\prime)^{-1}(g^\prime,h^\prime)=(h,h(g^\prime)^{-1}h^\prime)$. Hence $(g,h(g^\prime)^{-1}h^\prime)\in E\circ\wt{E^\prime}$ and $q(g,h)q(g^\prime,h^\prime)=g^{-1}h(g^\prime)^{-1}h^\prime\in q(E\circ\wt{E^\prime})$. 
\end{proof}

Since $E\circ\wt{E^\prime}\in\mca{E}$, $q(E\circ\wt{E^\prime})\in\mca{B}_\mca{E}$. Hence $BB^\prime\in\mca{B}_\mca{E}$. 

2. First we prove the second equality of \eqref{ebeggegbb}. Let $B\in\mca{B}$ and $E\subset\mca{G}(B\times_{t,t}B)$. Since $q(E)\subset B^{-1}B\in\mca{B}$ by the compatibility, we have $q(E)\in\mca{B}$. 

Conversely let $E\subset\mca{G}\times_{t,t}\mca{G}$ be such that $q(E)\in\mca{B}$. Since $1_M\in\mca{B}$, we have $B:=q(E)\cup1_M\in\mca{B}$. For any $(g,g^\prime)\in E$, $(g,g^\prime)=g(1_{s(g)},g^{-1}g^\prime)\in\mca{G}(B\times_{t,t}B)$. Hence $E\subset\mca{G}(B\times_{t,t}B)$. 

It is easy to check that $\mca{E}_\mca{B}$ is a connected left invariant fiberwise coarse structure of $t\colon\mca{G}\to M$ using the second expression of $\mca{E}_\mca{B}$ in \eqref{ebeggegbb}. (For $E$, $E^\prime\in\mca{E}_\mca{B}$, we have $q(E)$, $q(E^\prime)\in\mca{B}$ and $q(E\circ E^\prime)\subset q(E)q(E^\prime)\in\mca{B}$, hence $E\circ E^\prime\in\mca{E}_\mca{B}$.) 

3. $\mca{E}\mapsto\mca{B}_\mca{E}\mapsto\mca{E}_{\mca{B}_\mca{E}}$. We prove $\mca{E}_{\mca{B}_\mca{E}}=\mca{E}$. We have 
\begin{align*}
\mca{E}_{\mca{B}_\mca{E}}&=\left\{E\subset\mca{G}\times_{t,t}\mca{G}\ \middle|\ \text{$E\subset\mca{G}(B\times_{t,t}B)$ for some $B\in\mca{B}_\mca{E}$}\right\}\\
&=\left\{E\subset\mca{G}\times_{t,t}\mca{G}\ \middle|\ q(E)\in\mca{B}_\mca{E}\right\}, 
\end{align*}
\begin{equation*}
\mca{B}_\mca{E}=\left\{q(E)\subset\mca{G}\ \middle|\ E\in\mca{E}\right\}. 
\end{equation*}
The inclusion $\mca{E}_{\mca{B}_\mca{E}}\supset\mca{E}$ is obvious. Let $E\in\mca{E}_{\mca{B}_\mca{E}}$. So $q(E)\in\mca{B}_\mca{E}$. Then there exists $E^\prime\in\mca{E}$ such that $q(E)=q(E^\prime)=q(\mca{G}E^\prime)$. We have $\mca{G}E^\prime\in\mca{E}$ since $\mca{E}$ is left invariant. 

\begin{claim}
$E\subset\mca{G}E^\prime$. 
\end{claim}

\begin{proof}
Let $(g,h)\in E$. There exists $(g^\prime,h^\prime)\in E^\prime$ such that $g^{-1}h=(g^\prime)^{-1}h^\prime$. Then 
\begin{equation*}
(g,h)=g(1_{s(g)},g^{-1}h)=g(1_{s(g)},(g^\prime)^{-1}h^\prime)=g(g^\prime)^{-1}(g^\prime,h^\prime)\in\mca{G}E^\prime. \qedhere
\end{equation*}
\end{proof}

Therefore $E\in\mca{E}$ and $\mca{E}_{\mca{B}_\mca{E}}=\mca{E}$. 

$\mca{B}\mapsto\mca{E}_\mca{B}\mapsto\mca{B}_{\mca{E}_\mca{B}}$. We prove $\mca{B}_{\mca{E}_\mca{B}}=\mca{B}$. We have 
\begin{equation*}
\mca{B}_{\mca{E}_\mca{B}}=\left\{q(E)\subset\mca{G}\ \middle|\ E\in\mca{E}_\mca{B}\right\}, 
\end{equation*}
\begin{align*}
\mca{E}_\mca{B}&=\left\{E\subset\mca{G}\times_{t,t}\mca{G}\ \middle|\ \text{$E\subset\mca{G}(B\times_{t,t}B)$ for some $B\in\mca{B}$}\right\}\\
&=\left\{E\subset\mca{G}\times_{t,t}\mca{G}\ \middle|\ q(E)\in\mca{B}\right\}. 
\end{align*}
It follows that $\mca{B}_{\mca{E}_\mca{B}}\subset\mca{B}$. Conversely let $B\in\mca{B}$. Then $B^\prime:=B\cup1_M\in\mca{B}$ by the base-boundedness. Since $q(\mca{G}(B^\prime\times_{t,t}B^\prime))\subset(B^\prime)^{-1}B^\prime\in\mca{B}$ by the compatibility, $\mca{G}(B^\prime\times_{t,t}B^\prime)\in\mca{E}_\mca{B}$. We have $B\subset q(\mca{G}(B^\prime\times_{t,t}B^\prime))$. There exists $E\subset\mca{G}(B^\prime\times_{t,t}B^\prime)$ such that $B=q(E)$ by the axiom of choice. Since $E\in\mca{E}_\mca{B}$, $B\in\mca{B}_{\mca{E}_\mca{B}}$. Hence $\mca{B}_{\mca{E}_\mca{B}}\supset\mca{B}$. 
\end{proof}

\begin{prop}\label{gmgmttbbee}
Let $\mca{G}\rightrightarrows M$, $\mca{G}^\prime\rightrightarrows M^\prime$ be groupoids, $\mca{E}$ (resp. $\mca{E}^\prime$) be a connected left invariant fiberwise coarse structure of $t\colon\mca{G}\to M$ (resp. $t^\prime\colon\mca{G}^\prime\to M^\prime$) and $\varphi\in\Hom(\mca{G},\mca{G}^\prime)$. Then $\varphi$ is bornologous if and only if $\varphi$ is bounded with respect to $\mca{B}_\mca{E}$ and $\mca{B}_{\mca{E}^\prime}$. 
\end{prop}

\begin{proof}
This follows from 
\begin{equation*}
\begin{tikzcd}
\mca{G}\times_{t,t}\mca{G}\ar[r,"\varphi\times\varphi"]\ar[d,"q"']\ar[rd,phantom,"\circlearrowright"{xshift=-6}]&\mca{G}^\prime\times_{t^\prime,t^\prime}\mca{G}^\prime\ar[d,"q^\prime"]\\
\mca{G}\ar[r,"\varphi"']&\mca{G}^\prime
\end{tikzcd}
\end{equation*}
and Proposition \ref{glfcsbbcb}. 
\end{proof}

\begin{dfn}
Let $\mca{G}\rightrightarrows M$ be a topological groupoid such that $M$ is Hausdorff. Then we have base-bounded compatible bornologies $\mca{B}_\mca{G}$ and $\mca{B}^c_\mca{G}$ on $\mca{G}$. The connected left invariant fiberwise coarse structure of $t\colon\mca{G}\to M$ corresponding to $\mca{B}_\mca{G}$ (resp. $\mca{B}^c_\mca{G}$) is denoted by $\mca{E}_\mca{G}$ (resp. $\mca{E}^c_\mca{G}$). 

When $M=\mr{pt}$, we have $\mca{B}_\mca{G}=\mca{B}^c_\mca{G}$, hence $\mca{E}_\mca{G}=\mca{E}^c_\mca{G}$ and this coincides with the group-compact coarse structure on $\mca{G}$ (Definition \ref{gccsegggg}). 
\end{dfn}

\begin{dfn}
Let $\mca{G}\rightrightarrows M$ be a Lie groupoid, $N$ be a $C^\infty$ manifold, $\nu\colon N\to M$ be a $C^\infty$ map and $\rho\in\mca{A}(\nu,\mca{G})$. Then $\mca{B}_\rho$ is a base-bounded compatible bornology on $N\rtimes_\rho\mca{G}$. The connected left invariant fiberwise coarse structure of $t\colon N\rtimes_\rho\mca{G}\to N$ corresponding to $\mca{B}_\rho$ is denoted by $\mca{E}_\rho$. 
\end{dfn}

\begin{lem}
Let $\mca{G}\rightrightarrows M$ be a Lie groupoid, $N$ be a $C^\infty$ manifold, $\nu\colon N\to M$ be a $C^\infty$ map and $\rho\in\mca{A}(\nu,\mca{G})$. Then the following statements hold: 
\begin{enumerate}
\item $\varphi_\rho\colon(N\rtimes_\rho\mca{G},\mca{E}_\rho)\to(\mca{G},\mca{E}_\mca{G})$ is bornologous. 
\item For any $y\in N$, 
\begin{equation*}
\varphi_\rho\colon((N\rtimes_\rho\mca{G})^y,(\mca{E}_\rho)_y)\simeq(\mca{G}^{\nu(y)},(\mca{E}_\mca{G})_{\nu(y)})
\end{equation*}
as coarse spaces. 
\end{enumerate}
\end{lem}

\begin{proof}
1. Since $\varphi_\rho\colon(N\rtimes_\rho\mca{G},\mca{B}_\rho)\to(\mca{G},\mca{B}_\mca{G})$ is bounded, this follows from Proposition \ref{gmgmttbbee}. 

2. $\varphi_\rho\colon(N\rtimes_\rho\mca{G})^y\simeq\mca{G}^{\nu(y)}$ is bornologous by 1 and Proposition \ref{wywyffphiyyfy}. Conversely let $E\subset(N\rtimes_\rho\mca{G})^y\times(N\rtimes_\rho\mca{G})^y$ be such that $(\varphi_\rho\times\varphi_\rho)(E)\in(\mca{E}_\mca{G})_{\nu(y)}$. Then $\mca{B}_\mca{G}\ni q(\varphi_\rho\times\varphi_\rho)(E)=\varphi_\rho q(E)$, hence $q(E)\in\mca{B}_\rho$. Thus $E\in\mca{E}_\rho$ and $E\in(\mca{E}_\rho)_y$. 
\end{proof}

\begin{cor}\label{gmgmngngde}
Let $\mca{G}\rightrightarrows M$, $\mca{G}^\prime\rightrightarrows M^\prime$ be Lie groupoids, $N$, $N^\prime$ be $C^\infty$ manifolds, $\nu\colon N\to M$, $\nu^\prime\colon N^\prime\to M^\prime$ be $C^\infty$ maps, $\rho\in\mca{A}(\nu,\mca{G})$, $\rho^\prime\in\mca{A}(\nu^\prime,\mca{G}^\prime)$ and $\varphi\in\Hom(N\rtimes_\rho\mca{G},N^\prime\rtimes_{\rho^\prime}\mca{G}^\prime)$. For any $y\in N$, define $\ol{\varphi}(y,\cdot)\colon\mca{G}^{\nu(y)}\to(\mca{G}^\prime)^{\nu^\prime F_\varphi(y)}$ by 
\begin{equation*}
\begin{tikzcd}
(N\rtimes_\rho\mca{G})^y\ar[r,"\varphi"]\ar[d,dash,"\varphi_\rho"',"\sim"{sloped}]\ar[rd,phantom,"\circlearrowright"{xshift=-1,yshift=-1}]&(N^\prime\rtimes_{\rho^\prime}\mca{G}^\prime)^{F_\varphi(y)}\ar[d,dash,"\varphi_{\rho^\prime}"',"\sim"{sloped}]\\
\mca{G}^{\nu(y)}\ar[r,"\ol{\varphi}(y{,}\cdot)"']&(\mca{G}^\prime)^{\nu^\prime F_\varphi(y)}. 
\end{tikzcd}
\end{equation*}
Then the following statements hold: 
\begin{enumerate}
\item If $\varphi\colon N\rtimes_\rho\mca{G}\to N^\prime\rtimes_{\rho^\prime}\mca{G}^\prime$ is bounded with respect to $\mca{B}_\rho$ and $\mca{B}_{\rho^\prime}$, then $\ol{\varphi}(y,\cdot)$ is bornologous. 
\item If $\varphi\colon N\rtimes_\rho\mca{G}\to N^\prime\rtimes_{\rho^\prime}\mca{G}^\prime$ is a bibounded isomorphism with respect to $\mca{B}_\rho$ and $\mca{B}_{\rho^\prime}$, then $\ol{\varphi}(y,\cdot)$ is an isomorphism of coarse spaces (in particular a coarse equivalence). 
\end{enumerate}
\end{cor}

\begin{rem}
2 in Corollary \ref{gmgmngngde} is a ``generalization'' of 4.1.3 Lemma in \cite{Ma2}. 
\end{rem}

\section{The Teichm\"{u}ller spaces of an action}\label{888}
\subsection{The Teichm\"{u}ller spaces of an action}
\begin{dfn}\label{gmngprmagl}
Let $\mca{G}\rightrightarrows M$ be a Lie groupoid, $N_0$ be a $C^\infty$ manifold, $\nu_0\colon N_0\to M$ be a surjective submersion and $\rho_0\in\mca{A}(\nu_0,\mca{G})$. 
\begin{itemize}
\item $(N\rtimes_\rho\mca{G},\varphi)$ is a \emph{$\rho_0$-marked action groupoid} if: 
\begin{itemize}
\item $N$ is a $C^\infty$ manifold, $\nu\colon N\to M$ is a $C^\infty$ map and $\rho\in\mca{A}(\nu,\mca{G})$
\item $\varphi\colon N_0\rtimes_{\rho_0}\mca{G}\to N\rtimes_\rho\mca{G}$ is an action groupoid isomorphism (Definition \ref{acgrmois}). 
\end{itemize}
\item $\MAG(\rho_0)=\left\{\text{$\rho_0$-marked action groupoids}\right\}$. 
\item Let $(N\rtimes_\rho\mca{G},\varphi)\in\MAG(\rho_0)$. We say that $(N\rtimes_\rho\mca{G},\varphi)$ is \emph{bibounded} if $\varphi$ is bibounded with respect to the action groupoid bornologies $\mca{B}_{\rho_0}$ and $\mca{B}_\rho$ (Definition \ref{acgrbong}). 
\item $\MAG^{bb}(\rho_0)=\left\{\text{bibounded $\rho_0$-marked action groupoids}\right\}$. We have $(N_0\rtimes_{\rho_0}\mca{G},\id)\in\MAG^{bb}(\rho_0)$. 
\item Let $(N\rtimes_\rho\mca{G},\varphi)$, $(N^\prime\rtimes_{\rho^\prime}\mca{G},\varphi^\prime)\in\MAG(\rho_0)$ and $a\in\left\{o,nt,bnt\right\}$. We write $(N\rtimes_\rho\mca{G},\varphi)\mathrel{\ul{\sim_a}}(N^\prime\rtimes_{\rho^\prime}\mca{G},\varphi^\prime)$ if there exists a conjugacy $\psi\in\ul{\Hom}(N\rtimes_\rho\mca{G},N^\prime\rtimes_{\rho^\prime}\mca{G})$ such that $\psi\varphi\mathrel{\ul{\sim_a}}\varphi^\prime$. 
\begin{equation*}
\begin{tikzcd}[row sep=tiny]
&N\rtimes_\rho\mca{G}\ar[dd,"\psi"]\\
N_0\rtimes_{\rho_0}\mca{G}\ar[ru,"\varphi"]\ar[rd,"\varphi^\prime"']\\
&N^\prime\rtimes_{\rho^\prime}\mca{G}
\end{tikzcd}
\end{equation*}
Note that $\psi\varphi\mathrel{\ul{\sim_a}}\varphi^\prime$ if and only if $\varphi^\prime\varphi^{-1}\mathrel{\ul{\sim_a}}\psi$ by Lemma \ref{ulsimsim} and Lemma \ref{ulsimsim2}. For a semiconjugacy $\psi=F_\psi\rtimes\Psi\colon N\rtimes_\rho\mca{G}\to N^\prime\rtimes_{\rho^\prime}\mca{G}$, we have $\psi\in\ul{\Hom}(N\rtimes_\rho\mca{G},N^\prime\rtimes_{\rho^\prime}\mca{G})$ if and only if $F_\Psi=\id$ by Lemma \ref{ulid}. 
\end{itemize}
\end{dfn}

\begin{rem}
In the above definition we often do not need the assumption that $\nu_0$ is a \emph{surjective submersion}, but we assume so for the sake of simplicity. 
\end{rem}

\begin{rem}
Note that the notation $(N\rtimes_\rho\mca{G},\varphi)$ does not contain $\nu$. For $(N\rtimes_\rho\mca{G},\varphi)\in\MAG(\rho_0)$, $\nu$ is always meant to be the map $\nu=\nu_0F_\varphi^{-1}\colon N\to M$. 
\end{rem}

\begin{lem}
Let $\mca{G}\rightrightarrows M$ be a Lie groupoid, $N_0$ be a $C^\infty$ manifold, $\nu_0\colon N_0\to M$ be a surjective submersion and $\rho_0\in\mca{A}(\nu_0,\mca{G})$. Let $a\in\left\{o,nt,bnt\right\}$. Then $\mathrel{\ul{\sim_a}}$ is an equivalence relation on $\MAG(\rho_0)$. The equivalence class of $(N\rtimes_\rho\mca{G},\varphi)$ is denoted by $[N\rtimes_\rho\mca{G},\varphi]$. 
\end{lem}

\begin{proof}
We have $(N\rtimes_\rho\mca{G},\varphi)\mathrel{\ul{\sim_a}}(N\rtimes_\rho\mca{G},\varphi)$ for any $(N\rtimes_\rho\mca{G},\varphi)\in\MAG(\rho_0)$. 

Assume $(N\rtimes_\rho\mca{G},\varphi)\mathrel{\ul{\sim_a}}(N^\prime\rtimes_{\rho^\prime}\mca{G},\varphi^\prime)$. There exists a conjugacy $\psi\in\ul{\Hom}(N\rtimes_\rho\mca{G},N^\prime\rtimes_{\rho^\prime}\mca{G})$ such that $\psi\varphi\mathrel{\ul{\sim_a}}\varphi^\prime$. Then $\psi^{-1}$ is bounded by Proposition \ref{semibounded}. Hence $\psi^{-1}\varphi^\prime\mathrel{\ul{\sim_a}}\varphi$ by Lemma \ref{ulsimsim} and Lemma \ref{ulsimsim2}. So $(N^\prime\rtimes_{\rho^\prime}\mca{G},\varphi^\prime)\mathrel{\ul{\sim_a}}(N\rtimes_\rho\mca{G},\varphi)$. 

Assume $(N\rtimes_\rho\mca{G},\varphi)\mathrel{\ul{\sim_a}}(N^\prime\rtimes_{\rho^\prime}\mca{G},\varphi^\prime)$ and $(N^\prime\rtimes_{\rho^\prime}\mca{G},\varphi^\prime)\mathrel{\ul{\sim_a}}(N^{\prime\prime}\rtimes_{\rho^{\prime\prime}}\mca{G},\varphi^{\prime\prime})$. There exist conjugacies $\psi\in\ul{\Hom}(N\rtimes_\rho\mca{G},N^\prime\rtimes_{\rho^\prime}\mca{G})$, $\psi^\prime\in\ul{\Hom}(N^\prime\rtimes_{\rho^\prime}\mca{G},N^{\prime\prime}\rtimes_{\rho^{\prime\prime}}\mca{G})$ such that $\psi\varphi\mathrel{\ul{\sim_a}}\varphi^\prime$ and $\psi^\prime\varphi^\prime\mathrel{\ul{\sim_a}}\varphi^{\prime\prime}$. Since $\psi^\prime$ is bounded, $\psi^\prime\psi\varphi\mathrel{\ul{\sim_a}}\psi^\prime\varphi^\prime\mathrel{\ul{\sim_a}}\varphi^{\prime\prime}$. Hence $(N\rtimes_\rho\mca{G},\varphi)\mathrel{\ul{\sim_a}}(N^{\prime\prime}\rtimes_{\rho^{\prime\prime}}\mca{G},\varphi^{\prime\prime})$. 
\end{proof}

\begin{dfn}\label{Teistarsim}
The \emph{Teichm\"{u}ller spaces} of $\rho_0$ are defined as 
\begin{equation*}
T_a(\rho_0)=\MAG(\rho_0)/\ul{\sim_a}
\end{equation*}
for $a\in\left\{o,nt\right\}$ and 
\begin{equation*}
T_{bnt}(\rho_0)=\MAG^{bb}(\rho_0)/\ul{\sim_{bnt}}. 
\end{equation*}
We have 
\begin{equation*}
T_{bnt}(\rho_0)\to T_{nt}(\rho_0)\twoheadrightarrow T_o(\rho_0). 
\end{equation*}
\end{dfn}

\begin{prop}\label{gmncptbnt}
Let $\mca{G}\rightrightarrows M$ be a Lie groupoid, $N_0$ be a $C^\infty$ manifold, $\nu_0\colon N_0\to M$ be a surjective submersion and $\rho_0\in\mca{A}(\nu_0,\mca{G})$. If $N_0$ is compact or $t\colon\mca{G}\to M$ is proper, then $T_{bnt}(\rho_0)=T_{nt}(\rho_0)$. 
\end{prop}

\begin{proof}
Let $(N\rtimes_\rho\mca{G},\varphi)\in\MAG(\rho_0)$. If $N_0$ is compact or $t\colon\mca{G}\to M$ is proper, $\varphi$ is bibounded by Proposition \ref{semibounded}, hence $\MAG^{bb}(\rho_0)=\MAG(\rho_0)$. 

Let $(N\rtimes_\rho\mca{G},\varphi)$, $(N^\prime\rtimes_{\rho^\prime}\mca{G},\varphi^\prime)\in\MAG(\rho_0)$ be such that $(N\rtimes_\rho\mca{G},\varphi)\mathrel{\ul{\sim_{nt}}}(N^\prime\rtimes_{\rho^\prime}\mca{G},\varphi^\prime)$. There exists a conjugacy $\psi\in\ul{\Hom}(N\rtimes_\rho\mca{G},N^\prime\rtimes_{\rho^\prime}\mca{G})$ such that $\psi\varphi\mathrel{\ul{\sim_{nt}}}\varphi^\prime$. Let $P\colon N_0\times I\to N^\prime\rtimes_{\rho^\prime}\mca{G}$ be an $nt$-homotopy between $\psi\varphi$ and $\varphi^\prime$. 

\begin{itemize}
\item If $N_0$ is compact, $P(N_0\times I)$ is compact, hence $P(N_0\times I)\in\mca{B}_{\rho^\prime}$. 
\item If $t\colon\mca{G}\to M$ is proper, then $\mca{B}_{\rho^\prime}=P(N^\prime\rtimes_{\rho^\prime}\mca{G})$ by Proposition \ref{crhol}. Hence $P(N_0\times I)\in\mca{B}_{\rho^\prime}$. 
\end{itemize}
Therefore $\psi\varphi\mathrel{\ul{\sim_{bnt}}}\varphi^\prime$ and $(N\rtimes_\rho\mca{G},\varphi)\mathrel{\ul{\sim_{bnt}}}(N^\prime\rtimes_{\rho^\prime}\mca{G},\varphi^\prime)$. 
\end{proof}

In Section \ref{morlieorb} we will prove $T_{nt}(\rho_0)=T_o(\rho_0)$ if $\rho_0$ is locally free. 

\begin{lem}\label{parae}
Let $\mca{G}\rightrightarrows M$ be a Lie groupoid, $N_0$ be a $C^\infty$ manifold, $\nu_0\colon N_0\to M$ be a surjective submersion and $\rho_0\in\mca{A}(\nu_0,\mca{G})$. Let $a\in\left\{o,nt,bnt\right\}$. Then the following are equivalent: 
\begin{enumerate}
\item $T_a(\rho_0)$ is a singleton. 
\item For any $(N\rtimes_\rho\mca{G},\varphi)\in\MAG(\rho_0)$ ($\MAG^{bb}(\rho_0)$ when $a=bnt$), there exists a conjugacy $\varphi^\prime\in\ul{\Hom}(N_0\rtimes_{\rho_0}\mca{G},N\rtimes_\rho\mca{G})$ such that $\varphi\mathrel{\ul{\sim_a}}\varphi^\prime$. 
\end{enumerate}
\end{lem}

\begin{proof}
Each of the two implications follows from the diagram 
\begin{equation*}
\begin{tikzcd}[row sep=tiny]
&N_0\rtimes_{\rho_0}\mca{G}\ar[dd,"\exists\varphi^\prime"]\\
N_0\rtimes_{\rho_0}\mca{G}\ar[ru,"\id"]\ar[rd,"\varphi"']\\
&N\rtimes_\rho\mca{G}, 
\end{tikzcd}
\end{equation*}
where $\varphi^\prime\in\ul{\Hom}(N_0\rtimes_{\rho_0}\mca{G},N\rtimes_\rho\mca{G})$ is a conjugacy such that $\varphi^\prime\mathrel{\ul{\sim_a}}\varphi$. 
\end{proof}

\begin{example}
Let $\mca{G}\rightrightarrows M$ be a Lie groupoid. Then $\mca{A}(\id,\mca{G})=\left\{\tau\right\}$ and $T_{bnt}(\tau)=T_{nt}(\tau)=T_o(\tau)=\mr{pt}$. 

\begin{proof}
Let $(N\rtimes_\rho\mca{G},\varphi)\in\MAG(\tau)$. We have 
\begin{equation*}
\begin{tikzcd}[column sep=tiny]
M\ar[rr,"F_\varphi"{name=U}]\ar[rd,"\id"']&&N\ar[ld,"\nu"]\\
&M, \ar[to=U,phantom,"\circlearrowright"]
\end{tikzcd}
\end{equation*}
hence $\nu$ is a diffeomorphism. Since $\varphi_\tau$, $\varphi_\rho$ are isomorphisms, we can define the bottom map in the diagram 
\begin{equation*}
\begin{tikzcd}
M\rtimes_\tau\mca{G}\ar[r,"\sim","\varphi"{yshift=7}]\ar[d,"\sim"{sloped},"\varphi_\tau"']\ar[rd,phantom,"\circlearrowright"{xshift=-3,yshift=1}]&N\rtimes_\rho\mca{G}\ar[d,"\sim"{sloped},"\varphi_\rho"{xshift=7}]\\
\mca{G}\ar[r]&\mca{G}. 
\end{tikzcd}
\end{equation*}
Therefore $\varphi$ is a conjugacy and $T_{bnt}(\tau)=T_{nt}(\tau)=T_o(\tau)=\mr{pt}$ by Lemma \ref{parae}. 
\end{proof}
\end{example}

\subsection{$\MAG_{\id}(\rho_0)$}
\begin{dfn}\label{magidmagnfp}
Let $\mca{G}\rightrightarrows M$ be a Lie groupoid, $N_0$ be a $C^\infty$ manifold, $\nu_0\colon N_0\to M$ be a surjective submersion and $\rho_0\in\mca{A}(\nu_0,\mca{G})$. Define 
\begin{equation*}
\MAG_{\id}(\rho_0)=\left\{(N\rtimes_\rho\mca{G},\varphi)\in\MAG(\rho_0)\ \middle|\ N=N_0,F_\varphi=\id\right\}
\end{equation*}
and 
\begin{equation*}
\MAG^{bb}_{\id}(\rho_0)=\MAG_{\id}(\rho_0)\cap\MAG^{bb}(\rho_0). 
\end{equation*}
\end{dfn}

\begin{lem}\label{magid}
Let $\mca{G}\rightrightarrows M$ be a Lie groupoid, $N_0$ be a $C^\infty$ manifold, $\nu_0\colon N_0\to M$ be a surjective submersion and $\rho_0\in\mca{A}(\nu_0,\mca{G})$. Then 
\begin{align*}
T_a(\rho_0)&\simeq\MAG_{\id}(\rho_0)/\ul{\sim_a}\\
[N\rtimes_\rho\mca{G},\varphi]&\mapsto[N_0\rtimes_{\rho^{(F_\varphi,\id)}}\mca{G},(F_\varphi^{-1}\rtimes\id)\varphi]\\
[N_0\rtimes_\rho\mca{G},\varphi]&\mapsfrom[N_0\rtimes_\rho\mca{G},\varphi]
\end{align*}
for $a\in\left\{o,nt\right\}$, where $\rho^{(F_\varphi,\id)}\in\mca{A}(\nu_0,\mca{G})$ is defined in Lemma \ref{ngconj} and 
\begin{equation*}
T_{bnt}(\rho_0)\simeq\MAG^{bb}_{\id}(\rho_0)/\ul{\sim_{bnt}}
\end{equation*}
by similar maps. 
\end{lem}

\begin{proof}
Let $a\in\left\{o,nt\right\}$. We have $\MAG_{\id}(\rho_0)\subset\MAG(\rho_0)$ and $\MAG_{\id}(\rho_0)/\ul{\sim_a}\to\MAG(\rho_0)/\ul{\sim_a}$ is injective. Let $(N\rtimes_\rho\mca{G},\varphi)\in\MAG(\rho_0)$. Then 
\begin{equation*}
\begin{tikzcd}[row sep=tiny]
&N\rtimes_\rho\mca{G}\\
N_0\rtimes_{\rho_0}\mca{G}\ar[ru,"\varphi"]\ar[rd,"(F_\varphi^{-1}\rtimes\id)\varphi"'{xshift=5}]\\
&N_0\rtimes_{\rho^{(F_\varphi,\id)}}\mca{G}\ar[uu,"F_\varphi\rtimes\id"'{name=U}]\ar[from=2-1,to=U,phantom,"\circlearrowright"]
\end{tikzcd}
\end{equation*}
and $F_\varphi\rtimes\id\in\ul{\Hom}(N_0\rtimes_{\rho^{(F_\varphi,\id)}}\mca{G},N\rtimes_\rho\mca{G})$ is a conjugacy. Since $F_{(F_\varphi^{-1}\rtimes\id)\varphi}=\id$, $(N_0\rtimes_{\rho^{(F_\varphi,\id)}}\mca{G},(F_\varphi^{-1}\rtimes\id)\varphi)\in\MAG_{\id}(\rho_0)$. Hence $(N\rtimes_\rho\mca{G},\varphi)\mathrel{\ul{\sim_a}}(N_0\rtimes_{\rho^{(F_\varphi,\id)}}\mca{G},(F_\varphi^{-1}\rtimes\id)\varphi)$ and $\MAG_{\id}(\rho_0)/\ul{\sim_a}\simeq\MAG(\rho_0)/\ul{\sim_a}$. 
\end{proof}

\subsection{The mapping class groups}\label{mpcgs}
Let $\mca{G}\rightrightarrows M$ be a Lie groupoid, $N_0$ be a $C^\infty$ manifold, $\nu_0\colon N_0\to M$ be a surjective submersion and $\rho_0\in\mca{A}(\nu_0,\mca{G})$. We have an action $\MAG(\rho_0)\curvearrowleft\ul{\Aut}(N_0\rtimes_{\rho_0}\mca{G})$ defined by $(N\rtimes_\rho\mca{G},\varphi)\theta=(N\rtimes_\rho\mca{G},\varphi\theta)$. 

\subsubsection{$a\in\left\{o,nt\right\}$}
\begin{claim}
$T_a(\rho_0)\curvearrowleft\ul{\Aut}(N_0\rtimes_{\rho_0}\mca{G})$ by $[N\rtimes_\rho\mca{G},\varphi]\theta=[N\rtimes_\rho\mca{G},\varphi\theta]$. 
\end{claim}

\begin{proof}
Let $(N\rtimes_\rho\mca{G},\varphi)$, $(N^\prime\rtimes_{\rho^\prime}\mca{G},\varphi^\prime)\in\MAG(\rho_0)$ be such that $(N\rtimes_\rho\mca{G},\varphi)\mathrel{\ul{\sim_a}}(N^\prime\rtimes_{\rho^\prime}\mca{G},\varphi^\prime)$. There exists a conjugacy $\psi\in\ul{\Hom}(N\rtimes_\rho\mca{G},N^\prime\rtimes_{\rho^\prime}\mca{G})$ such that $\psi\varphi\mathrel{\ul{\sim_a}}\varphi^\prime$. We have $\psi\varphi\theta\mathrel{\ul{\sim_a}}\varphi^\prime\theta$ by Lemma \ref{ulsimsim}. Hence $(N\rtimes_\rho\mca{G},\varphi\theta)\mathrel{\ul{\sim_a}}(N^\prime\rtimes_{\rho^\prime}\mca{G},\varphi^\prime\theta)$. 
\end{proof}

Let 
\begin{equation*}
\ul{\Aut}(N_0\rtimes_{\rho_0}\mca{G})_{a,1}=\left\{\theta\in\ul{\Aut}(N_0\rtimes_{\rho_0}\mca{G})\ \middle|\ \theta\mathrel{\ul{\sim_a}}\id\right\}. 
\end{equation*}

\begin{claim}
$\ul{\Aut}(N_0\rtimes_{\rho_0}\mca{G})_{a,1}$ is a normal subgroup of $\ul{\Aut}(N_0\rtimes_{\rho_0}\mca{G})$ and acts trivially on $T_a(\rho_0)$. 
\end{claim}

\begin{proof}
Let $\theta\in\ul{\Aut}(N_0\rtimes_{\rho_0}\mca{G})_{a,1}$ and $[N\rtimes_\rho\mca{G},\varphi]\in T_a(\rho_0)$. Since $\theta\mathrel{\ul{\sim_a}}\id$, we have $\varphi\mathrel{\ul{\sim_a}}\varphi\theta$. So $(N\rtimes_\rho\mca{G},\varphi\theta)\mathrel{\ul{\sim_a}}(N\rtimes_\rho\mca{G},\varphi)$. 
\begin{equation*}
\begin{tikzcd}[row sep=tiny]
&N\rtimes_\rho\mca{G}\ar[dd,"\id"]\\
N_0\rtimes_{\rho_0}\mca{G}\ar[ru,"\varphi"]\ar[rd,"\varphi\theta"']\\
&N\rtimes_\rho\mca{G}
\end{tikzcd}
\end{equation*}
\end{proof}

Define 
\begin{equation*}
\MCG_a(\rho_0)=\ul{\Aut}(N_0\rtimes_{\rho_0}\mca{G})/\ul{\Aut}(N_0\rtimes_{\rho_0}\mca{G})_{a,1}. 
\end{equation*}
Hence we get an action 
\begin{equation*}
T_a(\rho_0)\curvearrowleft\MCG_a(\rho_0). 
\end{equation*}

\subsubsection{$a=bnt$}
Let 
\begin{equation*}
\ul{\Aut}^{bb}(N_0\rtimes_{\rho_0}\mca{G})=\left\{\theta\in\ul{\Aut}(N_0\rtimes_{\rho_0}\mca{G})\ \middle|\ \text{$\theta$ is bibounded with respect to $\mca{B}_{\rho_0}$}\right\}. 
\end{equation*}
We have $\MAG^{bb}(\rho_0)\curvearrowleft\ul{\Aut}^{bb}(N_0\rtimes_{\rho_0}\mca{G})$ and $T_{bnt}(\rho_0)\curvearrowleft\ul{\Aut}^{bb}(N_0\rtimes_{\rho_0}\mca{G})$ by Lemma \ref{ulsimsim2}. Let 
\begin{equation*}
\ul{\Aut}(N_0\rtimes_{\rho_0}\mca{G})_{bnt,1}=\left\{\theta\in\ul{\Aut}(N_0\rtimes_{\rho_0}\mca{G})\ \middle|\ \text{$\theta\mathrel{\ul{\sim_{bnt}}}\id$ with respect to $\mca{B}_{\rho_0}$}\right\}. 
\end{equation*}
Then $\ul{\Aut}(N_0\rtimes_{\rho_0}\mca{G})_{bnt,1}\subset\ul{\Aut}^{bb}(N_0\rtimes_{\rho_0}\mca{G})$ by Lemma \ref{bounboun}. This is a normal subgroup of $\ul{\Aut}^{bb}(N_0\rtimes_{\rho_0}\mca{G})$ and acts trivially on $T_{bnt}(\rho_0)$ by Lemma \ref{ulsimsim2}. Define 
\begin{equation*}
\MCG_{bnt}(\rho_0)=\ul{\Aut}^{bb}(N_0\rtimes_{\rho_0}\mca{G})/\ul{\Aut}(N_0\rtimes_{\rho_0}\mca{G})_{bnt,1}. 
\end{equation*}
Then 
\begin{equation*}
T_{bnt}(\rho_0)\curvearrowleft\MCG_{bnt}(\rho_0). 
\end{equation*}

\begin{rem}
Usually the elements of the mapping class groups are isotopy classes, not homotopy classes. So the name ``mapping class group'' here may not be appropriate, but we call them so in this paper. 
\end{rem}

\section{Foliation groupoids}\label{1010}
\subsection{The monodromy groupoids}
\begin{dfn}\label{c0xfmon}
Let $M$ be a $C^\infty$ manifold and $\mca{F}$ be a $C^\infty$ regular foliation of $M$. For a locally connected topological space $X$, let 
\begin{align*}
C^0(X,\mca{F})&=\left\{F\in C^0(X,M)\ \middle|\ \text{there exists $L\in\mca{F}$ such that $F(X)\subset L$}\right\}\\
&=\left\{F\ \middle|\ 
\begin{gathered}
\text{$L\in\mca{F}$ and $F\colon X\to L$ is a continuous map}\\
\text{with respect to the submanifold topology of $L$}
\end{gathered}
\right\}. 
\end{align*}
The second equality follows from Corollary \ref{mrfmssf}. For $\gamma$, $\gamma^\prime\in C^0(I,\mca{F})$, write $\gamma\sim_\mca{F}\gamma^\prime$ if there exists $H\in C^0(I\times I,\mca{F})$ such that: 
\begin{itemize}
\setlength\itemsep{0em}
\item $H(\cdot,0)=\gamma$, $H(\cdot,1)=\gamma^\prime$
\item $H(0,s)=\gamma(0)$, $H(1,s)=\gamma(1)$ for all $s\in I$. 
\end{itemize}
Then $\sim_\mca{F}$ is an equivalence relation on $C^0(I,\mca{F})$. Define $\Mon(\mca{F})=C^0(I,\mca{F})/{\sim}_\mca{F}$ and let $[\gamma]$ denote the equivalence class of $\gamma$ with respect to $\sim_\mca{F}$. 
\end{dfn}

\begin{prop}\label{monts1lgm}
Let $M$ be a $C^\infty$ manifold and $\mca{F}$ be a $C^\infty$ regular foliation of $M$. Consider 
\begin{align*}
t\colon\Mon(\mca{F})&\to M&s\colon\Mon(\mca{F})&\to M&\Mon(\mca{F})\times_{s,t}\Mon(\mca{F})&\to\Mon(\mca{F})\\
[\gamma]&\mapsto\gamma(1), &[\gamma]&\mapsto\gamma(0), &([\gamma],[\gamma^\prime])&\mapsto[\gamma\gamma^\prime], 
\end{align*}
\begin{equation*}
\begin{aligned}
\Mon(\mca{F})&\to\Mon(\mca{F})\\
[\gamma]&\mapsto[\gamma^{-1}]
\end{aligned}
\qquad\text{and}\qquad
\begin{aligned}
1\colon M&\to\Mon(\mca{F})\\
x&\mapsto[x]. 
\end{aligned}
\end{equation*}
\begin{equation*}
\begin{tikzpicture}[every label/.append style={font=\scriptsize},matrix of math nodes,decoration={markings,mark=at position0.5with{\arrow{>}}}]
\node(1){};
\node(2)[right=of 1,label=below:\gamma\gamma^\prime]{};
\node(3)[right=of 2]{};
\draw[postaction={decorate}](2.center)--node[label=above:\gamma]{}(1.center);
\draw[postaction={decorate}](3.center)--node[label=above:\gamma^\prime]{}(2.center);
\foreach\x in{1,2,3}\filldraw(\x)circle(1pt);
\end{tikzpicture}
\end{equation*}
Then $\Mon(\mca{F})\rightrightarrows M$ is a Lie groupoid with these structure maps, called the \emph{monodromy groupoid} of $\mca{F}$. $\Mon(\mca{F})$ may be non Hausdorff, non second countable. 
\end{prop}

\begin{rem}
Another way to define a groupoid structure is as follows: 
\begin{align*}
t\colon\Mon(\mca{F})&\to M&s\colon\Mon(\mca{F})&\to M&\Mon(\mca{F})\times_{s,t}\Mon(\mca{F})&\to\Mon(\mca{F})\\
[\gamma]&\mapsto\gamma(0), &[\gamma]&\mapsto\gamma(1), &([\gamma],[\gamma^\prime])&\mapsto[\gamma\gamma^\prime]. 
\end{align*}
\begin{equation*}
\begin{tikzpicture}[every label/.append style={font=\scriptsize},matrix of math nodes,decoration={markings,mark=at position0.5with{\arrow{>}}}]
\node(1){};
\node(2)[right=of 1,label=below:\gamma\gamma^\prime]{};
\node(3)[right=of 2]{};
\draw[postaction={decorate}](1.center)--node[label=above:\gamma]{}(2.center);
\draw[postaction={decorate}](2.center)--node[label=above:\gamma^\prime]{}(3.center);
\foreach\x in{1,2,3}\filldraw(\x)circle(1pt);
\end{tikzpicture}
\end{equation*}
This may be a more common way to define the concatenation of paths. 
\end{rem}

\begin{proof}[Proof of Proposition \ref{monts1lgm}]
Let $p=\dim\mca{F}$ and $q=\dim M-\dim\mca{F}$. Let $[\gamma]\in\Mon(\mca{F})$. There exist foliation charts $(V,\psi)$, $(U,\varphi)$ of $\mca{F}$ such that: 
\begin{itemize}
\setlength\itemsep{0em}
\item $\gamma(1)\in V$, $\gamma(0)\in U$
\item $\psi\colon V\to V_1\times V_2$, $\varphi\colon U\to U_1\times U_2$ are diffeomorphisms, where $V_1$, $U_1$ are simply connected open neighborhoods of $0$ in $\bb{R}^p$, $V_2$, $U_2$ are open neighborhoods of $0$ in $\bb{R}^q$
\item $\psi(\gamma(1))=(0,0)$, $\varphi(\gamma(0))=(0,0)$
\item the holonomy $\hol_\gamma\colon S\to T$ along $\gamma$ is defined and is a bijection, where $T=\psi^{-1}(\left\{0\right\}\times V_2)$, $S=\varphi^{-1}(\left\{0\right\}\times U_2)$. 
\end{itemize}
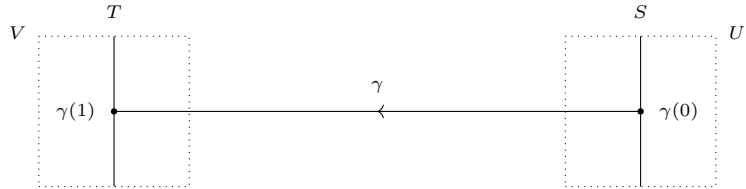
\begin{equation*}
\begin{tikzpicture}[every label/.append style={font=\scriptsize},matrix of math nodes,decoration={markings,mark=at position0.5with{\arrow{>}}}]
\def\a{2}
\def\b{1}
\def\c{7}
\node(1)at(0,0){};
\node(2)at(\a,0){};
\node(3)at(\a,\a){};
\node(4)at(0,\a)[label={[xshift=-8,yshift=-8]V}]{};
\node(5)at(\b,0){};
\node(6)at(\b,\a)[label=above:T]{};
\node(7)at(\c,0){};
\node(8)at(\c+\a,0){};
\node(9)at(\c+\a,\a)[label={[xshift=8,yshift=-8]U}]{};
\node(10)at(\c,\a){};
\node(11)at(\c+\b,0){};
\node(12)at(\c+\b,\a)[label=above:S]{};
\node(13)at(\b,\b){};
\node(14)at(\c+\b,\b){};
\node(15)at(\b,\b)[label=left:\gamma(1)]{};
\node(16)at(\c+\b,\b)[label=right:\gamma(0)]{};
\draw[dotted](1.center)--(2.center)--(3.center)--(4.center)--cycle;
\draw[dotted](7.center)--(8.center)--(9.center)--(10.center)--cycle;
\draw(5.center)--(6.center);
\draw(11.center)--(12.center);
\draw[postaction={decorate}](14.center)--node[label=above:\gamma]{}(13.center);
\foreach\x in{15,16}\filldraw(\x)circle(1pt);
\end{tikzpicture}
\end{equation*}
There exists a continuous map $H\colon I\times V_2\to M$ such that: 
\begin{itemize}
\setlength\itemsep{0em}
\item $H(\cdot,0)=\gamma$
\item $H(\cdot,v_2)\in C^0(I,\mca{F})$ for all $v_2\in V_2$
\item $H(1,v_2)=\psi^{-1}(0,v_2)$, $H(0,v_2)=\hol_\gamma^{-1}(\psi^{-1}(0,v_2))$ for all $v_2\in V_2$. 
\end{itemize}
For $(v_1,v_2,u_1)\in V_1\times V_2\times U_1$, let $\gamma^{v_1,v_2,u_1}\in C^0(I,\mca{F})$ be a concatenation of: 
\begin{itemize}
\setlength\itemsep{0em}
\item a continuous curve connecting $\psi^{-1}(0,v_2)$ and $\psi^{-1}(v_1,v_2)$ in a plaque of $(V,\psi)$
\item $H(\cdot,v_2)$
\item a continuous curve connecting $\varphi^{-1}(u_1,u_2)$ and $\varphi^{-1}(0,u_2)$ in a plaque of $(U,\varphi)$, where $u_2$ is defined by $\hol_\gamma^{-1}\psi^{-1}(0,v_2)=\varphi^{-1}(0,u_2)$. 
\end{itemize}
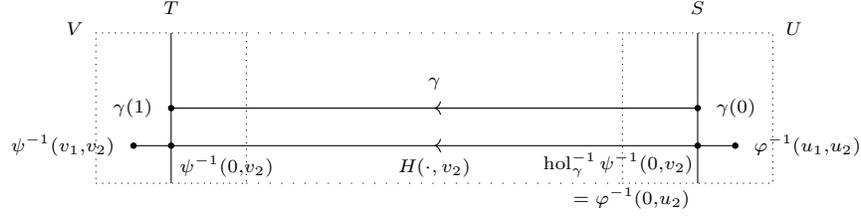
\begin{equation*}
\begin{tikzpicture}[every label/.append style={font=\scriptsize},matrix of math nodes,decoration={markings,mark=at position0.5with{\arrow{>}}}]
\def\a{2}
\def\b{1}
\def\c{7}
\node(1)at(0,0){};
\node(2)at(\a,0){};
\node(3)at(\a,\a){};
\node(4)at(0,\a)[label={[xshift=-8,yshift=-8]V}]{};
\node(5)at(\b,0){};
\node(6)at(\b,\a)[label=above:T]{};
\node(7)at(\c,0){};
\node(8)at(\c+\a,0){};
\node(9)at(\c+\a,\a)[label={[xshift=8,yshift=-8]U}]{};
\node(10)at(\c,\a){};
\node(11)at(\c+\b,0){};
\node(12)at(\c+\b,\a)[label=above:S]{};
\node(13)at(\b,\b){};
\node(14)at(\c+\b,\b){};
\node(15)at(\b/2,\b/2)[label=left:\psi^{-1}(v_1{,}v_2)]{};
\node(16)at(\c+\a-\b/2,\b/2)[label=right:\varphi^{-1}(u_1{,}u_2)]{};
\node(17)at(\b,\b)[label=left:\gamma(1)]{};
\node(18)at(\c+\b,\b)[label=right:\gamma(0)]{};
\node(19)at(\b,\b/2)[label={[xshift=21,yshift=-19]\psi^{-1}(0{,}v_2)}]{};
\node(20)at(\c+\b,\b/2)[label={[xshift=-30,yshift=-19]\hol_\gamma^{-1}\psi^{-1}(0{,}v_2)},label={[xshift=-25,yshift=-31]=\varphi^{-1}(0{,}u_2)}]{};
\draw[dotted](1.center)--(2.center)--(3.center)--(4.center)--cycle;
\draw[dotted](7.center)--(8.center)--(9.center)--(10.center)--cycle;
\draw[loosely dotted](2.center)--(7.center);
\draw[loosely dotted](3.center)--(10.center);
\draw(5.center)--(6.center);
\draw(11.center)--(12.center);
\draw[postaction={decorate}](14.center)--node[label=above:\gamma]{}(13.center);
\draw[postaction={decorate}](16.center)--node[label={[xshift=0,yshift=-19]H(\cdot,v_2)}]{}(15.center);
\foreach\x in{15,...,20}\filldraw(\x)circle(1pt);
\end{tikzpicture}
\end{equation*}
Let 
\begin{align*}
\theta\colon V_1\times V_2\times U_1&\to\Mon(\mca{F})\\
(v_1,v_2,u_1)&\mapsto[\gamma^{v_1,v_2,u_1}]. 
\end{align*}
Then 
\begin{equation*}
\left\{\theta(V_1\times V_2\times U_1)\ \middle|\ \gamma,V,U,V_i,U_i,\psi,\varphi,H\right\}
\end{equation*}
is a base for a topology of $\Mon(\mca{F})$ and 
\begin{equation*}
\left\{(\theta(V_1\times V_2\times U_1),\theta^{-1})\ \middle|\ \gamma,V,U,V_i,U_i,\psi,\varphi,H\right\}
\end{equation*}
is a $C^\infty$ atlas of $\Mon(\mca{F})$. The structure maps are $C^\infty$. We call $(\theta(V_1\times V_2\times U_1),\theta^{-1})$ a chart of $\Mon(\mca{F})$ \emph{centered} at $[\gamma]$. 
\end{proof}

\begin{prop}\label{monfmonfmb}
Let $M$, $M^\prime$ be $C^\infty$ manifolds, $\mca{F}$, $\mca{F}^\prime$ be $C^\infty$ regular foliations of $M$, $M^\prime$ and $F\colon M\to M^\prime$ be a $C^\infty$ map such that for each $L\in\mca{F}$, there exists $L^\prime\in\mca{F}^\prime$ satisfying $F(L)\subset L^\prime$. Then 
\begin{align*}
F_*\colon\Mon(\mca{F})&\to\Mon(\mca{F}^\prime)\\
[\gamma]&\mapsto[F\gamma]
\end{align*}
is a morphism with base map $F$. 
\end{prop}

To prove the proposition we use the following lemma (see Corollary 10.8 of \cite{Milnor_1963}), whose proof will be given in the next subsection. 

\begin{lem}\label{mfmgfedfxy}
Let $M$ be a $C^\infty$ manifold, $\mca{F}$ be a $C^\infty$ regular foliation of $M$, $g_\mca{F}$ be a leafwise Riemannian metric of $\mca{F}$, $d_\mca{F}$ be the leafwise metric induced by $g_\mca{F}$ and $K$ be a compact subset of $M$. Then there exists $\epsilon>0$ satisfying the following property: For any $x$, $y\in K$ in the same leaf of $\mca{F}$ with $d_\mca{F}(x,y)<\epsilon$, there exists a unique geodesic $c_{x,y}\colon I\to M$ in the leaf of length less than $\epsilon$ such that $c_{x,y}(0)=x$, $c_{x,y}(1)=y$. $c_{x,y}(t)$ depends continuously on $x$, $y$, $t$. 
\end{lem}

\begin{proof}[Proof of Proposition \ref{monfmonfmb}]
Let $[\gamma]\in\Mon(\mca{F})$. We will prove that $F_*$ is continuous at $[\gamma]$ and $C^\infty$ on an open neighborhood of $[\gamma]$. We have $F_*[\gamma]=[F\gamma]$. Take a chart $(\theta^\prime(V^\prime_1\times V^\prime_2\times U^\prime_1),(\theta^\prime)^{-1})$ of $\Mon(\mca{F}^\prime)$ centered at $[F\gamma]$: 
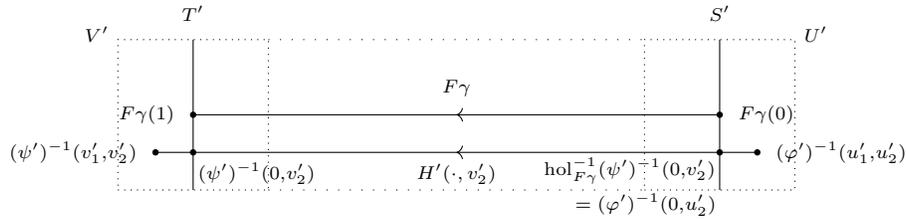
\begin{equation*}
\begin{tikzpicture}[every label/.append style={font=\scriptsize},matrix of math nodes,decoration={markings,mark=at position0.5with{\arrow{>}}}]
\def\a{2}
\def\b{1}
\def\c{7}
\node(1)at(0,0){};
\node(2)at(\a,0){};
\node(3)at(\a,\a){};
\node(4)at(0,\a)[label={[xshift=-8,yshift=-8]V^\prime}]{};
\node(5)at(\b,0){};
\node(6)at(\b,\a)[label=above:T^\prime]{};
\node(7)at(\c,0){};
\node(8)at(\c+\a,0){};
\node(9)at(\c+\a,\a)[label={[xshift=8,yshift=-8]U^\prime}]{};
\node(10)at(\c,\a){};
\node(11)at(\c+\b,0){};
\node(12)at(\c+\b,\a)[label=above:S^\prime]{};
\node(13)at(\b,\b){};
\node(14)at(\c+\b,\b){};
\node(15)at(\b/2,\b/2)[label=left:(\psi^\prime)^{-1}(v^\prime_1{,}v^\prime_2)]{};
\node(16)at(\c+\a-\b/2,\b/2)[label=right:(\varphi^\prime)^{-1}(u^\prime_1{,}u^\prime_2)]{};
\node(17)at(\b,\b)[label=left:F\gamma(1)]{};
\node(18)at(\c+\b,\b)[label=right:F\gamma(0)]{};
\node(19)at(\b,\b/2)[label={[xshift=24,yshift=-19](\psi^\prime)^{-1}(0{,}v^\prime_2)}]{};
\node(20)at(\c+\b,\b/2)[label={[xshift=-34,yshift=-19]\hol_{F\gamma}^{-1}(\psi^\prime)^{-1}(0{,}v^\prime_2)},label={[xshift=-28,yshift=-31]=(\varphi^\prime)^{-1}(0{,}u^\prime_2)}]{};
\draw[dotted](1.center)--(2.center)--(3.center)--(4.center)--cycle;
\draw[dotted](7.center)--(8.center)--(9.center)--(10.center)--cycle;
\draw[loosely dotted](2.center)--(7.center);
\draw[loosely dotted](3.center)--(10.center);
\draw(5.center)--(6.center);
\draw(11.center)--(12.center);
\draw[postaction={decorate}](14.center)--node[label=above:F\gamma]{}(13.center);
\draw[postaction={decorate}](16.center)--node[label={[xshift=0,yshift=-19]H^\prime(\cdot,v^\prime_2)}]{}(15.center);
\foreach\x in{15,...,20}\filldraw(\x)circle(1pt);
\end{tikzpicture}
\end{equation*}
\begin{align*}
\theta^\prime\colon V^\prime_1\times V^\prime_2\times U^\prime_1&\to\Mon(\mca{F}^\prime)\\
(v^\prime_1,v^\prime_2,u^\prime_1)&\mapsto[(\gamma^\prime)^{v^\prime_1,v^\prime_2,u^\prime_1}]
\end{align*}
We fix relatively compact foliation charts $V^\prime=W^\prime_\ell,\ldots,W^\prime_1=U^\prime$ of $\mca{F}^\prime$ and a partition $1=t_\ell>\cdots>t_0=0$ such that $F\gamma[t_{i-1},t_i]\subset W^\prime_i$ for all $i=1,\ldots,\ell$ which can be used to define the holonomy along $F\gamma$. Take a chart $(\theta(V_1\times V_2\times U_1),\theta^{-1})$ of $\Mon(\mca{F})$ centered at $[\gamma]$ such that $F\gamma^{(v_1,v_2,u_1)}[t_{i-1},t_i]\subset W^\prime_i$ for all $(v_1,v_2,u_1)\in V_1\times V_2\times U_1$ and $i=1,\ldots,\ell$. Note that $F(V)\subset V^\prime$, $F(U)\subset U^\prime$ in particular. 
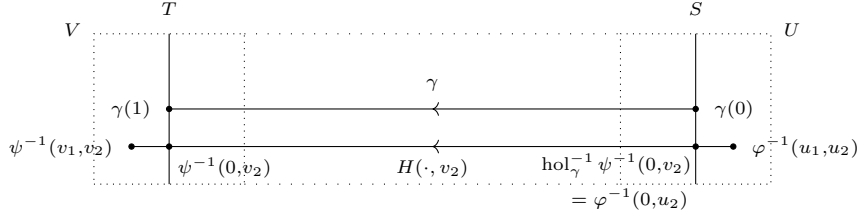
\begin{equation*}
\begin{tikzpicture}[every label/.append style={font=\scriptsize},matrix of math nodes,decoration={markings,mark=at position0.5with{\arrow{>}}}]
\def\a{2}
\def\b{1}
\def\c{7}
\node(1)at(0,0){};
\node(2)at(\a,0){};
\node(3)at(\a,\a){};
\node(4)at(0,\a)[label={[xshift=-8,yshift=-8]V}]{};
\node(5)at(\b,0){};
\node(6)at(\b,\a)[label=above:T]{};
\node(7)at(\c,0){};
\node(8)at(\c+\a,0){};
\node(9)at(\c+\a,\a)[label={[xshift=8,yshift=-8]U}]{};
\node(10)at(\c,\a){};
\node(11)at(\c+\b,0){};
\node(12)at(\c+\b,\a)[label=above:S]{};
\node(13)at(\b,\b){};
\node(14)at(\c+\b,\b){};
\node(15)at(\b/2,\b/2)[label=left:\psi^{-1}(v_1{,}v_2)]{};
\node(16)at(\c+\a-\b/2,\b/2)[label=right:\varphi^{-1}(u_1{,}u_2)]{};
\node(17)at(\b,\b)[label=left:\gamma(1)]{};
\node(18)at(\c+\b,\b)[label=right:\gamma(0)]{};
\node(19)at(\b,\b/2)[label={[xshift=21,yshift=-19]\psi^{-1}(0{,}v_2)}]{};
\node(20)at(\c+\b,\b/2)[label={[xshift=-30,yshift=-19]\hol_\gamma^{-1}\psi^{-1}(0{,}v_2)},label={[xshift=-25,yshift=-31]=\varphi^{-1}(0{,}u_2)}]{};
\draw[dotted](1.center)--(2.center)--(3.center)--(4.center)--cycle;
\draw[dotted](7.center)--(8.center)--(9.center)--(10.center)--cycle;
\draw[loosely dotted](2.center)--(7.center);
\draw[loosely dotted](3.center)--(10.center);
\draw(5.center)--(6.center);
\draw(11.center)--(12.center);
\draw[postaction={decorate}](14.center)--node[label=above:\gamma]{}(13.center);
\draw[postaction={decorate}](16.center)--node[label={[xshift=0,yshift=-19]H(\cdot,v_2)}]{}(15.center);
\foreach\x in{15,...,20}\filldraw(\x)circle(1pt);
\end{tikzpicture}
\end{equation*}
\begin{align*}
\theta\colon V_1\times V_2\times U_1&\to\Mon(\mca{F})\\
(v_1,v_2,u_1)&\mapsto[\gamma^{v_1,v_2,u_1}]
\end{align*}
Let $g^\prime$ be a Riemannian metric of $M^\prime$ and $(g^\prime)_{\mca{F}^\prime}$ be the leafwise Riemannian metric of $\mca{F}^\prime$ induced by $g^\prime$. Let $d^\prime$ be the metric induced by $g^\prime$ on the connected component of $M^\prime$ containing $F\gamma(I)$, and $d_{\mca{F}^\prime}$ be the leafwise metric of $\mca{F}^\prime$ induced by $(g^\prime)_{\mca{F}^\prime}$. For $(v_1,v_2,u_1)\in V_1\times V_2\times U_1$, put 
\begin{equation*}
c^{(v_1,v_2,u_1)}=(\gamma^\prime)^{(\psi^\prime F\psi^{-1}(v_1,v_2),p^\prime_1\varphi^\prime F\varphi^{-1}(u_1,u_2))}\in C^0(I,\mca{F}^\prime), 
\end{equation*}
where $p^\prime_1\colon U^\prime_1\times U^\prime_2\to U^\prime_1$ is the projection. Note that this curve has the same endpoints as $F\gamma^{(v_1,v_2,u_1)}$. By taking $V$ and $U$ small enough, $d^\prime(F\gamma^{(v_1,v_2,u_1)}(t),c^{(v_1,v_2,u_1)}(t))$ is uniformly small with respect to $v_1$, $v_2$, $u_1$ and $t$. Since $F\gamma^{(v_1,v_2,u_1)}(t)$ and $c^{(v_1,v_2,u_1)}(t)$ are in the same plaque of some $W^\prime_i$, 
\begin{equation*}
d_{\mca{F}^\prime}(F\gamma^{(v_1,v_2,u_1)}(t),c^{(v_1,v_2,u_1)}(t))\leq Cd^\prime(F\gamma^{(v_1,v_2,u_1)}(t),c^{(v_1,v_2,u_1)}(t))
\end{equation*}
for some $C>0$ independent of $v_1$, $v_2$, $u_1$ and $t$. Hence $d_{\mca{F}^\prime}(F\gamma^{(v_1,v_2,u_1)}(t),c^{(v_1,v_2,u_1)}(t))$ is uniformly small. By Lemma \ref{mfmgfedfxy} there exists a unique ``short'' geodesic in the leaf connecting $F\gamma^{(v_1,v_2,u_1)}(t)$ and $c^{(v_1,v_2,u_1)}(t)$ and it gives a homotopy: $F\gamma^{(v_1,v_2,u_1)}\sim_{\mca{F}^\prime}c^{(v_1,v_2,u_1)}$. Thus 
\begin{equation*}
F_*[\gamma^{(v_1,v_2,u_1)}]=[c^{(v_1,v_2,u_1)}]=[(\gamma^\prime)^{(\psi^\prime F\psi^{-1}(v_1,v_2),p^\prime_1\varphi^\prime F\varphi^{-1}(u_1,u_2))}]. 
\end{equation*}
Hence $F_*$ is represented by $(v_1,v_2,u_1)\mapsto(\psi^\prime F\psi^{-1}(v_1,v_2),p^\prime_1\varphi^\prime F\varphi^{-1}(u_1,u_2))$ with respect to the charts. Therefore $F_*$ is $C^\infty$. 
\end{proof}

\subsection{Proof of Lemma \ref{mfmgfedfxy}}
We follow \S10 of \cite{Milnor_1963}. 

\begin{lem}\label{mfmgfxfe}
Let $M$ be a $C^\infty$ manifold, $\mca{F}$ be a $C^\infty$ regular foliation of $M$, $g_\mca{F}$ be a leafwise Riemannian metric of $\mca{F}$ and $x_0\in M$. Then there exist an open neighborhood $U$ of $x_0$ in $M$ and $\epsilon>0$ with the following property: For any $v\in(T\mca{F})|_U$ with $\lVert v\rVert<\epsilon$, there exists a unique geodesic $c_v\colon[-2,2]\to M$ in a leaf such that $\frac{dc_v}{dt}(0)=v$. $c_v(t)$ depends smoothly on $v$ and $t$. 
\end{lem}

\begin{proof}
Let $n=\dim M$ and $p=\dim\mca{F}$. Let $(U^\prime,\varphi)$ be a foliation chart of $\mca{F}$ such that $x_0\in U^\prime$ and $\varphi(x_0)=0$. Let $J$ be an interval in $\bb{R}$ such that $0\in J$. For $\gamma\colon J\to U^\prime$, let $c=\varphi\gamma=
\begin{pmatrix}
c_1&\cdots&c_n
\end{pmatrix}
$. Let $\nabla$ be the Levi-Civita connection of $g_\mca{F}$ and $\Gamma_{ij}^k\colon U^\prime\to\bb{R}$ ($1\leq i,j,k\leq p$) be the Christoffel symbols of $g_\mca{F}$ with respect to $(U^\prime,\varphi)$. Let $v\in(T\mca{F})|_{U^\prime}$ and assume $\frac{d\gamma}{dt}(0)=v$. Then $\gamma$ is a geodesic with respect to $g_\mca{F}$ if and only if $\gamma(J)$ is contained in a leaf and $\nabla_{\frac{d\gamma}{dt}(t)}\frac{d\gamma}{dt}(t)=0$ for all $t\in J$. Let $a=
\begin{pmatrix}
a_1&\cdots&a_n
\end{pmatrix}
=\varphi(\gamma(0))\in\varphi(U^\prime)$, $b=
\begin{pmatrix}
b_1&\cdots&b_p
\end{pmatrix}
=\varphi_*v\in\bb{R}^p$. Then $\gamma$ is a geodesic if and only if 
\begin{equation}\label{ckdtckabt}
\begin{cases}
\frac{d^2c_k}{dt^2}(t)+\sum_{i,j=1}^p\Gamma_{ij}^k(c(t))\frac{dc_i}{dt}(t)\frac{dc_j}{dt}(t)=0&1\leq k\leq p\\
\frac{dc_k}{dt}(t)=0&p+1\leq k\leq n\\
c_k(0)=a_k&1\leq k\leq n\\
\frac{dc_k}{dt}(0)=b_k&1\leq k\leq p. 
\end{cases}
\end{equation}
Consider 
\begin{equation}\label{dcdtckuk}
\begin{cases}
\frac{dc_k}{dt}(t)=u_k(t)&1\leq k\leq p\\
\frac{dc_k}{dt}(t)=0&p+1\leq k\leq n\\
\frac{du_k}{dt}(t)=-\sum_{i,j=1}^p\Gamma_{ij}^k(c(t))u_i(t)u_j(t)&1\leq k\leq p\\
c_k(0)=a_k&1\leq k\leq n\\
u_k(0)=b_k&1\leq k\leq p. 
\end{cases}
\end{equation}
Then there is a bijection 
\begin{align*}
c&\mapsto
\begin{pmatrix}
c&\frac{dc}{dt}
\end{pmatrix}\\
c&\mapsfrom
\begin{pmatrix}
c&u
\end{pmatrix}
\end{align*}
between the solutions of \eqref{ckdtckabt} and those of \eqref{dcdtckuk}. There exist an open neighborhood $V$ of $0$ in $\bb{R}^n$ and $\epsilon_i>0$ ($i=1$, $2$) such that $\varphi(U^\prime)\supset V$ and for any $a\in V$ and $b\in\bb{R}^p$ such that $\lVert b\rVert<\epsilon_1$, there exists a unique solution $
\begin{pmatrix}
c_{a,b}&u_{a,b}
\end{pmatrix}
$ of \eqref{dcdtckuk} for $\lvert t\rvert\leq2\epsilon_2$, which depends smoothly on $a$, $b$ and $t$. Let $U=\varphi^{-1}(V)$. For $a\in V$, $\lVert b\rVert<\epsilon_1\epsilon_2$ and $\lvert t\rvert\leq2$, let $c^\prime_{a,b}(t)=c_{a,\epsilon_2^{-1}b}(\epsilon_2t)$. Then $c^\prime_{a,b}(0)=a$, $\frac{dc^\prime_{a,b}}{dt}(0)=b$ and $\varphi^{-1}c^\prime_{a,b}$ is a geodesic. 
\end{proof}

\begin{lem}\label{mfufxecxy}
Let $M$ be a $C^\infty$ manifold, $\mca{F}$ be a $C^\infty$ regular foliation of $M$, $g_\mca{F}$ be a leafwise Riemannian metric of $\mca{F}$, $d_\mca{F}$ be the leafwise metric induced by $g_\mca{F}$ and $x_0\in M$. Then there exist a foliation chart $U$ of $\mca{F}$ containing $x_0$ and $\epsilon>0$ with the following property: For any $x$, $y\in U$ with $d_\mca{F}(x,y)<\epsilon$, there exists a unique geodesic $c_{x,y}\colon I\to M$ in the leaf of length less than $\epsilon$ such that $c_{x,y}(0)=x$, $c_{x,y}(1)=y$. $c_{x,y}(t)$ depends continuously on $x$, $y$, $t$. 
\end{lem}

\begin{proof}
Take $U$, $\epsilon$ as in Lemma \ref{mfmgfxfe}. Consider 
\begin{align*}
F\colon\left\{v\in(T\mca{F})|_U\ \middle|\ \lVert v\rVert<\epsilon\right\}&\to\left\{(x,y)\in U\times M\ \middle|\ d_\mca{F}(x,y)<\epsilon\right\}\\
v&\mapsto(c_v(0),c_v(1)). 
\end{align*}
We may assume that $U$ is relatively compact in $M$. Then by taking $\epsilon$ small enough, the right hand side has a natural structure of a $C^\infty$ manifold. $F$ is $C^\infty$ and the Jacobian matrix of $F$ at $0\in T_{x_0}\mca{F}$ with respect to the obvious choices of coordinates is $
\begin{pmatrix}
I&O\\
I&I
\end{pmatrix}
$. By the inverse function theorem, taking $U$ and $\epsilon$ small enough, $F$ is a diffeomorphism onto its image. By taking $\epsilon$ small enough, we may also assume that for any two distinct plaques $P$, $P^\prime$ of $U$ in the same leaf of $\mca{F}$ and any $x\in P$, $x^\prime\in P^\prime$, we have $d_\mca{F}(x,x^\prime)>\epsilon$. Take an open neighborhood $U^\prime$ of $x_0$ in $M$ such that 
\begin{equation*}
F\left\{v\in(T\mca{F})|_U\ \middle|\ \lVert v\rVert<\epsilon\right\}\supset\left\{(x,y)\in U^\prime\times U^\prime\ \middle|\ \text{$x$, $y$ are in the same plaque of $U^\prime$}\right\}. 
\end{equation*}
Since the right hand side equals to $\left\{(x,y)\in U^\prime\times U^\prime\ \middle|\ d_\mca{F}(x,y)<\epsilon\right\}$, the lemma follows. 
\end{proof}

\begin{lem}\label{mfmkeuuxy}
Let $M$ be a $C^\infty$ manifold, $\mca{F}$ be a $C^\infty$ regular foliation of $M$, $g_\mca{F}$ be a leafwise Riemannian metric of $\mca{F}$, $d_\mca{F}$ be the leafwise metric induced by $g_\mca{F}$, $K$ be a compact subset of $M$, $\ms{U}$ be an open cover of $K$. Then there exists $\epsilon>0$ such that for any $x$, $y\in K$ in the same leaf with $d_\mca{F}(x,y)<\epsilon$, there exists $U\in\ms{U}$ such that $x$, $y\in U$. 
\end{lem}

\begin{proof}
For $x\in K$ and $\epsilon>0$, let $B_\epsilon(x)=\left\{y\in K\ \middle|\ d_\mca{F}(x,y)\leq\epsilon\right\}$. For $\epsilon>0$, consider 
\begin{equation*}
U_\epsilon=\left\{x\in K\ \middle|\ \text{there exists $U\in\ms{U}$ such that $B_\epsilon(x)\subset U$}\right\}. 
\end{equation*}
Then $U_\epsilon$ is open in $K$ and $K=\bigcup_{\epsilon>0}U_\epsilon$. Hence there exist $0<\epsilon_1<\cdots<\epsilon_k$ such that $K=\bigcup_{i=1}^kU_{\epsilon_i}$. For $0<\epsilon<\epsilon^\prime$, we have $U_\epsilon\supset U_{\epsilon^\prime}$. Thus $K=U_{\epsilon_1}$. $\epsilon_1$ satisfies the required condition of the lemma. 
\end{proof}

\begin{proof}[Proof of Lemma \ref{mfmgfedfxy}]
For $x\in K$, take $U_x$, $\epsilon_x$ as in Lemma \ref{mfufxecxy}. For the open cover $\ms{U}=\left\{K\cap U_x\ \middle|\ x\in K\right\}$, let $\epsilon>0$ be as in Lemma \ref{mfmkeuuxy}. This $\epsilon$ satisfies the condition of the lemma. 
\end{proof}

\subsection{Foliation groupoids}
\begin{dfn}[\cite{CM}]\label{folgrgxxx}
A Lie groupoid $\mca{G}\rightrightarrows M$ is called a \emph{foliation groupoid} if $\mca{G}^x_x$ is a discrete subset of $\mca{G}$ for each $x\in M$. 
\end{dfn}

Recall that for a Lie groupoid $\mca{G}\rightrightarrows M$, $\mca{F}_\mca{G}$ denotes the set of connected components of orbits of $\mca{G}$. 

\begin{prop}\label{folgconcorf}
Let $\mca{G}\rightrightarrows M$ be a foliation groupoid. Then $\mca{F}_\mca{G}$ is a regular $C^\infty$ foliation of $M$. 
\end{prop}

\begin{proof}
By Proposition \ref{gmstsif}, $\mca{F}_\mca{G}$ is a $C^\infty$ Stefan singular foliation of $M$. Since the leaves of $\mca{F}_\mca{G}$ have the same dimension, $\mca{F}_\mca{G}$ is a regular $C^\infty$ foliation of $M$. 
\end{proof}

\begin{dfn}
A Lie groupoid $\mca{G}\rightrightarrows M$ is said to be $s$-connected (resp. $s$-simply connected) if $\mca{G}_x$ is connected (resp. simply connected) for each $x\in M$. 
\end{dfn}

\begin{dfn}
Let $\mca{G}\rightrightarrows M$ be a Lie groupoid, $N$ be a $C^\infty$ manifold and $\nu\colon N\to M$ be a $C^\infty$ map. 
\begin{itemize}
\item For $\rho\in\mca{A}(\nu,\mca{G})$ and $y\in N$, 
\begin{equation*}
\Stab_\rho(y)=\left\{g\in\mca{G}\ \middle|\ \rho(y,g)=y\right\}=\varphi_\rho\left((N\rtimes_\rho\mca{G})^y_y\right)
\end{equation*}
is called the \emph{stabilizer} of $\rho$ at $y$. 
\item We say that $\rho\in\mca{A}(\nu,\mca{G})$ is \emph{locally free} if $\Stab_\rho(y)$ is a discrete subset of $\mca{G}$ for each $y\in N$. 
\item Let 
\begin{equation*}
\mca{A}_{LF}(\nu,\mca{G})=\left\{\rho\in\mca{A}(\nu,\mca{G})\ \middle|\ \text{$\rho$ is locally free}\right\}. 
\end{equation*}
\end{itemize}
\end{dfn}

\begin{example}\label{gmnmnrgyy}
Let $\mca{G}\rightrightarrows M$ be a Lie groupoid, $N$ be a $C^\infty$ manifold, $\nu\colon N\to M$ be a $C^\infty$ map and $\rho\in\mca{A}(\nu,\mca{G})$. For $y\in N$, we have 
\begin{equation*}
(N\rtimes_\rho\mca{G})^y_y=\left\{(y,g)\in N\rtimes_\rho\mca{G}\ \middle|\ \rho(y,g)=y\right\}\simeq\Stab_\rho(y). 
\end{equation*}
So $N\rtimes_\rho\mca{G}$ is a foliation groupoid if and only if $\rho\in\mca{A}_{LF}(\nu,\mca{G})$. For $\rho\in\mca{A}(\nu,\mca{G})$, define $\mca{F}_\rho=\mca{F}_{N\rtimes_\rho\mca{G}}$. Since $t^{-1}(y)=\left\{(y,g)\in N\rtimes_\rho\mca{G}\right\}\simeq\mca{G}^{\nu(y)}$ for $y\in N$, $N\rtimes_\rho\mca{G}$ is $s$-connected (resp. $s$-simply connected) if and only if $\mca{G}^{\nu(y)}$ is connected (resp. simply connected) for all $y\in N$. 
\end{example}

\begin{example}
Let $M$ be a $C^\infty$ manifold and $\mca{F}$ be a $C^\infty$ regular foliation of $M$. For $x\in M$, $\Mon(\mca{F})^x_x$ is discrete. Let $x\in L\in\mca{F}$. Then $\Mon(\mca{F})_x$ is diffeomorphic to the universal cover of $(L,x)$. Hence $\Mon(\mca{F})$ is an $s$-simply connected foliation groupoid. 
\end{example}

\begin{prop}\label{fgmonlod}
Let $\mca{G}\rightrightarrows M$ be a foliation groupoid. Then the following statements hold: 
\begin{enumerate}
\setlength\itemsep{0em}
\item There exists a unique $\varphi\in\Hom(\Mon(\mca{F}_\mca{G}),\mca{G})$ such that $F_\varphi=\id$. 
\item $\varphi$ is a local diffeomorphism. 
\item $\mca{G}$ is $s$-connected if and only if $\varphi$ is surjective. 
\item $\mca{G}$ is $s$-simply connected if and only if $\varphi$ is an isomorphism. 
\end{enumerate}
\end{prop}

\begin{proof}
See Proposition 1 and its proof in \cite{CM}. $\varphi$ can be described alternatively as follows. Let $[\gamma]\in\Mon(\mca{F}_\mca{G})$. Take $L\in\mca{F}_\mca{G}$ such that $\gamma(I)\subset L$. Let $\mca{G}_0$ be the identity component of $\mca{G}$ (see Lemma \ref{identitycomp}). Then $t\colon(\mca{G}_0)_{\gamma(0)}\to L$ is a principal $(\mca{G}_0)_{\gamma(0)}^{\gamma(0)}$ bundle by Lemma \ref{connorb}. Since $(\mca{G}_0)_{\gamma(0)}^{\gamma(0)}$ is discrete in $\mca{G}_0$, $t\colon(\mca{G}_0)_{\gamma(0)}\to L$ is a covering map. Let $\wt{\gamma}\colon I\to(\mca{G}_0)_{\gamma(0)}$ be the lift of $\gamma$ starting at $1_{\gamma(0)}$. Define $\varphi\colon\Mon(\mca{F}_\mca{G})\to\mca{G}_0$ by $\varphi[\gamma]=\wt{\gamma}(1)$. 
\end{proof}

\begin{cor}\label{gmngphigphi}
Let $\mca{G}\rightrightarrows M$ be a Lie groupoid, $N$ be a $C^\infty$ manifold, $\nu\colon N\to M$ be a $C^\infty$ map and $\rho\in\mca{A}_{LF}(\nu,\mca{G})$. Then the following statements hold: 
\begin{enumerate}
\setlength\itemsep{0em}
\item There exists a unique morphism $\varphi\in\Hom(\Mon(\mca{F}_\rho),N\rtimes_\rho\mca{G})$ such that $F_\varphi=\id$. 
\item $\mca{G}^{\nu(y)}$ is connected for all $y\in N$ if and only if $\varphi$ is surjective. 
\item $\mca{G}^{\nu(y)}$ is simply connected for all $y\in N$ if and only if $\varphi$ is an isomorphism. 
\end{enumerate}
\end{cor}

\begin{cor}\label{simmon}
Let $\mca{G}\rightrightarrows M$, $\mca{G}^\prime\rightrightarrows M^\prime$ be $s$-simply connected Lie groupoids, $N$, $N^\prime$ be $C^\infty$ manifolds, $\nu\colon N\to M$, $\nu^\prime\colon N^\prime\to M^\prime$ be $C^\infty$ maps and $\rho\in\mca{A}_{LF}(\nu,\mca{G})$, $\rho^\prime\in\mca{A}_{LF}(\nu^\prime,\mca{G}^\prime)$. Let $F\colon N\to N^\prime$ be a diffeomorphism such that $F(\mca{F}_\rho)=\mca{F}_{\rho^\prime}$. Then there exists a unique isomorphism $\varphi\in\Hom(N\rtimes_\rho\mca{G},N^\prime\rtimes_{\rho^\prime}\mca{G}^\prime)$ such that $F_\varphi=F$. 
\end{cor}

\begin{proof}
We have isomorphisms $\psi\colon\Mon(\mca{F}_\rho)\to N\rtimes_\rho\mca{G}$, $\psi^\prime\colon\Mon(\mca{F}_{\rho^\prime})\to N^\prime\rtimes_{\rho^\prime}\mca{G}^\prime$ by Corollary \ref{gmngphigphi}, and an isomorphism $F_*\colon\Mon(\mca{F}_\rho)\to\Mon(\mca{F}_{\rho^\prime})$ by Lemma \ref{monfmonfmb}, which give an isomorphism $\psi^\prime F_*\psi^{-1}\colon N\rtimes_\rho\mca{G}\to N^\prime\rtimes_{\rho^\prime}\mca{G}^\prime$ with base map $F$. The uniqueness follows from Proposition \ref{connuni} in the next subsection. 
\end{proof}

\subsection{Morphisms are determined by their base maps}
\begin{prop}\label{connuni}
Let $\mca{G}\rightrightarrows M$ be an $s$-connected Lie groupoid and $\mca{G}^\prime\rightrightarrows M^\prime$ be a foliation groupoid. Then 
\begin{align*}
b\colon\Hom(\mca{G},\mca{G}^\prime)&\to C^\infty(\mca{O}_\mca{G},\mca{O}_{\mca{G}^\prime})\\
\varphi&\mapsto F_\varphi
\end{align*}
is injective. 
\end{prop}

\begin{proof}
The map $b$ is well-defined by Lemma \ref{oosmhs}. Assume $\varphi_1$, $\varphi_2\in\Hom(\mca{G},\mca{G}^\prime)$ have the same base map $F$. Let $g\in\mca{G}$ and $x=s(g)\in M$. Since $\mca{G}_x$ is a connected manifold, it is path connected. Hence there exists a continuous curve $c\colon I\to\mca{G}_x$ such that $c(0)=1_x$, $c(1)=g$. Let $x\in O\in\mca{O}_\mca{G}$, $F(x)\in O^\prime\in\mca{O}_{\mca{G}^\prime}$. Then $F\colon O\to O^\prime$ is continuous by the definition of $C^\infty(\mca{O}_\mca{G},\mca{O}_{\mca{G}^\prime})$. We have 
\begin{equation*}
\begin{tikzcd}
\mca{G}_x\ar[r,"\varphi_i"]\ar[d,"t"']\ar[rd,phantom,"\circlearrowright"]&(\mca{G}^\prime)_{F(x)}\ar[d,"t^\prime"]\\
O\ar[r,"F"']&O^\prime. 
\end{tikzcd}
\end{equation*}
Since $(\mca{G}^\prime)^{F(x)}_{F(x)}$ is discrete in $\mca{G}^\prime$, $t^\prime\colon(\mca{G}^\prime)_{F(x)}\to O^\prime$ is a covering map by 4 in Proposition \ref{isoorb}. $Ftc\colon I\to O^\prime$ and $\varphi_ic\colon I\to(\mca{G}^\prime)_{F(x)}$ are continuous curves and $t^\prime\varphi_ic=Ftc$. So both $\varphi_1c$ and $\varphi_2c$ are continuous lifts of $Ftc$ starting at $1_{F(x)}$. Hence $\varphi_1c=\varphi_2c$ and $\varphi_1(g)=(\varphi_1c)(1)=(\varphi_2c)(1)=\varphi_2(g)$. 
\end{proof}

\begin{cor}
Let $\mca{G}\rightrightarrows M$ be an $s$-connected foliation groupoid (for example, fundamental groupoids, monodromy groupoids), $N_0$ be a $C^\infty$ manifold, $\nu_0\colon N_0\to M$ be a surjective submersion and $\rho_0\in\mca{A}(\nu_0,\mca{G})$. Then $T_{bnt}(\rho_0)=T_{nt}(\rho_0)=T_o(\rho_0)=\mr{pt}$. 
\end{cor}

\begin{proof}
Let $(N\rtimes_\rho\mca{G},\varphi)\in\MAG(\rho_0)$. 
Then $\varphi_\rho\varphi$, $\varphi_{\rho_0}\in\Hom(N_0\rtimes_{\rho_0}\mca{G},\mca{G})$ have the same base map: $F_{\varphi_\rho\varphi}=\nu F_\varphi=\nu_0=F_{\varphi_{\rho_0}}$. $N_0\rtimes_{\rho_0}\mca{G}$ is $s$-connected since $\mca{G}$ is $s$-connected. $\mca{G}$ is a foliation groupoid. Hence $\varphi_\rho\varphi=\varphi_{\rho_0}$ by Proposition \ref{connuni}. So $\varphi$ is a semiconjugacy. $\varphi^{-1}$ is also a semiconjugacy since $\varphi_{\rho_0}\varphi^{-1}=\varphi_\rho$. Therefore $\varphi$ is a conjugacy and $(N\rtimes_\rho\mca{G},\varphi)\mathrel{\ul{\sim_a}}(N_0\rtimes_{\rho_0}\mca{G},\id)$ for all $a\in\left\{o,nt,bnt\right\}$. 
\end{proof}

\begin{prop}\label{simconfolbij}
Let $\mca{G}\rightrightarrows M$ be an $s$-simply connected Lie groupoid and $\mca{G}^\prime\rightrightarrows M^\prime$ be a foliation groupoid. Then 
\begin{align*}
b\colon\Hom(\mca{G},\mca{G}^\prime)&\to C^\infty(\mca{O}_\mca{G},\mca{O}_{\mca{G}^\prime})\\
\varphi&\mapsto F_\varphi
\end{align*}
is bijective. 
\end{prop}

\begin{proof}
The map $b$ is injective by Proposition \ref{connuni}. We prove surjectivity. Let $F\in C^\infty(\mca{O}_\mca{G},\mca{O}_{\mca{G}^\prime})$. 

Let $g\in\mca{G}$ and $x=s(g)\in M$. Since $\mca{G}_x$ is simply connected, there exists a continuous curve $c\colon I\to\mca{G}_x$ such that $c(0)=1_x$, $c(1)=g$, which is unique up to homotopy relative to endpoints. Let $x\in O\in\mca{O}_\mca{G}$, $F(x)\in O^\prime\in\mca{O}_{\mca{G}^\prime}$. We have 
\begin{equation*}
\begin{tikzcd}
\mca{G}_x\ar[d,"t"']&(\mca{G}^\prime)_{F(x)}\ar[d,"t^\prime"]\\
O\ar[r,"F"']&O^\prime, 
\end{tikzcd}
\end{equation*}
where $t^\prime\colon(\mca{G}^\prime)_{F(x)}\to O^\prime$ is a covering. There exists a unique continuous map $\wt{c}\colon I\to(\mca{G}^\prime)_{F(x)}$ such that $t^\prime\wt{c}=Ftc$ and $\wt{c}(0)=1_{F(x)}$. Then $\varphi(g)=\wt{c}(1)$ does not depend on the choice of $c$ by the homotopy lifting property. We must show that $\varphi\colon\mca{G}\to\mca{G}^\prime$ is $C^\infty$, but it was not easy, so we will construct $\varphi$ in a different way. 

Since $s\colon\mca{G}\to M$ is a submersion, $\mca{F}=\left\{\mca{G}_x\ \middle|\ x\in M\right\}$ is a $C^\infty$ foliation of $\mca{G}$ (note that $\mca{G}$ may not be second countable or Hausdorff). We can define $\Mon(\mca{F})\rightrightarrows\mca{G}$ and show that it is a Lie groupoid (but the base manifold may not be second countable or Hausdorff). 

Since $\mca{G}$ is $s$-connected, $Ft\colon\mca{G}\to M^\prime$ is a leaf preserving $C^\infty$ map with respect to $\mca{F}$ and $\mca{F}_{\mca{G}^\prime}$. Since $\mca{G}$ may not be second countable or Hausdorff, we cannot apply Proposition \ref{monfmonfmb} directly, but the same proof shows that $Ft$ induces a morphism $(Ft)_*\colon\Mon(\mca{F})\to\Mon(\mca{F}_{\mca{G}^\prime})$ with base map $Ft$. 

For $g\in\mca{G}$, there exists a continuous map $c_g\colon I\to\mca{G}_{s(g)}$ such that $c_g(1)=g$ and $c_g(0)=1_{s(g)}$. Define 
\begin{align*}
\psi\colon\mca{G}&\to\Mon(\mca{F})\\
g&\mapsto[c_g]. 
\end{align*}
Since $\mca{G}$ is $s$-simply connected, $\psi$ is well-defined. 

\begin{claim}
$\psi$ is $C^\infty$. 
\end{claim}

\begin{proof}
This follows from the definition of charts of $\Mon(\mca{F})$. See the proof of Proposition \ref{monts1lgm}. 
\end{proof}

We will prove that the composition 
\begin{equation*}
\mca{G}\xrightarrow{\psi}\Mon(\mca{F})\xrightarrow{(Ft)_*}\Mon(\mca{F}_{\mca{G}^\prime})
\end{equation*}
is a morphism with base map $F$. Note that $\mca{G}\rightrightarrows M$ acts from the right on $ss\colon\Mon(\mca{F})\xrightarrow{s}\mca{G}\xrightarrow{s}M$ by 
\begin{align*}
\Mon(\mca{F})\times_{ss,t}\mca{G}&\to\Mon(\mca{F})\\
([\gamma],g)&\mapsto[\gamma g], 
\end{align*}
where 
\begin{align*}
(\gamma g)\colon I&\to\mca{G}_{s(g)}\\
t&\mapsto\gamma(t)g. 
\end{align*}

\begin{claim}
$\psi(g_1g_2)=(\psi(g_1)g_2)\psi(g_2)$ for all $(g_1,g_2)\in\mca{G}^{(2)}$. 
\end{claim}

\begin{proof}
Define $c\colon I\to\mca{G}^{s(g_2)}$ by 
\begin{equation*}
c(t)=
\begin{cases}
c_{g_2}(2t)&0\leq t\leq\frac{1}{2}\\
c_{g_1}(2t-1)g_2&\frac{1}{2}\leq t\leq1. 
\end{cases}
\end{equation*}
Then $c$ is continuous and $c(1)=g_1g_2$, $c(0)=1_{s(g_2)}$. We can take $c$ as $c_{g_1g_2}$. 
\end{proof}

Therefore $(Ft)_*\psi(g_1g_2)=(Ft)_*\psi(g_1)(Ft)_*\psi(g_2)$ for all $(g_1,g_2)\in\mca{G}^{(2)}$. We have $t\psi=\id$, $s\psi=1_s$, hence 
\begin{equation*}
t(Ft)_*\psi=Ftt\psi=Ft,\quad s(Ft)_*\psi=Fts\psi=Fs. 
\end{equation*}
So $(Ft)_*\psi\colon\mca{G}\to\Mon(\mca{F}_{\mca{G}^\prime})$ is a morphism with base map $F$. 

There exists a morphism $\theta\colon\Mon(\mca{F}_{\mca{G}^\prime})\to\mca{G}^\prime$ such that $F_\theta=\id$ by Proposition \ref{fgmonlod}. Then 
\begin{equation*}
\mca{G}\xrightarrow{\psi}\Mon(\mca{F})\xrightarrow{(Ft)_*}\Mon(\mca{F}_{\mca{G}^\prime})\xrightarrow{\theta}\mca{G}^\prime. 
\end{equation*}
is a morphism with base map $F$, hence $b$ is surjective. 
\end{proof}

\begin{rem}
Another proof might be given by using Lie algebroids and integrating a morphism of Lie algebroids to a morphism of Lie groupoids. 
\end{rem}

\subsection{Orbitwise homotopies correspond to $nt$-homotopies}\label{morlieorb}
\begin{lem}\label{gmgmpmigt}
Let $\mca{G}\rightrightarrows M$ be a Lie groupoid, $\mca{G}^\prime\rightrightarrows M^\prime$ be a foliation groupoid, $F$, $F^\prime\in C^\infty(\mca{O}_\mca{G},\mca{O}_{\mca{G}^\prime})$ and $H\colon M\times I\to M^\prime$ be an orbitwise homotopy between $F$ and $F^\prime$. Then there exists a unique continuous map $P\colon M\times I\to\mca{G}^\prime$ such that: 
\begin{itemize}
\setlength\itemsep{0em}
\item $P(\cdot,t)\in C^\infty(M,\mca{G}^\prime)$ for all $t\in I$
\item $t^\prime P(\cdot,t)=F$ for all $t\in I$
\item $P(\cdot,0)=1_F$
\item $s^\prime P=H$. 
\end{itemize}
\end{lem}

\begin{proof}
$\mca{F}_{\mca{G}^\prime}$ is a $C^\infty$ regular foliation by Proposition \ref{folgconcorf}. There exists a unique $\varphi\in\Hom(\Mon(\mca{F}_{\mca{G}^\prime}),\mca{G}^\prime)$ such that $F_\varphi=\id$ by 1 in Proposition \ref{fgmonlod}. Let 
\begin{align*}
\wt{P}\colon M\times I&\to\Mon(\mca{F}_{\mca{G}^\prime})\\
(x,t)&\mapsto[H(x,t\ \cdot\ )]. 
\end{align*}
\begin{equation*}
\begin{tikzpicture}[every label/.append style={font=\scriptsize},matrix of math nodes,decoration={markings,mark=at position0.5with{\arrow{>}}}]
\node(1)[label=above:F(x)]{};
\node(2)[below=.5 of 1,label=below:H(x{,}t)]{};
\draw[postaction={decorate}](1.center)--node[label=right:H(x{,}t\ \cdot\ )]{}(2.center);
\foreach\x in{1,2}\filldraw(\x)circle(1pt);
\end{tikzpicture}
\end{equation*}
Then: 
\begin{itemize}
\setlength\itemsep{0em}
\item $\wt{P}$ is continuous
\item $\wt{P}(\cdot,t)\in C^\infty(M,\Mon(\mca{F}_{\mca{G}^\prime}))$ for all $t\in I$
\item $s\wt{P}(\cdot,t)=F$ for all $t\in I$
\item $\wt{P}(\cdot,0)=1_F$
\item $t\wt{P}=H$. 
\end{itemize}
Define 
\begin{equation*}
P\colon M\times I\xrightarrow{\wt{P}}\Mon(\mca{F}_{\mca{G}^\prime})\xrightarrow{\varphi}\mca{G}^\prime\xrightarrow{\cdot^{-1}}\mca{G}^\prime. 
\end{equation*}
\begin{equation*}
\begin{tikzpicture}[every label/.append style={font=\scriptsize},matrix of math nodes,decoration={markings,mark=at position0.5with{\arrow{>}}}]
\node(1)[label=above:F(x)]{};
\node(2)[below=.5 of 1,label=below:H(x{,}t)]{};
\draw[postaction={decorate}](2.center)--node[label=right:P(x{,}t)]{}(1.center);
\foreach\x in{1,2}\filldraw(\x)circle(1pt);
\end{tikzpicture}
\end{equation*}
Then: 
\begin{itemize}
\setlength\itemsep{0em}
\item $P$ is continuous
\item $P(\cdot,t)\in C^\infty(M,\mca{G}^\prime)$ for all $t\in I$
\item $t^\prime P(\cdot,t)=F$ for all $t\in I$
\item $P(\cdot,0)=1_F$
\item $s^\prime P=H$. 
\end{itemize}

We show the uniqueness of $P$. Let $x\in M$ and $F(x)\in O^\prime\in\mca{O}_{\mca{G}^\prime}$. Then $s^\prime\colon(\mca{G}^\prime)^{F(x)}\to O^\prime$ is a covering and $s^\prime P(x,t)=H(x,t)$, so $P(x,\cdot)$ is a lift of $H(x,\cdot)$ which starts at $1_{F(x)}$. Hence $P$ is unique. 
\end{proof}

\begin{prop}\label{homcon}
Let $\mca{G}\rightrightarrows M$ be a Lie groupoid, $\mca{G}^\prime\rightrightarrows M^\prime$ be a foliation groupoid, $\varphi\in\Hom(\mca{G},\mca{G}^\prime)$, $F^\prime\in C^\infty(\mca{O}_\mca{G},\mca{O}_{\mca{G}^\prime})$ and $H$ be an orbitwise homotopy between $F_\varphi$ and $F^\prime$. Then there exist a unique $\varphi^\prime\in\Hom(\mca{G},\mca{G}^\prime)$ and a unique $nt$-homotopy $P$ between $\varphi$ and $\varphi^\prime$ such that $H=s^\prime P$. 
\end{prop}

\begin{proof}
By Lemma \ref{gmgmpmigt} there exists a unique continuous $P\colon M\times I\to\mca{G}^\prime$ such that: 
\begin{itemize}
\setlength\itemsep{0em}
\item $P(\cdot,t)\in C^\infty(M,\mca{G}^\prime)$ for all $t\in I$
\item $t^\prime P(\cdot,t)=F_\varphi$ for all $t\in I$
\item $P(\cdot,0)=1_{F_\varphi}$
\item $s^\prime P=H$. 
\end{itemize}
Let $\varphi^\prime=\varphi P(\cdot,1)$. 
\end{proof}

\begin{cor}\label{bggoogg}
The following statements hold: 
\begin{enumerate}
\item Let $\mca{G}\rightrightarrows M$ be an $s$-connected Lie groupoid and $\mca{G}^\prime\rightrightarrows M^\prime$ be a foliation groupoid. Then 
\begin{equation*}
b\colon\Hom(\mca{G},\mca{G}^\prime)/{\sim_o}\to C^\infty(\mca{O}_\mca{G},\mca{O}_{\mca{G}^\prime})/{\sim_o}
\end{equation*}
is injective. 
\item Let $\mca{G}\rightrightarrows M$ be an $s$-connected Lie groupoid, $N$, $N^\prime$ be $C^\infty$ manifolds, $\nu\colon N\to M$, $\nu^\prime\colon N^\prime\to M$ be $C^\infty$ maps and $\rho\in\mca{A}(\nu,\mca{G})$, $\rho^\prime\in\mca{A}_{LF}(\nu^\prime,\mca{G})$. Then 
\begin{equation*}
b\colon\ul{\Hom}(N\rtimes_\rho\mca{G},N^\prime\rtimes_{\rho^\prime}\mca{G})/\ul{\sim_o}\to\ul{C}^\infty(\mca{O}_\rho,\mca{O}_{\rho^\prime})/\ul{\sim_o}
\end{equation*}
is injective. 
\end{enumerate}
\end{cor}

\begin{proof}
1. Let $\varphi$, $\varphi^\prime\in\Hom(\mca{G},\mca{G}^\prime)$ be such that $F_\varphi\sim_oF_{\varphi^\prime}$. Let $H$ be an orbitwise homotopy between $F_\varphi$ and $F_{\varphi^\prime}$. By Proposition \ref{homcon} there exists a unique $\varphi^{\prime\prime}\in\Hom(\mca{G},\mca{G}^\prime)$ and a unique $nt$-homotopy $P$ between $\varphi$ and $\varphi^{\prime\prime}$ such that $s^\prime P=H$. Since $F_{\varphi^{\prime\prime}}=s^\prime P(\cdot,1)=H(\cdot,1)=F_{\varphi^\prime}$, we have $\varphi^{\prime\prime}=\varphi^\prime$ by Proposition \ref{connuni}. Hence $\varphi\sim_{nt}\varphi^\prime$, and $\varphi\sim_o\varphi^\prime$ by Lemma \ref{hnhnphipoh}. 

2. Proved similarly. 
\end{proof}

\begin{prop}
Let $\mca{G}\rightrightarrows M$ be a Lie groupoid, $\mca{G}^\prime\rightrightarrows M^\prime$ be a foliation groupoid, $\varphi$, $\varphi^\prime\in\Hom(\mca{G},\mca{G}^\prime)$ and $\eta$ be an orbitwise homotopy between $\varphi$ and $\varphi^\prime$. Then there exists a unique $nt$-homotopy $P$ between $\varphi$ and $\varphi^\prime$ such that $\eta(\cdot,t)=\varphi P(\cdot,t)$ for all $t\in I$. Hence 
\begin{equation*}
\left\{\text{$nt$-homotopies between $\varphi$ and $\varphi^\prime$}\right\}\simeq\left\{\text{orbitwise homotopies between $\varphi$ and $\varphi^\prime$}\right\}
\end{equation*}
and $\sim_o$ and $\sim_{nt}$ on $\Hom(\mca{G},\mca{G}^\prime)$ coincide. 
\end{prop}

\begin{proof}
By definition $F_{\eta(\cdot,t)}$ is an orbitwise homotopy between $F_\varphi$ and $F_{\varphi^\prime}$. By Proposition \ref{homcon}, there exist a unique $\varphi^{\prime\prime}\in\Hom(\mca{G},\mca{G}^\prime)$ and a unique $nt$-homotopy $P$ between $\varphi$ and $\varphi^{\prime\prime}$ such that $F_{\eta(\cdot,t)}=s^\prime P(\cdot,t)$ for all $t\in I$. 

\begin{claim}
$\eta(\cdot,t)=\varphi P(\cdot,t)$ for all $t\in I$. 
\end{claim}

\begin{proof}
Let $g\in\mca{G}$ and put $y=t(g)$, $x=s(g)$. 
\begin{equation*}
\begin{tikzpicture}[every label/.append style={font=\scriptsize},matrix of math nodes,decoration={markings,mark=at position0.5with{\arrow{>}}}]
\node(1)[label=below:y]{};
\node(2)[right=of 1,label=below:x]{};
\draw[postaction={decorate}](2.center)--node[label=above:g]{}(1.center);
\foreach\x in{1,2}\filldraw(\x)circle(1pt);
\end{tikzpicture}
\end{equation*}
\begin{equation*}
\begin{tikzpicture}[every label/.append style={font=\scriptsize},matrix of math nodes,decoration={markings,mark=at position0.5with{\arrow{>}}}]
\node(1)[label=below:F_{\eta(\cdot{,}t)}(y)]{};
\node(2)[right=2 of 1,label=below:F_{\eta(\cdot{,}t)}(x)]{};
\node(3)[above=.5 of 2,label=above:F_\varphi(x)]{};
\draw[postaction={decorate}](2.center)--node[label=above:\eta(g{,}t)]{}(1.center);
\draw[postaction={decorate}](2.center)--node[label=right:P(x{,}t)]{}(3.center);
\foreach\x in{1,2,3}\filldraw(\x)circle(1pt);
\end{tikzpicture}\qquad
\begin{tikzpicture}[every label/.append style={font=\scriptsize},matrix of math nodes,decoration={markings,mark=at position0.5with{\arrow{>}}}]
\node(1)[label=below:F_{\eta(\cdot{,}t)}(y)]{};
\node(2)[right=2 of 1,label=below:F_{\eta(\cdot{,}t)}(x)]{};
\node(3)[above=.5 of 2,label=above:F_\varphi(x)]{};
\draw[postaction={decorate}](2.center)--node[label=above:(\varphi P(\cdot{,}t))(g)]{}(1.center);
\draw[postaction={decorate}](2.center)--node[label=right:P(x{,}t)]{}(3.center);
\foreach\x in{1,2,3}\filldraw(\x)circle(1pt);
\end{tikzpicture}
\end{equation*}
We have $\eta(g,t)P(x,t)^{-1}$, $(\varphi P(\cdot,t))(g)P(x,t)^{-1}\in(\mca{G}^\prime)_{F_\varphi(x)}$ for all $t\in I$, and they are $\varphi(g)$ when $t=0$. Let $F_\varphi(x)\in O^\prime\in\mca{O}_{\mca{G}^\prime}$. Since $t^\prime\colon(\mca{G}^\prime)_{F_\varphi(x)}\to O^\prime$ is a covering, $\eta(g,t)P(x,t)^{-1}=(\varphi P(\cdot,t))(g)P(x,t)^{-1}$ for all $t\in I$. Hence $\eta(g,t)=(\varphi P(\cdot,t))(g)$. 
\end{proof}

Therefore $\varphi^\prime=\varphi^{\prime\prime}$. 
\end{proof}

\begin{cor}
Let $\mca{G}\rightrightarrows M$ be a Lie groupoid, $N$, $N^\prime$ be $C^\infty$ manifolds, $\nu\colon N\to M$, $\nu^\prime\colon N^\prime\to M$ be $C^\infty$ maps and $\rho\in\mca{A}(\nu,\mca{G})$, $\rho^\prime\in\mca{A}_{LF}(\nu^\prime,\mca{G})$. Then $\mathrel{\ul{\sim_o}}$ and $\mathrel{\ul{\sim_{nt}}}$ on $\ul{\Hom}(N\rtimes_\rho\mca{G},N^\prime\rtimes_{\rho^\prime}\mca{G})$ coincide. 
\end{cor}

\begin{cor}\label{tnteqto}
Let $\mca{G}\rightrightarrows M$ be a Lie groupoid, $N_0$ be a $C^\infty$ manifold, $\nu_0\colon N_0\to M$ be a surjective submersion and $\rho_0\in\mca{A}_{LF}(\nu_0,\mca{G})$. Then $T_{nt}(\rho_0)=T_o(\rho_0)$. 
\end{cor}

\section{$T_o(\rho_0)\to\mca{A}_{LF}(\mca{O}_{\rho_0},\mca{G})/(\ul{\Diff}(\mca{O}_{\rho_0})_1\times\ul{\Aut}(\mca{G}))$}\label{1111}
\begin{dfn}
Let $\mca{G}\rightrightarrows M$ be a Lie groupoid, $N_0$ be a $C^\infty$ manifold, $\nu_0\colon N_0\to M$ be a $C^\infty$ map and $\rho_0\in\mca{A}_{LF}(\nu_0,\mca{G})$. Define 
\begin{equation*}
\mca{A}_{LF}(\mca{O}_{\rho_0},\mca{G})=\left\{\rho\in\mca{A}_{LF}(\nu_0,\mca{G})\ \middle| \ 
\begin{gathered}
\text{$\mca{O}_\rho=\mca{O}_{\rho_0}$ as partitions into $C^\infty$ submanifolds}\\
\text{(not just as partitions into subsets)}
\end{gathered}
\right\}, 
\end{equation*}
\begin{equation*}
\ul{\Diff}(\mca{O}_{\rho_0})=\left\{F\in\ul{\Diff}(N_0)\ \middle|\ F,F^{-1}\in C^\infty(\mca{O}_{\rho_0},\mca{O}_{\rho_0})\right\}, 
\end{equation*}
\begin{equation*}
\ul{\Diff}(\mca{O}_{\rho_0})_1=\left\{F\in\ul{\Diff}(\mca{O}_{\rho_0})\ \middle|\ F\mathrel{\ul{\sim_o}}\id\right\}. 
\end{equation*}
Then $\ul{\Diff}(\mca{O}_{\rho_0})_1$ is a normal subgroup of $\ul{\Diff}(\mca{O}_{\rho_0})$ by Lemma \ref{ulsimsim} and define 
\begin{equation*}
\ul{\MCG}(\mca{O}_{\rho_0})=\ul{\Diff}(\mca{O}_{\rho_0})/\ul{\Diff}(\mca{O}_{\rho_0})_1. 
\end{equation*}
\end{dfn}

\begin{cor}
Let $\mca{G}\rightrightarrows M$ be a Lie groupoid, $N_0$ be a $C^\infty$ manifold, $\nu_0\colon N_0\to M$ be a surjective submersion and $\rho_0\in\mca{A}_{LF}(\nu_0,\mca{G})$. We have a homomorphism 
\begin{align*}
b\colon\ul{\Aut}(N_0\rtimes_{\rho_0}\mca{G})&\to\ul{\Diff}(\mca{O}_{\rho_0})\\
\theta&\mapsto F_\theta
\end{align*}
satisfying $b(\ul{\Aut}(N_0\rtimes_{\rho_0}\mca{G})_{o,1})\subset\ul{\Diff}(\mca{O}_{\rho_0})_1$. 
\begin{enumerate}
\item If $\mca{G}$ is $s$-connected, then $b$ is injective. 
\item If $\mca{G}$ is $s$-simply connected, then $b$ is bijective and $b(\ul{\Aut}(N_0\rtimes_{\rho_0}\mca{G})_{o,1})=\ul{\Diff}(\mca{O}_{\rho_0})_1$. Hence 
\begin{equation*}
\MCG_o(N_0\rtimes_{\rho_0}\mca{G})\simeq\ul{\MCG}(\mca{O}_{\rho_0}). 
\end{equation*}
\end{enumerate}
\end{cor}

\begin{proof}
The map $b$ is well-defined by Lemma \ref{oosmhs}. 1 follows from Proposition \ref{connuni}. For 2, bijectivity of $b$ follows from Proposition \ref{simconfolbij} and $b(\ul{\Aut}(N_0\rtimes_{\rho_0}\mca{G})_{o,1})=\ul{\Diff}(\mca{O}_{\rho_0})_1$ follows from corollary \ref{bggoogg}. 
\end{proof}

\begin{lem}
Let $\mca{G}\rightrightarrows M$ be a Lie groupoid, $N_0$ be a $C^\infty$ manifold, $\nu_0\colon N_0\to M$ be a $C^\infty$ map and $\rho_0\in\mca{A}_{LF}(\nu_0,\mca{G})$. Then the action $\mca{A}(\nu_0,\mca{G})\curvearrowleft\ul{\Diff}(N_0)\times\ul{\Aut}(\mca{G})$ in Proposition \ref{gmnnmfnnfn} restricts to the action $\mca{A}_{LF}(\mca{O}_{\rho_0},\mca{G})\curvearrowleft\ul{\Diff}(\mca{O}_{\rho_0})\times\ul{\Aut}(\mca{G})$. 
\end{lem}

\begin{proof}
Let $\rho\in\mca{A}_{LF}(\mca{O}_{\rho_0},\mca{G})$ and $(F,\Phi)\in\ul{\Diff}(\mca{O}_{\rho_0})\times\ul{\Aut}(\mca{G})$. We have $\rho^{(F,\Phi)}\in\mca{A}_{LF}(\nu_0,\mca{G})$. Since $F\rtimes\Phi\colon N_0\rtimes_{\rho^{(F,\Phi)}}\mca{G}\to N_0\rtimes_\rho\mca{G}$ is a conjugacy (in particular an isomorphism), $F(\mca{O}_{\rho^{(F,\Phi)}})=\mca{O}_\rho=\mca{O}_{\rho_0}$, hence $\mca{O}_{\rho^{(F,\Phi)}}=F^{-1}(\mca{O}_{\rho_0})=\mca{O}_{\rho_0}$ as partitions into $C^\infty$ submanifolds. Thus $\rho^{(F,\Phi)}\in\mca{A}_{LF}(\mca{O}_{\rho_0},\mca{G})$. 
\end{proof}

\begin{lem}
Let $\mca{G}\rightrightarrows M$ be a Lie groupoid, $N_0$ be a $C^\infty$ manifold, $\nu_0\colon N_0\to M$ be a surjective submersion and $\rho_0\in\mca{A}_{LF}(\nu_0,\mca{G})$. Then we have a map 
\begin{align*}
\MAG_{\id}(\rho_0)&\to\mca{A}_{LF}(\mca{O}_{\rho_0},\mca{G})\\
(N_0\rtimes_\rho\mca{G},\varphi)&\mapsto\rho
\end{align*}
(see Definition \ref{magidmagnfp} for the definition of $\MAG_{\id}(\rho_0)$) and it induces a map 
\begin{equation*}
\MAG_{\id}(\rho_0)/\ul{\sim_o}\to\mca{A}_{LF}(\mca{O}_{\rho_0},\mca{G})/(\ul{\Diff}(\mca{O}_{\rho_0})_1\times\ul{\Aut}(\mca{G})). 
\end{equation*}
\end{lem}

\begin{proof}
Let $(N_0\rtimes_\rho\mca{G},\varphi)\in\MAG_{\id}(\rho_0)$. Then $\nu=\nu_0$, $\rho$ is locally free and $\mca{O}_\rho=\mca{O}_{\rho_0}$ as partitions into $C^\infty$ submanifolds, hence $\rho\in\mca{A}_{LF}(\mca{O}_{\rho_0},\mca{G})$. 

Let $(N_0\rtimes_\rho\mca{G},\varphi)$, $(N_0\rtimes_{\rho^\prime}\mca{G},\varphi^\prime)\in\MAG_{\id}(\rho_0)$ be such that $(N_0\rtimes_\rho\mca{G},\varphi)\mathrel{\ul{\sim_o}}(N_0\rtimes_{\rho^\prime}\mca{G},\varphi^\prime)$. There exists a conjugacy $\psi=F_\psi\rtimes\Psi\in\ul{\Hom}(N_0\rtimes_\rho\mca{G},N_0\rtimes_{\rho^\prime}\mca{G})$ such that $\psi\varphi\mathrel{\ul{\sim_o}}\varphi^\prime$. 
We have $\rho=(\rho^\prime)^{(F_\psi,\Psi)}$ and $F_\psi\in\ul{\Diff}(N_0)$. Since $F_\psi(\mca{O}_{\rho_0})=F_\psi(\mca{O}_\rho)=\mca{O}_{\rho^\prime}=\mca{O}_{\rho_0}$ as partitions into $C^\infty$ submanifolds, $F_\psi\in\ul{\Diff}(\mca{O}_{\rho_0})$. By taking the base maps of $\psi\varphi\mathrel{\ul{\sim_o}}\varphi^\prime$, we get $F_\psi\mathrel{\ul{\sim_o}}\id$. Thus $F_\psi\in\ul{\Diff}(\mca{O}_{\rho_0})_1$. Finally $\Psi\in\ul{\Aut}(\mca{G})$ follows from Lemma \ref{ulid}. 
\end{proof}

By composing with the bijection 
\begin{align*}
T_o(\rho_0)&\simeq\MAG_{\id}(\rho_0)/\ul{\sim_o}\\
[N\rtimes_\rho\mca{G},\varphi]&\mapsto[N_0\rtimes_{\rho^{(F_\varphi,\id)}}\mca{G},(F_\varphi^{-1}\rtimes\id)\varphi]
\end{align*}
in Lemma \ref{magid}, we get a map 
\begin{align*}
T_o(\rho_0)&\to\mca{A}_{LF}(\mca{O}_{\rho_0},\mca{G})/(\ul{\Diff}(\mca{O}_{\rho_0})_1\times\ul{\Aut}(\mca{G}))\\
[N\rtimes_\rho\mca{G},\varphi]&\mapsto\rho^{(F_\varphi,\id)}. 
\end{align*}

\begin{prop}
Let $\mca{G}\rightrightarrows M$ be an $s$-connected Lie groupoid, $N_0$ be a $C^\infty$ manifold, $\nu_0\colon N_0\to M$ be a surjective submersion and $\rho_0\in\mca{A}_{LF}(\nu_0,\mca{G})$. Then 
\begin{equation*}
T_o(\rho_0)\to\mca{A}_{LF}(\mca{O}_{\rho_0},\mca{G})/(\ul{\Diff}(\mca{O}_{\rho_0})_1\times\ul{\Aut}(\mca{G}))
\end{equation*}
is injective. 
\end{prop}

\begin{proof}
It suffices to show that $\MAG_{\id}(\rho_0)/\ul{\sim_o}\to\mca{A}_{LF}(\mca{O}_{\rho_0},\mca{G})/(\ul{\Diff}(\mca{O}_{\rho_0})_1\times\ul{\Aut}(\mca{G}))$ is injective. Let $(N_0\rtimes_\rho\mca{G},\varphi)$, $(N_0\rtimes_{\rho^\prime}\mca{G},\varphi^\prime)\in\MAG_{\id}(\rho_0)$ and assume that there exist $F\in\ul{\Diff}(\mca{O}_{\rho_0})_1$ and $\Phi\in\ul{\Aut}(\mca{G})$ such that $\rho=(\rho^\prime)^{(F,\Phi)}$. Let $\psi=F\rtimes\Phi\in\ul{\Hom}(N_0\rtimes_\rho\mca{G},N_0\rtimes_{\rho^\prime}\mca{G})$. This is a conjugacy. We have $(\varphi^\prime)^{-1}\psi\varphi\in\ul{\End}(N_0\rtimes_{\rho_0}\mca{G})$ and $F_{(\varphi^\prime)^{-1}\psi\varphi}=F\mathrel{\ul{\sim_o}}\id=F_{\id}$. Hence $(\varphi^\prime)^{-1}\psi\varphi\mathrel{\ul{\sim_o}}\id$ by Corollary \ref{bggoogg}. Thus $\psi\varphi\mathrel{\ul{\sim_o}}\varphi^\prime$ and $(N_0\rtimes_\rho\mca{G},\varphi)\mathrel{\ul{\sim_o}}(N_0\rtimes_{\rho^\prime}\mca{G},\varphi^\prime)$. 
\end{proof}

\begin{prop}\label{gmsimiso}
Let $\mca{G}\rightrightarrows M$ be an $s$-simply connected Lie groupoid, $N_0$ be a $C^\infty$ manifold, $\nu_0\colon N_0\to M$ be a surjective submersion and $\rho_0\in\mca{A}_{LF}(\nu_0,\mca{G})$. Then we have 
\begin{equation*}
T_o(\rho_0)\simeq\mca{A}_{LF}(\mca{O}_{\rho_0},\mca{G})/(\ul{\Diff}(\mca{O}_{\rho_0})_1\times\ul{\Aut}(\mca{G})). 
\end{equation*}
\end{prop}

\begin{proof}
It suffices to show $\MAG_{\id}(\rho_0)/\ul{\sim_o}\to\mca{A}_{LF}(\mca{O}_{\rho_0},\mca{G})/(\ul{\Diff}(\mca{O}_{\rho_0})_1\times\ul{\Aut}(\mca{G}))$ is surjective. Let $\rho\in\mca{A}_{LF}(\mca{O}_{\rho_0},\mca{G})$. There exists a unique isomorphism $\varphi\colon N_0\rtimes_{\rho_0}\mca{G}\to N_0\rtimes_\rho\mca{G}$ such that $F_\varphi=\id$ by Corollary \ref{simmon}. We have $(N_0\rtimes_\rho\mca{G},\varphi)\in\MAG_{\id}(\rho_0)$ and it maps to $\rho$. 
\end{proof}

\begin{example}
The statement of Proposition \ref{gmsimiso} is not true if $\mca{G}\rightrightarrows M$ is not $s$-simply connected. For $n\in\bb{Z}_{>0}$, let 
\begin{align*}
\rho_n\colon\U(1)\times\U(1)&\to\U(1)\\
(z,w)&\mapsto zw^n. 
\end{align*}
Then $\U(1)\stackrel{\rho_1}{\curvearrowleft}\U(1)$ is the action by right multiplication. We have $\rho_n\in\mca{A}_{LF}(\mca{O}_{\rho_1},\U(1))$. $T_o(\rho_1)$ is a singleton by Example \ref{tnto}, which we will see later. On the other hand 
\begin{equation*}
\mca{A}_{LF}(\mca{O}_{\rho_1},\U(1))/(\Diff(\U(1))_1\times\Aut(\U(1)))
\end{equation*}
is an infinite set. Hence 
\begin{equation*}
T_o(\rho_1)\to\mca{A}_{LF}(\mca{O}_{\rho_1},\U(1))/(\Diff(\U(1))_1\times\Aut(\U(1)))
\end{equation*}
is not surjective. 
\end{example}

\section{Semiconjugacy-to-Conjugacy Theorem for $o$ and $nt$}\label{1313}
\subsection{Lemma on the normal spaces to orbits}
\begin{lem}\label{tsiso}
Let $\mca{G}\rightrightarrows M$ be a Lie groupoid and $\mca{G}\stackrel{\gamma}{\curvearrowleft}\mca{G}\times\mca{G}$ be the action by multiplication (Example \ref{canaction}). Then the following statements hold: 
\begin{enumerate}
\item Let $\wt{O}\in\mca{O}_\gamma$. Then there exists $O\in\mca{O}_\mca{G}$ such that $t(\wt{O})$, $s(\wt{O})\subset O$. The maps $t,s\colon\wt{O}\to O$ are $C^\infty$. For any $g\in\wt{O}$, 
\begin{equation*}
t_*\colon T_g\mca{G}/T_g\wt{O}\to T_{t(g)}M/T_{t(g)}O,\quad s_*\colon T_g\mca{G}/T_g\wt{O}\to T_{s(g)}M/T_{s(g)}O
\end{equation*}
are isomorphisms. 
\item Let $\sigma\colon M\to\mca{G}$ be a $C^\infty$ map such that $t\sigma=\id$. Let $O\in\mca{O}_\mca{G}$. Then there exists $\wt{O}\in\mca{O}_\gamma$ such that $\sigma(O)\subset\wt{O}$. We have $t(\wt{O})$, $s(\wt{O})\subset O$. The map $\sigma\colon O\to\wt{O}$ is $C^\infty$. For any $x\in O$, 
\begin{equation*}
\sigma_*\colon T_xM/T_xO\to T_{\sigma(x)}\mca{G}/T_{\sigma(x)}\wt{O}
\end{equation*}
is the inverse of $t_*\colon T_{\sigma(x)}\mca{G}/T_{\sigma(x)}\wt{O}\to T_xM/T_xO$. 
\end{enumerate}
\end{lem}

\begin{proof}
Let 
\begin{equation*}
\theta_i\colon\mca{G}\rtimes_\gamma(\mca{G}\times\mca{G})\xrightarrow{\varphi_\gamma}\mca{G}\times\mca{G}\xrightarrow{p_i}\mca{G}, 
\end{equation*}
where $p_i$ is the $i$-th projection. Since $F_{\varphi_\gamma}=(t,s)\colon\mca{G}\to M\times M$, the base maps of $\theta_1$, $\theta_2$ are $t$, $s$. Hence there exist $O_i\in\mca{O}_\mca{G}$ such that $t(\wt{O})\subset O_1$, $s(\wt{O})\subset O_2$. We have $O_1=O_2$, so let $O=O_i$. The maps $t$, $s\colon\wt{O}\to O$ are $C^\infty$ by Lemma \ref{oosmhs}. Hence 
\begin{equation}\label{tgmsgm}
t_*\colon T_g\mca{G}/T_g\wt{O}\to T_{t(g)}M/T_{t(g)}O,\quad s_*\colon T_g\mca{G}/T_g\wt{O}\to T_{s(g)}M/T_{s(g)}O
\end{equation}
are defined. 

Assume that $\sigma\colon M\to\mca{G}$ is a $C^\infty$ map such that $t\sigma=\id$. Consider $\id$, $\id\sigma\colon\mca{G}\to\mca{G}$. Then $\sigma\in\Nat(\id,\id\sigma)$. Hence $\sigma$ is the base map of 
\begin{align*}
\sigma t\rtimes(\id\times\id\sigma)\colon\mca{G}&\to\mca{G}\rtimes_\gamma(\mca{G}\times\mca{G})\\
g&\mapsto(\sigma(t(g)),(g,\sigma(t(g))^{-1}g\sigma(s(g))))
\end{align*}
by Proposition \ref{natmor}. So there exists $\wt{O}\in\mca{O}_\gamma$ such that $\sigma(O)\subset\wt{O}$. The map $\sigma\colon O\to\wt{O}$ is $C^\infty$ by Lemma \ref{oosmhs}. So $\sigma_*\colon T_xM/T_xO\to T_{\sigma(x)}\mca{G}/T_{\sigma(x)}\wt{O}$ is defined and $t_*\sigma_*=\id$. 

As $\sigma$ we can take $1\colon M\to\mca{G}$. Since we have $t1=\id$, $s1=\id$, the maps $t_*$, $s_*$ in \eqref{tgmsgm} are surjective. 

We have $\mca{G}^{t(g)}$, $\mca{G}_{s(g)}\subset\wt{O}$ since $\wt{O}$ is $\mca{G}\times\mca{G}$-invariant, $g\in\wt{O}$, $\mca{G}^{t(g)}=g\mca{G}$ and $\mca{G}_{s(g)}=\mca{G}g$. Hence $T_g\mca{G}^{t(g)}+T_g\mca{G}_{s(g)}\subset T_g\wt{O}$. Thus 
\begin{align*}
\dim T_g\mca{G}/T_g\wt{O}&\leq\dim T_g\mca{G}/(T_g\mca{G}^{t(g)}+T_g\mca{G}_{s(g)})\\
&=\dim\mca{G}-(\dim\mca{G}^{t(g)}+\dim\mca{G}_{s(g)}-\dim\mca{G}^{t(g)}_{s(g)})\\
&=\dim M-(\dim\mca{G}_{s(g)}-\dim\mca{G}^{s(g)}_{s(g)})\\
&=\dim M-\dim O=\dim T_{t(g)}M/T_{t(g)}O
\end{align*}
by Proposition \ref{isoorb}. Therefore $t_*$, $s_*$ in \eqref{tgmsgm} are isomorphisms. 
\end{proof}

\subsection{The cases of $o$ and $nt$}
\begin{prop}\label{btop}
Let $\mca{G}\rightrightarrows M$, $\mca{G}^\prime\rightrightarrows M^\prime$ be Lie groupoids, $N$, $N^\prime$ be $C^\infty$ manifolds, $\nu\colon N\to M$, $\nu^\prime\colon N^\prime\to M^\prime$ be $C^\infty$ maps, $\rho\in\mca{A}(\nu,\mca{G})$, $\rho^\prime\in\mca{A}(\nu^\prime,\mca{G}^\prime)$ and $\varphi$, $\varphi^\prime\in\Hom(N\rtimes_\rho\mca{G},N^\prime\rtimes_{\rho^\prime}\mca{G}^\prime)$. 
Assume $\varphi$ is an isomorphism and $\varphi^\prime=F_{\varphi^\prime}\rtimes\Phi$ is a semiconjugacy. Then the following statements hold: 
\begin{enumerate}
\setlength\itemsep{0em}
\item If $\varphi\sim_o\varphi^\prime$ and $\Phi$ is an isomorphism, then for any $O\in\mca{O}_\rho$, we have $F_{\varphi^\prime}(O)\in\mca{O}_{\rho^\prime}$, the map $F_{\varphi^\prime}\colon O\to F_{\varphi^\prime}(O)$ is a diffeomorphism and 
\begin{align*}
\mca{O}_\rho&\to\mca{O}_{\rho^\prime}\\
O&\mapsto F_{\varphi^\prime}(O)
\end{align*}
is bijective. In particular, $F_{\varphi^\prime}\colon N\to N^\prime$ is bijective (but may not be a diffeomorphism. See Example \ref{notadiff}). 
\item If $\varphi\sim_{nt}\varphi^\prime$ and $\Phi$ is an isomorphism, then $\varphi^\prime$ is a conjugacy. 
\end{enumerate}
\end{prop}

\begin{proof}
Put $\mca{H}=N\rtimes_\rho\mca{G}$, $\mca{H}^\prime=N^\prime\rtimes_{\rho^\prime}\mca{G}^\prime$ and $F=F_\varphi$, $F^\prime=F_{\varphi^\prime}\colon N\to N^\prime$. Then $F$ is a diffeomorphism. 

$1$. Assume $\varphi\sim_o\varphi^\prime$ and $\Phi\colon\mca{G}\to\mca{G}^\prime$ is an isomorphism. Let $O\in\mca{O}_\rho$. Then $O^\prime=F(O)\in\mca{O}_{\rho^\prime}$ since $\varphi$ is an isomorphism. The map $F\colon O\to O^\prime$ is a diffeomorphism by Lemma \ref{oosmhs}. $F^\prime\colon O\to O^\prime$ is defined since $\varphi\sim_o\varphi^\prime$, and is $C^\infty$ by Lemma \ref{oosmhs}. $F$, $F^\prime\colon O\to O^\prime$ are $C^0$ homotopic by the definition of $\varphi\sim_o\varphi^\prime$. Let $O_0$ be a connected component of $O$. Then $O^\prime_0=F(O_0)$ is a connected component of $O^\prime$ and $F\colon O_0\to O^\prime_0$ is a diffeomorphism. $F^\prime\colon O_0\to O^\prime_0$ is defined and is $C^\infty$, and $F$, $F^\prime\colon O_0\to O^\prime_0$ are $C^0$ homotopic. Let $\mca{G}_0$, $\mca{G}^\prime_0$, $\mca{H}_0$, $\mca{H}^\prime_0$ be the identity components of $\mca{G}$, $\mca{G}^\prime$, $\mca{H}$, $\mca{H}^\prime$. Then $O_0$ is an orbit of $\mca{H}_0$ by Lemma \ref{connorb}. 

\begin{claim}
$F^\prime\colon O_0\to O^\prime_0$ is a covering map. 
\end{claim}

\begin{proof}
Let $y\in O_0$. Since $\varphi_\rho\colon\mca{H}\to\mca{G}$ is an action morphism, $\varphi_\rho\colon\mca{H}^y\to\mca{G}^{\nu(y)}$ is a diffeomorphism. Hence $\varphi_\rho\colon(\mca{H}_0)^y\to(\mca{G}_0)^{\nu(y)}$ is a diffeomorphism. We have 
\begin{equation*}
\begin{tikzcd}
(\mca{G}_0)^{\nu(y)}\ar[r,"\Phi"]\ar[rd,phantom,"\circlearrowright"]&(\mca{G}^\prime_0)^{\nu^\prime F^\prime(y)}\\
(\mca{H}_0)^y\ar[r,"\varphi^\prime"]\ar[u,"\sim"{sloped},"\varphi_\rho"{xshift=-6}]\ar[d,"s"']\ar[rd,phantom,"\circlearrowright"{xshift=-4}]&(\mca{H}^\prime_0)^{F^\prime(y)}\ar[u,"\sim"{sloped},"\varphi_{\rho^\prime}"']\ar[d,"s^\prime"]\\
O_0\ar[r,"F^\prime"']&O^\prime_0, 
\end{tikzcd}
\end{equation*}
where the upper square is obtained by restricting 
\begin{equation*}
\begin{tikzcd}
\mca{G}\ar[r,"\Phi"]&\mca{G}^\prime\\
\mca{H}\ar[r,"\varphi^\prime"']\ar[u,"\varphi_\rho"]\ar[ru,phantom,"\circlearrowright"]&\mca{H}^\prime. \ar[u,"\varphi_{\rho^\prime}"']
\end{tikzcd}
\end{equation*}
$s\colon(\mca{H}_0)^y\to O_0$ is a principal $(\mca{H}_0)^y_y$ bundle and $s^\prime\colon(\mca{H}^\prime_0)^{F^\prime(y)}\to O^\prime_0$ is a principal $(\mca{H}^\prime_0)^{F^\prime(y)}_{F^\prime(y)}$ bundle. Let $\rho_0\in\mca{A}(\nu,\mca{G}_0)$, $\rho^\prime_0\in\mca{A}(\nu^\prime,\mca{G}^\prime_0)$ be the restrictions of $\rho$, $\rho^\prime$. Hence 
\begin{equation*}
\begin{tikzcd}
\Stab_{\rho_0}(y)\bs(\mca{G}_0)^{\nu(y)}\ar[r,"\Phi"]\ar[rd,phantom,"\circlearrowright"]&\Stab_{\rho^\prime_0}(F^\prime(y))\bs(\mca{G}^\prime_0)^{\nu^\prime F^\prime(y)}\\
(\mca{H}_0)^y_y\bs(\mca{H}_0)^y\ar[r,"\varphi^\prime"]\ar[u,"\sim"{sloped},"\varphi_\rho"{xshift=-6}]\ar[d,"s"',"\sim"{sloped}]\ar[rd,phantom,"\circlearrowright"{xshift=-15}]&(\mca{H}^\prime_0)^{F^\prime(y)}_{F^\prime(y)}\bs(\mca{H}^\prime_0)^{F^\prime(y)}\ar[u,"\sim"{sloped},"\varphi_{\rho^\prime}"']\ar[d,"\sim"{sloped},"s^\prime"{xshift=7}]\\
O_0\ar[r,"F^\prime"']&O^\prime_0. 
\end{tikzcd}
\end{equation*}
Since $\Phi\colon\mca{G}_0\to\mca{G}^\prime_0$ is an isomorphism, the top horizontal map is decomposed as 
\begin{equation*}
\Stab_{\rho_0}(y)\bs(\mca{G}_0)^{\nu(y)}\simeq\Phi(\Stab_{\rho_0}(y))\bs(\mca{G}^\prime_0)^{\nu^\prime F^\prime(y)}\to\Stab_{\rho^\prime_0}(F^\prime(y))\bs(\mca{G}^\prime_0)^{\nu^\prime F^\prime(y)}
\end{equation*}
The second map is a fiber bundle. Since $\dim O_0=\dim O^\prime_0$, the total space and the base space have the same dimension, so it must be a covering map. 
\end{proof}

\begin{claim}
For $y\in O_0$, $F^\prime_*\colon\pi_1(O_0,y)\to\pi_1(O^\prime_0,F^\prime(y))$ is an isomorphism. 
\end{claim}

\begin{proof}
$F^\prime\colon O_0\to O^\prime_0$ is $C^0$ homotopic to a diffeomorphism $F\colon O_0\to O^\prime_0$. Let $f_t\colon O_0\to O^\prime_0$ be a $C^0$ homotopy between $F^\prime$ and $F$. There exists a $C^0$ homotopy $h_t\colon O^\prime_0\to O^\prime_0$ such that: 
\begin{itemize}
\setlength\itemsep{0em}
\item $h_0=\id$
\item $h_t\colon O_0\to O_0$ is a homeomorphism for all $t\in I$
\item $h_tF^\prime(y)=f_t(y)$ for all $t\in I$. 
\end{itemize}
Then $h_t^{-1}\circ f_t\colon O_0\to O^\prime_0$ is a based $C^0$ homotopy between $F^\prime$ and $h_1^{-1}\circ F$. Hence $F^\prime_*=(h_1^{-1}\circ F)_*\colon\pi_1(O_0,y)\to\pi_1(O^\prime_0,F^\prime(y))$ and 
\begin{equation*}
\begin{tikzcd}[column sep=tiny]
\pi_1(O_0,y)\ar[rr,"(h_1^{-1}\circ F)_*"{name=U}]\ar[rd,"\sim"{sloped},"F_*"']&&\pi_1(O^\prime_0,F^\prime(y))\\
&\pi_1(O^\prime_0,F(y)). \ar[ru,"\sim"{sloped},"(h_1^{-1})_*"']\ar[to=U,phantom,"\circlearrowright"]
\end{tikzcd}
\end{equation*}
So $F^\prime_*$ is an isomorphism. 
\end{proof}

Therefore the degree of the covering $F^\prime\colon O_0\to O^\prime_0$ is $1$, which means $F^\prime\colon O_0\to O^\prime_0$ is a diffeomorphism. Hence $F^\prime\colon O\to O^\prime$ is a diffeomorphism. $F^\prime\colon N\to N^\prime$ is bijective. 

$2$. Assume that $\varphi\sim_{nt}\varphi^\prime$ and $\Phi$ is an isomorphism. $F^\prime\colon N\to N^\prime$ is bijective by $1$. For $y\in N$, let $y\in O\in\mca{O}_\rho$ and put $O^\prime=F(O)\in\mca{O}_{\rho^\prime}$. Then $F^\prime\colon O\to O^\prime$ is a diffeomorphism by $1$. We have 
\begin{equation*}
\begin{tikzcd}[row sep=scriptsize]
T_yN\ar[r,"F^\prime_*"]&T_{F^\prime(y)}N^\prime\\
T_yO\ar[r,"\sim"]\ar[u,phantom,sloped,"\subset"]&T_{F^\prime(y)}O^\prime. \ar[u,phantom,sloped,"\subset"]
\end{tikzcd}
\end{equation*}

\begin{claim}
$F^\prime_*\colon T_yN/T_yO\to T_{F^\prime(y)}N^\prime/T_{F^\prime(y)}O^\prime$ is an isomorphism. 
\end{claim}

\begin{proof}
There exists $P\in\Nat(\varphi,\varphi^\prime)$. Then $P\colon N\to\mca{H}^\prime$ is a $C^\infty$ map such that $t^\prime P=F$ and $\varphi P=\varphi^\prime$. 
%
%
Hence $PF^{-1}\colon N^\prime\to\mca{H}^\prime$ is $C^\infty$ and $t^\prime PF^{-1}=\id$. Let $\gamma^\prime\in\mca{A}((t^\prime,s^\prime),\mca{H}^\prime\times\mca{H}^\prime)$ be the action by multiplication (Example \ref{canaction}), where $(t^\prime,s^\prime)\colon\mca{H}^\prime\to N^\prime\times N^\prime$. By Lemma \ref{tsiso} (letting $\sigma=PF^{-1}$): 
\begin{itemize}
\setlength\itemsep{0em}
\item there exists $\wt{O^\prime}\in\mca{O}_{\gamma^\prime}$ such that $PF^{-1}(O^\prime)\subset\wt{O^\prime}$
\item $PF^{-1}\colon O^\prime\to\wt{O^\prime}$ is $C^\infty$
\item $s^\prime_*\colon T_{P(y)}\mca{H}^\prime/T_{P(y)}\wt{O^\prime}\to T_{F^\prime(y)}N^\prime/T_{F^\prime(y)}O^\prime$ is an isomorphism
\item $(PF^{-1})_*\colon T_{F(y)}N^\prime/T_{F(y)}O^\prime\to T_{P(y)}\mca{H}^\prime/T_{P(y)}\wt{O^\prime}$ is an isomorphism. 
\end{itemize}
Since $F^\prime=s^\prime\circ(PF^{-1})\circ F$, we have 
\begin{equation*}
\begin{tikzcd}[column sep=huge]
T_yN/T_yO\ar[r,"F^\prime_*"]\ar[d,"\sim"{sloped},"F_*"']\ar[rd,phantom,"\circlearrowright"{xshift=4,yshift=-2}]&T_{F^\prime(y)}N^\prime/T_{F^\prime(y)}O^\prime\\
T_{F(y)}N^\prime/T_{F(y)}O^\prime\ar[r,"\sim","(PF^{-1})_*"']&T_{P(y)}\mca{H}^\prime/T_{P(y)}\wt{O^\prime}. \ar[u,"\sim"{sloped},"s^\prime_*"']
\end{tikzcd}
\end{equation*}
Hence $F^\prime_*\colon T_yN/T_yO\to T_{F^\prime(y)}N^\prime/T_{F^\prime(y)}O^\prime$ is an isomorphism. 
\end{proof}

Therefore $F^\prime_*\colon T_yN\to T_{F^\prime(y)}N^\prime$ is an isomorphism. So $F^\prime$ is a diffeomorphism by the inverse function theorem. 

\begin{claim}
The map 
\begin{align*}
(F^\prime)^{-1}\rtimes\Phi^{-1}\colon N^\prime\rtimes_{\rho^\prime}\mca{G}^\prime&\to N\rtimes_\rho\mca{G}\\
(y^\prime,g^\prime)&\mapsto((F^\prime)^{-1}(y^\prime),\Phi^{-1}(g^\prime))
\end{align*}
is defined. 
\end{claim}

\begin{proof}
$t^\prime\Phi=F_\Phi t$ implies $F_\Phi^{-1}t^\prime=t\Phi^{-1}$ and 
$\nu^\prime F^\prime=F_\Phi\nu$ implies $F_\Phi^{-1}\nu^\prime=\nu(F^\prime)^{-1}$. For $(y^\prime,g^\prime)\in N^\prime\rtimes_{\rho^\prime}\mca{G}^\prime$, we have 
\begin{equation*}
\nu(F^\prime)^{-1}(y^\prime)=F_\Phi^{-1}\nu^\prime(y^\prime)=F_\Phi^{-1}t^\prime(g^\prime)=t\Phi^{-1}(g^\prime). 
\end{equation*}
Hence $((F^\prime)^{-1}(y^\prime),\Phi^{-1}(g^\prime))\in N\rtimes_\rho\mca{G}$. 
\end{proof}

We have $(\varphi^\prime)^{-1}=(F^\prime)^{-1}\rtimes\Phi^{-1}$, so $(\varphi^\prime)^{-1}$ is a semiconjugacy. 
%
%
Hence $\varphi^\prime$ is a conjugacy. 
\end{proof}

\begin{rem}
The proof shows the assumption $\varphi\sim_{nt}\varphi^\prime$ in $2$ of the theorem can be relaxed to the following: $\varphi\sim_o\varphi^\prime$ and that there exists $P\in\Nat(\varphi,\varphi^\prime)$. 
\end{rem}

\subsection{Presentations of Teichm\"{u}ller spaces by action morphisms}
\begin{dfn}
Let $\mca{G}\rightrightarrows M$ be a Lie groupoid, $N_0$ be a $C^\infty$ manifold, $\nu_0\colon N_0\to M$ be a $C^\infty$ map and $\rho_0\in\mca{A}(\nu_0,\mca{G})$. We have $\mca{A}(\id,\mca{G})=\left\{\tau\right\}$, where $\tau(x,g)=s(g)$ for all $x\in M$, $g\in\mca{G}$ such that $x=t(g)$. Then $\mca{G}\simeq M\rtimes_\tau\mca{G}$. Define 
\begin{align*}
\ul{\Hom}(N_0\rtimes_{\rho_0}\mca{G},\mca{G})&=\ul{\Hom}(N_0\rtimes_{\rho_0}\mca{G},M\rtimes_\tau\mca{G})\\
&=\left\{c\in\Hom(N_0\rtimes_{\rho_0}\mca{G},\mca{G})\ \middle|\ F_c=\nu_0\right\}, 
\end{align*}
\begin{equation*}
\ul{\Hom}(N_0\rtimes_{\rho_0}\mca{G},\mca{G})^\times=\ul{\Hom}(N_0\rtimes_{\rho_0}\mca{G},\mca{G})\cap\Hom(N_0\rtimes_{\rho_0}\mca{G},\mca{G})^\times
\end{equation*}
and 
\begin{equation*}
\ul{\Hom}^{bb}(N_0\rtimes_{\rho_0}\mca{G},\mca{G})^\times=\left\{c\in\ul{\Hom}(N_0\rtimes_{\rho_0}\mca{G},\mca{G})^\times\ \middle|\ 
\begin{gathered}
c^!\colon N_0\rtimes_{\rho_0}\mca{G}\to N_0\rtimes_{\rho_c}\mca{G}\\
\text{is bibounded with}\\
\text{respect to $\mca{B}_{\rho_0}$ and $\mca{B}_{\rho_c}$}
\end{gathered}
\right\}. 
\end{equation*}
\end{dfn}

\begin{lem}
Let $\mca{G}\rightrightarrows M$ be a Lie groupoid, $N_0$ be a $C^\infty$ manifold, $\nu_0\colon N_0\to M$ be a $C^\infty$ map and $\rho_0\in\mca{A}(\nu_0,\mca{G})$. Then the following statements hold: 
\begin{enumerate}
\item $\ul{\Aut}(\mca{G})\curvearrowright\ul{\Hom}(N_0\rtimes_{\rho_0}\mca{G},\mca{G})^\times$ by post composition. \label{authompo}
\item $\ul{\Hom}^{bb}(N_0\rtimes_{\rho_0}\mca{G},\mca{G})^\times$ is an $\ul{\Aut}(\mca{G})$-invariant subset of $\ul{\Hom}(N_0\rtimes_{\rho_0}\mca{G},\mca{G})^\times$ with respect to the action in \ref{authompo}. 
\item For $a\in\left\{o,nt\right\}$, the action in \ref{authompo} descends to $\ul{\Aut}(\mca{G})\curvearrowright\ul{\Hom}(N_0\rtimes_{\rho_0}\mca{G},\mca{G})^\times/\ul{\sim_a}$. 
\item The action in \ref{authompo} induces $\ul{\Aut}(\mca{G})\curvearrowright\ul{\Hom}^{bb}(N_0\rtimes_{\rho_0}\mca{G},\mca{G})^\times/\ul{\sim_{bnt}}$, where $\mathrel{\ul{\sim_{bnt}}}$ is with respect to $\mca{B}_\mca{G}$. The isomorphism $\mca{G}\simeq M\rtimes_\tau\mca{G}$ takes $\mca{B}_\mca{G}$ to $\mca{B}_\tau$. 
\end{enumerate}
\end{lem}

\begin{proof}
1. This follows from Lemma \ref{compacmor} and 1 in Proposition \ref{isoactdiff}. 

2. Let $\Phi\in\ul{\Aut}(\mca{G})$ and $c\in\ul{\Hom}^{bb}(N_0\rtimes_{\rho_0}\mca{G},\mca{G})^\times$. Since $c^!=(t,c)$ and $(\Phi c)^!=(t,\Phi c)$, we get $(\Phi c)^!=(\id\rtimes\Phi)c^!$. We have 
\begin{equation*}
\begin{tikzcd}[row sep=tiny]
&N_0\rtimes_{\rho_c}\mca{G}\ar[dd,"\id\rtimes\Phi"{name=U}]\ar[r,equal]&[-15]N_0\times_{\nu_0,t}\mca{G}\\
N_0\rtimes_{\rho_0}\mca{G}\ar[ru,"c^!"]\ar[rd,"(\Phi c)^!"']\ar[to=U,phantom,"\circlearrowright"]\\
&N_0\rtimes_{\rho_{\Phi c}}\mca{G}\ar[r,equal]&N_0\times_{\nu_0,t}\mca{G}. 
\end{tikzcd}
\end{equation*}
Since $\id\rtimes\Phi$ is a conjugacy, $\id\rtimes\Phi$ is bibounded by Proposition \ref{semibounded}. Hence $(\Phi c)^!$ is bibounded. 

3. This follows from Lemma \ref{ulsimsim}. 

4. Let $\Phi\in\ul{\Aut}(\mca{G})$ and $c$, $c^\prime\in\ul{\Hom}^{bb}(N_0\rtimes_{\rho_0}\mca{G},\mca{G})^\times$ be such that $c\mathrel{\ul{\sim_{bnt}}}c^\prime$. Since $F_\Phi=\id$ is proper, $\Phi$ is bounded with respect to $\mca{B}_\mca{G}$ by Proposition \ref{properbounded}. Hence $\Phi c\mathrel{\ul{\sim_{bnt}}}\Phi c^\prime$ by Lemma \ref{ulsimsim2}. 
\end{proof}

\begin{lem}
Let $\mca{G}\rightrightarrows M$ be a Lie groupoid, $N_0$ be a $C^\infty$ manifold, $\nu_0\colon N_0\to M$ be a surjective submersion and $\rho_0\in\mca{A}(\nu_0,\mca{G})$. For $\Phi\in\ul{\Aut}(\mca{G})$ and $(N\rtimes_\rho\mca{G},\varphi)\in\MAG(\rho_0)$, we have 
\begin{equation*}
\begin{tikzcd}[row sep=tiny]
&N\rtimes_\rho\mca{G}\ar[dd,"\id\rtimes\Phi"{name=U}]\\
N_0\rtimes_{\rho_0}\mca{G}\ar[ru,"\varphi"]\ar[rd,"(\id\rtimes\Phi)\varphi"']\ar[to=U,phantom,"\circlearrowright"]\\
&N\rtimes_{\rho^{(\id,\Phi^{-1})}}\mca{G}, 
\end{tikzcd}
\end{equation*}
where $\rho^{(\id,\Phi^{-1})}\in\mca{A}(\nu,\mca{G})$ is defined in Lemma \ref{ngconj}. Define 
\begin{equation*}
\Phi(N\rtimes_\rho\mca{G},\varphi)=(N\rtimes_{\rho^{(\id,\Phi^{-1})}}\mca{G},(\id\rtimes\Phi)\varphi)\in\MAG(\rho_0). 
\end{equation*}
Then we get an action $\ul{\Aut}(\mca{G})\curvearrowright\MAG(\rho_0)$ and $\Phi(N\rtimes_\rho\mca{G},\varphi)\mathrel{\ul{\sim_a}}(N\rtimes_\rho\mca{G},\varphi)$ for any $a\in\left\{o,nt,bnt\right\}$. The subsets $\MAG_{\id}(\rho_0)$, $\MAG^{bb}(\rho_0)$ and $\MAG_{\id}^{bb}(\rho_0)$ of $\MAG(\rho_0)$ are $\ul{\Aut}(\mca{G})$-invariant. 
\end{lem}

\begin{proof}
These are obvious. 
\end{proof}

\begin{prop}\label{magidhom}
Let $\mca{G}\rightrightarrows M$ be a Lie groupoid, $N_0$ be a $C^\infty$ manifold, $\nu_0\colon N_0\to M$ be a surjective submersion and $\rho_0\in\mca{A}(\nu_0,\mca{G})$. Then there is a bijection 
\begin{align*}
\MAG_{\id}(\rho_0)&\simeq\ul{\Hom}(N_0\rtimes_{\rho_0}\mca{G},\mca{G})^\times\\
(N_0\rtimes_\rho\mca{G},\varphi)&\mapsto\varphi_\rho\varphi\\
(N_0\rtimes_{\rho_c}\mca{G},c^!)&\mapsfrom c. 
\end{align*}
This is $\ul{\Aut}(\mca{G})$-equivariant and $\MAG^{bb}_{\id}(\rho_0)$ corresponds to $\ul{\Hom}^{bb}(N_0\rtimes_{\rho_0}\mca{G},\mca{G})^\times$. 
\end{prop}

\begin{proof}
The maps are well-defined. We have $c\mapsto(N_0\rtimes_{\rho_c}\mca{G},c^!)\mapsto c$. On the other hand, $(N_0\rtimes_\rho\mca{G},\varphi)\mapsto\varphi_\rho\varphi\mapsto(N_0\rtimes_{\rho_{\varphi_\rho\varphi}}\mca{G},(\varphi_\rho\varphi)^!)$. Since $(\varphi_\rho\varphi)^!=(t,\varphi_\rho\varphi)=\varphi$ and $\rho_{\varphi_\rho\varphi}=\rho_0((\varphi_\rho\varphi)^!)^{-1}=\rho_0\varphi^{-1}=\rho$, $(N_0\rtimes_{\rho_{\varphi_\rho\varphi}}\mca{G},(\varphi_\rho\varphi)^!)=(N_0\rtimes_\rho\mca{G},\varphi)$. Hence the maps are inverse to each other. 
\end{proof}

\begin{rem}
Since $\ul{\Hom}(N_0\rtimes_{\rho_0}\mca{G},\mca{G})^\times$ and $\ul{\Hom}^{bb}(N_0\rtimes_{\rho_0}\mca{G},\mca{G})^\times$ are subsets of $C^\infty(N_0\rtimes_{\rho_0}\mca{G},\mca{G})$, they can be given the strong or weak topologies. So we can topologize $T_a(\rho_0)$ via the bijections above and the quotient topology. 
\end{rem}

\begin{prop}\label{ampre}
Let $\mca{G}\rightrightarrows M$ be a Lie groupoid, $N_0$ be a $C^\infty$ manifold, $\nu_0\colon N_0\to M$ be a surjective submersion and $\rho_0\in\mca{A}(\nu_0,\mca{G})$. Then the following statements hold: 
\begin{enumerate}
\item There is a bijection 
\begin{align*}
T_{bnt}(\rho_0)&\simeq\ul{\Aut}(\mca{G})\bs(\ul{\Hom}^{bb}(N_0\rtimes_{\rho_0}\mca{G},\mca{G})^\times/\ul{\sim_{bnt}})\\
[N\rtimes_\rho\mca{G},\varphi]&\mapsto[\varphi_\rho\varphi]\\
[N_0\rtimes_{\rho_c}\mca{G},c^!]&\mapsfrom[c]. 
\end{align*}
\item There is a bijection 
\begin{align*}
T_{nt}(\rho_0)&\simeq\ul{\Aut}(\mca{G})\bs(\ul{\Hom}(N_0\rtimes_{\rho_0}\mca{G},\mca{G})^\times/\ul{\sim_{nt}})\\
[N\rtimes_\rho\mca{G},\varphi]&\mapsto[\varphi_\rho\varphi]\\
[N_0\rtimes_{\rho_c}\mca{G},c^!]&\mapsfrom[c]. 
\end{align*}
\item There is a surjective map 
\begin{align*}
T_o(\rho_0)&\twoheadrightarrow\ul{\Aut}(\mca{G})\bs(\ul{\Hom}(N_0\rtimes_{\rho_0}\mca{G},\mca{G})^\times/\ul{\sim_o})\\
[N\rtimes_\rho\mca{G},\varphi]&\mapsto[\varphi_\rho\varphi]. 
\end{align*}
\end{enumerate}
\end{prop}

\begin{proof}
1. We have 
\begin{align}
\MAG_{\id}^{bb}(\rho_0)&\simeq\ul{\Hom}^{bb}(N_0\rtimes_{\rho_0}\mca{G},\mca{G})^\times\label{maghom}\\
(N_0\rtimes_\rho\mca{G},\varphi)&\mapsto\varphi_\rho\varphi\nonumber\\
(N_0\rtimes_{\rho_c}\mca{G},c^!)&\mapsfrom c\nonumber
\end{align}
by Proposition \ref{magidhom}. 

Let $(N_0\rtimes_\rho\mca{G},\varphi)$, $(N_0\rtimes_{\rho^\prime}\mca{G},\varphi^\prime)\in\MAG_{\id}^{bb}(\rho_0)$ be such that $(N_0\rtimes_\rho\mca{G},\varphi)\mathrel{\ul{\sim_{bnt}}}(N_0\rtimes_{\rho^\prime}\mca{G},\varphi^\prime)$. There exists a conjugacy $\psi=F_\psi\rtimes\Psi\in\ul{\Hom}(N_0\rtimes_\rho\mca{G},N_0\rtimes_{\rho^\prime}\mca{G})$ such that $\psi\varphi\mathrel{\ul{\sim_{bnt}}}\varphi^\prime$. $\Psi\in\ul{\Aut}(\mca{G})$ since $\nu_0$ is surjective. Since $\varphi_{\rho^\prime}\colon N_0\rtimes_{\rho^\prime}\mca{G}\to\mca{G}$ is bounded, $\Psi\varphi_\rho\varphi=\varphi_{\rho^\prime}\psi\varphi\mathrel{\ul{\sim_{bnt}}}\varphi_{\rho^\prime}\varphi^\prime$ by Lemma \ref{ulsimsim2}. Hence we get the map 
\begin{equation*}
\MAG_{\id}^{bb}(\rho_0)/\ul{\sim_{bnt}}\to\ul{\Aut}(\mca{G})\bs(\ul{\Hom}^{bb}(N_0\rtimes_{\rho_0}\mca{G},\mca{G})^\times/\ul{\sim_{bnt}}). 
\end{equation*}

Let $c$, $c^\prime\in\ul{\Hom}^{bb}(N_0\rtimes_{\rho_0}\mca{G},\mca{G})^\times$ be such that $c\mathrel{\ul{\sim_{bnt}}}c^\prime$. There exists a continuous map $P\colon N_0\times I\to\mca{G}$ such that: 
\begin{itemize}
\setlength\itemsep{0em}
\item $P(\cdot,t)\in C^\infty(N_0,\mca{G})$ for all $t\in I$
\item $\nu_0=tP(\cdot,t)=sP(\cdot,t)$ for all $t\in I$
\item $P(\cdot,0)=1_{\nu_0}$
\item $c^\prime=cP(\cdot,1)$
\item $P(N_0\times I)\in\mca{B}_\mca{G}$. 
\end{itemize}
Define 
\begin{align*}
\wt{P}\colon N_0\times I&\to N_0\rtimes_{\rho_c}\mca{G}\\
(y,t)&\mapsto(y,P(y,t)). 
\end{align*}
Then 
\begin{itemize}
\setlength\itemsep{0em}
\item $\wt{P}$ is continuous
\item $\wt{P}(\cdot,t)\in C^\infty(N_0,N_0\rtimes_{\rho_c}\mca{G})$ for all $t\in I$
\item $t\wt{P}(\cdot,t)=\id$
\item $\wt{P}(\cdot,0)=1_{\id}$
\item $\nu_0s\wt{P}(y,t)=\nu_0\rho_c(y,P(y,t))=sP(y,t)=\nu_0(y)$ for all $y\in N_0$
\item $\varphi_{\rho_c}\wt{P}(N_0\times I)=P(N_0\times I)\in\mca{B}_\mca{G}$, hence $\wt{P}(N_0\times I)\in\mca{B}_{\rho_c}$. 
\end{itemize}
Put $\wt{P}_1=\wt{P}(\cdot,1)$. Then $\wt{P}$ is a bounded $nt$-homotopy between $c^!$ and $c^!\wt{P}_1$, and $c^!\mathrel{\ul{\sim_{bnt}}}c^!\wt{P}_1$. Let $\psi=(c^!\wt{P}_1)((c^\prime)^!)^{-1}\in\ul{\Hom}(N_0\rtimes_{\rho_{c^\prime}}\mca{G},N_0\rtimes_{\rho_c}\mca{G})$. We have 
\begin{equation*}
\begin{tikzcd}[row sep=small]
&N_0\rtimes_{\rho_c}\mca{G}\\
N_0\rtimes_{\rho_0}\mca{G}\ar[ru,shift right,"c^!\wt{P}_1"']\ar[ur,shift left,"c^!"]\ar[rd,"(c^\prime)^!"']\\
&N_0\rtimes_{\rho_{c^\prime}}\mca{G}. \ar[uu,"\psi"']
\end{tikzcd}
\end{equation*}

\begin{claim}
$\varphi_{\rho_c}(c^!\wt{P}_1)=c^\prime$. 
\end{claim}

\begin{proof}
For $(y,g)\in N_0\rtimes_{\rho_0}\mca{G}$, 
\begin{align*}
\varphi_{\rho_c}(c^!\wt{P}_1)(y,g)&=\varphi_{\rho_c}(\wt{P}_1(y)^{-1}c^!(y,g)\wt{P}_1(\rho_0(y,g)))\\
&=P(y,1)^{-1}c(y,g)P(\rho_0(y,g),1)\\
&=(cP(\cdot,1))(y,g)=c^\prime(y,g). \qedhere
\end{align*}
\end{proof}


\begin{claim}
$\varphi_{\rho_c}\psi=\varphi_{\rho_{c^\prime}}$, hence $\psi$ is a semiconjugacy. 
\end{claim}

\begin{proof}
We have $\varphi_{\rho_c}\psi(c^\prime)^!=\varphi_{\rho_c}(c^!\wt{P}_1)=c^\prime=\varphi_{\rho_{c^\prime}}(c^\prime)^!$, so $\varphi_{\rho_c}\psi=\varphi_{\rho_{c^\prime}}$. 
\end{proof}

Since $\psi(c^\prime)^!=c^!\wt{P}_1\mathrel{\ul{\sim_{bnt}}}c^!$, we have $\psi\mathrel{\ul{\sim_{bnt}}}c^!((c^\prime)^!)^{-1}$. In particular $\psi\sim_{nt}c^!((c^\prime)^!)^{-1}$, so $\psi$ is a conjugacy by Proposition \ref{btop}. Therefore $(N_0\rtimes_{\rho_c}\mca{G},c^!)\mathrel{\ul{\sim_{bnt}}}(N_0\rtimes_{\rho_{c^\prime}}\mca{G},(c^\prime)^!)$ and we get the map 
\begin{equation*}
\ul{\Hom}^{bb}(N_0\rtimes_{\rho_0}\mca{G},\mca{G})^\times/\ul{\sim_{bnt}}\to\MAG_{\id}^{bb}(\rho_0)/\ul{\sim_{bnt}}. 
\end{equation*}
By the $\ul{\Aut}(\mca{G})$-equivariance of \eqref{maghom}, we get 
\begin{equation*}
\ul{\Aut}(\mca{G})\bs(\ul{\Hom}^{bb}(N_0\rtimes_{\rho_0}\mca{G},\mca{G})^\times/\ul{\sim_{bnt}})\to\MAG_{\id}^{bb}(\rho_0)/\ul{\sim_{bnt}}. 
\end{equation*}

Therefore we have 
\begin{align*}
\MAG_{\id}^{bb}(\rho_0)/\ul{\sim_{bnt}}&\simeq\ul{\Aut}(\mca{G})\bs(\ul{\Hom}^{bb}(N_0\rtimes_{\rho_0}\mca{G},\mca{G})^\times/\ul{\sim_{bnt}})\\
[N_0\rtimes_\rho\mca{G},\varphi]&\mapsto[\varphi_\rho\varphi]\\
[N_0\rtimes_{\rho_c}\mca{G},c^!]&\mapsfrom[c]. 
\end{align*}
By composing with the bijection 
\begin{align*}
T_{bnt}(\rho_0)&\simeq\MAG_{\id}^{bb}(\rho_0)/\ul{\sim_{bnt}}\\
[N\rtimes_\rho\mca{G},\varphi]&\mapsto[N_0\rtimes_{\rho^{(F_\varphi,\id)}}\mca{G},(F_\varphi^{-1}\rtimes\id)\varphi]\\
[N_0\rtimes_\rho\mca{G},\varphi]&\mapsfrom[N_0\rtimes_\rho\mca{G},\varphi], 
\end{align*}
we obtain the desired bijection. 

2 and 3 can be proved similarly. 
\end{proof}

\section{Semiconjugacy-to-Conjugacy Theorem for $bnt$}\label{1212}
\subsection{Preparations}
\subsubsection{The target fibers are paracompact}
In this section we prove the following lemma. 

\begin{lem}\label{tarpara}
Let $\mca{G}\rightrightarrows M$ be a Lie groupoid and $x\in M$. Then $\mca{G}^x$ is paracompact. 
\end{lem}

We need some preparations for the proof. 

\begin{dfn}
Let $X$ be a topological space and $\ms{A}\subset P(X)$. We say that $\ms{A}$ is \emph{countably locally finite} if there exist a countable set $A$ and $\ms{A}_a\subset P(X)$ for each $a\in A$ such that: 
\begin{itemize}
\setlength\itemsep{0em}
\item $\ms{A}_a$ is locally finite for each $a\in A$
\item $\ms{A}=\bigcup_{a\in A}\ms{A}_a$. 
\end{itemize}
\end{dfn}

\begin{lem}
Let $X$ be a regular Hausdorff space. Then $X$ is paracompact if and only if for any open covering $\ms{A}$ of $X$, there exists a countably locally finite open covering of $X$ which refines $\ms{A}$. 
\end{lem}

\begin{proof}
See Lemma 41.3 in \cite{Munk}. 
\end{proof}

\begin{lem}\label{regpara}
Let $X$ be a regular Hausdorff space, $A$ be a countable set and $X_a\subset X$ be a closed paracompact subset for each $a\in A$ such that $X=\bigcup_{a\in A}X_a^\circ$, where $X_a^\circ$ is the interior of $X_a$ in $X$. Then $X$ is paracompact. 
\end{lem}

\begin{proof}
This is Exercise 7 (b) of \S41 in \cite{Munk}. Let $\ms{A}$ be an open covering of $X$. For $a\in A$, $\ms{A}_a=\left\{U\cap X_a\ \middle|\ U\in\ms{A}\right\}$ is an open covering of $X_a$. By paracompactness, there exists a locally finite open covering $\ms{B}_a$ of $X_a$ which refines $\ms{A}_a$. For each $U\in\ms{B}_a$, take an open subset $V_U$ of $X$ such that $U=V_U\cap X_a$. Let $\ms{C}_a=\left\{V_U\cap X_a^\circ\ \middle|\ U\in\ms{B}_a\right\}$. Then $\bigcup\ms{C}_a=X_a^\circ$. 

\begin{claim}
$\ms{C}_a$ is locally finite in $X$. 
\end{claim}

\begin{proof}
Let $x\in X$. If $x\in X_a$, there exists an open neighborhood $U_x$ of $x$ in $X_a$ which intersects nontrivially with only finitely many elements of $\ms{B}_a$. Take an open subset $V_x$ of $X$ such that $U_x=X_a\cap V_x$. Then $V_x$ is an open neighborhood of $x$ in $X$ which has only finitely many non empty intersections with elements of $\ms{C}_a$. 

If $x\in X\setminus X_a$, $X\setminus X_a$ is an open neighborhood of $x$ in $X$ which has no nonempty intersections with elements of $\ms{C}_a$. 
\end{proof}

Then $\ms{C}=\bigcup_{a\in A}\ms{C}_a$ is countably locally finite. $\ms{C}$ is an open covering of $X$ which refines $\ms{A}$. 
\end{proof}

\begin{lem}\label{paracomcompara}
Let $X$ be a paracompact space and $Y$ be a compact space. Then $X\times Y$ is paracompact. 
\end{lem}

\begin{proof}
A proof can be found on the internet. 
\end{proof}

\begin{proof}[Proof of Lemma \ref{tarpara}]
Let $x\in O\in\mca{O}_\mca{G}$. Then $s\colon\mca{G}^x\to O$ is a principal $\mca{G}^x_x$-bundle by Proposition \ref{isoorb}. Since $\mca{G}^x_x$ is a Lie group by Proposition \ref{isoorb}, it is paracompact. Let $O_0$ be a connected component of $O$. Then $s\colon\mca{G}^x_{O_0}\to O_0$ is a principal $\mca{G}^x_x$-bundle. 

\begin{claim}
$O_0$ is $\sigma$-compact. 
\end{claim}

\begin{proof}
Since $M$ is a second countable locally compact Hausdorff space, it is paracompact. So there exists a Riemannian metric on $M$. Hence an immersed submanifold $O_0$ admits a Riemannian metric. Thus $O_0$, with the submanifold topology, is metrizable. Therefore $O_0$ is paracompact. Since $O_0$ is connected paracompact locally compact Hausdorff, it is $\sigma$-compact. 
\end{proof}

There exist a countable set $A$ and a compact contractible subset $D_a\subset O_0$ for each $a\in A$ such that $O_0=\bigcup_{a\in A}D_a^\circ$. $\mca{G}^x_{D_a}$ is homeomorphic to $D_a\times\mca{G}^x_x$. $\mca{G}^x_{D_a}$ is a closed paracompact subset of $\mca{G}^x_{O_0}$ by Lemma \ref{paracomcompara}. We have $(\mca{G}^x_{D_a})^\circ=\mca{G}^x_{D_a^\circ}$ in $\mca{G}^x_{O_0}$. Hence $\mca{G}^x_{O_0}=\bigcup_{a\in A}(\mca{G}^x_{D_a})^\circ$. 

Since $\mca{G}^x$ is a Hausdorff $C^\infty$ manifold, it is locally compact Hausdorff, hence regular Hausdorff. Thus $\mca{G}^x_{O_0}$ is regular Hausdorff. Therefore $\mca{G}^x_{O_0}$ is paracompact by Lemma \ref{regpara}. Since $O_0$ is open in $O$, $\mca{G}^x_{O_0}$ is open in $\mca{G}^x$. $\mca{G}^x$ is a disjoint union of open paracompact subsets. Hence $\mca{G}^x$ is paracompact. 
\end{proof}

\subsubsection{The diffeomorphisms induced by actions}
\begin{lem}\label{onuonud}
Let $\mca{G}\rightrightarrows M$ be a Lie groupoid, $N$ be a $C^\infty$ manifold, $\nu\colon N\to M$ be a $C^\infty$ map, $\rho\in\mca{A}(\nu,\mca{G})$ and $O\in\mca{O}_\rho$. Then the following statements hold: 
\begin{enumerate}
\item There exists a unique $\ul{O}\in\mca{O}_\mca{G}$ such that $\nu(O)\subset\ul{O}$. The map $\nu\colon O\to\ul{O}$ is a surjective submersion. Hence $O\cap\nu^{-1}(x)$ is an embedded submanifold of $O$ for any $x\in\ul{O}$. 
\item For any $x$, $x^\prime\in\ul{O}$ and $g\in\mca{G}^x_{x^\prime}$, $\rho(\cdot,g)\colon O\cap\nu^{-1}(x)\to O\cap\nu^{-1}(x^\prime)$ is a diffeomorphism. 
\end{enumerate}
(Note that $\nu^{-1}(x)$ may not be an embedded submanifold of $N$ for some $x\in M$, which complicates the proof of 2.) 
\end{lem}

\begin{proof}
Put $\mca{H}=N\rtimes_\rho\mca{G}$. 

1. Since $\nu\colon N\to M$ is the base map of $\varphi_\rho\colon\mca{H}\to\mca{G}$, there exists a unique $\ul{O}\in\mca{O}_\mca{G}$ such that $\nu(O)\subset\ul{O}$. Take any $y_0\in O$ and put $x_0=\nu(y_0)\in\ul{O}$. 

For $x\in\ul{O}$, there exists $g\in\mca{G}^{x_0}_x$. We have $\rho(y_0,g)\in O$ and $\nu\rho(y_0,g)=x$. Hence $\nu\colon O\to\ul{O}$ is surjective. 

$s\colon\mca{G}^{x_0}\to\ul{O}$ is a submersion by Proposition \ref{isoorb}. Hence there exist an open neighborhood $U$ of $x_0$ in $\ul{O}$ and a $C^\infty$ map $\sigma\colon U\to\mca{G}^{x_0}$ such that $s\sigma=\id$ and $\sigma(x_0)=1_{x_0}$. Let 
\begin{align*}
\wt{\sigma}\colon U&\to O\\
x&\mapsto\rho(y_0,\sigma(x)). 
\end{align*}
Then $\wt{\sigma}$ is $C^\infty$, $\nu\wt{\sigma}=\id$ and $\wt{\sigma}(x_0)=y_0$. Hence $\nu_*\colon T_{y_0}O\to T_{x_0}\ul{O}$ is surjective. Therefore $\nu\colon O\to\ul{O}$ is a submersion. 

2. Fix $y\in O\cap\nu^{-1}(x)$. Then $s\colon\mca{H}^y\to O$ is a surjective submersion and $O\cap\nu^{-1}(x)$, $O\cap\nu^{-1}(x^\prime)$ are embedded submanifolds of $O$. We have 
\begin{equation*}
\begin{tikzcd}
\mca{H}^y_{O\cap\nu^{-1}(x)}\ar[r]\ar[d,"s"']\ar[rd,phantom,"\circlearrowright"{xshift=-2,yshift=3}]&\mca{H}^y\ar[d,"s"]\\
O\cap\nu^{-1}(x)\ar[r]&O, 
\end{tikzcd}
\qquad
\begin{tikzcd}
\mca{H}^y_{O\cap\nu^{-1}(x^\prime)}\ar[r]\ar[d,"s"']\ar[rd,phantom,"\circlearrowright"{xshift=-2,yshift=3}]&\mca{H}^y\ar[d,"s"]\\
O\cap\nu^{-1}(x^\prime)\ar[r]&O, 
\end{tikzcd}
\end{equation*}
where $s\colon\mca{H}^y_{O\cap\nu^{-1}(x)}\to O\cap\nu^{-1}(x)$, $s\colon\mca{H}^y_{O\cap\nu^{-1}(x^\prime)}\to O\cap\nu^{-1}(x^\prime)$ are surjective submersions by Lemma \ref{fiberproduct}. Consider 
\begin{align*}
R_{(s(\cdot),g)}\colon\mca{H}^y_{O\cap\nu^{-1}(x)}&\to\mca{H}^y_{O\cap\nu^{-1}(x^\prime)}\\
(y,g^\prime)&\mapsto(y,g^\prime g). 
\end{align*}
We have 
\begin{equation}\label{HHOO}
\begin{tikzcd}
\mca{H}^y_{O\cap\nu^{-1}(x)}\ar[r,"R_{(s(\cdot),g)}"]\ar[d,"s"']\ar[rd,phantom,"\circlearrowright"]&\mca{H}^y_{O\cap\nu^{-1}(x^\prime)}\ar[d,"s"]\\
O\cap\nu^{-1}(x)\ar[r,"\rho(\cdot{,}g)"']&O\cap\nu^{-1}(x^\prime). 
\end{tikzcd}
\end{equation}

\begin{claim}
$R_{(s(\cdot),g)}\colon\mca{H}^y_{O\cap\nu^{-1}(x)}\to\mca{H}^y_{O\cap\nu^{-1}(x^\prime)}$ is $C^\infty$. 
\end{claim}

\begin{proof}
$\mca{H}^y_{O\cap\nu^{-1}(x)}\times_{s,t}\mca{H}$ is an embedded submanifold of $\mca{H}^y_{O\cap\nu^{-1}(x)}\times\mca{H}$, hence of $\mca{H}\times_{s,t}\mca{H}$. Now $R_{(s(\cdot),g)}$ can be expressed as the composition 
\begin{align*}
\mca{H}^y_{O\cap\nu^{-1}(x)}&\to\mca{H}^y_{O\cap\nu^{-1}(x)}\times_{s,t}\mca{H}\to\mca{H}, \\
(y,g^\prime)&\mapsto((y,g^\prime),(\rho(y,g^\prime),g))
\end{align*}
where the second map is the restriction of the composition map $\mca{H}\times_{s,t}\mca{H}\to\mca{H}$, which is $C^\infty$. 
\end{proof}

By taking local $C^\infty$ sections of the left $s$ in Diagram \eqref{HHOO}, we see that $\rho(\cdot,g)\colon O\cap\nu^{-1}(x)\to O\cap\nu^{-1}(x^\prime)$ is $C^\infty$. By replacing $g$ with $g^{-1}$, we obtain 
\begin{equation*}
\begin{tikzcd}[column sep=large]
\mca{H}^y_{O\cap\nu^{-1}(x)}\ar[d,"s"']\ar[rd,phantom,"\circlearrowright"]&\mca{H}^y_{O\cap\nu^{-1}(x^\prime)}\ar[d,"s"]\ar[l,"R_{(s(\cdot),g^{-1})}"']\\
O\cap\nu^{-1}(x)&O\cap\nu^{-1}(x^\prime), \ar[l,"\rho(\cdot{,}g^{-1})"]
\end{tikzcd}
\end{equation*}
which shows the inverse map $\rho(\cdot,g^{-1})\colon O\cap\nu^{-1}(x^\prime)\to O\cap\nu^{-1}(x)$ is $C^\infty$ as well. 
\end{proof}

\subsubsection{The degree of a proper $C^\infty$ map}
\begin{dfn}
Let $M$, $M^\prime$ be connected paracompact Hausdorff oriented $C^\infty$ manifolds of dimension $n$ and $f\colon M\to M^\prime$ be a proper $C^\infty$ map. Then there exists a unique $\deg f\in\bb{R}$, called the \emph{degree} of $f$, such that 
\begin{equation*}
\begin{tikzcd}
H_c^n(M^\prime;\bb{R})\ar[r,"f^*"]\ar[d,"\sim"{sloped},"\int_{M^\prime}"']\ar[rd,phantom,"\circlearrowright"{xshift=-6,yshift=2}]&H_c^n(M;\bb{R})\ar[d,"\sim"{sloped},"\int_M"{xshift=7}]\\
\bb{R}\ar[r,"\deg f"']&\bb{R}, 
\end{tikzcd}
\end{equation*}
where $H_c^n$ is the $n$-th de Rham cohomology with compact support. 
\end{dfn}

\begin{prop}\label{degdiffp}
Let $M$, $M^\prime$ be connected oriented $C^\infty$ manifolds of dimension $n$ and $f\colon M\to M^\prime$ be a proper $C^\infty$ map. Then the following statements hold: 
\begin{enumerate}
\item $\deg f\in\bb{Z}$. 
\item If $f$ is a diffeomorphism, then $\deg f=\pm1$. 
\item If $\deg f\neq0$, then $f$ is surjective. 
\item If $f^\prime\colon M\to M^\prime$ is a proper $C^\infty$ map and $f$, $f^\prime$ are properly $C^\infty$ homotopic, then $\deg f=\deg f^\prime$. 
\end{enumerate}
\end{prop}

\begin{lem}\label{mfgone}
Let $M$ be a $C^\infty$ manifold possibly with boundary, $E$ be a closed subset of $M$ and $f\in C^0(M,\bb{R}^N)$. Assume that $f$ is $C^\infty$ on an open neighborhood of $E$ in $M$. Let $\delta\in C^0(M,\bb{R})$ be such that $\delta>0$. Then there exists $g\in C^\infty(M,\bb{R}^N)$ such that $\lVert f(x)-g(x)\rVert<\delta(x)$ for all $x\in M$ and $g=f$ on $E$. 
\end{lem}

\begin{proof}
This is Theorem 6.21 of \cite{Lee}. 
\end{proof}

\begin{lem}\label{nseonc}
Let $M$ be a $C^\infty$ manifold possibly with boundary, $M^\prime$ be a $C^\infty$ manifold without boundary, $E$ be a closed subset of $M$ and $f\in C^0(M,M^\prime)$. Assume that $f$ is $C^\infty$ on an open neighborhood of $E$ in $M$. If $U$ is an open neighborhood of $f$ in $C^0(M,M^\prime)$ with respect to the strong $C^0$ topology, then there exists $g\in U\cap C^\infty(M,M^\prime)$ such that $g=f$ on $E$. 
\end{lem}

\begin{proof}
There exists a $C^\infty$ embedding: $M^\prime\subset\bb{R}^N$. Let $V$ be a tubular neighborhood of $M^\prime$ in $\bb{R}^N$ and $\pi\colon V\to M^\prime$ be the projection. There exists $\delta\in C^0(M,\bb{R})$ such that: 
\begin{itemize}
\setlength\itemsep{0em}
\item $\delta>0$
\item if $h\in C^0(M,\bb{R}^N)$ satisfies $\lVert f(x)-h(x)\rVert<\delta(x)$ for all $x\in M$, then $h(M)\subset V$ and $\pi h\in U$. 
\end{itemize}
By Lemma \ref{mfgone}, there exists $g^\prime\in C^\infty(M,\bb{R}^N)$ such that $\lVert f(x)-g^\prime(x)\rVert<\delta(x)$ for all $x\in M$ and $g^\prime=f$ on $E$. Let $g=\pi g^\prime\in C^\infty(M,M^\prime)$. Then $g\in U$ and $g=f$ on $E$. 
\end{proof}

\begin{lem}\label{conorihod}
Let $M$, $M^\prime$ be $C^\infty$ manifolds without boundary and $f$, $f^\prime\colon M\to M^\prime$ be proper $C^\infty$ maps. If $f$ and $f^\prime$ are properly $C^0$ homotopic, then they are properly $C^\infty$ homotopic. 
\end{lem}

\begin{proof}
Let $P$ be the set of proper maps in $C^0(M\times I,M^\prime)$. Then $P$ is open in $C^0(M\times I,M^\prime)$ with respect to the strong $C^0$ topology. (See 1.5 Theorem of \cite{Hirsch}. Manifolds are assumed to be without boundary there, but the same proof works for manifolds with boundary.) Let $H\colon M\times I\to M^\prime$ be a proper $C^0$ homotopy between $f$ and $f^\prime$. Define $H^\prime\colon M\times I\to M^\prime$ by 
\begin{equation*}
H^\prime(x,t)=
\begin{cases}
H(x,0)&0\leq t\leq\frac{1}{3}\\
H(x,3t-1)&\frac{1}{3}\leq t\leq\frac{2}{3}\\
H(x,1)&\frac{2}{3}\leq t\leq1. 
\end{cases}
\end{equation*}
Then $H^\prime\in P$ and $H^\prime$ is $C^\infty$ on an open neighborhood of $M\times\left\{0,1\right\}$ in $M\times I$. By Lemma \ref{nseonc} there exists $H^{\prime\prime}\in P\cap C^\infty(M\times I,M^\prime)$ such that $H^{\prime\prime}(\cdot,t)=H^\prime(\cdot,t)$ for $t=0$, $1$. So $H^{\prime\prime}$ is a proper $C^\infty$ homotopy between $f$ and $f^\prime$. 
\end{proof}

\subsection{Theorem for $bnt$ and corollaries}
\begin{thm}\label{ctop}
Let $\mca{G}\rightrightarrows M$, $\mca{G}^\prime\rightrightarrows M^\prime$ be Lie groupoids, $N$, $N^\prime$ be $C^\infty$ manifolds, $\nu\colon N\to M$, $\nu^\prime\colon N^\prime\to M^\prime$ be $C^\infty$ maps, $\rho\in\mca{A}(\nu,\mca{G})$, $\rho^\prime\in\mca{A}(\nu^\prime,\mca{G}^\prime)$ and $\varphi$, $\varphi^\prime\in\Hom(N\rtimes_\rho\mca{G},N^\prime\rtimes_{\rho^\prime}\mca{G}^\prime)$. 
Assume $\varphi$ is an isomorphism and $\varphi^\prime=F_{\varphi^\prime}\rtimes\Phi$ is a semiconjugacy. If: 
\begin{itemize}
\setlength\itemsep{0em}
\item $\varphi\sim_{bnt}\varphi^\prime$ with respect to $\mca{B}_{N^\prime\rtimes_{\rho^\prime}\mca{G}^\prime}$
\item $F_\Phi\colon\nu(N)\to M^\prime$ is injective
\item $F_\Phi\colon\ul{O}\to M^\prime$ is an immersion for all $\ul{O}\in\mca{O}_\mca{G}$ such that $\ul{O}\subset\nu(N)$
\item $\nu^\prime$ is a surjective submersion, 
\end{itemize}
then $\varphi^\prime$ is a conjugacy. 
\end{thm}

The proof of the theorem will be given in Section \ref{prth176}. Here we look at corollaries of the theorem. 

\begin{cor}
Let $\mca{G}\rightrightarrows M$ be an $s$-simply connected Lie groupoid, $\mca{G}^\prime\rightrightarrows M^\prime$ be a foliation groupoid, $N$, $N^\prime$ be $C^\infty$ manifolds, $\nu\colon N\to M$ be a surjective $C^\infty$ map, $\nu^\prime\colon N^\prime\to M^\prime$ be a surjective submersion, $\rho\in\mca{A}(\nu,\mca{G})$, $\rho^\prime\in\mca{A}(\nu^\prime,\mca{G}^\prime)$ and $\varphi$, $\varphi^\prime\in\Hom(N\rtimes_\rho\mca{G},N^\prime\rtimes_{\rho^\prime}\mca{G}^\prime)$. Assume that: 
\begin{itemize}
\setlength\itemsep{0em}
\item $\varphi\sim_{bnt}\varphi^\prime$ with respect to $\mca{B}_{N^\prime\rtimes_{\rho^\prime}\mca{G}^\prime}$
\item $\varphi$ is an isomorphism
\item there exists a $C^\infty$ map $\beta\colon M\to M^\prime$ such that: 
\begin{itemize}
\setlength\itemsep{0em}
\item \hfill$\begin{aligned}[t]
\begin{tikzcd}
N\ar[r,"F_{\varphi^\prime}"]\ar[d,"\nu"']\ar[rd,phantom,"\circlearrowright"]&N^\prime\ar[d,"\nu^\prime"]\\
M\ar[r,"\beta"']&M^\prime
\end{tikzcd}
\end{aligned}$\hfill\null
\item $\beta\colon M\to M^\prime$ is injective
\item $\beta\colon\ul{O}\to M^\prime$ is an immersion for all $\ul{O}\in\mca{O}_\mca{G}$. 
\end{itemize}
\end{itemize}
Then $\varphi^\prime$ is an isomorphism and $\beta$ is a diffeomorphism. 
\end{cor}

\begin{proof}
First we prove $\beta\in C^\infty(\mca{O}_\mca{G},\mca{O}_{\mca{G}^\prime})$. Let $\ul{O}\in\mca{O}_\mca{G}$ and fix $x_0\in\ul{O}$. Since $\nu\colon N\to M$ is surjective, there exists $y_0\in N$ such that $\nu(y_0)=x_0$. Let $y_0\in O\in\mca{O}_\rho$. There exist a unique $O^\prime\in\mca{O}_{\rho^\prime}$ such that $F_{\varphi^\prime}(O)\subset O^\prime$, and a unique $\ul{O^\prime}\in\mca{O}_{\mca{G}^\prime}$ such that $\nu^\prime(O^\prime)\subset\ul{O^\prime}$. Since $\nu\colon O\to\ul{O}$ is surjective by Lemma \ref{onuonud}, it has a section, hence $\beta(\ul{O})\subset\ul{O^\prime}$. Now we have 
\begin{equation*}
\begin{tikzcd}
O\ar[r,"F_{\varphi^\prime}"]\ar[d,"\nu"']\ar[rd,phantom,"\circlearrowright"]&O^\prime\ar[d,"\nu^\prime"]\\
\ul{O}\ar[r,"\beta"']&\ul{O^\prime}. 
\end{tikzcd}
\end{equation*}
Since $\nu\colon O\to\ul{O}$ is a surjective submersion by Lemma \ref{onuonud}, it has a local $C^\infty$ section at any point of $\ul{O}$, so $\beta\colon\ul{O}\to\ul{O^\prime}$ is $C^\infty$. Hence $\beta\in C^\infty(\mca{O}_\mca{G},\mca{O}_{\mca{G}^\prime})$. 

There exists a unique $\Phi\in\Hom(\mca{G},\mca{G}^\prime)$ such that $F_\Phi=\beta$ by Proposition \ref{simconfolbij}. We have $\varphi_{\rho^\prime}\varphi^\prime=\Phi\varphi_\rho$ since the base maps coincide, $N\rtimes_\rho\mca{G}$ is $s$-connected and $\mca{G}^\prime$ is a foliation groupoid. So $\varphi^\prime$ is a semiconjugacy and apply Theorem \ref{ctop}. 
\end{proof}

\begin{cor}
Let $\mca{G}\rightrightarrows M$, $\mca{G}^\prime\rightrightarrows M^\prime$ be Lie groupoids and $\varphi$, $\varphi^\prime\in\Hom(\mca{G},\mca{G}^\prime)$. If: 
\begin{itemize}
\setlength\itemsep{0em}
\item $\varphi$ is an isomorphism
\item $\varphi\sim_{bnt}\varphi^\prime$ with respect to $\mca{B}_{\mca{G}^\prime}$
\item $F_{\varphi^\prime}\colon M\to M^\prime$ is injective
\item $F_{\varphi^\prime}\colon O\to M^\prime$ is an immersion for all $O\in\mca{O}_\mca{G}$, 
\end{itemize}
then $\varphi^\prime$ is an isomorphism. 
\end{cor}

\begin{proof}
We have $\mca{A}(\id,\mca{G})=\left\{\tau\right\}$, $\mca{A}(\id,\mca{G}^\prime)=\left\{\tau^\prime\right\}$ and $M\rtimes_\tau\mca{G}\simeq\mca{G}$, $M^\prime\rtimes_{\tau^\prime}\mca{G}^\prime\simeq\mca{G}^\prime$. Apply Theorem \ref{ctop} to $\tau$ and $\tau^\prime$. 
\end{proof}

\begin{cor}
Let $\mca{G}\rightrightarrows M$, $\mca{G}^\prime\rightrightarrows M^\prime$ be Lie groupoids, $N$, $Q$, $N^\prime$, $Q^\prime$ be $C^\infty$ manifolds, $\nu\colon N\to Q$, $\mu\colon Q\to M$, $\nu^\prime\colon N^\prime\to Q^\prime$, $\mu^\prime\colon Q^\prime\to M^\prime$ be $C^\infty$ maps and $\rho\in\mca{A}(\mu\nu,\mca{G})$, $\ul{\rho}\in\mca{A}(\mu,\mca{G})$, $\rho^\prime\in\mca{A}(\mu^\prime\nu^\prime,\mca{G}^\prime)$, $\ul{\rho^\prime}\in\mca{A}(\mu^\prime,\mca{G}^\prime)$, where $\nu\colon N\to Q$ is a $\mca{G}$-equivariant surjective map and $\nu^\prime\colon N^\prime\to Q^\prime$ is a $\mca{G}^\prime$-equivariant surjective submersion. Let $\varphi$, $\varphi^\prime\in\Hom(N\rtimes_\rho\mca{G},N^\prime\rtimes_{\rho^\prime}\mca{G}^\prime)$ and assume that there exists $\Phi\in\Hom(Q\rtimes_\ul{\rho}\mca{G},Q^\prime\rtimes_\ul{\rho^\prime}\mca{G}^\prime)$ such that 
\begin{equation*}
\begin{tikzcd}
N\rtimes_\rho\mca{G}\ar[r,"\varphi^\prime"]\ar[d,"\nu\rtimes\id"']\ar[rd,phantom,"\circlearrowright"]&N^\prime\rtimes_{\rho^\prime}\mca{G}^\prime\ar[d,"\nu^\prime\rtimes\id"]\\
Q\rtimes_\ul{\rho}\mca{G}\ar[r,"\Phi"']&Q^\prime\rtimes_\ul{\rho^\prime}\mca{G}^\prime. 
\end{tikzcd}
\end{equation*}
(Note that $F_{\varphi^\prime}$ takes each fiber of $\nu$ to a fiber of $\nu^\prime$ and $\varphi_{\rho^\prime}\varphi^\prime(y,g)$ is constant along the fibers of $\nu$.) If: 
\begin{itemize}
\setlength\itemsep{0em}
\item $\varphi$ is an isomorphism
\item $\varphi\sim_{bnt}\varphi^\prime$ with respect to $\mca{B}_{N^\prime\rtimes_{\rho^\prime}\mca{G}^\prime}$
\item $F_\Phi\colon Q\to Q^\prime$ is injective
\item $F_\Phi\colon\ul{O}\to Q^\prime$ is an immersion for each $\ul{O}\in\mca{O}_\ul{\rho}$, 
\end{itemize}
then $\varphi^\prime$ and $\Phi$ are isomorphisms. 
\end{cor}

\begin{proof}
Let $\sigma\in\mca{A}(\nu,Q\rtimes_\ul{\rho}\mca{G})$, $\sigma^\prime\in\mca{A}(\nu^\prime,Q^\prime\rtimes_\ul{\rho^\prime}\mca{G}^\prime)$ be the actions corresponding to $\rho$, $\rho^\prime$ by Proposition \ref{nqmaphic}. Then we have a commutative diagram 
\begin{equation*}
\begin{tikzcd}
N\rtimes_\sigma(Q\rtimes_\ul{\rho}\mca{G})\ar[r,"\sim"]\ar[rd,"\varphi_\sigma"']&N\rtimes_\rho\mca{G}\ar[r,"\varphi^\prime"]\ar[d,"\nu\rtimes\id"']&N^\prime\rtimes_{\rho^\prime}\mca{G}^\prime\ar[d,"\nu^\prime\rtimes\id"]\ar[r,"\sim"]&N^\prime\rtimes_{\sigma^\prime}(Q^\prime\rtimes_\ul{\rho^\prime}\mca{G}^\prime)\ar[ld,"\varphi_{\sigma^\prime}"]\\
&Q\rtimes_\ul{\rho}\mca{G}\ar[r,"\Phi"']&Q^\prime\rtimes_\ul{\rho^\prime}\mca{G}^\prime. 
\end{tikzcd}
\end{equation*}
Since the isomorphism $N^\prime\rtimes_{\rho^\prime}\mca{G}^\prime\simeq N^\prime\rtimes_{\sigma^\prime}(Q^\prime\rtimes_\ul{\rho^\prime}\mca{G}^\prime)$ takes $\mca{B}_{N^\prime\rtimes_{\rho^\prime}\mca{G}^\prime}$ to $\mca{B}_{N^\prime\rtimes_{\sigma^\prime}(Q^\prime\rtimes_\ul{\rho^\prime}\mca{G}^\prime)}$, we can apply Theorem \ref{ctop}. 
\end{proof}

\begin{cor}
Let $G$, $G^\prime$ be Lie groups, $N$, $Q$, $N^\prime$, $Q^\prime$ be $C^\infty$ manifolds, $\rho\in\mca{A}(N,G)$, $\ul{\rho}\in\mca{A}(Q,G)$, $\rho^\prime\in\mca{A}(N^\prime,G^\prime)$, $\ul{\rho^\prime}\in\mca{A}(Q^\prime,G^\prime)$, $\nu\colon N\to Q$ be a $G$-equivariant surjective $C^\infty$ map, $\nu^\prime\colon N^\prime\to Q^\prime$ be a $G^\prime$-equivariant surjective submersion and $\varphi$, $\varphi^\prime\in\Hom(N\rtimes_\rho G,N^\prime\rtimes_{\rho^\prime}G^\prime)$. Assume that $\varphi$ is an isomorphism and there exists $\Phi\in\Hom(Q\rtimes_\ul{\rho}G,Q^\prime\rtimes_\ul{\rho^\prime}G^\prime)$ such that 
\begin{equation*}
\begin{tikzcd}
N\rtimes_\rho G\ar[r,"\varphi^\prime"]\ar[d,"\nu\rtimes\id"']\ar[rd,phantom,"\circlearrowright"]&N^\prime\rtimes_{\rho^\prime}G^\prime\ar[d,"\nu^\prime\rtimes\id"]\\
Q\rtimes_\ul{\rho}G\ar[r,"\Phi"']&Q^\prime\rtimes_\ul{\rho^\prime}G^\prime. 
\end{tikzcd}
\end{equation*}
(Note that $F_{\varphi^\prime}$ takes the fibers of $\nu$ to the fibers of $\nu^\prime$ and $\varphi_{\rho^\prime}\varphi^\prime(y,g)$ is constant along the fibers of $\nu$.) If: 
\begin{itemize}
\setlength\itemsep{0em}
\item $\varphi\sim_{bnt}\varphi^\prime$ with respect to $\mca{B}_{N^\prime\rtimes_{\rho^\prime}G^\prime}$
\item $F_\Phi\colon Q\to Q^\prime$ is injective
\item $F_\Phi\colon\ul{O}\to Q^\prime$ is an immersion for all $\ul{O}\in\mca{O}_\ul{\rho}$, 
\end{itemize}
then $\varphi^\prime$ and $\Phi$ are isomorphisms. 
\end{cor}

\begin{cor}
Let $\mca{G}\rightrightarrows M$, $\mca{G}^\prime\rightrightarrows M^\prime$, $\mca{G}^{\prime\prime}\rightrightarrows M^{\prime\prime}$ be Lie groupoids, $N$, $N^\prime$, $N^{\prime\prime}$ be $C^\infty$ manifolds, $\nu\colon N\to M$, $\nu^\prime\colon N^\prime\to M^\prime$, $\nu^{\prime\prime}\colon N^{\prime\prime}\to M^{\prime\prime}$ be $C^\infty$ maps and $\rho\in\mca{A}(\nu,\mca{G})$, $\rho^\prime\in\mca{A}(\nu^\prime,\mca{G}^\prime)$, $\rho^{\prime\prime}\in\mca{A}(\nu^{\prime\prime},\mca{G}^{\prime\prime})$. Let 
\begin{equation*}
\begin{tikzcd}[row sep=tiny]
&N^\prime\rtimes_{\rho^\prime}\mca{G}^\prime\ar[dd,"\psi"]\\
N\rtimes_\rho\mca{G}\ar[ru,"\varphi"]\ar[rd,"\varphi^\prime"']\\
&N^{\prime\prime}\rtimes_{\rho^{\prime\prime}}\mca{G}^{\prime\prime}. 
\end{tikzcd}
\end{equation*}
If: 
\begin{itemize}
\setlength\itemsep{0em}
\item $\varphi$, $\varphi^\prime$ are isomorphisms, $\psi=F_\psi\rtimes\Psi$ is a semiconjugacy
\item $\psi\varphi\sim_{bnt}\varphi^\prime$ with respect to $\mca{B}_{N^{\prime\prime}\rtimes_{\rho^{\prime\prime}}\mca{G}^{\prime\prime}}$
\item $F_\Psi\colon\nu^\prime(N^\prime)\to M^{\prime\prime}$ is injective
\item $F_\Psi\colon\ul{O^\prime}\to M^{\prime\prime}$ is an immersion for any $\ul{O^\prime}\in\mca{O}_{\mca{G}^\prime}$ such that $\ul{O^\prime}\subset\nu^\prime(N^\prime)$
\item $\nu^{\prime\prime}$ is a surjective submersion, 
\end{itemize}
then $\psi$ is a conjugacy. 
\end{cor}

\begin{proof}
We have $\varphi^\prime\varphi^{-1}\sim_{bnt}\psi$ by Lemma \ref{simsim2}. Apply Theorem \ref{ctop} to see that $\psi$ is a conjugacy. 
\end{proof}

\subsection{Proof of Theorem \ref{ctop}}\label{prth176}
Put $\mca{H}=N\rtimes_\rho\mca{G}$, $\mca{H}^\prime=N^\prime\rtimes_{\rho^\prime}\mca{G}^\prime$ and $F=F_\varphi$, $F^\prime=F_{\varphi^\prime}\colon N\to N^\prime$. Then $F$ is a diffeomorphism. There exists a continuous map $P\colon N\times I\to\mca{H}^\prime$ such that: 
\begin{itemize}
\setlength\itemsep{0em}
\item $P(\cdot,t)\in C^\infty(N,\mca{H}^\prime)$ for all $t\in I$
\item $t^\prime P(\cdot,t)=F$ for all $t\in I$
\item $P(\cdot,0)=1_F$
\item $\varphi P(\cdot,1)=\varphi^\prime$
\item $B:=P(N\times I)\in\mca{B}_{\mca{H}^\prime}$. 
\end{itemize}

\begin{claim}
The map 
\begin{equation*}
H:=s^\prime P\colon N\times I\to N^\prime
\end{equation*}
is a proper continuous map. 
\end{claim}

\begin{proof}
The map $H$ is continuous. Let $K\subset N^\prime$ be any compact subset. Then there exists a compact subset $K^\prime\subset\mca{H}^\prime$ such that $B\cap(s^\prime)^{-1}(K)\subset K^\prime$ (see Proposition \ref{loccombor}). 

\begin{claim}
$H^{-1}(K)\subset F^{-1}( t^\prime(K^\prime))\times I$. 
\end{claim}

\begin{proof}
Let $(y,t)\in H^{-1}(K)$, ie $s^\prime P(y,t)\in K$. Then $P(y,t)\in B\cap(s^\prime)^{-1}(K)\subset K^\prime$. Hence $F(y)\in t^\prime(K^\prime)$, so $y\in F^{-1}( t^\prime(K^\prime))$. 
\end{proof}

Since $F^{-1}( t^\prime(K^\prime))\times I$ is compact, $H^{-1}(K)$ is compact. 
\end{proof}

We have $H(\cdot,0)=F$, which is a diffeomorphism, and $H(\cdot,1)=F^\prime$. 

\begin{claim}
$F^\prime\colon N\to N^\prime$ is surjective. 
\end{claim}

\begin{proof}
Let $N_0$ be a connected component of $N$. Then $N^\prime_0=H(N_0,0)$ is a connected component of $N^\prime$. The map $H\colon N_0\times I\to N^\prime_0$ is proper. $N_0$, $N^\prime_0$ are diffeomorphic connected $C^\infty$ manifolds. Assume that $N_0$ is not orientable. Let $\wt{N_0}\to N_0$, $\wt{N^\prime_0}\to N^\prime_0$ be the orientation coverings. Since $H(\cdot,0)$ is a diffeomorphism, it induces a diffeomorphism $\wt{H(\cdot,0)}\colon\wt{N_0}\to\wt{N^\prime_0}$ such that 
\begin{equation*}
\begin{tikzcd}
\wt{N_0}\ar[r,"\wt{H(\cdot,0)}"]\ar[d]\ar[rd,phantom,"\circlearrowright"]&\wt{N^\prime_0}\ar[d]\\
N_0\ar[r,"H(\cdot{,}0)"']&N^\prime_0
\end{tikzcd}
\end{equation*}
by taking differentials (The total spaces of orientation coverings consist of the orientations of the tangent space at each point). There exists a unique continuous map $\wt{H}\colon\wt{N_0}\times I\to\wt{N^\prime_0}$ such that $\wt{H}(\cdot,0)=\wt{H(\cdot,0)}$ and 
\begin{equation}\label{nininn}
\begin{tikzcd}
\wt{N_0}\times I\ar[r,"\wt{H}"]\ar[d]\ar[rd,phantom,"\circlearrowright"]&\wt{N^\prime_0}\ar[d]\\
N_0\times I\ar[r,"H"']&N^\prime_0
\end{tikzcd}
\end{equation}
since $\wt{N^\prime_0}\to N^\prime_0$ is a covering space. Then $\wt{H}$ is proper since $H$ and the left vertical map in \eqref{nininn} are proper. Now $\wt{N_0}$, $\wt{N^\prime_0}$ are connected oriented $C^\infty$ manifolds and $\wt{H}(\cdot,0)$, $\wt{H}(\cdot,1)\colon\wt{N_0}\to\wt{N^\prime_0}$ are properly $C^0$ homotopic, where $\wt{H}(\cdot,0)$ is a diffeomorphism. Hence the degree of $\wt{H}(\cdot,1)$ is $\pm1$ by Proposition \ref{degdiffp} and Lemma \ref{conorihod}. So $\wt{H}(\cdot,1)$ is surjective, otherwise the degree is $0$ by Proposition \ref{degdiffp}. Thus $H(\cdot,1)\colon N_0\to N^\prime_0$ is surjective since we have 
\begin{equation*}
\begin{tikzcd}
\wt{N_0}\ar[r,"\wt{H(\cdot,1)}"]\ar[d]\ar[rd,phantom,"\circlearrowright"]&\wt{N^\prime_0}\ar[d]\\
N_0\ar[r,"H(\cdot{,}1)"']&N^\prime_0. 
\end{tikzcd}
\end{equation*}
If $N_0$ is orientable, we can skip most of the above argument. 

Therefore $F^\prime=H(\cdot,1)\colon N\to N^\prime$ is surjective. 
\end{proof}

(It follows that $F_\Phi\colon M\to M^\prime$ is surjective, but we will not use this.) 

\begin{claim}
For any $y\in N$, 
\begin{align*}
\eta\colon\mca{H}^y\times I&\to(\mca{H}^\prime)^{F^\prime(y)}\\
((y,g),t)&\mapsto P(y,1)^{-1}\varphi(y,g)P(\rho(y,g),t)
\end{align*}
is a proper continuous map. 
\end{claim}

\begin{proof}
Let $K$ be a compact subset of $(\mca{H}^\prime)^{F^\prime(y)}$. Recall $B=P(N\times I)\in\mca{B}_{\mca{H}^\prime}$. Since $B^{-1}\in\mca{B}_{\mca{H}^\prime}$, there exists a compact subset $K^\prime$ of $\mca{H}^\prime$ such that $KB^{-1}\subset K^\prime$ (see Proposition \ref{loccombor}). Put $K^{\prime\prime}=P(y,1)K^\prime\subset(\mca{H}^\prime)^{F(y)}$. It is compact by the first Claim in the proof of Proposition \ref{loccombor}. So $\varphi^{-1}(K^{\prime\prime})\subset\mca{H}^y$ is compact as well. 

\begin{claim}
$\eta^{-1}(K)\subset\varphi^{-1}(K^{\prime\prime})\times I$. 
\end{claim}

\begin{proof}
Let $((y,g),t)\in\eta^{-1}(K)$, ie $P(y,1)^{-1}\varphi(y,g)P(\rho(y,g),t)\in K$. Then 
\begin{equation*}
\varphi(y,g)\in P(y,1)KB^{-1}\subset P(y,1)K^\prime=K^{\prime\prime}. 
\end{equation*}
Hence $(y,g)\in\varphi^{-1}(K^{\prime\prime})$. 
\end{proof}

The compact subset $K$ is closed in $(\mca{H}^\prime)^{F^\prime(y)}$ since $(\mca{H}^\prime)^{F^\prime(y)}$ is Hausdorff. Hence $\eta^{-1}(K)$ is closed in $\mca{H}^y\times I$. Since $\varphi^{-1}(K^{\prime\prime})\times I$ is compact, $\eta^{-1}(K)$ is compact. 
\end{proof}

For any $y\in N$, $\eta(\cdot,0)=P(y,1)^{-1}\varphi\colon\mca{H}^y\to(\mca{H}^\prime)^{F^\prime(y)}$ is a diffeomorphism and $\eta(\cdot,1)=\varphi^\prime\colon\mca{H}^y\to(\mca{H}^\prime)^{F^\prime(y)}$. 

\begin{claim}
$\varphi^\prime\colon\mca{H}\to\mca{H}^\prime$ is surjective. 
\end{claim}

\begin{proof}
Since $F^\prime\colon N\to N^\prime$ is surjective, it suffices to prove that $\varphi^\prime\colon\mca{H}^y\to(\mca{H}^\prime)^{F^\prime(y)}$ is surjective for all $y\in N$. 

Let $y\in N$ and put $y^\prime=F^\prime(y)\in N^\prime$. Let $\mca{H}^y_0$ be a connected component of $\mca{H}^y$. Then $(\mca{H}^\prime)^{y^\prime}_0:=\eta(\mca{H}^y_0,0)$ is a connected component of $(\mca{H}^\prime)^{y^\prime}$. Since $\mca{H}^y$ is paracompact by Lemma \ref{tarpara} and $\mca{H}^y_0$ is closed in $\mca{H}^y$, $\mca{H}^y_0$ is paracompact. So $\mca{H}^y_0$, $(\mca{H}^\prime)^{y^\prime}_0$ are connected Hausdorff paracompact $C^\infty$ manifolds and $\eta\colon\mca{H}^y_0\times I\to(\mca{H}^\prime)^{y^\prime}_0$ is a proper $C^0$ homotopy, where $\eta(\cdot,0)$ is a diffeomorphism. The same argument as before shows $\eta(\cdot,1)\colon\mca{H}^y_0\to(\mca{H}^\prime)^{y^\prime}_0$ is surjective. Hence $\varphi^\prime=\eta(\cdot,1)\colon\mca{H}^y\to(\mca{H}^\prime)^{y^\prime}$ is surjective. 
\end{proof}

(It follows that $\Phi\colon\mca{G}\to\mca{G}^\prime$ is surjective, but we will not use this.) 

\begin{rem}
Since $\mca{H}$, $\mca{H}^\prime$ may not be second countable or Hausdorff, it is not clear whether a similar argument can be applied to 
\begin{align*}
\mca{H}\times I&\to\mca{H}^\prime\\
((y,g),t)&\mapsto P(y,t)^{-1}\varphi(y,g)P(\rho(y,g),t)
\end{align*}
to show directly that $\varphi^\prime$ is surjective. 
\end{rem}

Let $O\in\mca{O}_\rho$ and put $O^\prime=F(O)\in\mca{O}_{\rho^\prime}$. Then $F\colon O\to O^\prime$ is a diffeomorphism by Lemma \ref{oosmhs}. 
The map 
\begin{align*}
\mca{O}_\rho&\to\mca{O}_{\rho^\prime}\\
O&\mapsto F(O)
\end{align*}
is bijective. It follows that $F^\prime\colon O\to O^\prime$ is surjective since $F^\prime\colon N\to N^\prime$ is surjective. Since $\nu\colon N\to M$, $\nu^\prime\colon N^\prime\to M^\prime$ are the base maps of $\varphi_\rho\colon\mca{H}\to\mca{G}$, $\varphi_{\rho^\prime}\colon\mca{H}^\prime\to\mca{G}^\prime$, there exist unique $\ul{O}\in\mca{O}_\mca{G}$, $\ul{O^\prime}\in\mca{O}_{\mca{G}^\prime}$ such that $\nu(O)\subset\ul{O}$, $\nu^\prime(O^\prime)\subset\ul{O^\prime}$. By 
\begin{equation*}
\begin{tikzcd}
N\ar[r,"F^\prime"]\ar[d,"\nu"']\ar[rd,phantom,"\circlearrowright"]&N^\prime\ar[d,"\nu^\prime"]\\
M\ar[r,"F_\Phi"']&M^\prime, 
\end{tikzcd}
\end{equation*}
we have $F_\Phi(\ul{O})\subset\ul{O^\prime}$ and 
\begin{equation*}
\begin{tikzcd}
O\ar[r,"F^\prime"]\ar[d,"\nu"']\ar[rd,phantom,"\circlearrowright"]&O^\prime\ar[d,"\nu^\prime"]\\
\ul{O}\ar[r,"F_\Phi"']&\ul{O^\prime}. 
\end{tikzcd}
\end{equation*}
By Lemma \ref{oosmhs}, $F_\Phi\colon\ul{O}\to\ul{O^\prime}$ is $C^\infty$. By Lemma \ref{onuonud}, $\nu\colon O\to\ul{O}$, $\nu^\prime\colon O^\prime\to\ul{O^\prime}$ are surjective submersions, and $O\cap\nu^{-1}(x)$, $O^\prime\cap(\nu^\prime)^{-1}(x^\prime)$ for any $x\in\ul{O}$, $x^\prime\in\ul{O^\prime}$ are embedded submanifolds of $O$, $O^\prime$. For each $x\in\ul{O}$, we obtain a $C^\infty$ map $F^\prime\colon O\cap\nu^{-1}(x)\to O^\prime\cap(\nu^\prime)^{-1}(F_\Phi(x))$. 

\begin{claim}
There exists a regular point $y_0\in O$ of $F^\prime\colon O\to O^\prime$. 
\end{claim}

\begin{proof}
Let $O_0$ be a connected component of $O$. Then $O^\prime_0=F(O_0)$ is a connected component of $O^\prime$. $F^\prime\colon O_0\to O^\prime_0$ is defined, and is surjective since $F^\prime\colon O\to O^\prime$ is surjective. Since $N$ is assumed to be second countable, $O_0$ is second countable by the same argument using a Riemannian metric as in the proof of Claim in the proof of Lemma \ref{tarpara}. Assume that there are no regular points of $F^\prime\colon O_0\to O^\prime_0$. Then $F^\prime(O_0)=O^\prime_0$ is the set of critical values of $F^\prime\colon O_0\to O^\prime_0$. Since $O_0$ is second countable, we can apply Sard's theorem to the map $F^\prime\colon O_0\to O^\prime_0$. Hence $F^\prime(O_0)=O^\prime_0$ is of measure zero in $O^\prime_0$. This implies $\dim O^\prime_0=0$. Then any point of $O_0$ must be a regular point of $F^\prime\colon O_0\to O^\prime_0$, which is a contradiction. So there exists a regular point $y_0\in O_0$ of $F^\prime\colon O_0\to O^\prime_0$. This is a regular point of $F^\prime\colon O\to O^\prime$. 
\end{proof}

Let $y_0\in O$ be a regular point of $F^\prime\colon O\to O^\prime$. Put $y^\prime_0=F^\prime(y_0)\in O^\prime$, $x_0=\nu(y_0)\in\ul{O}$ and $x^\prime_0=F_\Phi(x_0)\in\ul{O^\prime}$. Consider 
\begin{equation}\label{ttooto}
\begin{tikzcd}
0\ar[d]&0\ar[d]\\
T_{y_0}(O\cap\nu^{-1}(x_0))\ar[r,"F^\prime_*"]\ar[d]\ar[rd,phantom,"\circlearrowright"{xshift=-7}]&T_{y^\prime_0}(O^\prime\cap(\nu^\prime)^{-1}(x^\prime_0))\ar[d]\\
T_{y_0}O\ar[r,"\sim","F^\prime_*"']\ar[d,"\nu_*"']\ar[rd,phantom,"\circlearrowright"{yshift=-3}]&T_{y^\prime_0}O^\prime\ar[d,"\nu^\prime_*"]\\
T_{x_0}\ul{O}\ar[r,"(F_\Phi)_*"']\ar[d]&T_{x^\prime_0}\ul{O^\prime}\ar[d]\\
0&0. 
\end{tikzcd}
\end{equation}
It follows that $(F_\Phi)_*\colon T_{x_0}\ul{O}\to T_{x^\prime_0}\ul{O^\prime}$ is surjective. Since $\ul{O}\subset\nu(N)$, $(F_\Phi)_*\colon T_{x_0}\ul{O}\to T_{x^\prime_0}M^\prime$ is injective by assumption. So $(F_\Phi)_*\colon T_{x_0}\ul{O}\to T_{x^\prime_0}\ul{O^\prime}$ is an isomorphism and $\dim\ul{O}=\dim\ul{O^\prime}$. By \eqref{ttooto}, $F^\prime_*\colon T_{y_0}(O\cap\nu^{-1}(x_0))\to T_{y^\prime_0}(O^\prime\cap(\nu^\prime)^{-1}(x^\prime_0))$ is also an isomorphism. 

Let $y\in O$ and put $y^\prime=F^\prime(y)\in O^\prime$, $x=\nu(y)\in\ul{O}$ and $x^\prime=F_\Phi(x)\in\ul{O^\prime}$. Then $(F_\Phi)_*\colon T_x\ul{O}\to T_{x^\prime}\ul{O^\prime}$ is an isomorphism since $F_\Phi\colon\ul{O}\to M^\prime$ is an immersion by assumption. We have 
\begin{equation}\label{tooprito}
\begin{tikzcd}
0\ar[d]&0\ar[d]\\
T_y(O\cap\nu^{-1}(x))\ar[r,"F^\prime_*"]\ar[d]\ar[rd,phantom,"\circlearrowright"{xshift=-7}]&T_{y^\prime}(O^\prime\cap(\nu^\prime)^{-1}(x^\prime))\ar[d]\\
T_yO\ar[r,"F^\prime_*"']\ar[d,"\nu_*"']\ar[rd,phantom,"\circlearrowright"{yshift=-3}]&T_{y^\prime}O^\prime\ar[d,"\nu^\prime_*"]\\
T_x\ul{O}\ar[r,"\sim","(F_\Phi)_*"']\ar[d]&T_{x^\prime}\ul{O^\prime}\ar[d]\\
0&0. 
\end{tikzcd}
\end{equation}

\begin{claim}
$F^\prime_*\colon T_y(O\cap\nu^{-1}(x))\to T_{y^\prime}(O^\prime\cap(\nu^\prime)^{-1}(x^\prime))$ is an isomorphism. 
\end{claim}

\begin{proof}
Since $y_0$, $y\in O$, there exists $g\in\mca{G}$ such that $\rho(y_0,g)=y$. Then $g\in\mca{G}^{x_0}_x$. By Lemma \ref{onuonud}, $\rho(\cdot,g)\colon O\cap\nu^{-1}(x_0)\to O\cap\nu^{-1}(x)$ is a diffeomorphism. Since $\varphi^\prime$ is a semiconjugacy, we obtain $\rho^\prime(y^\prime_0,\Phi(g))=y^\prime$. We have $\Phi(g)\in(\mca{G}^\prime)^{x^\prime_0}_{x^\prime}$ and $\rho^\prime(\cdot,\Phi(g))\colon O^\prime\cap(\nu^\prime)^{-1}(x^\prime_0)\to O^\prime\cap(\nu^\prime)^{-1}(x^\prime)$ is a diffeomorphism by Lemma \ref{onuonud}. Since 
\begin{equation*}
\begin{tikzcd}
O\cap\nu^{-1}(x_0)\ar[r,"F^\prime"]\ar[d,"\sim"{sloped},"\rho(\cdot{,}g)"']\ar[rd,phantom,"\circlearrowright"]&O^\prime\cap(\nu^\prime)^{-1}(x^\prime_0)\ar[d,"\sim"{sloped},"\rho^\prime(\cdot{,}\Phi(g))"{xshift=7}]\\
O\cap\nu^{-1}(x)\ar[r,"F^\prime"']&O^\prime\cap(\nu^\prime)^{-1}(x^\prime), 
\end{tikzcd}
\end{equation*}
we get 
\begin{equation*}
\begin{tikzcd}
T_{y_0}(O\cap\nu^{-1}(x_0))\ar[r,"\sim","F^\prime_*"{yshift=7}]\ar[d,"\sim"{sloped},"\rho(\cdot{,}g)_*"']\ar[rd,phantom,"\circlearrowright"]&T_{y^\prime_0}(O^\prime\cap(\nu^\prime)^{-1}(x^\prime_0))\ar[d,"\sim"{sloped},"\rho^\prime(\cdot{,}\Phi(g))_*"{xshift=7}]\\
T_y(O\cap\nu^{-1}(x))\ar[r,"F^\prime_*"']&T_{y^\prime}(O^\prime\cap(\nu^\prime)^{-1}(x^\prime)). 
\end{tikzcd}
\end{equation*}
Hence $F^\prime_*\colon T_y(O\cap\nu^{-1}(x))\to T_{y^\prime}(O^\prime\cap(\nu^\prime)^{-1}(x^\prime))$ is an isomorphism. 
\end{proof}

Thus $F^\prime_*\colon T_yO\to T_{y^\prime}O^\prime$ is an isomorphism by \eqref{tooprito}. 

\begin{claim}
$F^\prime_*\colon T_yN\to T_{y^\prime}N^\prime$ is an isomorphism. 
\end{claim}

\begin{proof}
This part is the same as in the proof of 2 in Proposition \ref{btop}. 
\end{proof}

Therefore $F^\prime\colon N\to N^\prime$ is a local diffeomorphism and proper. Let $N_0$ be a connected component of $N$. Then $N^\prime_0=F(N_0)$ is a connected component of $N^\prime$, and $F^\prime\colon N_0\to N^\prime_0$ is a local diffeomorphism and proper. 
%
It follows that $F^\prime\colon N_0\to N^\prime_0$ is a covering map by Proposition 4.46 of \cite{Lee}. For $y\in N_0$, $F^\prime_*\colon\pi_1(N_0,y)\to\pi_1(N^\prime_0,F^\prime(y))$ is an isomorphism since $F^\prime$ is $C^0$ homotopic to a diffeomorphism $F$. So $F^\prime\colon N_0\to N^\prime_0$ must be a diffeomorphism and $F^\prime\colon N\to N^\prime$ is a diffeomorphism. 

\begin{claim}
For any $y\in N$, $\varphi^\prime\colon\mca{H}^y_y\to(\mca{H}^\prime)^{F^\prime(y)}_{F^\prime(y)}$ is an isomorphism. 
\end{claim}

\begin{proof}
This follows from 
\begin{equation*}
\begin{tikzcd}[row sep=tiny]
&(\mca{H}^\prime)^{F(y)}_{F(y)}\ar[dd,"\sim"{sloped,name=U},"P(y{,}1)^{-1}\cdot P(y{,}1)"{xshift=7}]\\
\mca{H}^y_y\ar[ru,"\sim"{sloped},"\varphi"{yshift=7}]\ar[rd,"\varphi^\prime"']\ar[to=U,phantom,"\circlearrowright"{xshift=4}]\\
&(\mca{H}^\prime)^{F^\prime(y)}_{F^\prime(y)}. 
\end{tikzcd}
\end{equation*}
\end{proof}

\begin{claim}
$\varphi^\prime\colon\mca{H}\to\mca{H}^\prime$ is injective. 
\end{claim}

\begin{proof}
Let $h$, $h^\prime\in\mca{H}$ be such that $\varphi^\prime(h)=\varphi^\prime(h^\prime)$. By 
\begin{equation}\label{hhnnp}
\begin{tikzcd}
\mca{H}\ar[r,"(t{,}s)"]\ar[d,"\varphi^\prime"']\ar[rd,phantom,"\circlearrowright"{xshift=1,yshift=-2}]&N\times N\ar[d,"\sim"{sloped},"F^\prime\times F^\prime"{xshift=7}]\\
\mca{H}^\prime\ar[r,"(t^\prime{,}s^\prime)"']&N^\prime\times N^\prime, 
\end{tikzcd}
\end{equation}
we have $(t,s)(h)=(t,s)(h^\prime)$, hence $h^\prime h^{-1}\in\mca{H}^{t(h)}_{t(h)}$ is defined. Since $\varphi^\prime(h^\prime h^{-1})=1_{F^\prime t(h)}$, we get $h^\prime h^{-1}=1_{t(h)}$ by the previous claim. So $h=h^\prime$. 
\end{proof}

\begin{claim}
For any $h\in\mca{H}$, $\varphi^\prime_*\colon T_h\mca{H}\to T_{\varphi^\prime(h)}\mca{H}^\prime$ is an isomorphism. 
\end{claim}

\begin{proof}
Since $\dim\mca{H}=\dim\mca{H}^\prime$, it suffices to prove $\varphi^\prime_*$ is injective. Let $X\in T_h\mca{H}$ be such that $\varphi^\prime_*X=0$. By taking the differentials of \eqref{hhnnp}, we have 
\begin{equation*}
\begin{tikzcd}
T_h\mca{H}\ar[r,"(t{,}s)_*"]\ar[d,"\varphi^\prime_*"']\ar[rd,phantom,"\circlearrowright"]&T_{t(h)}N\times T_{s(h)}N\ar[d,"\sim"{sloped},"F^\prime_*\times F^\prime_*"{xshift=7}]\\
T_{\varphi^\prime(h)}\mca{H}^\prime\ar[r,"(t^\prime{,}s^\prime)_*"']&T_{t^\prime\varphi^\prime(h)}N^\prime\times T_{s^\prime\varphi^\prime(h)}N^\prime. 
\end{tikzcd}
\end{equation*}
It follows that $(t,s)_*X=0$. 
%
Since 
\begin{equation*}
\begin{tikzcd}
\mca{H}^{t(h)}_{s(h)}\ar[r,"\sim","\cdot h^{-1}"{yshift=7}]\ar[d,"\varphi^\prime"']\ar[rd,phantom,"\circlearrowright"{xshift=3,yshift=-1}]&\mca{H}^{t(h)}_{t(h)}\ar[d,"\sim"{sloped},"\varphi^\prime"{xshift=7}]\\
(\mca{H}^\prime)^{F^\prime t(h)}_{F^\prime s(h)}\ar[r,"\sim","\cdot\varphi^\prime(h)^{-1}"']&(\mca{H}^\prime)^{F^\prime t(h)}_{F^\prime t(h)}, 
\end{tikzcd}
\end{equation*}
$\varphi^\prime\colon\mca{H}^{t(h)}_{s(h)}\to(\mca{H}^\prime)^{F^\prime t(h)}_{F^\prime s(h)}$ is a diffeomorphism. Noting 
\begin{equation*}
\ker(t,s)_*=T_h\mca{H}^{t(h)}_{s(h)},\qquad\ker(t^\prime,s^\prime)_*=T_{\varphi^\prime(h)}(\mca{H}^\prime)^{F^\prime t(h)}_{F^\prime s(h)}
\end{equation*}
by Proposition \ref{isoorb}, we get an isomorphism $\varphi^\prime_*\colon\ker(t,s)_*\to\ker(t^\prime,s^\prime)_*$. 
Since $X\in\ker(t,s)_*$, we see that $X=0$. 
\end{proof}

Hence $\varphi^\prime\colon\mca{H}\to\mca{H}^\prime$ is a diffeomorphism by the inverse function theorem. By Proposition \ref{conjdiffinj}, $\varphi^\prime$ is a conjugacy.

\subsection{Examples}
\begin{example}
This example shows the assumption that $F_\Phi\colon\ul{O}\to M^\prime$ is an immersion in Theorem \ref{ctop} is necessary. Let $\Pair(\bb{R})\rightrightarrows\bb{R}$ be the pair groupoid of $\bb{R}$. Let 
\begin{gather*}
\mca{G}=\mca{G}^\prime=\Pair(\bb{R})=\bb{R}\times\bb{R}, \quad M=M^\prime=\bb{R}, \\
\nu=\nu^\prime=\id\colon\bb{R}\to\bb{R},\quad\mca{A}(\id,\Pair(\bb{R}))=\left\{\rho\right\}, \\
\bb{R}\rtimes_\rho\Pair(\bb{R})\overset{\varphi}{\underset{\varphi^\prime}{\rightrightarrows}}\bb{R}\rtimes_\rho\Pair(\bb{R}), \\
\varphi=\id,\quad\varphi^\prime(x,(x,y))=(x^3,(x^3,y^3)), \\
\Phi(x,y)=(x^3,y^3),\quad F_\Phi(x)=x^3. 
\end{gather*}
$\varphi$ is an isomorphism. We have 
\begin{equation*}
\begin{tikzcd}
\bb{R}\rtimes_\rho\Pair(\bb{R})\ar[r,"\varphi^\prime"]\ar[d,"\varphi_\rho"']\ar[rd,phantom,"\circlearrowright"]&\bb{R}\rtimes_\rho\Pair(\bb{R})\ar[d,"\varphi_\rho"]\\
\Pair(\bb{R})\ar[r,"\Phi"']&\Pair(\bb{R}). 
\end{tikzcd}
\end{equation*}
$\varphi^\prime$ is a semiconjugacy but is not a conjugacy. $F_\Phi$ is injective but is not an immersion. Let 
\begin{align*}
P\colon\bb{R}\times I&\to\Pair(\bb{R})\\
(x,t)&\mapsto(x,(1-t)x+tx^3). 
\end{align*}
Then $P(\bb{R}\times I)\in\mca{B}_{\Pair(\bb{R})}$ and $\varphi\sim_{bnt}\varphi^\prime$. 
\end{example}

\begin{example}
This example shows that we cannot drop the boundedness of $nt$-homotopy in Theorem \ref{ctop}, and that we need $\Phi$ being an isomorphism in Proposition \ref{btop}. Define $\rho\in\mca{A}(\bb{R},\bb{R})$ by $\rho(x,t)=x+t$. We have 
\begin{equation*}
\begin{tikzcd}
\bb{R}\rtimes_\rho\bb{R}\ar[r,"\id\times\id"]\ar[d,shift left]\ar[d,shift right]\ar[rd,phantom,"\circlearrowright"{xshift=-5}]&\bb{R}\rtimes_\rho\bb{R}\ar[d,shift left]\ar[d,shift right]\\
\bb{R}\ar[r,"\id"']&\bb{R}, 
\end{tikzcd}
\qquad
\begin{tikzcd}
\bb{R}\rtimes_\rho\bb{R}\ar[r,"0\times0"]\ar[d,shift left]\ar[d,shift right]\ar[rd,phantom,"\circlearrowright"{xshift=-5}]&\bb{R}\rtimes_\rho\bb{R}\ar[d,shift left]\ar[d,shift right]\\
\bb{R}\ar[r,"0"']&\bb{R}. 
\end{tikzcd}
\end{equation*}
$\id\times\id$ is an isomorphism. $0\times0$ is a semiconjugacy but not a conjugacy. 

\begin{claim}
The morphisms $\id\times\id$ and $0\times0$ are $nt$-homotopic. 
\end{claim}

\begin{proof}
For $s\in I$, define $P_s\in C^\infty(\bb{R},\bb{R})$ by $P_s(x)=-sx$, and $F_s\colon\bb{R}\to\bb{R}$ by $F_s(x)=\rho(x,P_s(x))=(1-s)x$. Let 
\begin{align*}
\varphi_s=(\id\times\id)(t,P_st)\colon\bb{R}\rtimes_\rho\bb{R}&\to\bb{R}\rtimes_\rho\bb{R}\\
(x,t)&\mapsto(F_s(x),\ol{\varphi_s}(x,t)), 
\end{align*}
where the notation $(\id\times\id)(t,P_st)$ is from Lemma \ref{hnhncnnb} and $t$ is the target map. Then 
\begin{equation*}
\ol{\varphi_s}(x,t)=-P_s(x)+t+P_s(x+t)=sx+t-s(x+t)=(1-s)t, 
\end{equation*}
so $\varphi_0=\id\times\id$ and $\varphi_1=0\times0$ are $nt$-homotopic. 
\end{proof}
\end{example}

\begin{example}
Let $G$, $G^\prime$ be Lie groups and $\rho\in\mca{A}(\mr{pt},G)$, $\rho^\prime\in\mca{A}(\mr{pt},G^\prime)$. Then $\Hom(\mr{pt}\rtimes_\rho G,\mr{pt}\rtimes_{\rho^\prime}G^\prime)\simeq\Hom(G,G^\prime)$. $\varphi$, $\varphi^\prime\in\Hom(G,G^\prime)$ are orbitwise homotopic if and only if they are homotopic through homomorphisms. So $\varphi$ being an isomorphism does not imply $\varphi^\prime$ being an isomorphism. This also shows that $\Phi$ being an isomorphism in 1 of Proposition \ref{btop} is necessary. $\varphi$, $\varphi^\prime\in\Hom(G,G^\prime)$ are $nt$-homotopic if and only if they are conjugate by an element of the connected component of $1$ in $G^\prime$. 
\end{example}

\begin{example}\label{notadiff}
We cannot prove that $F^\prime$ is a diffeomorphism in $1$ of Proposition \ref{btop} as the following example shows. Define $\rho$, $\rho^\prime\in\mca{A}(\bb{R},\bb{R})$ by $\rho(x,t)=xe^t$ and $\rho^\prime(x,t)=xe^{3t}$. Then $\mca{O}_\rho=\mca{O}_{\rho^\prime}=\{\bb{R}_{<0},\{0\},\bb{R}_{>0}\}$. We have an isomorphism 
\begin{equation*}
\begin{tikzcd}
\bb{R}\rtimes_\rho\bb{R}\ar[r,"\id\times\frac{1}{3}"]\ar[d,shift left]\ar[d,shift right]\ar[rd,phantom,"\circlearrowright"{xshift=-5}]&\bb{R}\rtimes_{\rho^\prime}\bb{R}\ar[d,shift left]\ar[d,shift right]\\
\bb{R}\ar[r,"\id"']&\bb{R}. 
\end{tikzcd}
\end{equation*}
(Actually it is a conjugacy.) Define $F^\prime\colon\bb{R}\to\bb{R}$ by $F^\prime(x)=x^3$. Then $F^\prime$ is not a diffeomorphism and 
\begin{equation*}
\begin{tikzcd}
\bb{R}\rtimes_\rho\bb{R}\ar[r,"F^\prime\times\id"]\ar[d,shift left]\ar[d,shift right]\ar[rd,phantom,"\circlearrowright"{xshift=-5}]&\bb{R}\rtimes_{\rho^\prime}\bb{R}\ar[d,shift left]\ar[d,shift right]\\
\bb{R}\ar[r,"F^\prime"']&\bb{R}
\end{tikzcd}
\end{equation*}
is a semiconjugacy but not a conjugacy. 

\begin{claim}
The morphisms $\id\times\frac{1}{3}$ and $F^\prime\times\id$ are orbitwise homotopic but not $nt$-homotopic. 
\end{claim}

\begin{proof}
For $s\in I$, define $F_s\colon\bb{R}\to\bb{R}$ by $F_s(x)=sx^3+(1-s)x$. Since $F_0=\id$ and $F_1=F^\prime$, $\id$ and $F^\prime$ are orbitwise homotopic. 

Let us find $P_s\colon\bb{R}\setminus\{0\}\to\bb{R}$ such that $F_s(x)=\rho^\prime(F_0(x),P_s(x))$ for all $x\in\bb{R}\setminus\{0\}$. We have 
\begin{gather*}
sx^3+(1-s)x=xe^{3P_s(x)}, \\
sx^2+1-s=e^{3P_s(x)}, 
\end{gather*}
hence 
\begin{equation*}
P_s(x)=\frac{1}{3}\log(sx^2+1-s). 
\end{equation*}
For $0\leq s<1$, $P_s$ extends to a unique $C^\infty$ function on $\bb{R}$, while $P_1(x)=\frac{1}{3}\log(x^2)$ does not extend to a continuous function on $\bb{R}$. 

Consider 
\begin{align*}
\varphi_s=\left(\id\times\frac{1}{3}\right)(t,P_st)\colon(\bb{R}\setminus\{0\})\rtimes_\rho\bb{R}&\to\bb{R}\rtimes_{\rho^\prime}\bb{R}\\
(x,t)&\mapsto(F_s(x),\ol{\varphi_s}(x,t)), 
\end{align*}
where the notation $\left(\id\times\frac{1}{3}\right)(t,P_st)$ is that defined in Lemma \ref{hnhncnnb} and $t$ is the target map. Then 
\begin{align*}
\ol{\varphi_s}(x,t)&=-P_s(x)+\frac{1}{3}t+P_s(\rho(x,t))\\
&=-\frac{1}{3}\log(sx^2+1-s)+\frac{1}{3}t+\frac{1}{3}\log(sx^2e^{2t}+1-s)
\end{align*}
and 
\begin{gather*}
\ol{\varphi_0}(x,t)=\frac{1}{3}t, \\
\ol{\varphi_1}(x,t)=-\frac{1}{3}\log(x^2)+\frac{1}{3}t+\frac{1}{3}\log(x^2e^{2t})=t. 
\end{gather*}
So $\varphi_s$ has a unique continuous extension 
\begin{align*}
\varphi_s\colon\bb{R}\rtimes_\rho\bb{R}&\to\bb{R}\rtimes_{\rho^\prime}\bb{R}\\
(x,t)&\mapsto(F_s(x),\ol{\varphi_s}(x,t)). 
\end{align*}
Hence $\id\times\frac{1}{3}$ and $F^\prime\times\id$ are orbitwise homotopic. But $\id\times\frac{1}{3}$ and $F^\prime\times\id$ are not $nt$-homotopic since if they are, $F^\prime\times\id$ is a conjugacy by 2 of Proposition \ref{btop}, hence $F^\prime$ is a diffeomorphism, which is a contradiction. 
\end{proof}
\end{example}

\subsection{``Cocycle rigidity'' implies ``parameter rigidity''}
\begin{thm}\label{cosp}
Let $\mca{G}\rightrightarrows M$ be a Lie groupoid, $N_0$ be a $C^\infty$ manifold, $\nu_0\colon N_0\to M$ be a surjective submersion and $\rho_0\in\mca{A}(\nu_0,\mca{G})$. Assume that for any $c\in\ul{\Hom}^{bb}(N_0\rtimes_{\rho_0}\mca{G},\mca{G})^\times$, there exists $\Phi\in\ul{\End}(\mca{G})$ such that $c\mathrel{\ul{\sim_{bnt}}}\Phi\varphi_{\rho_0}$ with respect to $\mca{B}_\mca{G}$. Then $T_{bnt}(\rho_0)$ is a singleton. 
\end{thm}

\begin{proof}
Let $(N_0\rtimes_\rho\mca{G},\varphi)\in\MAG_{\id}^{bb}(\rho_0)$. Then $\varphi_\rho\varphi\in\ul{\Hom}^{bb}(N_0\rtimes_{\rho_0}\mca{G},\mca{G})^\times$ by Proposition \ref{magidhom}. By assumption there exists $\Phi\in\ul{\End}(\mca{G})$ such that $\varphi_\rho\varphi\mathrel{\ul{\sim_{bnt}}}\Phi\varphi_{\rho_0}$. Note that $F_{\varphi_\rho\varphi}=\nu_0$. There exists a continuous map $P\colon N_0\times I\to\mca{G}$ such that: 
\begin{itemize}
\setlength\itemsep{0em}
\item $P(\cdot,t)\in C^\infty(N_0,\mca{G})$ for all $t\in I$
\item $\nu_0=tP(\cdot,t)=sP(\cdot,t)$ for all $t\in I$
\item $P(\cdot,0)=1_{\nu_0(\cdot)}$
\item $\Phi\varphi_{\rho_0}=(\varphi_\rho\varphi)P(\cdot,1)$
\item $P(N_0\times I)\in\mca{B}_\mca{G}$. 
\end{itemize}
Define 
\begin{align*}
\wt{P}\colon N_0\times I&\to N_0\rtimes_\rho\mca{G}\\
(y,t)&\mapsto(y,P(y,t)). 
\end{align*}
Then 
\begin{itemize}
\setlength\itemsep{0em}
\item $\wt{P}$ is continuous
\item $\wt{P}(\cdot,t)\in C^\infty(N_0,N\rtimes_\rho\mca{G})$ for all $t\in I$
\item $t\wt{P}(\cdot,t)=\id$
\item $\wt{P}(\cdot,0)=1_{\id(\cdot)}$
\item $\nu_0s\wt{P}(y,t)=\nu_0\rho(y,P(y,t))=sP(y,t)=\nu_0(y)$ for all $(y,t)\in N_0\times I$
\item $\varphi_\rho\wt{P}(N_0\times I)=P(N_0\times I)\in\mca{B}_\mca{G}$, hence $\wt{P}(N_0\times I)\in\mca{B}_\rho$. 
\end{itemize}
Let $\varphi^\prime=\varphi\wt{P}(\cdot,1)\in\ul{\Hom}(N_0\rtimes_{\rho_0}\mca{G},N_0\rtimes_\rho\mca{G})$. Then $\varphi\mathrel{\ul{\sim_{bnt}}}\varphi^\prime$ with respect to $\mca{B}_\rho$. Since $\varphi_\rho\varphi^\prime=(\varphi_\rho\varphi)P(\cdot,1)=\Phi\varphi_{\rho_0}$, $\varphi^\prime$ is a semiconjugacy. $\varphi^\prime$ is a conjugacy by Theorem \ref{ctop}. Hence $(N_0\rtimes_\rho\mca{G},\varphi)\mathrel{\ul{\sim_{bnt}}}(N_0\rtimes_{\rho_0}\mca{G},\id)$ and $T_{bnt}(\rho_0)$ is a singleton. 
\end{proof}

\begin{rem}
Theorem \ref{cosp} is a generalization of Proposition 2.3 in \cite{MM}, where they assume, among other things, $M$ is a point, $\mca{G}$ is contractible, $N$ is compact and $\rho$ is locally free. Although not mentioned explicitly, Asaoka \cite{A} essentially proves the theorem with the assumption that $M$ is a point, $\mca{G}$ is simply connected, $N$ is compact and $\rho$ is locally free. 
\end{rem}

\section{The Teichm\"{u}ller spaces of transitive actions}\label{1414}
\subsection{Based Lie groupoids}
\begin{dfn}\label{blghom0}
We say $(\mca{G}\rightrightarrows M,x)$ is a \emph{based Lie groupoid} if $\mca{G}\rightrightarrows M$ is a Lie groupoid and $x\in M$. For based Lie groupoids $(\mca{G}\rightrightarrows M,x)$ and $(\mca{G}^\prime\rightrightarrows M^\prime,x^\prime)$, let 
\begin{gather*}
\Hom_0(\mca{G},\mca{G}^\prime)=\left\{\varphi\in\Hom(\mca{G},\mca{G}^\prime)\ \middle|\ F_\varphi(x)=x^\prime\right\}, \\
\Hom_0(\mca{G},\mca{G}^\prime)^\times=\Hom(\mca{G},\mca{G}^\prime)^\times\cap\Hom_0(\mca{G},\mca{G}^\prime). 
\end{gather*}
For $\varphi$, $\varphi^\prime\in\Hom_0(\mca{G},\mca{G}^\prime)$, we write $\varphi\sim_{o,0}\varphi^\prime$ if there exists an orbitwise homotopy $\eta$ between $\varphi$ and $\varphi^\prime$ such that $\eta(\cdot,t)\in\Hom_0(\mca{G},\mca{G}^\prime)$ for all $t\in I$. 
\end{dfn}

\begin{dfn}
Let $\mca{G}\rightrightarrows M$ be a Lie groupoid, $N$, $N^\prime$ be $C^\infty$ manifolds, $\nu\colon N\to M$, $\nu^\prime\colon N^\prime\to M$ be $C^\infty$ maps, $y\in N$, $y^\prime\in N^\prime$ be such that $\nu(y)=\nu^\prime(y^\prime)$, and $\rho\in\mca{A}(\nu,\mca{G})$, $\rho^\prime\in\mca{A}(\nu^\prime,\mca{G})$. Let 
\begin{equation*}
\ul{\Hom}_0(N\rtimes_\rho\mca{G},N^\prime\rtimes_{\rho^\prime}\mca{G})=\left\{\varphi\in\ul{\Hom}(N\rtimes_\rho\mca{G},N^\prime\rtimes_{\rho^\prime}\mca{G})\ \middle|\ F_\varphi(y)=y^\prime\right\}. 
\end{equation*}
For $\varphi$, $\varphi^\prime\in\ul{\Hom}_0(N\rtimes_\rho\mca{G},N^\prime\rtimes_{\rho^\prime}\mca{G}^\prime)$, we write $\varphi\mathrel{\ul{\sim_{o,0}}}\varphi^\prime$ if there exists an orbitwise homotopy $\eta$ between $\varphi$ and $\varphi^\prime$ such that $\eta(\cdot,t)\in\ul{\Hom}_0(N\rtimes_\rho\mca{G},N^\prime\rtimes_{\rho^\prime}\mca{G})$ for all $t\in I$. 
\end{dfn}

\begin{dfn}
Let $\mca{G}\rightrightarrows M$ be a Lie groupoid, $N_0$ be a $C^\infty$ manifold, $\nu_0\colon N_0\to M$ be a surjective submersion, $\rho_0\in\mca{A}(\nu_0,\mca{G})$ and $y_0\in N_0$. For $(N\rtimes_\rho\mca{G},\varphi)\in\MAG(\rho_0)$, letting $y=F_\varphi(y_0)\in N$, $(N\rtimes_\rho\mca{G},y)$ is a based Lie groupoid. In this case we write $(N\rtimes_\rho\mca{G},y,\varphi)$. 

For $(N\rtimes_\rho\mca{G},y,\varphi)$, $(N^\prime\rtimes_{\rho^\prime}\mca{G},y^\prime,\varphi^\prime)\in\MAG(\rho_0)$, we write $(N\rtimes_\rho\mca{G},y,\varphi)\mathrel{\ul{\sim_{o,0}}}(N^\prime\rtimes_{\rho^\prime}\mca{G},y^\prime,\varphi^\prime)$ if there exists a conjugacy $\psi\in\ul{\Hom}_0(N\rtimes_\rho\mca{G},N^\prime\rtimes_{\rho^\prime}\mca{G})$ such that $\psi\varphi\mathrel{\ul{\sim_{o,0}}}\varphi^\prime$ in $\ul{\Hom}_0(N_0\rtimes_{\rho_0}\mca{G},N^\prime\rtimes_{\rho^\prime}\mca{G})$. 
\end{dfn}

\begin{dfn}
Let $\mca{G}\rightrightarrows M$ be a Lie groupoid. A \emph{bisection} of $\mca{G}$ is a $C^\infty$ map $\sigma\colon M\to\mca{G}$ such that: 
\begin{itemize}
\setlength\itemsep{0em}
\item $t\sigma=\id$
\item $s\sigma\colon M\to M$ is a diffeomorphism. 
\end{itemize}
Let $\sigma$ be a bisection of $\mca{G}$. Define 
\begin{align*}
\tau\colon M&\to\mca{G}\\
x&\mapsto\sigma((s\sigma)^{-1}(x))^{-1}. 
\end{align*}
Then $\tau$ is $C^\infty$, $t\tau=\id$ and $s\tau=(s\sigma)^{-1}$. Hence $\tau$ is a bisection of $\mca{G}$. The map 
\begin{align*}
\mca{G}&\to\mca{G}\\
g&\mapsto\sigma(t(g))^{-1}g\sigma(s(g))
\end{align*}
is an isomorphism with base map $s\sigma\colon M\to M$. 
\end{dfn}

\begin{lem}
Let $\mca{G}\rightrightarrows M$ be a Lie groupoid, $N$ be a $C^\infty$ manifold, $\nu\colon N\to M$ be a $C^\infty$ map, $\rho\in\mca{A}(\nu,\mca{G})$ and $\sigma$ be a bisection of $\mca{G}$. Define 
\begin{align*}
\varphi_\sigma\colon N\rtimes_\rho\mca{G}&\to N\rtimes_\rho\mca{G}\\
(y,g)&\mapsto(\rho(y,\sigma\nu(y)),\sigma(t(g))^{-1}g\sigma(s(g))). 
\end{align*}
Then $\varphi_\sigma$ is a conjugacy with base map $\rho(\cdot,\sigma\nu(\cdot))$. 
\end{lem}

\begin{proof}
We have 
\begin{equation*}
\varphi_\sigma((y,g)(\rho(y,g),g^\prime))=\varphi_\sigma(y,gg^\prime)=(\rho(y,\sigma\nu(y)),\sigma(t(g))^{-1}gg^\prime\sigma(s(g^\prime))), 
\end{equation*}
\begin{align*}
\varphi_\sigma(y,g)\varphi_\sigma(\rho(y,g),g^\prime)&=(\rho(y,\sigma\nu(y)),\sigma(t(g))^{-1}g\sigma(s(g)))\\
&\qquad(\rho(y,g\sigma(t(g))),\sigma(t(g^\prime))^{-1}g^\prime\sigma(s(g^\prime)))\\
&=(\rho(y,\sigma\nu(y)),\sigma(t(g))^{-1}gg^\prime\sigma(s(g^\prime)))
\end{align*}
and $\varphi_\sigma\varphi_\tau=\id$, $\varphi_\tau\varphi_\sigma=\id$. 
\end{proof}

\begin{dfn}
A Lie groupoid $\mca{G}\rightrightarrows M$ is \emph{transitive} if $\mca{O}_\mca{G}=\left\{O\right\}$ and the inclusion map $\iota\colon O\to M$ is a diffeomorphism, where $O$ is equipped with the $C^\infty$ manifold structure given in Proposition \ref{isoorb}. 
\end{dfn}

\begin{lem}
Let $\mca{G}\rightrightarrows M$ be a Lie groupoid and $x\in M$. Then $\mca{G}$ is transitive if and only if $s\colon\mca{G}^x\to M$ is a surjective submersion. (This latter condition is equivalent to (ii) in Proposition 1.3.3 of \cite{Mackenzie}. Hence it is equivalent to the \emph{local triviality} defined in Definition 1.3.2 of \cite{Mackenzie}, though our definition of a Lie groupoid is different from that in \cite{Mackenzie}.) 
\end{lem}

\begin{proof}
Necessity follows since $s\colon\mca{G}^x\to O$ is a surjective submersion and $\id\colon O\to M$ is a diffeomorphism. For sufficiency, we have 
\begin{equation*}
\begin{tikzcd}[column sep=tiny]
\mca{G}^x\ar[rr,"s"{name=U}]\ar[rd,"s"']&&M\\
&O. \ar[ru,equal]\ar[to=U,phantom,"\circlearrowright"]
\end{tikzcd}
\end{equation*}
Since $s\colon\mca{G}^x\to M$ has local $C^\infty$ sections, $\id\colon M\to O$ is $C^\infty$. 
\end{proof}

\begin{dfn}
Let $\mca{G}\rightrightarrows M$ be a Lie groupoid, $N$ be a $C^\infty$ manifold, $\nu\colon N\to M$ be a $C^\infty$ map and $\rho\in\mca{A}(\nu,\mca{G})$. We say that $\rho$ is \emph{transitive} if the action groupoid $N\rtimes_\rho\mca{G}$ is transitive. 
\end{dfn}

\begin{example}
For the Lie groupoid $S^1\rtimes_{\rho_d}S^1_d$ in Example \ref{discs1s1}, we have $\mca{O}_{\rho_d}=\left\{O_d\right\}$, but $O_d$ has the discrete topology, hence $\iota\colon O_d\to S^1$ is not a diffeomorphism. So $S^1\rtimes_{\rho_d}S^1_d$ and $\rho_d$ are not transitive. 
\end{example}

\begin{lem}
Let $\mca{G}\rightrightarrows M$ be a Lie groupoid such that $\mca{G}$ is second countable. Then $\mca{G}\rightrightarrows M$ is transitive if and only if $\mca{O}_\mca{G}=\left\{O\right\}$. 
\end{lem}

\begin{proof}
Assume $\mca{O}_\mca{G}=\left\{O\right\}$. Let $x\in O$. Since $\mca{G}^x$ is second countable and $s\colon\mca{G}^x\to O$ is a continuous open map, $O$ is second countable by Lemma \ref{topopsec}. Let $\iota\colon O\to M$ be the inclusion map. If $\iota\colon O\to M$ has no regular points, then $\iota(O)=M$ is the set of critical values of $\iota$. By the second countability of $O$ and Sard's theorem, $M$ has measure zero in $M$. Hence $\dim M=0$ and any point of $O$ is a regular point of $\iota\colon O\to M$, which is a contradiction. Thus $\iota\colon O\to M$ has a regular point and $\dim O\geq\dim M$. For any $x\in O$, $\iota_*\colon T_xO\to T_xM$ is an isomorphism since $\iota\colon O\to M$ is an immersion. Therefore $\iota\colon O\to M$ is a diffeomorphism by the inverse function theorem. 
\end{proof}

\begin{prop}
Let $\mca{G}\rightrightarrows M$ be a transitive Lie groupoid, $N_0$ be a $C^\infty$ manifold, $\nu_0\colon N_0\to M$ be a surjective submersion, $\rho_0\in\mca{A}(\nu_0,\mca{G})$ and $y_0\in N_0$. Then $\mathrel{\ul{\sim_{o,0}}}$ is equal to $\mathrel{\ul{\sim_o}}$ on $\MAG(\rho_0)$. 
\end{prop}

\begin{proof}
Let $(N\rtimes_\rho\mca{G},y,\varphi)$, $(N^\prime\rtimes_{\rho^\prime}\mca{G},y^\prime,\varphi^\prime)\in\MAG(\rho_0)$ be such that $(N\rtimes_\rho\mca{G},\varphi)\mathrel{\ul{\sim_o}}(N^\prime\rtimes_{\rho^\prime}\mca{G},\varphi^\prime)$. There exists a conjugacy $\psi\in\ul{\Hom}(N\rtimes_\rho\mca{G},N^\prime\rtimes_{\rho^\prime}\mca{G})$ such that $\psi\varphi\mathrel{\ul{\sim_o}}\varphi^\prime$. Let $\eta$ be an orbitwise homotopy between $\varphi^\prime$ and $\psi\varphi$ such that $\eta(\cdot,t)\in\ul{\Hom}(N_0\rtimes_{\rho_0}\mca{G},N^\prime\rtimes_{\rho^\prime}\mca{G})$ for all $t\in I$. Let $x_0=\nu_0(y_0)=\nu^\prime(y^\prime)\in M$ and $y^\prime\in O^\prime\in\mca{O}_{\rho^\prime}$. Then 
\begin{align*}
\mca{G}^{x_0}&\to O^\prime\\
g&\mapsto\rho^\prime(y^\prime,g)
\end{align*}
is a principal bundle. There exists a continuous map $c\colon I\to\mca{G}^{x_0}$ such that $c(0)=1_{x_0}$ and $F_{\eta(\cdot,t)}(y_0)=\rho^\prime(y^\prime,c(t))$ for all $t\in I$. Since 
\begin{equation*}
sc(t)=\nu^\prime\rho^\prime(y^\prime,c(t))=\nu^\prime F_{\eta(\cdot,t)}(y_0)=\nu_0(y_0)=x_0, 
\end{equation*}
we have $c\colon I\to\mca{G}^{x_0}_{x_0}$. Let $\mca{IG}=\left\{g\in\mca{G}\ \middle|\ t(g)=s(g)\right\}$. Then $t=s\colon\mca{IG}\to M$ is a Lie group bundle since $\mca{G}\rightrightarrows M$ is transitive (see Proposition 1.3.9 of \cite{Mackenzie}). Thus there exists a $C^\infty$ section $\sigma_t$ of $\mca{IG}\to M$ for each $t\in I$ such that: 
\begin{itemize}
\setlength\itemsep{0em}
\item $\sigma_t(x_0)=c(t)$ for all $t\in I$
\item $\sigma_0(x)=1_x$ for all $x\in M$
\item the map 
\begin{align*}
M\times I&\to\mca{IG}\\
(x,t)&\mapsto\sigma_t(x)
\end{align*}
is continuous. 
\end{itemize}
Let $\eta^\prime(\cdot,t)=\varphi_{\sigma_t^{-1}}\eta(\cdot,t)\in\Hom(N_0\rtimes_{\rho_0}\mca{G},N^\prime\rtimes_{\rho^\prime}\mca{G})$. Then 
\begin{gather*}
F_{\eta^\prime(\cdot,t)}=\rho^\prime(\cdot,\sigma_t^{-1}\nu^\prime(\cdot))F_{\eta(\cdot,t)}, \\
\nu^\prime F_{\eta^\prime(\cdot,t)}=s\sigma_t^{-1}\nu^\prime F_{\eta(\cdot,t)}=\nu^\prime F_{\eta(\cdot,t)}=\nu_0, \\
F_{\eta^\prime(\cdot,t)}(y_0)=\rho^\prime(\cdot,\sigma_t^{-1}\nu^\prime(\cdot))\rho^\prime(y^\prime,c(t))=\rho^\prime(\rho^\prime(y^\prime,c(t)),\sigma_t(x_0)^{-1})=y^\prime, 
\end{gather*}
hence $\eta^\prime(\cdot,t)\in\ul{\Hom}_0(N_0\rtimes_{\rho_0}\mca{G},N^\prime\rtimes_{\rho^\prime}\mca{G})$ for all $t\in I$. We have $\eta^\prime(\cdot,0)=\varphi^\prime$, $\eta^\prime(\cdot,1)=\varphi_{\sigma_1^{-1}}\psi\varphi=\psi^\prime\varphi$, where $\psi^\prime=\varphi_{\sigma_1^{-1}}\psi\in\ul{\Hom}(N\rtimes_\rho\mca{G},N^\prime\rtimes_{\rho^\prime}\mca{G})$ is a conjugacy. Since $F_{\psi^\prime}(y)=y^\prime$, $\varphi^\prime\mathrel{\ul{\sim_{o,0}}}\psi^\prime\varphi$ and $(N\rtimes_\rho\mca{G},y,\varphi)\mathrel{\ul{\sim_{o,0}}}(N^\prime\rtimes_{\rho^\prime}\mca{G},y^\prime,\varphi^\prime)$. 
\end{proof}

\subsection{Morphisms between transitive based Lie groupoids}
\begin{dfn}\label{hnhnyychy}
Let $(\mca{H}\rightrightarrows N,y)$, $(\mca{H}^\prime\rightrightarrows N^\prime,y^\prime)$ be transitive based Lie groupoids, $H$ be a closed subgroup of $\mca{H}^y_y$ and $H^\prime$ be a closed subgroup of $(\mca{H}^\prime)^{y^\prime}_{y^\prime}$. Define 
\begin{equation*}
C^\infty(\mca{H}^y,(\mca{H}^\prime)^{y^\prime})_{H,H^\prime}=\left\{\theta\in C^\infty(\mca{H}^y,(\mca{H}^\prime)^{y^\prime})\ \middle|\ 
\begin{gathered}
\theta(H)\subset H^\prime\\
\theta(kh)=\theta(k)\theta(h)\text{ for }k\in H,h\in\mca{H}^y
\end{gathered}
\right\}, 
\end{equation*}
\begin{equation*}
C^\infty(\mca{H}^y,(\mca{H}^\prime)^{y^\prime})_{H,H^\prime}^\times=\left\{\theta\in C^\infty(\mca{H}^y,(\mca{H}^\prime)^{y^\prime})_{H,H^\prime}\ \middle|\ \text{$\theta$ is a diffeomorphism}\right\}. 
\end{equation*}
For $\theta$, $\theta^\prime\in C^\infty(\mca{H}^y,(\mca{H}^\prime)^{y^\prime})_{H,H^\prime}$, write $\theta\sim_{o,0}\theta^\prime$ if there exists a continuous map $J\colon\mca{H}^y\times I\to(\mca{H}^\prime)^{y^\prime}$ such that: 
\begin{itemize}
\setlength\itemsep{0em}
\item $J(\cdot,t)\in C^\infty(\mca{H}^y,(\mca{H}^\prime)^{y^\prime})_{H,H^\prime}$ for all $t\in I$
\item $J(\cdot,0)=\theta$, $J(\cdot,1)=\theta^\prime$. 
\end{itemize}
\end{dfn}

\begin{prop}
Let $(\mca{G}\rightrightarrows M,x)$, $(\mca{G}^\prime\rightrightarrows M^\prime,x^\prime)$ be transitive based Lie groupoids. Then the following statements hold: 
\begin{enumerate}
\item We have a bijection 
\begin{align*}
\Hom_0(\mca{G},\mca{G}^\prime)&\simeq C^\infty(\mca{G}^x,(\mca{G}^\prime)^{x^\prime})_{\mca{G}_x^x,(\mca{G}^\prime)_{x^\prime}^{x^\prime}}\\
\varphi&\mapsto\theta_\varphi\\
\varphi_\theta&\mapsfrom\theta, 
\end{align*}
where $\theta_\varphi=\varphi|_{\mca{G}^x}$ and $\varphi_\theta(g)=\theta(h)^{-1}\theta(hg)$ for a choice of $h\in\mca{G}_{t(g)}^x$. 
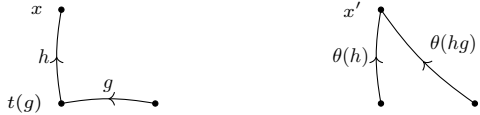
\begin{equation*}
\begin{tikzpicture}[every label/.append style={font=\scriptsize},matrix of math nodes,decoration={markings,mark=at position0.5with{\arrow{>}}}]
\node(1)[label=left:x]{};
\node(2)[below=of 1,label=left:t(g)]{};
\node(3)[right=of 2]{};
\node(4)[right=4 of 1,label=left:x^\prime]{};
\node(5)[below=of 4]{};
\node(6)[right=of 5]{};
\draw[postaction={decorate}](2.center)to[out=100,in=260]node[left,scale=.75]{h}(1.center);
\draw[postaction={decorate}](3.center)to[out=170,in=10]node[above,scale=.75]{g}(2.center);
\draw[postaction={decorate}](5.center)to[out=100,in=260]node[left,scale=.75]{\theta(h)}(4.center);
\draw[postaction={decorate}](6.center)to[out=145,in=305]node[above right,scale=.75]{\theta(hg)}(4.center);
\foreach\x in{1,...,6}\filldraw(\x)circle(1pt);
\end{tikzpicture}
\end{equation*}
\item The bijection above induces a bijection 
\begin{equation*}
\Hom_0(\mca{G},\mca{G}^\prime)^\times\simeq C^\infty(\mca{G}^x,(\mca{G}^\prime)^{x^\prime})_{\mca{G}_x^x,(\mca{G}^\prime)_{x^\prime}^{x^\prime}}^\times. 
\end{equation*}
\item $\sim_{o,0}$ on $\Hom_0(\mca{G},\mca{G}^\prime)$ corresponds to $\sim_{o,0}$ on $C^\infty(\mca{G}^x,(\mca{G}^\prime)^{x^\prime})_{\mca{G}_x^x,(\mca{G}^\prime)_{x^\prime}^{x^\prime}}$. 
\end{enumerate}
\end{prop}

\begin{proof}
1. Let $\theta\in C^\infty(\mca{G}^x,(\mca{G}^\prime)^{x^\prime})_{\mca{G}_x^x,(\mca{G}^\prime)_{x^\prime}^{x^\prime}}$. For $g\in\mca{G}$, take $h\in\mca{G}_{t(g)}^x$. Define 
\begin{align*}
\varphi_\theta\colon\mca{G}&\to\mca{G}^\prime\\
g&\mapsto\theta(h)^{-1}\theta(hg). 
\end{align*}
To show the well-definedness, take $h^\prime\in\mca{G}_{t(g)}^x$. Then 
\begin{align*}
\theta(h)^{-1}\theta(hg)&=\theta(h)^{-1}\theta(h\textcolor{red}{(h^\prime)^{-1}h^\prime}g)=\theta(h)^{-1}\theta(h(h^\prime)^{-1})\theta(h^\prime g)\\
&=\theta(h)^{-1}\theta(h(h^\prime)^{-1})\textcolor{red}{\theta(h^\prime)\theta(h^\prime)^{-1}}\theta(h^\prime g)\\
&=\theta(h)^{-1}\theta(h)\theta(h^\prime)^{-1}\theta(h^\prime g)\\
&=\theta(h^\prime)^{-1}\theta(h^\prime g). 
\end{align*}
Hence $\varphi_\theta$ is well-defined. Since $s\colon\mca{G}^x\to M$ is a surjective submersion, $h$ can be taken smoothly on a small neighborhood of $g$, hence $\varphi_\theta$ is $C^\infty$. Since $\mca{G}_x^x\bs\mca{G}^x\simeq M$, $(\mca{G}^\prime)_{x^\prime}^{x^\prime}\bs(\mca{G}^\prime)^{x^\prime}\simeq M^\prime$, there exists a unique $\ol{\theta}\colon M\to M^\prime$ such that 
\begin{equation*}
\begin{tikzcd}
\mca{G}^x\ar[r,"\theta"]\ar[d,"s"']\ar[rd,phantom,"\circlearrowright"]&(\mca{G}^\prime)^{x^\prime}\ar[d,"s^\prime"]\\
M\ar[r,"\ol{\theta}"']&M^\prime. 
\end{tikzcd}
\end{equation*}
We have 
\begin{gather*}
t^\prime\varphi_\theta(g)=s^\prime\theta(h)=\ol{\theta}s(h)=\ol{\theta}t(g), \\
s^\prime\varphi_\theta(g)=s^\prime\theta(hg)=\ol{\theta}s(hg)=\ol{\theta}s(g)
\end{gather*}
hence 
\begin{equation*}
\begin{tikzcd}
\mca{G}\ar[r,"\varphi_\theta"]\ar[d,shift left]\ar[d,shift right]\ar[rd,phantom,"\circlearrowright"{xshift=2,yshift=-2}]&\mca{G}^\prime\ar[d,shift left]\ar[d,shift right]\\
M\ar[r,"\ol{\theta}"']&M^\prime. 
\end{tikzcd}
\end{equation*}
For $g$, $g^\prime\in\mca{G}$ such that $s(g)=t(g^\prime)$, take $h\in\mca{G}_{t(g)}^x$ and $h^\prime\in\mca{G}_{t(g^\prime)}^x$. Then 
\begin{align*}
\varphi_\theta(g)\varphi_\theta(g^\prime)&=\theta(h)^{-1}\theta(hg)\theta(h^\prime)^{-1}\theta(h^\prime g^\prime)\\
&=\theta(h)^{-1}\theta(hg(h^\prime)^{-1}h^\prime)\theta(h^\prime)^{-1}\theta(h^\prime g^\prime)\\
&=\theta(h)^{-1}\theta(hg(h^\prime)^{-1})\theta(h^\prime)\theta(h^\prime)^{-1}\theta(h^\prime g^\prime)\\
&=\theta(h)^{-1}\theta(hg(h^\prime)^{-1})\theta(h^\prime g^\prime)\\
&=\theta(h)^{-1}\theta(hg(h^\prime)^{-1}h^\prime g^\prime)\\
&=\theta(h)^{-1}\theta(hgg^\prime)\\
&=\varphi_\theta(gg^\prime). 
\end{align*}
We have $\theta(1_x)=1_{x^\prime}$ and $\ol{\theta}(x)=x^\prime$. Hence $\varphi_\theta\in\Hom_0(\mca{G},\mca{G}^\prime)$. 

For the composition $\theta\mapsto\varphi_\theta\mapsto\theta_{\varphi_\theta}$, we have $\theta_{\varphi_\theta}(g)=\varphi_\theta(g)=\theta(g)$, and for $\varphi\mapsto\theta_\varphi\mapsto\varphi_{\theta_\varphi}$, we have $\varphi_{\theta_\varphi}(g)=\theta_\varphi(h)^{-1}\theta_\varphi(hg)=\varphi(h)^{-1}\varphi(hg)=\varphi(g)$. 

2. Let $\varphi\in\Hom_0(\mca{G},\mca{G}^\prime)^\times$. 
Then $\varphi\colon\mca{G}^x\to(\mca{G}^\prime)^{x^\prime}$ is a diffeomorphism. Hence $\theta_\varphi\in C^\infty(\mca{G}^x,(\mca{G}^\prime)^{x^\prime})_{\mca{G}_x^x,(\mca{G}^\prime)_{x^\prime}^{x^\prime}}^\times$. 

Conversely let $\theta\in C^\infty(\mca{G}^x,(\mca{G}^\prime)^{x^\prime})_{\mca{G}_x^x,(\mca{G}^\prime)_{x^\prime}^{x^\prime}}^\times$. We will construct the inverse of $\varphi_\theta^!\colon\mca{G}\to M\times_{\ol{\theta},t^\prime}\mca{G}^\prime$. Define 
\begin{align*}
M\times_{\ol{\theta},t^\prime}\mca{G}^\prime&\to\mca{G}\\
(y,g^\prime)&\mapsto h^{-1}\theta^{-1}(\theta(h)g^\prime), 
\end{align*}
where $h\in\mca{G}_y^x$. 
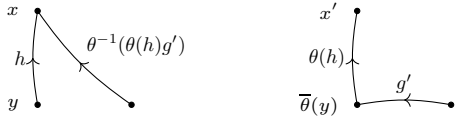
\begin{equation*}
\begin{tikzpicture}[every label/.append style={font=\scriptsize},matrix of math nodes,decoration={markings,mark=at position0.5with{\arrow{>}}}]
\node(1)[label=left:x]{};
\node(2)[below=of 1,label=left:y]{};
\node(3)[right=of 2]{};
\node(4)[right=4 of 1,label=left:x^\prime]{};
\node(5)[below=of 4,label=left:\ol{\theta}(y)]{};
\node(6)[right=of 5]{};
\draw[postaction={decorate}](2.center)to[out=100,in=260]node[left,scale=.75]{h}(1.center);
\draw[postaction={decorate}](3.center)to[out=145,in=305]node[above right,scale=.75]{\theta^{-1}(\theta(h)g^\prime)}(1.center);
\draw[postaction={decorate}](5.center)to[out=100,in=260]node[left,scale=.75]{\theta(h)}(4.center);
\draw[postaction={decorate}](6.center)to[out=170,in=10]node[above,scale=.75]{g^\prime}(5.center);
\foreach\x in{1,...,6}\filldraw(\x)circle(1pt);
\end{tikzpicture}
\end{equation*}
To check the well-definedness, let $h^\prime\in\mca{G}_y^x$. Then 
\begin{align*}
(h^\prime)^{-1}\theta^{-1}(\theta(h^\prime)g^\prime)&=\textcolor{red}{h^{-1}h}(h^\prime)^{-1}\theta^{-1}(\theta(h^\prime)g^\prime)\\
&=h^{-1}\textcolor{red}{\theta^{-1}\theta}(h(h^\prime)^{-1}\theta^{-1}(\theta(h^\prime)g^\prime))\\
&=h^{-1}\theta^{-1}(\textcolor{red}{\theta(h(h^\prime)^{-1})}\theta(h^\prime)g^\prime)\\
&=h^{-1}\theta^{-1}(\theta(h)g^\prime). 
\end{align*}
Hence the map is well-defined. Since $s\colon\mca{G}^x\to M$ is a surjective submersion, $h$ can be taken smoothly on a small neighborhood of $(y,g^\prime)$. So the map is $C^\infty$. To see that it is the inverse of $\varphi_\theta^!$, compute the compositions: 
\begin{align*}
(y,g^\prime)&\mapsto h^{-1}\theta^{-1}(\theta(h)g^\prime)\mapsto(y,\varphi_\theta(h^{-1}\theta^{-1}(\theta(h)g^\prime)))\\
&=(y,\theta(h)^{-1}\theta(h)g^\prime)=(y,g^\prime), 
\end{align*}
\begin{align*}
g&\mapsto(t(g),\varphi_\theta(g))\mapsto h^{-1}\theta^{-1}(\theta(h)\varphi_\theta(g))\\
&=h^{-1}\theta^{-1}(\theta(h)\theta(h)^{-1}\theta(hg))=g. 
\end{align*}
Hence $\varphi_\theta^!\colon\mca{G}\to M\times_{\ol{\theta},t^\prime}\mca{G}^\prime$ is a diffeomorphism. 

3. Omitted. 
\end{proof}

\begin{dfn}
Let $(\mca{G}\rightrightarrows M,x)$ be a based Lie groupoid and $H$, $H^\prime$ be closed subgroups of $\mca{G}^x_x$. We have the map $s\colon\mca{G}^x\to M$. Define 
\begin{equation*}
\ul{C}^\infty(\mca{G}^x,\mca{G}^x)_{H,H^\prime}=\left\{\theta\in C^\infty(\mca{G}^x,\mca{G}^x)\ \middle|\ 
\begin{gathered}
s\theta=s,\theta(H)\subset H^\prime\\
\theta(hg)=\theta(h)\theta(g)\text{ for }h\in H,g\in\mca{G}^x
\end{gathered}
\right\}, 
\end{equation*}
\begin{equation*}
\ul{C}^\infty(\mca{G}^x,\mca{G}^x)_{H,H^\prime}^\times=\left\{\theta\in\ul{C}^\infty(\mca{G}^x,\mca{G}^x)_{H,H^\prime}\ \middle|\ \text{$\theta$ is a diffeomorphism}\right\}. 
\end{equation*}
For $\theta$, $\theta^\prime\in\ul{C}^\infty(\mca{G}^x,\mca{G}^x)_{H,H^\prime}$, write $\theta\mathrel{\ul{\sim_{o,0}}}\theta^\prime$ if there exists a continuous map $J\colon\mca{G}^x\times I\to\mca{G}^x$ such that: 
\begin{itemize}
\setlength\itemsep{0em}
\item $J(\cdot,t)\in\ul{C}^\infty(\mca{G}^x,\mca{G}^x)_{H,H^\prime}$ for all $t\in I$
\item $J(\cdot,0)=\theta$, $J(\cdot,1)=\theta^\prime$. 
\end{itemize}
\end{dfn}

\begin{cor}\label{homhomtimessim}
Let $(\mca{G}\rightrightarrows M,x)$ be a based Lie groupoid, $N$, $N^\prime$ be $C^\infty$ manifolds, $\nu\colon N\to M$, $\nu^\prime\colon N^\prime\to M$ be $C^\infty$ maps, $\rho\in\mca{A}(\nu,\mca{G})$, $\rho^\prime\in\mca{A}(\nu^\prime,\mca{G})$ be transitive actions and $y\in N$, $y^\prime\in N^\prime$ be such that $\nu(y)=\nu^\prime(y^\prime)=x$. Let $H=\Stab_\rho(y)$, $H^\prime=\Stab_{\rho^\prime}(y^\prime)$. Then the following statements hold: 
\begin{enumerate}
\item We have a bijection 
\begin{align}
\Hom_0(N\rtimes_\rho\mca{G},N^\prime\rtimes_{\rho^\prime}\mca{G})&\simeq C^\infty(\mca{G}^x,\mca{G}^x)_{H,H^\prime}\label{ngnpggxgxhh}\\
\varphi&\mapsto\theta_\varphi\nonumber\\
\varphi_\theta&\mapsfrom\theta, \nonumber
\end{align}
where 
\begin{equation*}
\theta_\varphi(g)=\varphi_{\rho^\prime}\varphi(y,g)
\end{equation*}
for $g\in\mca{G}^x$ and 
\begin{gather*}
\varphi_\theta(\rho(y,g),g^\prime)=(\rho^\prime(y^\prime,\theta(g)),\theta(g)^{-1}\theta(gg^\prime)), \\
F_{\varphi_\theta}(\rho(y,g))=\rho^\prime(y^\prime,\theta(g))
\end{gather*}
for $g\in\mca{G}^x$ and $g^\prime\in\mca{G}^{s(g)}$. 
\item The bijection \eqref{ngnpggxgxhh} induces bijections 
\begin{gather*}
\ul{\Hom}_0(N\rtimes_\rho\mca{G},N^\prime\rtimes_{\rho^\prime}\mca{G})\simeq\ul{C}^\infty(\mca{G}^x,\mca{G}^x)_{H,H^\prime}, \\
\ul{\Hom}_0(N\rtimes_\rho\mca{G},N^\prime\rtimes_{\rho^\prime}\mca{G})^\times\simeq\ul{C}^\infty(\mca{G}^x,\mca{G}^x)_{H,H^\prime}^\times. 
\end{gather*}
\item $\mathrel{\ul{\sim_{o,0}}}$ on $\ul{\Hom}_0(N\rtimes_\rho\mca{G},N^\prime\rtimes_{\rho^\prime}\mca{G})$ corresponds to $\mathrel{\ul{\sim_{o,0}}}$ on $\ul{C}^\infty(\mca{G}^x,\mca{G}^x)_{H,H^\prime}$. 
\end{enumerate}
\end{cor}

\begin{proof}
1. This follows from 
\begin{equation*}
(N\rtimes_\rho\mca{G})^y\simeq\mca{G}^x, \ (N^\prime\rtimes_{\rho^\prime}\mca{G})^{y^\prime}\simeq\mca{G}^x, \ (N\rtimes_\rho\mca{G})^y_y\simeq H, \ (N^\prime\rtimes_{\rho^\prime}\mca{G})^{y^\prime}_{y^\prime}\simeq H^\prime. 
\end{equation*}

2. Let $\varphi\in\ul{\Hom}_0(N\rtimes_\rho\mca{G},N^\prime\rtimes_{\rho^\prime}\mca{G}^\prime)$. We have 
\begin{equation*}
s\theta_\varphi(g)=s\varphi_{\rho^\prime}\varphi(y,g)=\nu^\prime s^\prime\varphi(y,g)=\nu^\prime F_\varphi s(y,g)=\nu s(y,g)=\nu\rho(y,g)=s(g)
\end{equation*}
for all $g\in\mca{G}^x$. Hence $s\theta_\varphi=s$ and $\theta_\varphi\in\ul{C}^\infty(\mca{G}^x,\mca{G}^x)_{H,H^\prime}$. 

Conversely let $\theta\in\ul{C}^\infty(\mca{G}^x,\mca{G}^x)_{H,H^\prime}$. Then 
\begin{equation*}
\nu^\prime F_{\varphi_\theta}\rho(y,g)=\nu^\prime\rho^\prime(y^\prime,\theta(g))=s\theta(g)=s(g)=\nu\rho(y,g)
\end{equation*}
for all $g\in\mca{G}^x$. Hence $\nu^\prime F_{\varphi_\theta}=\nu$ and $\varphi_\theta\in\ul{\Hom}_0(N\rtimes_\rho\mca{G},N^\prime\rtimes_{\rho^\prime}\mca{G})$. 

3. Omitted. 
\end{proof}

\begin{lem}\label{compos}
Let $(\mca{H}\rightrightarrows N,y)$, $(\mca{H}^\prime\rightrightarrows N^\prime,y^\prime)$, $(\mca{H}^{\prime\prime}\rightrightarrows N^{\prime\prime},y^{\prime\prime})$ be transitive based Lie groupoids and 
\begin{equation*}
\theta\in C^\infty(\mca{H}^y,(\mca{H}^\prime)^{y^\prime})_{\mca{H}_y^y,(\mca{H}^\prime)_{y^\prime}^{y^\prime}}, \quad\theta^\prime\in C^\infty((\mca{H}^\prime)^{y^\prime},(\mca{H}^{\prime\prime})^{y^{\prime\prime}})_{(\mca{H}^\prime)_{y^\prime}^{y^\prime},(\mca{H}^{\prime\prime})_{y^{\prime\prime}}^{y^{\prime\prime}}}. 
\end{equation*}
Then we have $\varphi_{\theta^\prime\theta}=\varphi_{\theta^\prime}\varphi_\theta$. 
\end{lem}

\begin{proof}
We have $(\varphi_{\theta^\prime\theta})|_{\mca{H}^y}=\theta^\prime\theta=\varphi_{\theta^\prime}|_{(\mca{H}^\prime)^{y^\prime}}\varphi_\theta|_{\mca{H}^y}=(\varphi_{\theta^\prime}\varphi_\theta)|_{\mca{H}^y}$. Hence $\varphi_{\theta^\prime\theta}=\varphi_{\theta^\prime}\varphi_\theta$. 
\end{proof}

\subsection{Transitive actions of Lie groupoids}
\begin{example}
Let $(\mca{G}\rightrightarrows M,x)$ be a transitive based Lie groupoid and $H$ be a closed subgroup of $\mca{G}^x_x$. Since $s\colon\mca{G}^x\to M$ is a principal $\mca{G}^x_x$ bundle, the quotient map $\mca{G}^x\to H\bs\mca{G}^x$ and $s\colon H\bs\mca{G}^x\to M$ are surjective submersions. Define $\rho_H\in\mca{A}(s,\mca{G})$ by 
\begin{align*}
\rho_H\colon(H\bs\mca{G}^x)\times_{s,t}\mca{G}&\to H\bs\mca{G}^x\\
(Hg,g^\prime)&\mapsto Hgg^\prime. 
\end{align*}
Take $H\in H\bs\mca{G}^x$ as a base point. Since 
\begin{equation*}
\begin{tikzcd}[column sep=tiny]
((H\bs\mca{G}^x)\rtimes_{\rho_H}\mca{G})^H\ar[rr,"\sim"{name=U},"\varphi_{\rho_H}"{yshift=7}]\ar[rd,"s"']&&\mca{G}^x\ar[ld]\\
&H\bs\mca{G}^x, \ar[to=U,phantom,"\circlearrowright"{xshift=-3,yshift=-1}]
\end{tikzcd}
\end{equation*}
$s\colon((H\bs\mca{G}^x)\rtimes_{\rho_H}\mca{G})^H\to H\bs\mca{G}^x$ is a surjective submersion. Hence $\rho_H$ is transitive. We have $\Stab_{\rho_H}(H)=H$. 
\end{example}

\begin{lem}
Let $\mca{G}\rightrightarrows M$ be a Lie groupoid, $\nu\colon N\to M$ be a surjective submersion and $\rho\in\mca{A}(\nu,\mca{G})$ be a transitive action. Let $y\in N$ and put $x=\nu(y)\in M$, $H=\Stab_\rho(y)\subset\mca{G}^x_x$. Then $\mca{G}\rightrightarrows M$ is transitive and the map 
\begin{align*}
(H\bs\mca{G}^x)\rtimes_{\rho_H}\mca{G}&\to N\rtimes_\rho\mca{G}\\
(Hg,g^\prime)&\mapsto(\rho(y,g),g^\prime)
\end{align*}
is a conjugacy. 
\end{lem}

\begin{proof}
Since $\nu=F_{\varphi_\rho}$ and $N\in\mca{O}_{N\rtimes_\rho\mca{G}}$, there exists $O\in\mca{O}_\mca{G}$ such that $\nu(N)\subset O$. Hence $O=M$ as sets. We have 
\begin{equation*}
\begin{tikzcd}[column sep=tiny]
N\ar[rr,"\nu"{name=U}]\ar[rd,"\nu"']&&M\\
&O. \ar[ru,equal]\ar[to=U,phantom,"\circlearrowright"]
\end{tikzcd}
\end{equation*}
$\nu\colon N\to O$ is $C^\infty$ by Lemma \ref{oosmhs}. $\nu\colon N\to M$ has local $C^\infty$ sections since it is a surjective submersion. Hence $\id\colon M\to O$ is $C^\infty$. Thus $\mca{G}\rightrightarrows M$ is transitive. 

$s\colon(N\rtimes_\rho\mca{G})^y\to N$ is a principal $(N\rtimes_\rho\mca{G})^y_y$ bundle since $\rho$ is transitive. Hence $s\colon(N\rtimes_\rho\mca{G})^y_y\bs(N\rtimes_\rho\mca{G})^y\to N$ is a diffeomorphism. We have 
\begin{equation*}
\begin{tikzcd}
(N\rtimes_\rho\mca{G})^y_y\bs(N\rtimes_\rho\mca{G})^y\ar[r,"s"]\ar[d,dash,"\sim"{sloped}]&N. \\
H\bs\mca{G}^x\ar[ru,"\rho(y{,}\cdot)"',""{pos=.3,name=U}]
\ar[from=1-1,to=U,phantom,"\circlearrowright"]
\end{tikzcd}
\end{equation*}
So 
\begin{align*}
H\bs\mca{G}^x&\to N\\
Hg&\mapsto\rho(y,g)
\end{align*}
is a diffeomorphism. 
\begin{align*}
(H\bs\mca{G}^x)\rtimes_{\rho_H}\mca{G}&\to N\rtimes_\rho\mca{G}\\
(Hg,g^\prime)&\mapsto(\rho(y,g),g^\prime)
\end{align*}
is a conjugacy. 
\end{proof}

\subsection{The Teichm\"{u}ller space of a transitive action}
\begin{dfn}\label{gmxththhh}
Let $(\mca{G}\rightrightarrows M,x)$ be a based Lie groupoid, $H$ be a closed subgroup of $\mca{G}^x_x$ and $\theta$, $\theta^\prime\in\ul{C}^\infty(\mca{G}^x,\mca{G}^x)_{H,\mca{G}^x_x}^\times$. We write $\theta\sim\theta^\prime$ if $\theta(H)=\theta^\prime(H)$ and $\theta\mathrel{\ul{\sim_{o,0}}}\theta^\prime$ in $\ul{C}^\infty(\mca{G}^x,\mca{G}^x)_{H,\theta(H)}$. 
\end{dfn}

\begin{dfn}\label{gmxdiffgxth}
For a transitive based Lie groupoid $(\mca{G}\rightrightarrows M,x)$ and a closed subgroup $H$ of $\mca{G}^x_x$, let 
\begin{equation*}
\ul{\Diff}(\mca{G}^x)=\left\{\theta\in\Diff(\mca{G}^x)\ \middle|\ s\theta=s\right\}
\end{equation*}
and 
\begin{align*}
\ul{\Diff}(\mca{G}^x)_H&=\ul{C}^\infty(\mca{G}^x,\mca{G}^x)_{H,\mca{G}^x_x}^\times\\
&=\left\{\theta\in\ul{\Diff}(\mca{G}^x)\ \middle|\ \text{$\theta(hg)=\theta(h)\theta(g)$ for $h\in H$, $g\in\mca{G}^x$}\right\}. 
\end{align*}
We have 
\begin{align*}
\ul{\Aut}(\mca{G})\simeq\ul{\Hom}_0(M\rtimes\mca{G},M\rtimes\mca{G})^\times&\simeq\ul{C}^\infty(\mca{G}^x,\mca{G}^x)_{\mca{G}^x_x,\mca{G}^x_x}^\times=\ul{\Diff}(\mca{G}^x)_{\mca{G}^x_x}\\
\Phi&\mapsto\theta_\Phi. 
\end{align*}
There is an action $\ul{C}^\infty(\mca{G}^x,\mca{G}^x)_{\mca{G}^x_x,\mca{G}^x_x}^\times\curvearrowright\ul{C}^\infty(\mca{G}^x,\mca{G}^x)_{H,\mca{G}^x_x}^\times$. 
\end{dfn}

\begin{prop}\label{trans}
Let $(\mca{G}\rightrightarrows M,x)$ be a transitive based Lie groupoid and $H$ be a closed subgroup of $\mca{G}^x_x$. Recall that $\rho_H$ is the action of $\mca{G}\rightrightarrows M$ on $s\colon H\bs\mca{G}^x\to M$ by right multiplication. Then we have a bijection 
\begin{align*}
T_o(\rho_H)&\simeq\ul{\Diff}(\mca{G}^x)_{\mca{G}^x_x}\bs(\ul{\Diff}(\mca{G}^x)_H/{\sim})\\
[N\rtimes_\rho\mca{G},\varphi]&\mapsto[\theta_\varphi]\\
[(\theta(H)\bs\mca{G}^x)\rtimes_{\rho_{\theta(H)}}\mca{G},\varphi_\theta]&\mapsfrom[\theta]. 
\end{align*}
\end{prop}

The proof of this proposition comes after the following corollary. 

\begin{cor}\label{tblgcgsgxx}
Let $G$ be a Lie group and $H$ be a closed subgroup of $G$. Then we have a bijection 
\begin{align*}
T_o(\rho_H)&\simeq\Aut(G)\bs(\Diff(G)_H/{\sim})\\
[N\rtimes_\rho G,\varphi]&\mapsto[\theta_\varphi]\\
[(\theta(H)\bs G)\rtimes_{\rho_{\theta(H)}}G,\varphi_\theta]&\mapsfrom[\theta], 
\end{align*}
where 
\begin{itemize}
\setlength\itemsep{0em}
\item $H\bs G\stackrel{\rho_H}{\curvearrowleft}G$ is the action by right multiplication
\item $\Diff(G)_H=\left\{\theta\in\Diff(G)\ \middle|\ \text{$\theta(hg)=\theta(h)\theta(g)$ for all $h\in H$ and $g\in G$}\right\}$
\item $\theta\sim\theta^\prime$ for $\theta$, $\theta^\prime\in\Diff(G)_H$ if $\theta(H)=\theta^\prime(H)$ and there exists a continuous map $J\colon G\times I\to G$ such that: 
\begin{itemize}
\setlength\itemsep{0em}
\item $J(\cdot,t)\in C^\infty(G,G)$ for all $t\in I$
\item $J(H,t)\subset\theta(H)$ for all $t\in I$
\item $J(hg,t)=J(h,t)J(g,t)$ for all $h\in H$, $g\in G$, $t\in I$
\item $J(\cdot,0)=\theta$, $J(\cdot,1)=\theta^\prime$
\end{itemize}
\item $\Aut(G)$ acts by post composition
\item $\varphi\colon(H\bs G)\rtimes_{\rho_H}G\to N\rtimes_\rho G$ is an isomorphism and $\theta_\varphi(g)=\varphi_\rho\varphi(H,g)$ for $g\in G$
\item $\varphi_\theta(Hg,g^\prime)=(\theta(H)\theta(g),\theta(g)^{-1}\theta(gg^\prime))$ for $g$, $g^\prime\in G$. 
\end{itemize}
\end{cor}

\begin{proof}[Proof of Proposition \ref{trans}]
Let $(N\rtimes_\rho\mca{G},y,\varphi)\in\MAG(\rho_H)$. Then $\varphi\in\ul{\Hom}_0((H\bs\mca{G}^x)\rtimes_{\rho_H}\mca{G},N\rtimes_\rho\mca{G})^\times$ by Proposition \ref{isoactdiff}. Since $H=\Stab_{\rho_H}(H)$, we have $\theta_\varphi\in\ul{C}^\infty(\mca{G}^x,\mca{G}^x)_{H,\Stab_\rho(y)}^\times\subset\ul{C}^\infty(\mca{G}^x,\mca{G}^x)_{H,\mca{G}^x_x}^\times$ by 3 of Corollary \ref{homhomtimessim}. Define 
\begin{align*}
\MAG(\rho_H)&\to\ul{\Diff}(\mca{G}^x)_{\mca{G}^x_x}\bs(\ul{\Diff}(\mca{G}^x)_H/{\sim})\\
(N\rtimes_\rho\mca{G},\varphi)&\mapsto[\theta_\varphi]. 
\end{align*}

Let $(N\rtimes_\rho\mca{G},y,\varphi)$, $(N^\prime\rtimes_{\rho^\prime}\mca{G},y^\prime,\varphi^\prime)\in\MAG(\rho_H)$ be such that $(N\rtimes_\rho\mca{G},y,\varphi)\mathrel{\ul{\sim_{o,0}}}(N^\prime\rtimes_{\rho^\prime}\mca{G},y^\prime,\varphi^\prime)$. There exists a conjugacy $\psi=F_\psi\rtimes\Psi\colon N\rtimes_\rho\mca{G}\to N^\prime\rtimes_{\rho^\prime}\mca{G}$ such that: 
\begin{itemize}
\setlength\itemsep{0em}
\item $\Psi\in\ul{\Aut}(\mca{G})$
\item $F_\psi(y)=y^\prime$
\item $\psi\varphi\mathrel{\ul{\sim_{o,0}}}\varphi^\prime$ in $\ul{\Hom}_0((H\bs\mca{G}^x)\rtimes_{\rho_H}\mca{G},N^\prime\rtimes_{\rho^\prime}\mca{G})$. 
\end{itemize}
Then $\theta_{\psi\varphi}\mathrel{\ul{\sim_{o,0}}}\theta_{\varphi^\prime}$ in $\ul{C}^\infty(\mca{G}^x,\mca{G}^x)_{H,\Stab_{\rho^\prime}(y^\prime)}$ by 4 in Corollary \ref{homhomtimessim}. We have 
\begin{align*}
\psi\varphi(\{H\}\times H)&=\psi\varphi((H\bs\mca{G}^x)\rtimes_{\rho_H}\mca{G})^H_H=(N^\prime\rtimes_{\rho^\prime}\mca{G})^{y^\prime}_{y^\prime}\\
&=\varphi^\prime((H\bs\mca{G}^x)\rtimes_{\rho_H}\mca{G})^H_H=\varphi^\prime(\{H\}\times H)
\end{align*}
since $\psi\varphi$, $\varphi^\prime$ are isomorphisms and $F_{\psi\varphi}(H)=y^\prime=F_{\varphi^\prime}(H)$. By applying $\varphi_{\rho^\prime}$, we get $\theta_{\psi\varphi}(H)=\Stab_{\rho^\prime}(y^\prime)=\theta_{\varphi^\prime}(H)$. Since $\theta_{\psi\varphi}\mathrel{\ul{\sim_{o,0}}}\theta_{\varphi^\prime}$ in $\ul{C}^\infty(\mca{G}^x,\mca{G}^x)_{H,\theta_{\psi\varphi}(H)}$, $\theta_{\psi\varphi}\sim\theta_{\varphi^\prime}$ in $\ul{C}^\infty(\mca{G}^x,\mca{G}^x)_{H,\mca{G}^x_x}^\times$. We have 
\begin{equation*}
\theta_{\psi\varphi}(g)=\varphi_{\rho^\prime}\psi\varphi(H,g)=\Psi\varphi_\rho\varphi(H,g)=\Psi\theta_\varphi(g)
\end{equation*}
for $g\in\mca{G}^x$. Hence $\theta_{\psi\varphi}=\Psi\theta_\varphi$. Therefore $[\theta_\varphi]=[\theta_{\varphi^\prime}]$ in $\ul{\Diff}(\mca{G}^x)_{\mca{G}^x_x}\bs(\ul{\Diff}(\mca{G}^x)_H/{\sim})$. So we get a map 
\begin{equation}\label{tordgx}
T_o(\rho_H)\to\ul{\Diff}(\mca{G}^x)_{\mca{G}^x_x}\bs(\ul{\Diff}(\mca{G}^x)_H/{\sim}). 
\end{equation}

Conversely let $\theta\in\ul{\Diff}(\mca{G}^x)_H=\ul{C}^\infty(\mca{G}^x,\mca{G}^x)_{H,\mca{G}^x_x}^\times$. Then $\theta(H)$ is a closed subgroup of $\mca{G}^x_x$. Since $\theta\in\ul{C}^\infty(\mca{G}^x,\mca{G}^x)_{H,\theta(H)}^\times$, we get 
$\varphi_\theta\in\ul{\Hom}_0((H\bs\mca{G}^x)\rtimes_{\rho_H}\mca{G},(\theta(H)\bs\mca{G}^x)\rtimes_{\rho_{\theta(H)}}\mca{G})^\times$ by 3 in Corollary \ref{homhomtimessim}. Since $\ol{\theta}\colon H\bs\mca{G}^x\to\theta(H)\bs\mca{G}^x$ is a diffeomorphism, $\varphi_\theta$ is an isomorphism by Proposition \ref{isoactdiff}. 
%
%
%
%
Thus we get a map 
\begin{align*}
\ul{\Diff}(\mca{G}^x)_H&\to\MAG(\rho_H)\\
\theta&\mapsto[(\theta(H)\bs\mca{G}^x)\rtimes_{\rho_{\theta(H)}}\mca{G},\varphi_\theta]. 
\end{align*}

Let $\theta$, $\theta^\prime\in\ul{\Diff}(\mca{G}^x)_H$ be such that $\theta\sim\theta^\prime$. So we have $\theta(H)=\theta^\prime(H)$ and $\theta\mathrel{\ul{\sim_{o,0}}}\theta^\prime$ in $\ul{C}^\infty(\mca{G}^x,\mca{G}^x)_{H,\theta^\prime(H)}$. Hence 
\begin{equation*}
\varphi_\theta\mathrel{\ul{\sim_{o,0}}}\varphi_{\theta^\prime}\ \ \text{in}\ \ \ul{\Hom}_0((H\bs\mca{G}^x)\rtimes_{\rho_H}\mca{G},(\theta^\prime(H)\bs\mca{G}^x)\rtimes_{\rho_{\theta^\prime(H)}}\mca{G})
\end{equation*}
by 4 in Corollary \ref{homhomtimessim}. So we get a map 
\begin{equation*}
\ul{\Diff}(\mca{G}^x)_H/{\sim}\to T_o(\rho_H). 
\end{equation*}

Let $\theta\in\ul{\Diff}(\mca{G}^x)_H$ and $\Theta\in\ul{\Diff}(\mca{G}^x)_{\mca{G}^x_x}$. Then $\theta\in\ul{C}^\infty(\mca{G}^x,\mca{G}^x)_{H,\theta(H)}^\times$ and $\Theta\in\ul{C}^\infty(\mca{G}^x,\mca{G}^x)_{\theta(H),\Theta\theta(H)}^\times$. We have 
\begin{equation*}
\varphi_\Theta=\varphi_{\Theta,\theta(H),\Theta\theta(H)}\in\ul{\Hom}_0((\theta(H)\bs\mca{G}^x)\rtimes_{\rho_{\theta(H)}}\mca{G},(\Theta\theta(H)\bs\mca{G}^x)\rtimes_{\rho_{\Theta\theta(H)}}\mca{G})^\times
\end{equation*}
by 3 in Corollary \ref{homhomtimessim}. $\Theta\theta\in\ul{C}^\infty(\mca{G}^x,\mca{G}^x)_{H,\Theta\theta(H)}^\times$ and $\varphi_{\Theta\theta}=\varphi_\Theta\varphi_\theta$ by Lemma \ref{compos}: 
\begin{equation*}
\begin{tikzcd}[row sep=tiny]
&(\theta(H)\bs\mca{G}^x)\rtimes_{\rho_{\theta(H)}}\mca{G}\ar[dd,"\varphi_\Theta"]\\
(H\bs\mca{G}^x)\rtimes_{\rho_H}\mca{G}\ar[ru,"\varphi_\theta"]\ar[rd,"\varphi_{\Theta\theta}"']\ar[r,phantom,"\circlearrowright"pos=.65]&\ \\
&(\Theta\theta(H)\bs\mca{G}^x)\rtimes_{\rho_{\Theta\theta(H)}}\mca{G}. 
\end{tikzcd}
\end{equation*}
Thus $\varphi_\Theta\varphi_\theta\mathrel{\ul{\sim_{o,0}}}\varphi_{\theta^\prime}$. Since $\Theta\in\ul{C}^\infty(\mca{G}^x,\mca{G}^x)_{\mca{G}^x_x,\mca{G}^x_x}^\times$ as well, there corresponds 
\begin{equation*}
\ol{\varphi_\Theta}=\varphi_{\Theta,\mca{G}^x_x,\mca{G}^x_x}\in\ul{\Hom}_0(M\rtimes\mca{G},M\rtimes\mca{G})^\times=\ul{\Aut}(\mca{G})
\end{equation*}
by 3 in Corollary \ref{homhomtimessim}. We have 
\begin{gather*}
\ol{\varphi_\Theta}(\mca{G}^x_xg,g^\prime)=(\mca{G}^x_x\Theta(g),\Theta(g)^{-1}\Theta(gg^\prime))=(\mca{G}^x_xg,\Theta(g)^{-1}\Theta(gg^\prime)), \\
\varphi_\Theta(\theta(H)g,g^\prime)=(\Theta\theta(H)\Theta(g),\Theta(g)^{-1}\Theta(gg^\prime)). 
\end{gather*}
Hence $\varphi_\Theta=\ol{\Theta}\rtimes\ol{\varphi_\Theta}$ and $\varphi_\Theta$ is a conjugacy. Thus 
\begin{equation*}
[(\theta(H)\bs\mca{G}^x)\rtimes_{\rho_{\theta(H)}}\mca{G},\varphi_\theta]=[(\Theta\theta(H)\bs\mca{G}^x)\rtimes_{\rho_{\Theta\theta(H)}}\mca{G},\varphi_{\Theta\theta}]\ \ \text{in}\ \ T_o(\rho_H). 
\end{equation*}
So we get a map 
\begin{equation}\label{dgxdgxtr}
\ul{\Diff}(\mca{G}^x)_{\mca{G}^x_x}\bs(\ul{\Diff}(\mca{G}^x)_H/{\sim})\to T_o(\rho_H). 
\end{equation}

We need to check the maps \eqref{tordgx} and \eqref{dgxdgxtr} are inverse to each other. First we have 
\begin{equation*}
[\theta]\mapsto[(\theta(H)\bs\mca{G}^x)\rtimes_{\rho_{\theta(H)}}\mca{G},\varphi_\theta]\mapsto[\theta_{\varphi_\theta}]. 
\end{equation*}
Since $\theta_{\varphi_\theta}=\theta$, this composition is the identity. The other one is 
\begin{equation*}
[N\rtimes_\rho\mca{G},\varphi]\mapsto[\theta_\varphi]\mapsto[(\theta_\varphi(H)\bs\mca{G}^x)\rtimes_{\rho_{\theta_\varphi(H)}}\mca{G},\varphi_{\theta_\varphi}]. 
\end{equation*}
Since 
\begin{equation*}
\begin{tikzcd}[row sep=tiny]
&(\theta_\varphi(H)\bs\mca{G}^x)\rtimes_{\rho_{\theta_\varphi(H)}}\mca{G}\ar[dd]\\
(H\bs\mca{G}^x)\rtimes_{\rho_H}\mca{G}\ar[ru,"\varphi_{\theta_\varphi}"]\ar[rd,"\varphi"']\ar[r,phantom,"\circlearrowright"pos=.65]&\ \\
&N\rtimes_\rho\mca{G}, 
\end{tikzcd}
\end{equation*}
$[(\theta_\varphi(H)\bs\mca{G}^x)\rtimes_{\rho_{\theta_\varphi(H)}}\mca{G},\varphi_{\theta_\varphi}]=[N\rtimes_\rho\mca{G},\varphi]$. 
\end{proof}

\section{The Teichm\"{u}ller spaces of transitive actions: Examples}\label{1717}
\subsection{The Teichm\"{u}ller space of $G\curvearrowleft G$}\label{1515}
Let $G$ be a Lie group and $G\stackrel{\rho}{\curvearrowleft}G$ be the action by right multiplication. Then we have $T_{nt}(\rho)=T_o(\rho)$ by Corollary \ref{tnteqto}, and by Corollary \ref{tblgcgsgxx}, 
\begin{equation*}
T_o(\rho)\simeq\Aut(G)\bs(\Diff(G,1)/\text{based homotopy}), 
\end{equation*}
where $\Diff(G,1)=\left\{\theta\in\Diff(G)\ \middle|\ \theta(1)=1\right\}$. 

\begin{lem}
If $G$ is connected, then the map 
\begin{equation*}
\Diff(G,1)/\text{based homotopy}\to\Diff(G)/\text{homotopy}
\end{equation*}
induced by the inclusion is bijective. Hence 
\begin{equation*}
T_o(\rho)\simeq\Aut(G)\bs(\Diff(G)/\text{homotopy}). 
\end{equation*}
\end{lem}

\begin{proof}
It is surjective. We prove injectivity. Let $\theta$, $\theta^\prime\in\Diff(G,1)$ be such that $\theta$ is homotopic to $\theta^\prime$. Let $c\colon I\to G$ be the loop obtained from a homotopy between $\theta$ and $\theta^\prime$ by evaluating $1\in G$. We have an action $[G,G]_0\curvearrowleft\pi_1(G,1)$ (see Section \ref{onemore}) and $[\theta]=[\theta^\prime][c]$ in $[G,G]_0$. The figure below shows $[\theta^\prime][c]=[\theta^\prime]$. 
\begin{equation*}
\begin{tikzpicture}[every label/.append style={font=\scriptsize},matrix of math nodes,decoration={markings,mark=at position0.5with{\arrow{>}}}]
\def\a{2};
\def\b{1};
\def\c{1.4};
\node(1)at(0,0)[label=left:0]{};
\node(2)at(\a,0)[label=right:\theta^\prime]{};
\node(3)at(\a,\a)[label=right:\theta^\prime]{};
\node(4)at(0,\a)[label=left:1]{};
\node(5)at(\b,0)[label=below:1]{};
\node(6)at(\b,\a)[label=above:1]{};
\node(7)at(0,\c)[label=left:t]{};
\node(8)at(\a,\c)[label=right:R_{c(t)}\theta^\prime]{};
\draw(1.center)--(2.center)--(3.center)--(4.center)--(1.center);
\draw[postaction={decorate}](5.center)--node[label=right:c]{}(6.center);
\draw(7.center)--(8.center);
\draw[->](\a+1.2,\a/2)--(\a+2,\a/2);
\node at(\a+2.9,\a/2){G};
\end{tikzpicture}
\end{equation*}
Hence $\theta$ and $\theta^\prime$ are based homotopic. 
\end{proof}

\begin{example}
If $G$ is contractible, then $\lvert T_o(\rho)\rvert=1$. 
\end{example}

\begin{example}\label{tnto}
For $G=\bb{T}^n$, we have $\lvert T_o(\rho)\rvert=1$. 

\begin{proof}
We have $\Aut(G)\simeq\GL(n,\bb{Z})$ and $\Aut(G)\simeq\Diff(G)/\text{homotopy}$. 
\end{proof}
\end{example}

\begin{example}
For $G=\SU(2)$, we have $\lvert T_o(\rho)\rvert=2$. 

\begin{proof}
We have a diffeomorphism $G\simeq S^3$. Hence 
\begin{align*}
\Diff(G)/\text{homotopy}&\simeq\left\{\pm1\right\}\\
[\theta]&\mapsto\deg\theta. 
\end{align*}
For $\theta\in\Diff(G)$, $\deg\theta=1$ if and only if $\theta$ is orientation preserving. Since $\Aut(G)$ consists of inner automorphisms, which must be orientation preserving, 
\begin{equation*}
\Aut(G)\bs(\Diff(G)/\text{homotopy})\simeq\left\{\pm1\right\}. \qedhere
\end{equation*}
\end{proof}
\end{example}

\begin{example}
For $G=\PSL(2,\bb{R})$, we have $\lvert T_o(\rho)\rvert=2$. 
\end{example}

\subsection{The space of marked lattices}\label{marked lattices}\label{1616}
\begin{dfn}
Let $\Gamma$ be a group and $G$ be a Lie group. Put 
\begin{equation*}
\mca{H}(\Gamma,G)=\left\{\alpha\in\Hom(\Gamma,G)\ \middle|\ \begin{gathered}
\alpha\text{ is injective}, \\
\alpha(\Gamma)\text{ is discrete in }G
\end{gathered}\right\}. 
\end{equation*}
We have an action $\Aut(G)\curvearrowright\mca{H}(\Gamma,G)$ by post composition. 
\end{dfn}

\begin{prop}\label{cocompact}
Let $G$ be a connected Lie group, $\Gamma$ be a torsion free discrete subgroup of $G$ and $\alpha\in\mca{H}(\Gamma,G)$. Then $\Gamma$ is cocompact in $G$ if and only if $\alpha(\Gamma)$ is cocompact in $G$. 
\end{prop}

\begin{proof}
There exists a maximal compact subgroup $K$ of $G$ and $G/K$ is contractible (Cartan--Iwasawa--Malcev). The actions $\Gamma\curvearrowright G/K$, $\alpha(\Gamma) \curvearrowright G/K$ by left multiplication are properly discontinuous and free, so $\Gamma\bs G/K$, $\alpha(\Gamma)\bs G/K$ are connected orientable manifolds. Let $d=\dim G/K$. We have 
\begin{equation*}
H_d(\Gamma;\bb{Z})\simeq H_d(\Gamma\bs G/K;\bb{Z}),\quad H_d(\alpha(\Gamma);\bb{Z})\simeq H_d(\alpha(\Gamma)\bs G/K;\bb{Z}). 
\end{equation*}
Therefore, 
\begin{align*}
\Gamma\text{ is cocompact}&\Leftrightarrow\Gamma\bs G/K\text{ is compact}\\
&\Leftrightarrow H_d(\Gamma\bs G/K;\bb{Z})\simeq\bb{Z}\\
&\Leftrightarrow H_d(\alpha(\Gamma)\bs G/K;\bb{Z})\simeq\bb{Z}\\
&\Leftrightarrow\alpha(\Gamma)\bs G/K\text{ is compact}\\
&\Leftrightarrow\alpha(\Gamma)\text{ is cocompact}. \qedhere
\end{align*}
\end{proof}

\begin{dfn}
Let $(X,x)$ be a pointed topological space. For $[c]$, $[c^\prime]\in\pi_1(X,x)$, their product is defined as $[c][c^\prime]=[cc^\prime]$, where $cc^\prime$ is the curve that follows $c^\prime$ first and then follows $c$ both with twice the speed. We chose this convention so as to make it compatible with the definition of the composition of elements in a monodromy groupoid. Note that this is different from the definition which is commonly used. 
\end{dfn}

\begin{dfn}\label{gxgxpxx}
Let $\Gamma$ be a group, $X$ be a simply connected, locally path connected topological space and $\Gamma\curvearrowright X$ be a continuous, properly discontinuous, free action. Take $x\in X$. Then $p\colon(X,x)\to(\Gamma\bs X,\Gamma x)$ is a covering map. 

Let $\Pi_1(\Gamma\bs X)\rightrightarrows\Gamma\bs X$ be the fundamental groupoid of $\Gamma\bs X$. Then $\pi_1(\Gamma\bs X,\Gamma x)=\Pi_1(\Gamma\bs X)^{\Gamma x}_{\Gamma x}$ and $s\colon\Pi_1(\Gamma\bs X)^{\Gamma x}\to\Gamma\bs X$ is a left principal $\pi_1(\Gamma\bs X,\Gamma x)$ bundle. This covering is isomorphic to $p$: 
\begin{equation*}
\begin{tikzcd}[column sep=tiny]
\Pi_1(\Gamma\bs X)^{\Gamma x}\ar[rr,""{name=U}]\ar[rd,"s"']&&X\ar[ld,"p"]\\
&\Gamma\bs X\ar[to=U,phantom,"\circlearrowright"{xshift=-3}]
\end{tikzcd}
\end{equation*}
by 
\begin{equation}\label{pgxgxx}
\begin{aligned}
\Pi_1(\Gamma\bs X)^{\Gamma x}&\simeq X\\
[c]&\mapsto\wt{c}(0), 
\end{aligned}
\end{equation}
where $\wt{c}\colon I\to X$ is the lift of $c$ such that $\wt{c}(1)=x$. To compute the induced action $\pi_1(\Gamma\bs X,\Gamma x)\curvearrowright X$, let $[c]\in\pi_1(\Gamma\bs X,\Gamma x)$ and $y\in X$. Let $\wt{c^\prime}\colon I\to X$ be a continuous curve such that $\wt{c^\prime}(1)=x$, $\wt{c^\prime}(0)=y$. Put $c^\prime=p\wt{c^\prime}$. Then $[c^\prime]\in\Pi_1(\Gamma\bs X)^{\Gamma x}$ and $[c^\prime]\mapsto y$ by \eqref{pgxgxx}. We have $[c][c^\prime]=[cc^\prime]\mapsto\wt{cc^\prime}(0)$ by \eqref{pgxgxx}, where $\wt{cc^\prime}\colon I\to X$ is a continuous curve such that $\wt{cc^\prime}(1)=x$. Hence the action $\pi_1(\Gamma\bs X,\Gamma x)\curvearrowright X$ is given by $[c]y=\wt{cc^\prime}(0)$. 

Therefore we have an isomorphism 
\begin{equation*}
\iota_x\colon\pi_1(\Gamma\bs X,\Gamma x)\to\Gamma
\end{equation*}
defined as follows: For $[c]\in\pi_1(\Gamma\bs X,\Gamma x)$, let $\wt{c}$ be the lift of $c$ such that $\wt{c}(1)=x$. Then $\wt{c}(0)=(\iota_x[c])x$ for a unique element $\iota_x[c]\in\Gamma$. 
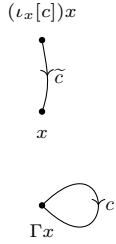
\begin{equation*}
\begin{tikzpicture}[every label/.append style={font=\scriptsize},matrix of math nodes,decoration={markings,mark=at position0.5with{\arrow{>}}}]
\clip(-.7,-1.8)rectangle(1.1,1.6);
\node(1)[label=below:x]{};
\node(2)[above=2em of 1,label=above:(\iota_x{[}c{]})x]{};
\node(3)[below=of 1,label=below:\Gamma x]{};
\node(4)at(3){};
\draw[postaction={decorate}](2.center)to[out=-80,in=70]node[right,scale=.75]{\wt{c}}(1.center);
\draw[postaction={decorate}](3.center)..controls+(1,1)and+(1,-1)..(4.center)node[pos=.5,right,scale=.75]{c};
\foreach\x in{1,2,3}\filldraw(\x)circle(1pt);
\end{tikzpicture}
\end{equation*}
\end{dfn}

\begin{dfn}
Let $G$ be a Lie group, $\Gamma$ be a discrete subgroup of $G$ and $\Gamma\bs G\stackrel{\rho_\Gamma}{\curvearrowleft}G$ be the action by right multiplication. By Corollary \ref{tblgcgsgxx}, we have $T_o(\rho_\Gamma)\simeq\Aut(G)\bs(\Diff(G)_\Gamma/{\sim})$. Consider the map 
\begin{align*}
\Diff(G)_\Gamma&\to\mca{H}(\Gamma,G)\\
\theta&\mapsto\theta|_\Gamma. 
\end{align*}
It induces a map 
\begin{equation*}
r\colon\Diff(G)_\Gamma/{\sim}\to\mca{H}(\Gamma,G). 
\end{equation*}
This is $\Aut(G)$-equivariant, hence we get 
\begin{equation*}
\Aut(G)\bs(\Diff(G)_\Gamma/{\sim})\to\Aut(G)\bs\mca{H}(\Gamma,G). 
\end{equation*}
\end{dfn}

\begin{lem}\label{injective}
Let $G$ be a Lie group and $\Gamma$ be a discrete subgroup of $G$. If $G$ is contractible, then $r\colon\Diff(G)_\Gamma/{\sim}\to\mca{H}(\Gamma,G)$ is injective. 
\end{lem}

\begin{proof}
Let $\theta_1$, $\theta_2\in\Diff(G)_\Gamma$ be such that $\theta_1|_\Gamma=\theta_2|_\Gamma$. Then $\theta_1(\Gamma)=\theta_2(\Gamma)$ is a discrete subgroup of $G$ and $\theta_i$ induces a map $\ol{\theta_i}$: 
\begin{equation*}
\begin{tikzcd}
G\ar[r,"\theta_i"]\ar[d]\ar[rd,phantom,"\circlearrowright"{xshift=2,yshift=-2}]&G\ar[d]\\
\Gamma\bs G\ar[r,"\ol{\theta_i}"']&\theta_i(\Gamma)\bs G. 
\end{tikzcd}
\end{equation*}

\begin{claim}
We have 
\begin{equation*}
\begin{tikzcd}
\pi_1(\Gamma\bs G,\Gamma)\ar[r,"\ol{\theta_i}_*"]\ar[d,"\sim"{sloped},"\iota_1"']\ar[rd,phantom,"\circlearrowright"{xshift=-7}]&\pi_1(\theta_i(\Gamma)\bs G,\theta_i(\Gamma))\ar[d,"\sim"{sloped},"\iota_1"']\\
\Gamma\ar[r,"\theta_i|_\Gamma"']&\theta_i(\Gamma). 
\end{tikzcd}
\end{equation*}
Note that vertical maps are defined since $G$ is simply connected. 
\end{claim}

\begin{proof}
Let $\gamma\in\Gamma$. Then the commutativity can be observed from the following figure. 
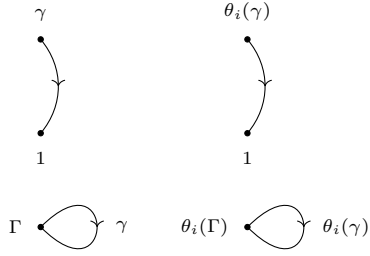
\begin{equation*}
\begin{tikzpicture}[every label/.append style={font=\scriptsize},matrix of math nodes,decoration={markings,mark=at position0.5with{\arrow{>}}}]
\clip(-.5,-1.7)rectangle(4.5,2);
\node(1)[label=below:1]{};
\node(2)[above=1cm of 1,label=above:\gamma]{};
\node(3)[below=1cm of 1,label=left:\Gamma]{};
\node(4)at(3){};
\draw[postaction={decorate}](2.center)to[out=-50,in=50](1.center);
\draw[postaction={decorate}](3.center)..controls+(1,1)and+(1,-1)..(4.center)node[pos=.5,label=right:\gamma]{};
\node(5)[right=2.5cm of 1,label=below:1]{};
\node(6)[above=1cm of 5,label={\theta_i(\gamma)}]{};
\node(7)[below=1cm of 5,label=left:\theta_i(\Gamma)]{};
\node(8)at(7){};
\draw[postaction={decorate}](6.center)to[out=-50,in=50](5.center);
\draw[postaction={decorate}](7.center)..controls+(1,1)and+(1,-1)..(8.center)node[pos=.5,label=right:\theta_i(\gamma)]{};
\foreach\x in{1,2,3,5,6,7}\filldraw(\x)circle(1pt);
\end{tikzpicture}\qedhere
\end{equation*}
\end{proof}

Hence $\ol{\theta_1}_*=\ol{\theta_2}_*\colon\pi_1(\Gamma\bs G,\Gamma)\to\pi_1(\theta_1(\Gamma)\bs G,\theta_1(\Gamma))$. Since $\theta_1(\Gamma)\bs G$ is aspherical, $\ol{\theta_1}$ is based homotopic to $\ol{\theta_2}$. See for example Proposition 1B.9 in \cite{Hatcher}. So there exists a continuous map $H\colon(\Gamma\bs G)\times I\to\theta_1(\Gamma)\bs G$ such that: 
\begin{itemize}
\setlength\itemsep{0em}
\item $H(\Gamma,t)=\theta_1(\Gamma)$ for $t\in I$
\item $H(\cdot,t)$ is $C^\infty$ for $t\in I$
\item $H(\cdot,0)=\ol{\theta_1}$, $H(\cdot,1)=\ol{\theta_2}$. 
\end{itemize}
Then there exists a unique continuous map $\wt{H}\colon G\times I\to G$ such that $\wt{H}(1,0)=1$ and 
\begin{equation}\label{hh}
\begin{tikzcd}
G\times I\ar[r,"\wt{H}"]\ar[d]\ar[rd,phantom,"\circlearrowright"{xshift=2,yshift=-1}]&G\ar[d]\\
(\Gamma\bs G)\times I\ar[r,"H"']&\theta_1(\Gamma)\bs G. 
\end{tikzcd}
\end{equation}
We have $\wt{H}(\cdot,t)$ is $C^\infty$ for $t\in I$, and $\wt{H}(\cdot,0)=\theta_1$. 

\begin{claim}
$\wt{H}(\gamma,t)=\theta_1(\gamma)$ for $\gamma\in\Gamma$ and $t\in I$. 
\end{claim}

\begin{proof}
In Diagram \eqref{hh}, $\{\gamma\}\times I$ above left goes to $\{\Gamma\}\times I$ below left and then to the point $\theta_1(\Gamma)$ below right. Hence $\wt{H}(\{\gamma\}\times I)$ must be a point since the right vertical map is a covering. So $\wt{H}(\gamma,t)=\wt{H}(\gamma,0)=\theta_1(\gamma)$ for all $t\in I$. 
\end{proof}

In particular, we have $\wt{H}(1,1)=1$, hence $\wt{H}(\cdot,1)=\theta_2$. 

\begin{claim}
$\wt{H}(\gamma g,t)=\theta_1(\gamma)\wt{H}(g,t)$ for $\gamma\in\Gamma$, $g\in G$, $t\in I$. 
\end{claim}

\begin{proof}
We have 
\begin{equation*}
\begin{tikzcd}
G\ar[r,"\gamma\ \cdot"]\ar[rd]&G\ar[r,"\wt{H}(\cdot{,}t)"]\ar[d]\ar[rd,phantom,"\circlearrowright"{xshift=2,yshift=-2}]\ar[ld,phantom,"\circlearrowright",pos=.15]&G\ar[d]\\
\ &\Gamma\bs G\ar[r,"H(\cdot{,}t)"']&\theta_1(\Gamma)\bs G, 
\end{tikzcd}\qquad
\begin{tikzcd}
G\ar[r,"\wt{H}(\cdot{,}t)"]\ar[d]\ar[rd,phantom,"\circlearrowright"{xshift=2,yshift=-2}]&G\ar[r,"\theta_1(\gamma)\ \cdot"]\ar[d]\ar[rd,phantom,"\circlearrowright",pos=.18]&G\ar[ld]\\
\Gamma\bs G\ar[r,"H(\cdot{,}t)"']&\theta_1(\Gamma)\bs G. &\ 
\end{tikzcd}
\end{equation*}
The top composed maps are lifts of the same base map and the left map sends $1$ to $\wt{H}(\gamma,t)$ and the right map sends $1$ to $\theta_1(\gamma)$, and they are equal, so two maps must be the same. 
\end{proof}

Therefore we get $\wt{H}(\gamma g,t)=\wt{H}(\gamma,t)\wt{H}(g,t)$ for $\gamma\in\Gamma$, $g\in G$, $t\in I$ and hence $\theta_1\sim\theta_2$. 
\end{proof}

\subsection{The Teichm\"{u}ller space of $\Gamma\bs S\curvearrowleft S$}
Let $S$ be a simply connected solvable Lie group, $\Gamma$ be a lattice in $S$ and $\Gamma\bs S\stackrel{\rho_\Gamma}{\curvearrowleft}S$ be the action by right multiplication. It is known that $S$ is diffeomorphic to $\bb{R}^{\dim S}$, $\Gamma$ is torsion free and cocompact in $S$. We have $T_{bnt}(\rho_\Gamma)=T_{nt}(\rho_\Gamma)=T_o(\rho_\Gamma)$ by Proposition \ref{gmncptbnt} and Corollary \ref{tnteqto}. Consider the map 
\begin{equation*}
r\colon\Diff(S)_\Gamma/{\sim}\to\mca{H}(\Gamma,S)
\end{equation*}
defined in Section \ref{marked lattices}. 

\begin{lem}
The map $r$ is surjective. 
\end{lem}

\begin{proof}
Take $\alpha\in\mca{H}(\Gamma,S)$. Then $\alpha(\Gamma)$ is a lattice in $S$ by Proposition \ref{cocompact} and $\alpha$ gives an isomorphism $\Gamma\simeq\alpha(\Gamma)$ between lattices. By $(7{,}4^\prime)$ Corollary of \cite{Witte_1995}, there exists $\theta\in\Diff(S)_\Gamma$ such that $\theta|_\Gamma=\alpha$. 
\end{proof}

\begin{prop}\label{torautshgs}
We have 
\begin{equation*}
T_o(\rho_\Gamma)\simeq\Aut(S)\bs\mca{H}(\Gamma,S). 
\end{equation*}
\end{prop}

\begin{proof}
This follows from Proposition \ref{trans}, Lemma \ref{injective} and the above lemma. 
\end{proof}

Recall that $S$ is of type (R) if all eigenvalues of $\Ad(s)$ are real for any $s\in S$. 

\begin{cor}
If $S$ is of type (R), then $\lvert T_o(\rho_\Gamma)\rvert=1$. 
\end{cor}

\begin{proof}
Take $\alpha\in\mca{H}(\Gamma,S)$. Then $\alpha(\Gamma)$ is a lattice in $S$ by Proposition \ref{cocompact} and $\alpha$ gives an isomorphism $\Gamma\simeq\alpha(\Gamma)$ between lattices. By a theorem of Saito in \cite{Sa}, there exists $\Phi\in\Aut(S)$ such that $\Phi|_\Gamma=\alpha$. 
\end{proof}

\begin{cor}
If $S$ is nilpotent, then $\lvert T_o(\rho_\Gamma)\rvert=1$. 
\end{cor}

\begin{proof}
This is because a nilpotent Lie group is of type (R). One can also apply a theorem by Malcev instead of a theorem by Saito. 
\end{proof}

\section{The Teichm\"{u}ller space of $\Gamma\bs\PSL(2,\bb{R})\curvearrowleft\PSL(2,\bb{R})$}\label{1818}
Let $G=\PSL(2,\bb{R})$, $\Gamma$ be a torsion free cocompact lattice in $G$ and $\Gamma\bs G\stackrel{\rho_\Gamma}{\curvearrowleft}G$ be the action by right multiplication. We have 
\begin{equation*}
T_o(\rho_\Gamma)\simeq T_{nt}(\rho_\Gamma)\simeq T_{bnt}(\rho_\Gamma)
\end{equation*}
since $\rho_\Gamma$ is locally free and $\Gamma\bs G$ is compact (see Corollary \ref{tnteqto} and Proposition \ref{gmncptbnt}). We compute $T_o(\rho_\Gamma)$ in this section. Let $\wt{G}=\wt{\PSL}(2,\bb{R})$ be the universal covering group of $G$ and $p\colon\wt{G}\to G$ be the covering map. Let $Z=\ker p$. Then $Z$ is the center of $\wt{G}$. Since $G$ is homotopy equivalent to $S^1$, $Z\simeq\bb{Z}$ (not canonically). Let $\wt{\Gamma}=p^{-1}(\Gamma)$. We have 
\begin{equation*}
\begin{tikzcd}[row sep=14]
1\ar[r]&Z\ar[r]\ar[d,equal]&\wt{G}\ar[r]&G\ar[r]&1\\
1\ar[r]&Z\ar[r]&\wt{\Gamma}\ar[r]\ar[u,phantom,sloped,"\subset"]&\Gamma\ar[r]\ar[u,phantom,sloped,"\subset"]&1. 
\end{tikzcd}
\end{equation*}
$\wt{\Gamma}\bs\wt{G}\simeq\Gamma\bs G$ and $\wt{\Gamma}$ is a torsion free cocompact lattice in $\wt{G}$. By Borel density theorem (see 5.18 of \cite{Rag}), $Z$ coincides with the center of $\wt{\Gamma}$. If $\alpha\in\mca{H}(\Gamma,G)$ (resp. $\wt{\alpha}\in\mca{H}(\wt{\Gamma},\wt{G})$), then $\alpha(\Gamma)$ (resp. $\wt{\alpha}(\wt{\Gamma})$) is a torsion free cocompact lattice in $G$ (resp. $\wt{G}$) by Proposition \ref{cocompact}. 

What we will prove in this section can be summarized in the following diagram (we do not explain the detail here): 
\begin{equation*}
\begin{tikzcd}[row sep=10]
&[-20]1\ar[dd]&[-40]&[-15]\\
\ \\
\Hom(\wt{\Gamma},Z)\ar[dd]&\Hom(\wt{\Gamma},Z)\ar[dd]\\
\ \\
\Aut(\wt{G})\bs\mca{H}(\wt{\Gamma},\wt{G})\ar[r,phantom,"\curvearrowleft"]\ar[dd,"\ol{\pi}"]&\Out(\wt{\Gamma})\ar[rrd,dash,"\sim"{sloped}]\ar[dd,"\ol{\pi}_0"{yshift=-3}]\\
&&T_o(\rho_\Gamma)\ar[r,phantom,"\curvearrowleft"]&\MCG_o(\rho_\Gamma)\\
\Aut(G)\bs\mca{H}(\Gamma,G)\ar[r,phantom,"\curvearrowleft"]\ar[uu,bend left,"\ol{\sigma}"]&\Out(\Gamma)\ar[rrd,dash,"\sim"{sloped}]\ar[dd]\ar[uu,bend left,"\ol{\sigma}_0"]\\
&&T(\Sigma)\ar[r,phantom,"\curvearrowleft"]&\MCG^\pm(\Sigma)\\
&1. 
\ar[from=5-1,to=6-3,dash,"\sim"{sloped},crossing over]\ar[from=7-1,to=8-3,dash,"\sim"{sloped},crossing over]
\end{tikzcd}
\end{equation*}

\subsection{The construction of $\pi\colon\mca{H}(\wt{\Gamma},\wt{G})\to\mca{H}(\Gamma,G)$}
\begin{lem}\label{descends}
For any $\wt{\alpha}\in\mca{H}(\wt{\Gamma},\wt{G})$, there exists a unique $\alpha\in\mca{H}(\Gamma,G)$ such that 
\begin{equation*}
\begin{tikzcd}
Z\ar[r,"\sim"]\ar[d]\ar[rd,phantom,"\circlearrowright"]&Z\ar[d]\\
\wt{\Gamma}\ar[r,"\wt{\alpha}"]\ar[d]\ar[rd,phantom,"\circlearrowright"]&\wt{G}\ar[d,"p"]\\
\Gamma\ar[r,"\alpha"']&G
\end{tikzcd}
\end{equation*}
and $\wt{\alpha}(\wt{\Gamma})=p^{-1}(\alpha(\Gamma))$. 
\end{lem}

\begin{proof}
Let $\wt{\gamma}\in Z$. Then $\wt{\alpha}(\wt{\gamma})$ commutes with $\wt{\alpha}(\wt{\Gamma})$. Since $\wt{\alpha}(\wt{\Gamma})$ is a lattice in $\wt{G}$, $\wt{\alpha}(\wt{\gamma})\in Z$ by Borel density theorem. Hence $\wt{\alpha}(Z)\subset Z$. So there exists a unique homomorphism $\alpha\colon\Gamma\to G$ such that 
\begin{equation*}
\begin{tikzcd}
Z\ar[r]\ar[d]\ar[rd,phantom,"\circlearrowright"]&Z\ar[d]\\
\wt{\Gamma}\ar[r,"\wt{\alpha}"]\ar[d]\ar[rd,phantom,"\circlearrowright"]&\wt{G}\ar[d]\\
\Gamma\ar[r,"\alpha"']&G. 
\end{tikzcd}
\end{equation*}
By the discreteness of $\wt{\alpha}(\wt{\Gamma})$ in $\wt{G}$, we can show that $\alpha(\Gamma)$ is discrete in $G$. We have $\wt{\alpha}(\wt{\Gamma})\cap Z=\wt{\alpha}(Z)$. So $\alpha$ is injective, hence $\alpha\in\mca{H}(\Gamma,G)$. The situation is as follows: 
\begin{equation}\label{diagram}
\begin{tikzcd}[row sep=small,column sep=small]
1\ar[d]&1\ar[d]&1\ar[d]\\
Z\ar[r,"\sim"]\ar[d]&\wt{\alpha}(Z)\ar[r,hook]\ar[d]&Z\ar[d]\\
\wt{\Gamma}\ar[r,"\sim"]\ar[d]&\wt{\alpha}(\wt{\Gamma})\ar[r,hook]\ar[d]&\wt{\alpha(\Gamma)}\ar[d]\\
\Gamma\ar[r,"\sim"]\ar[d]&\alpha(\Gamma)\ar[r,equal]\ar[d]&\alpha(\Gamma)\ar[d]\\
1&1&1, 
\end{tikzcd}
\end{equation}
where $\wt{\alpha(\Gamma)}=p^{-1}(\alpha(\Gamma))$. 

Let $\bb{H}^2=\left\{z\in\bb{C}\ \middle|\ \im z>0\right\}$ be the hyperbolic plane, on which $G$ acts through 
\begin{align*}
G&\to\Isom(\bb{H}^2)\\
\pm
\begin{pmatrix}
a&b\\
c&d
\end{pmatrix}
&\mapsto\left(z\mapsto\frac{az+b}{cz+d}\right). 
\end{align*}
The actions $\Gamma\curvearrowright\bb{H}^2$, $\alpha(\Gamma)\curvearrowright\bb{H}^2$ are properly discontinuous and free, so $\Gamma\bs\bb{H}^2$, $\alpha(\Gamma)\bs\bb{H}^2$ are connected oriented closed $2$-manifolds. 

\begin{notation}\label{gh2agh2yb}
In this section objects on $\Gamma\bs\bb{H}^2$ and $\alpha(\Gamma)\bs\bb{H}^2$ coming from objects on $\bb{H}^2$ are denoted by putting overlines on them. For example, if $y\in\bb{H}^2$, the projections of $y$ to $\Gamma\bs\bb{H}^2$ and $\alpha(\Gamma)\bs\bb{H}^2$ are denoted by $\ol{y}$. Whether $\ol{y}$ is a point of $\Gamma\bs\bb{H}^2$ or $\alpha(\Gamma)\bs\bb{H}^2$ will be understood from the context. 
\end{notation}

Let $g\geq2$ be the genus of $\Gamma\bs\bb{H}^2$. Then $\Gamma\bs\bb{H}^2$ is homeomorphic to a surface which is obtained from a $4g$-gon by identifying sides. The homeomorphism is chosen so that $\ol{i}$ corresponds to the unique point coming from the vertices of the $4g$-gon. The figure below shows the case of $g=2$. 
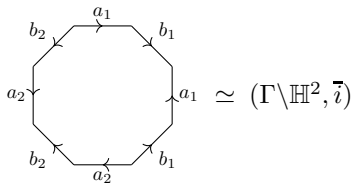
\begin{equation*}
\begin{tikzpicture}[matrix of math nodes,decoration={markings,mark=at position0.5with{\arrow{>}}}]
\foreach\x in{0,...,7}\node(\x)at(45/2+45*\x:10mm){};
\draw[postaction={decorate}](1.center)--node[above right,scale=.75]{b_1}(0.center);
\draw[postaction={decorate}](2.center)--node[above,scale=.75]{a_1}(1.center);
\draw[postaction={decorate}](5.center)--node[below left,scale=.75]{b_2}(4.center);
\draw[postaction={decorate}](6.center)--node[below,scale=.75]{a_2}(5.center);
\draw[postaction={decorate}](2.center)--node[above left,scale=.75]{b_2}(3.center);
\draw[postaction={decorate}](3.center)--node[left,scale=.75]{a_2}(4.center);
\draw[postaction={decorate}](6.center)--node[below right,scale=.75]{b_1}(7.center);
\draw[postaction={decorate}](7.center)--node(8)[right,scale=.75]{a_1}(0.center);
\node[right=0em of 8]{\simeq\ (\Gamma\bs\bb{H}^2,\ol{i})};
\end{tikzpicture}
\end{equation*}
Then we have 
\begin{equation*}
\Gamma=\left\langle a_1,\ldots,a_g,b_1,\ldots,b_g\ \middle|\ [a_1,b_1]\cdots[a_g,b_g]=1\right\rangle, 
\end{equation*}
where $[a_i,b_i]=a_ib_ia_i^{-1}b_i^{-1}$. Take $\wt{a}_i$, $\wt{b}_i\in\wt{\Gamma}$ such that $p\wt{a}_i=a_i$, $p\wt{b}_i=b_i$. Let $c$ be a generator of $Z$. 

\begin{claim}
We have 
\begin{equation*}
[\wt{a}_1,\wt{b}_1]\cdots[\wt{a}_g,\wt{b}_g]=c^{\pm(2-2g)}. 
\end{equation*}
\end{claim}

\begin{proof}
We have a right principal $G$ bundle $p\colon\bb{H}^2\times_\Gamma G\to\Gamma\bs\bb{H}^2$, where $\Gamma$ acts on $G$ by left multiplication. Since $\pi_1(G)$ is abelian, the fundamental group of a fiber of $p$ does not depend on a choice of a base point, hence the local system $\pi_1(p^{-1}(\cdot))$ on $\Gamma\bs\bb{H}^2$ is defined. Since $G$ is connected, the classifying space $BG$ is simply connected. Hence the local system $\pi_1(p^{-1}(\cdot))$ is trivial. Let $c_\Gamma\in H^2(\Gamma\bs\bb{H}^2;\pi_1(p^{-1}(\cdot)))\simeq H^2(\Gamma\bs\bb{H}^2;\pi_1(G))$ be the first obstruction to the existence of a continuous section of $\bb{H}^2\times_\Gamma G\to\Gamma\bs\bb{H}^2$. Let $[\Gamma\bs\bb{H}^2]\in H_2(\Gamma\bs\bb{H}^2;\bb{Z})$ be the fundamental class of $\Gamma\bs\bb{H}^2$. By Lemma 2 of \cite{Milnor_1958}, $c_\Gamma[\Gamma\bs\bb{H}^2]$ corresponds to $([\wt{a}_1,\wt{b}_1]\cdots[\wt{a}_g,\wt{b}_g])^{\pm1}$ through the isomorphism $\iota_1\colon\pi_1(G)\simeq Z$. Let $\partial\bb{H}^2=\bb{R}\cup\left\{\infty\right\}$. We have an action $G\curvearrowright\partial\bb{H}^2$. Let $\bb{H}^2\times_\Gamma\partial\bb{H}^2\to\Gamma\bs\bb{H}^2$ be the associated bundle. The map 
\begin{align*}
G&\to\partial\bb{H}^2\\
g&\mapsto g0
\end{align*}
is a $\Gamma$ equivariant homotopy equivalence since 
\begin{equation*}
\begin{bmatrix}
\cos\theta&\sin\theta\\
-\sin\theta&\cos\theta
\end{bmatrix}
0=\tan\theta. 
\end{equation*}
This induces a fiberwise homotopy equivalence $\bb{H}^2\times_\Gamma G\to\bb{H}^2\times_\Gamma\partial\bb{H}^2$. Therefore $c_\Gamma$ coincides with the first obstruction to the existence of a continuous section of $\bb{H}^2\times_\Gamma\partial\bb{H}^2\to\Gamma\bs\bb{H}^2$. Since $\bb{H}^2\times_\Gamma\partial\bb{H}^2\to\Gamma\bs\bb{H}^2$ is isomorphic to the unit tangent bundle $T^1(\Gamma\bs\bb{H}^2)\to\Gamma\bs\bb{H}^2$, $c_\Gamma=e(T(\Gamma\bs\bb{H}^2))$, where $e(T(\Gamma\bs\bb{H}^2))\in H^2(\Gamma\bs\bb{H}^2;\bb{Z})$ is the Euler class of $T(\Gamma\bs\bb{H}^2)$. Hence $c_\Gamma[\Gamma\bs\bb{H}^2]=e(T(\Gamma\bs\bb{H}^2))[\Gamma\bs\bb{H}^2]=2-2g$. 
\end{proof}

Since $\Gamma\stackrel{\alpha}{\simeq}\alpha(\Gamma)$, we have 
\begin{equation*}
\alpha(\Gamma)=\left\langle\alpha(a_1),\ldots,\alpha(a_g),\alpha(b_1),\ldots,\alpha(b_g)\ \middle|\ [\alpha(a_1),\alpha(b_1)]\cdots[\alpha(a_g),\alpha(b_g)]=1\right\rangle, 
\end{equation*}
$p\wt{\alpha}(\wt{a}_i)=\alpha(a_i)$, $p\wt{\alpha}(\wt{b}_i)=\alpha(b_i)$ and $\wt{\alpha}(\wt{a}_i)$, $\wt{\alpha}(\wt{b}_i)\in\wt{\alpha(\Gamma)}$. 

\begin{claim}
There exists a homeomorphism $f\colon(\Gamma\bs\bb{H}^2,\ol{i})\to(\alpha(\Gamma)\bs\bb{H}^2,\ol{i})$ such that 
\begin{equation*}
\begin{tikzcd}
\pi_1(\Gamma\bs\bb{H}^2,\ol{i})\ar[r,"f_*"]\ar[d,"\iota_i"']\ar[rd,phantom,"\circlearrowright"{xshift=-3}]&\pi_1(\alpha(\Gamma)\bs\bb{H}^2,\ol{i})\ar[d,"\iota_i"]\\
\Gamma\ar[r,"\alpha"']&\alpha(\Gamma). 
\end{tikzcd}
\end{equation*}
\end{claim}

\begin{proof}
Since $\alpha(\Gamma)\bs\bb{H}^2$ is aspherical, there exists a continuous map $f_0\colon(\Gamma\bs\bb{H}^2,\ol{i})\to(\alpha(\Gamma)\bs\bb{H}^2,\ol{i})$ such that 
\begin{equation*}
\begin{tikzcd}
\pi_1(\Gamma\bs\bb{H}^2,\ol{i})\ar[r,"(f_0)_*"]\ar[d,"\sim"{sloped},"\iota_i"']\ar[rd,phantom,"\circlearrowright"{xshift=-3}]&\pi_1(\alpha(\Gamma)\bs\bb{H}^2,\ol{i})\ar[d,"\iota_i","\sim"'{sloped}]\\
\Gamma\ar[r,"\alpha"']&\alpha(\Gamma). 
\end{tikzcd}
\end{equation*}
See for example Proposition 1B.9 in \cite{Hatcher}. We see that $f_0$ is a homotopy equivalence by using $\alpha^{-1}$ and Proposition 1B.9 in \cite{Hatcher}. 

By Theorem 8.9 (Dehn--Nielsen--Baer theorem) of \cite{FMar}, there exists a homeomorphism $f_1\colon\Gamma\bs\bb{H}^2\to\alpha(\Gamma)\bs\bb{H}^2$ which is homotopic to $f_0$. Let $h_t$ be a homotopy between $f_0$ and $f_1$. There exists a homotopy $h^\prime_t$ between $f_1$ and a homeomorphism $f\colon(\Gamma\bs\bb{H}^2,\ol{i})\to(\alpha(\Gamma)\bs\bb{H}^2,\ol{i})$ such that the loop 
\begin{align*}
I&\to\alpha(\Gamma)\bs\bb{H}^2\\
t&\mapsto 
\begin{cases}
h_{2t}(\ol{i})&0\leq t\leq\frac{1}{2}\\
h^\prime_{2t-1}(\ol{i})&\frac{1}{2}\leq t\leq1
\end{cases}
\end{align*}
is homotopic to the constant loop $\ol{i}$ relative to endpoints. Let $h^{\prime\prime}_s$ be such a homotopy of loops. 
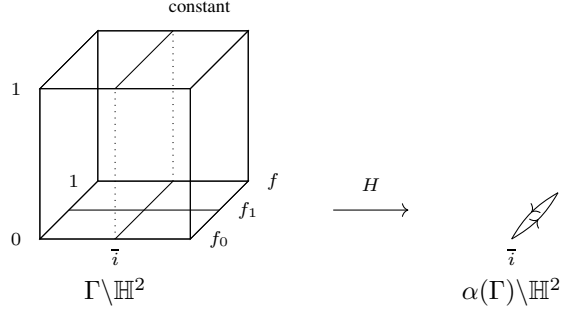
\begin{equation*}
\begin{tikzpicture}[every label/.append style={font=\scriptsize},matrix of math nodes,decoration={markings,mark=at position0.5with{\arrow{>}}}]
\newcommand\x{2}
\newcommand\y{2}
\newcommand\z{2}
\newcommand\vv{2}
\coordinate(O)at(0,0,0);
\coordinate(A)at(0,\y,0);
\coordinate(B)at(0,\y,\z);
\coordinate(C)at(0,0,\z);
\coordinate(D)at(\x,0,0);
\coordinate(E)at(\x,\y,0);
\coordinate(F)at(\x,\y,\z);
\coordinate(G)at(\x,0,\z);
\draw(O)--(C)--(G)--(D)--cycle;
\draw(O)--(A)--(E)--(D)--cycle;
\draw(O)--(A)--(B)--(C)--cycle;
\draw(D)--(E)--(F)--(G)--cycle;
\draw(C)--(B)--(F)--(G)--cycle;
\draw(A)--(B)--(F)--(E)--cycle;

\node(1)at(\x/\vv,0,\z)[label={[yshift=-17]\ol{i}}]{};
\node(2)at(\x/\vv,\y,0)[label={[xshift=10]\text{constant}}]{};
\draw(1.center)--(\x/\vv,0,0);
\draw(\x/\vv,\y,\z)--(2.center);
\draw[dotted](1.center)--(\x/\vv,\y,\z);
\draw[dotted](\x/\vv,0,0)--(2.center);
\node(3)at(\x,0,\z/2)[label=right:f_1]{};
\node at(D)[label=right:f]{};
\draw(0,0,\z/2)--(3.center);
\node at(G)[label=right:f_0]{};
\node at(\x/2,-.7,\z){\Gamma\bs\bb{H}^2};
\node at(C)[label=left:0]{};
\node at(O)[label=left:1]{};
\node at(B)[label=left:1]{};

\draw[->,xshift=100](0,0,\z/2)--node[label=above:H]{}(1,0,\z/2);

\begin{scope}[xshift=150]
\node(4)at(\x/2,0,\z)[label={[yshift=-17]\ol{i}}]{};
\draw[postaction={decorate}](4.center)to[out=30,in=240](\x/2,0,\z/5);
\draw[postaction={decorate}](\x/2,0,\z/5)to[out=210,in=60](\x/2,0,\z);
\node at(\x/2,-.7,\z){\alpha(\Gamma)\bs\bb{H}^2};
\end{scope}
\end{tikzpicture}
\end{equation*}
By the homotopy extension property for the pair 
\begin{equation*}
\left((\Gamma\bs\bb{H}^2)\times I,((\Gamma\bs\bb{H}^2)\times\{0,1\})\cup(\{\ol{i}\}\times I)\right), 
\end{equation*}
there exists a continuous map $H\colon(\Gamma\bs\bb{H}^2)\times I\times I\to\alpha(\Gamma)\bs\bb{H}^2$ such that: 
\begin{itemize}
\setlength\itemsep{0em}
\item $H\left(\cdot,\frac{1}{2}t,0\right)=h_t$, $H\left(\cdot,\frac{1}{2}t+\frac{1}{2},0\right)=h^\prime_t$ for all $t\in I$
\item $H(\ol{i},t,s)=h^{\prime\prime}_s(t)$ for all $t$, $s\in I$
\item $H(\cdot,0,s)=f_0$, $H(\cdot,1,s)=f$ for all $s\in I$. 
\end{itemize}
Then $H(\cdot,t,1)$ is a based homotopy between $f_0$ and $f$. Hence $f$ satisfies the required condition. 
\end{proof}

This shows 
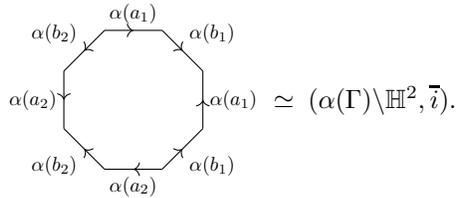
\begin{equation*}
\begin{tikzpicture}[matrix of math nodes,decoration={markings,mark=at position0.5with{\arrow{>}}}]
\foreach\x in{0,...,7}\node(\x)at(45/2+45*\x:10mm){};
\draw[postaction={decorate}](1.center)--node[above right,scale=.75]{\alpha(b_1)}(0.center);
\draw[postaction={decorate}](2.center)--node[above,scale=.75]{\alpha(a_1)}(1.center);
\draw[postaction={decorate}](5.center)--node[below left,scale=.75]{\alpha(b_2)}(4.center);
\draw[postaction={decorate}](6.center)--node[below,scale=.75]{\alpha(a_2)}(5.center);
\draw[postaction={decorate}](2.center)--node[above left,scale=.75]{\alpha(b_2)}(3.center);
\draw[postaction={decorate}](3.center)--node[left,scale=.75]{\alpha(a_2)}(4.center);
\draw[postaction={decorate}](6.center)--node[below right,scale=.75]{\alpha(b_1)}(7.center);
\draw[postaction={decorate}](7.center)--node(8)[right,scale=.75]{\alpha(a_1)}(0.center);
\node[right=0em of 8]{\simeq\ (\alpha(\Gamma)\bs\bb{H}^2,\ol{i}). };
\end{tikzpicture}
\end{equation*}
So as in the first claim, we get 
\begin{equation*}
\wt{\alpha}(c)^{\pm(2-2g)}=[\wt{\alpha}(\wt{a}_1),\wt{\alpha}(\wt{b}_1)]\cdots[\wt{\alpha}(\wt{a}_g),\wt{\alpha}(\wt{b}_g)]=c^{\pm(2-2g)}. 
\end{equation*}
Hence $\wt{\alpha}(c)=c^{\pm1}$ and $\wt{\alpha}(Z)=Z$. 

By diagram \eqref{diagram}, $\wt{\alpha}(\wt{\Gamma})=\wt{\alpha(\Gamma)}$. 
\end{proof}

\begin{lem}
We have an isomorphism 
\begin{align*}
\Aut(\wt{G})&\to\Aut(G)\\
\wt{\Phi}&\mapsto\Phi
\end{align*}
such that 
\begin{equation*}
\begin{tikzcd}
\wt{G}\ar[r,"\wt{\Phi}"]\ar[d]\ar[rd,phantom,"\circlearrowright"]&\wt{G}\ar[d]\\
G\ar[r,"\Phi"']&G. 
\end{tikzcd}
\end{equation*}
\end{lem}

\begin{proof}
For $\wt{\Phi}\in\Aut(\wt{G})$, we have $\wt{\Phi}(Z)=Z$, hence it induces $\Phi\in\Aut(G)$. For the other direction, let $\Phi\in\Aut(G)$. There exists a unique continuous map $\wt{\Phi}\colon\wt{G}\to\wt{G}$ such that $\wt{\Phi}(1)=1$ and 
\begin{equation*}
\begin{tikzcd}
\wt{G}\ar[r,"\wt{\Phi}"]\ar[d]\ar[rd,phantom,"\circlearrowright"]&\wt{G}\ar[d]\\
G\ar[r,"\Phi"']&G. 
\end{tikzcd}
\end{equation*}

\begin{claim}
We have $\wt{\Phi}\left(\wt{g}\wt{g}^\prime\right)=\wt{\Phi}(\wt{g})\wt{\Phi}\left(\wt{g}^\prime\right)$ for $\wt{g}$, $\wt{g}^\prime\in\wt{G}$. 
\end{claim}

\begin{proof}
The bottom composed maps in the diagrams
\begin{equation*}
\begin{tikzcd}
\wt{G}\ar[r,"\cdot\ \wt{g}^\prime"]\ar[d]\ar[rd,phantom,"\circlearrowright"]&\wt{G}\ar[r,"\wt{\Phi}"]\ar[d]\ar[rd,phantom,"\circlearrowright"]&\wt{G}\ar[d]\\
G\ar[r,"\cdot\ p\wt{g}^\prime"']&G\ar[r,"\Phi"']&G, 
\end{tikzcd}\qquad
\begin{tikzcd}
\wt{G}\ar[r,"\wt{\Phi}"]\ar[d]\ar[rd,phantom,"\circlearrowright"]&\wt{G}\ar[r,"\cdot\ \wt{\Phi}(\wt{g}^\prime)"]\ar[d]\ar[rd,phantom,"\circlearrowright"]&\wt{G}\ar[d]\\
G\ar[r,"\Phi"']&G\ar[r,"\cdot\ \Phi(p\wt{g}^\prime)"']&G
\end{tikzcd}
\end{equation*}
coincide and both of the top composed maps take $1$ to $\wt{\Phi}(\wt{g}^\prime)$, they must be the same. 
\end{proof}
\end{proof}

By Lemma \ref{descends}, we obtain a map 
\begin{align*}
\pi\colon\mca{H}(\wt{\Gamma},\wt{G})&\to\mca{H}(\Gamma,G)\\
\wt{\alpha}&\mapsto\alpha. 
\end{align*}
This is $(\Aut(\wt{G})\stackrel{\sim}{\to}\Aut(G))$-equivariant and defines a map 
\begin{equation*}
\ol{\pi}\colon\Aut(\wt{G})\bs\mca{H}(\wt{\Gamma},\wt{G})\to\Aut(G)\bs\mca{H}(\Gamma,G). 
\end{equation*}

\subsection{$T_o(\rho_\Gamma)\simeq\Aut(\wt{G})\bs\mca{H}(\wt{\Gamma},\wt{G})$}
In this section we prove the following theorem. 

\begin{thm}\label{traghg}
We have a natural bijection 
\begin{equation*}
T_o(\rho_\Gamma)\simeq\Aut(\wt{G})\bs\mca{H}(\wt{\Gamma},\wt{G}). 
\end{equation*}
\end{thm}

First, we have 
\begin{equation*}
T_o(\rho_\Gamma)\simeq\Aut(G)\bs(\Diff(G)_\Gamma/{\sim})
\end{equation*}
by Corollary \ref{tblgcgsgxx}. We will compute the right hand side to show Theorem \ref{traghg}. 

\begin{lem}\label{diffggdiffwgwg}
There is a natural bijection 
\begin{align}
\Diff(G)_\Gamma&\simeq\Diff(\wt{G})_{\wt{\Gamma}}\label{dgdgwgw}\\
\theta&\mapsto\wt{\theta}\nonumber
\end{align}
such that 
\begin{equation*}
\begin{tikzcd}
\wt{G}\ar[r,"\wt{\theta}"]\ar[d]\ar[rd,phantom,"\circlearrowright"]&\wt{G}\ar[d]\\
G\ar[r,"\theta"']&G. 
\end{tikzcd}
\end{equation*}
\end{lem}

\begin{proof}
For $\theta\in\Diff(G)_\Gamma$, there exists a unique continuous map $\wt{\theta}\colon\wt{G}\to\wt{G}$ such that $\wt{\theta}(1)=1$ and 
\begin{equation*}
\begin{tikzcd}
\wt{G}\ar[r,"\wt{\theta}"]\ar[d]\ar[rd,phantom,"\circlearrowright"]&\wt{G}\ar[d]\\
G\ar[r,"\theta"']&G. 
\end{tikzcd}
\end{equation*}
By using $\theta^{-1}$ we can show that $\wt{\theta}\in\Diff(\wt{G})$. 

\begin{claim}
$\wt{\theta}(\wt{\gamma}\wt{g})=\wt{\theta}(\wt{\gamma})\wt{\theta}(\wt{g})$ for all $\wt{\gamma}\in\wt{\Gamma}$, $\wt{g}\in\wt{G}$. 
\end{claim}

\begin{proof}
For any $\wt{\gamma}\in\wt{\Gamma}$, the bottom maps of the diagrams 
\begin{equation*}
\begin{tikzcd}
\wt{G}\ar[r,"\wt{\gamma}\ \cdot"]\ar[d]\ar[rd,phantom,"\circlearrowright"]&\wt{G}\ar[r,"\wt{\theta}"]\ar[d]\ar[rd,phantom,"\circlearrowright"]&\wt{G}\ar[d]\\
G\ar[r,"p\wt{\gamma}\ \cdot"']&G\ar[r,"\theta"']&G, 
\end{tikzcd}\qquad
\begin{tikzcd}
\wt{G}\ar[r,"\wt{\theta}"]\ar[d]\ar[rd,phantom,"\circlearrowright"]&\wt{G}\ar[r,"\wt{\theta}(\wt{\gamma})\ \cdot"]\ar[d]\ar[rd,phantom,"\circlearrowright"]&\wt{G}\ar[d]\\
G\ar[r,"\theta"']&G\ar[r,"\theta(p\wt{\gamma})\ \cdot"']&G
\end{tikzcd}
\end{equation*}
coincide and the top maps send $1$ to $\wt{\theta}(\wt{\gamma})$, so they are equal. 
\end{proof}

Hence $\wt{\theta}\in\Diff(\wt{G})_{\wt{\Gamma}}$. 

For the other direction let $\wt{\theta}\in\Diff(\wt{G})_{\wt{\Gamma}}$. Since $\wt{\theta}|_{\wt{\Gamma}}\in\mca{H}(\wt{\Gamma},\wt{G})$, we have $\wt{\theta}(Z)=Z$ by Lemma \ref{descends}. Hence there exists a unique $C^\infty$ map $\theta\colon G\to G$ such that 
\begin{equation*}
\begin{tikzcd}
\wt{G}\ar[r,"\wt{\theta}"]\ar[d]\ar[rd,phantom,"\circlearrowright"]&\wt{G}\ar[d]\\
G\ar[r,"\theta"']&G. 
\end{tikzcd}
\end{equation*}
By using $\wt{\theta}^{-1}\in\Diff(\wt{G})_{\wt{\theta}(\wt{\Gamma})}$, we can show $\theta\in\Diff(G)$. Since $\wt{\theta}(\wt{\gamma}\wt{g})=\wt{\theta}(\wt{\gamma})\wt{\theta}(\wt{g})$ for all $\wt{\gamma}\in\wt{\Gamma}$, $\wt{g}\in\wt{G}$, we get $\theta(p\wt{\gamma}p\wt{g})=\theta(p\wt{\gamma})\theta(p\wt{g})$. Therefore $\theta\in\Diff(G)_\Gamma$. 
\end{proof}

\begin{lem}\label{diffsimdiffsim}
The bijection \eqref{dgdgwgw} induces a bijection 
\begin{equation*}
\Diff(G)_\Gamma/{\sim}\simeq\Diff(\wt{G})_{\wt{\Gamma}}/{\sim}. 
\end{equation*}
\end{lem}

\begin{proof}
Let $\theta$, $\theta^\prime\in\Diff(G)_\Gamma$ be such that $\theta\sim\theta^\prime$. Then there exists a continuous map $J\colon G\times I\to G$ such that: 
\begin{itemize}
\setlength\itemsep{0em}
\item $J(\cdot,0)=\theta$, $J(\cdot,1)=\theta^\prime$
\item $J(\cdot,t)\in C^\infty(G,G)$ for all $t\in I$
\item $J(\gamma g,t)=J(\gamma,t)J(g,t)$ for all $\gamma\in\Gamma$, $g\in G$, $t\in I$
\item $J(\Gamma,t)\subset J(\Gamma,0)$ for all $t\in I$
\item $J(\Gamma,1)=J(\Gamma,0)$. 
\end{itemize}
See Corollary \ref{tblgcgsgxx}. There exists a unique continuous map $\wt{J}\colon\wt{G}\times I\to\wt{G}$ such that $\wt{J}(1,0)=1$ and 
\begin{equation*}
\begin{tikzcd}
\wt{G}\times I\ar[r,"\wt{J}"]\ar[d]\ar[rd,phantom,"\circlearrowright"]&\wt{G}\ar[d]\\
G \times I\ar[r,"J"']&G. 
\end{tikzcd}
\end{equation*}
Then we have: 
\begin{itemize}
\setlength\itemsep{0em}
\item $\wt{J}(1,t)=1$ for all $t\in I$
\item $\wt{J}(\cdot,0)=\wt{\theta}$, $\wt{J}(\cdot,1)=\wt{\theta^\prime}$
\item $\wt{J}(\cdot,t)\in C^\infty(\wt{G},\wt{G})$ for all $t\in I$
\item $\wt{J}(\wt{\gamma}\wt{g},t)=\wt{J}(\wt{\gamma},t)\wt{J}(\wt{g},t)$ for all $\wt{\gamma}\in\wt{\Gamma}$, $\wt{g}\in\wt{G}$ and $t\in I$
\item $\wt{J}(\wt{\gamma},t)=\wt{J}(\wt{\gamma},0)$ for all $\wt{\gamma}\in\wt{\Gamma}$ and $t\in I$ since $J(\gamma,t)=J(\gamma,0)$ for all $\gamma\in\Gamma$ and $t\in I$. 
\end{itemize}
Hence $\wt{\theta}\sim\wt{\theta^\prime}$. 

Conversely let $\theta$, $\theta^\prime\in\Diff(G)_\Gamma$ be such that $\wt{\theta}\sim\wt{\theta^\prime}$. There exists a continuous map $\wt{J}\colon\wt{G}\times I\to\wt{G}$ such that: 
\begin{itemize}
\setlength\itemsep{0em}
\item $\wt{J}(\cdot,0)=\wt{\theta}$, $\wt{J}(\cdot,1)=\wt{\theta^\prime}$
\item $\wt{J}(\cdot,t)\in C^\infty(\wt{G},\wt{G})$ for all $t\in I$
\item $\wt{J}(\wt{\gamma}\wt{g},t)=\wt{J}(\wt{\gamma},t)\wt{J}(\wt{g},t)$ for all $\wt{\gamma}\in\wt{\Gamma}$, $\wt{g}\in\wt{G}$ and $t\in I$
\item $\wt{J}(\wt{\Gamma},t)\subset\wt{J}(\wt{\Gamma},0)$ for all $t\in I$
\item $\wt{J}(\wt{\Gamma},1)=\wt{J}(\wt{\Gamma},0)$. 
\end{itemize}
We have $\wt{J}(\wt{\gamma},t)=\wt{J}(\wt{\gamma},0)$ for all $\wt{\gamma}\in\wt{\Gamma}$. In particular $\wt{J}(Z,t)=\wt{\theta}(Z)=Z$ for all $t\in I$. Hence there exists a unique continuous map $J\colon G\times I\to G$ such that 
\begin{equation*}
\begin{tikzcd}
\wt{G}\times I\ar[r,"\wt{J}"]\ar[d]\ar[rd,phantom,"\circlearrowright"]&\wt{G}\ar[d]\\
G \times I\ar[r,"J"']&G. 
\end{tikzcd}
\end{equation*}
Then we have: 
\begin{itemize}
\setlength\itemsep{0em}
\item $J(\cdot,0)=\theta$, $J(\cdot,1)=\theta^\prime$
\item $J(\cdot,t)\in C^\infty(G,G)$ for all $t\in I$
\item $J(\gamma g,t)=J(\gamma,t)J(g,t)$ for all $\gamma\in\Gamma$, $g\in G$, $t\in I$
\item $J(\Gamma,t)=J(\Gamma,0)$ for all $t\in I$. 
\end{itemize}
Therefore $\theta\sim\theta^\prime$. 
\end{proof}

Note that $\Diff(G)_\Gamma/{\sim}\simeq\Diff(\wt{G})_{\wt{\Gamma}}/{\sim}$ is $(\Aut(G)\simeq\Aut(\wt{G}))$-equivariant since 
\begin{equation*}
\begin{tikzcd}
\wt{G}\ar[r,"\wt{\theta}"]\ar[d]\ar[rd,phantom,"\circlearrowright"]&\wt{G}\ar[r,"\wt{\Phi}"]\ar[d]\ar[rd,phantom,"\circlearrowright"]&\wt{G}\ar[d]\\
G\ar[r,"\theta"']&G\ar[r,"\Phi"']&G
\end{tikzcd}
\end{equation*}
for $\theta\in\Diff(G)_\Gamma$ and $\Phi\in\Aut(G)$. Hence 
\begin{equation*}
\Aut(G)\bs(\Diff(G)_\Gamma/{\sim})\simeq\Aut(\wt{G})\bs(\Diff(\wt{G})_{\wt{\Gamma}}/{\sim}). 
\end{equation*}
(This shows $T(\rho_\Gamma)\simeq T(\rho_{\wt{\Gamma}})$ by Corollary \ref{tblgcgsgxx}, where $\wt{\Gamma}\bs\wt{G}\stackrel{\rho_\wt{\Gamma}}{\curvearrowleft}\wt{G}$ is the action by right multiplication.) Recall the map 
\begin{equation}\label{dffwgsimgg}
\begin{aligned}
r\colon\Diff(\wt{G})_{\wt{\Gamma}}/{\sim}&\to\mca{H}(\wt{\Gamma},\wt{G})\\
[\wt{\theta}]&\mapsto\wt{\theta}|_{\wt{\Gamma}}
\end{aligned}
\end{equation}
considered in Section \ref{marked lattices}. Since $\wt{G}$ is contractible, $r$ is injective by Lemma \ref{injective}. 

\begin{lem}
The map $r$ is surjective. 
\end{lem}

\begin{proof}
Let $\wt{\alpha}\in\mca{H}(\wt{\Gamma},\wt{G})$. We have $\wt{\Gamma}\bs\wt{G}\simeq\Gamma\bs G$, and $\wt{\alpha}(\wt{\Gamma})\bs\wt{G}=\wt{\alpha(\Gamma)}\bs\wt{G}\simeq\alpha(\Gamma)\bs G$ by Lemma \ref{descends}. The manifolds $\wt{\Gamma}\bs\wt{G}, \wt{\alpha}(\wt{\Gamma})\bs\wt{G}$ are Haken. By a theorem of Waldhausen (Corollary 6.5 in \cite{Waldhausen}), there exists a homeomorphism $f\colon(\wt{\Gamma}\bs\wt{G},\wt{\Gamma})\to(\wt{\alpha}(\wt{\Gamma})\bs\wt{G},\wt{\alpha}(\wt{\Gamma}))$ such that 
\begin{equation*}
\begin{tikzcd}
\pi_1(\wt{\Gamma}\bs\wt{G},\wt{\Gamma})\ar[r,"f_*"]\ar[d,"\sim"{sloped},"\iota_1"']\ar[rd,phantom,"\circlearrowright"{xshift=-5}]&\pi_1(\wt{\alpha}(\wt{\Gamma})\bs\wt{G},\wt{\alpha}(\wt{\Gamma}))\ar[d,"\sim"{sloped},"\iota_1"{xshift=7}]\\
\wt{\Gamma}\ar[r,"\wt{\alpha}"']&\wt{\alpha}(\wt{\Gamma}). 
\end{tikzcd}
\end{equation*}
Since a homeomorphism between $3$-dimensional $C^\infty$ manifolds is homotopic to a $C^1$ diffeomorphism (see the first corollary in \cite{Munkres}) and a $C^1$ diffeomorphism is homotopic to a $C^\infty$ diffeomorphism (see 2.7 Theorem in \cite{Hirsch}), we may assume that $f$ is a $C^\infty$ diffeomorphism (by a similar argument in the proof of Lemma \ref{descends} to make a homotopy basepoint preserving). 

There exists a unique continuous map $\wt{f}\colon\wt{G}\to\wt{G}$ such that $\wt{f}(1)=1$ and 
\begin{equation*}
\begin{tikzcd}
\wt{G}\ar[r,"\wt{f}"]\ar[d]\ar[rd,phantom,"\circlearrowright"{xshift=5,yshift=-2}]&\wt{G}\ar[d]\\
\wt{\Gamma}\bs\wt{G}\ar[r,"f"']&\wt{\alpha}(\wt{\Gamma})\bs\wt{G}. 
\end{tikzcd}
\end{equation*}
By using $f^{-1}$, we see that $\wt{f}\in\Diff(\wt{G})$. 

\begin{claim}
$\wt{f}(\wt{\gamma})=\wt{\alpha}(\wt{\gamma})$ for all $\wt{\gamma}\in\wt{\Gamma}$. 
\end{claim}

\begin{proof}
This can be seen from the following figure. 
\begin{equation*}
\begin{tikzpicture}[every label/.append style={font=\scriptsize},matrix of math nodes,decoration={markings,mark=at position0.5with{\arrow{>}}}]
\clip(-1,-1.7)rectangle(5.5,1.5);
\node(1)[label=left:1]{};
\node(2)[above=of 1,label=left:\wt{\gamma}]{};
\node(3)[below=of 1,label=left:\wt{\Gamma}]{};
\node(4)at(3){};
\draw[postaction={decorate}](2.center)to[out=-50,in=50](1.center);
\draw[postaction={decorate}](3.center)..controls+(1,1)and+(1,-1)..(4.center)node[pos=.5,label=right:\wt{\gamma}]{};

\node(5)[right=3 of 1,label=right:1]{};
\node(6)[above=of 5,label=right:\wt{f}(\wt{\gamma})\ {=}\ \wt{\alpha}(\wt{\gamma})]{};
\node(7)[below=of 5,label=left:\wt{\alpha}(\wt{\Gamma})]{};
\node(8)at(7){};
\draw[postaction={decorate}](6.center)to[out=-50,in=50](5.center);
\draw[postaction={decorate}](7.center)..controls+(1,1)and+(1,-1)..(8.center)node[pos=.5,label=right:\wt{\alpha}(\wt{\gamma})]{};
\foreach\x in{1,2,3,5,6,7}\filldraw(\x)circle(1pt);
\end{tikzpicture}\qedhere
\end{equation*}
\end{proof}

\begin{claim}
$\wt{f}(\wt{\gamma}\wt{g})=\wt{f}(\wt{\gamma})\wt{f}(\wt{g})$ for all $\wt{\gamma}\in\wt{\Gamma}$, $\wt{g}\in\wt{G}$. 
\end{claim}

\begin{proof}
We have 
\begin{equation*}
\begin{tikzcd}
\wt{G}\ar[r,"\wt{\gamma}\ \cdot"]\ar[rd]&\wt{G}\ar[r,"\wt{f}"]\ar[d]\ar[rd,phantom,"\circlearrowright"{xshift=4,yshift=-3}]\ar[ld,phantom,"\circlearrowright",pos=.15]&\wt{G}\ar[d]\\
\ &\wt{\Gamma}\bs\wt{G}\ar[r,"f"']&\wt{\alpha}(\wt{\Gamma})\bs\wt{G}, 
\end{tikzcd}\qquad
\begin{tikzcd}
\wt{G}\ar[r,"\wt{f}"]\ar[d]\ar[rd,phantom,"\circlearrowright"{xshift=3,yshift=-3}]&\wt{G}\ar[r,"\wt{\alpha}(\wt{\gamma})\ \cdot"]\ar[d]\ar[rd,phantom,"\circlearrowright",pos=.18]&\wt{G}\ar[ld]\\
\wt{\Gamma}\bs\wt{G}\ar[r,"f"']&\wt{\alpha}(\wt{\Gamma})\bs\wt{G}. &\ 
\end{tikzcd}
\end{equation*}
Since both of the top horizontal compositions in the diagrams are lifts of $f$ and map $1$ to $\wt{f}(\wt{\gamma})$, they must be equal. 
\end{proof}

Hence $\wt{f}\in\Diff(\wt{G})_{\wt{\Gamma}}$ and $r[\wt{f}]=\wt{\alpha}$. 
\end{proof}

This completes the proof of Theorem \ref{traghg}.

\subsection{The construction of a section $\sigma$ of $\pi$}
In this section we will construct a section $\sigma$ of the map $\pi\colon\mca{H}(\wt{\Gamma},\wt{G})\to\mca{H}(\Gamma,G)$. Let $\alpha\in\mca{H}(\Gamma,G)$ and $y\in\bb{H}^2$. We have an isomorphism $\Gamma\stackrel{\alpha}{\simeq}\alpha(\Gamma)$. 

By Proposition 1B.9 in \cite{Hatcher}, there exists a homotopy equivalence $f\colon(\Gamma\bs\bb{H}^2,\ol{i})\to(\alpha(\Gamma)\bs\bb{H}^2,\ol{y})$ such that 
\begin{equation*}
\begin{tikzcd}
\pi_1(\Gamma\bs\bb{H}^2,\ol{i})\ar[r,"f_*"]\ar[d,"\sim"{sloped},"\iota_i"']\ar[rd,phantom,"\circlearrowright"{xshift=-3}]&\pi_1(\alpha(\Gamma)\bs\bb{H}^2,\ol{y})\ar[d,"\sim"{sloped},"\iota_y"{xshift=7}]\\
\Gamma\ar[r,"\alpha"']&\alpha(\Gamma). 
\end{tikzcd}
\end{equation*}
Since a homotopy equivalence between connected orientable closed hyperbolic surfaces is homotopic to a diffeomorphism (see 8.3.2 of \cite{FMar}), we may assume that $f$ is a diffeomorphism (by the same argument in the proof of the second claim in the proof of Lemma \ref{descends} to make a homotopy basepoint preserving). 

Define a map 
\begin{equation*}
\wh{f}\colon T^1(\Gamma\bs\bb{H}^2)\to T^1(\alpha(\Gamma)\bs\bb{H}^2)
\end{equation*}
by $\wh{f}(v)=\frac{f_*(v)}{\lVert f_*(v)\rVert}$. Put $u=(i,i)\in T_i^1\bb{H}^2$. There exists a unique $w\in T_y^1\bb{H}^2$ such that $\ol{w}=\wh{f}(\ol{u})$. Define $Df$ by 
\begin{equation*}
\begin{tikzcd}
\Gamma g\ar[r,phantom,"\in"]\ar[d,mapsto]&[-20pt]\Gamma\bs G\ar[r,"Df"]\ar[d,dash,sloped,"\sim"]\ar[rd,phantom,"\circlearrowright"{xshift=2,yshift=-2}]&\alpha(\Gamma)\bs G\ar[d,dash,sloped,"\sim"]\ar[r,phantom,"\ni"]&[-20pt]\alpha(\Gamma)g\ar[d,mapsto]\\
\Gamma g_*u\ar[r,phantom,"\in"]&\Gamma\bs T^1\bb{H}^2\ar[r]\ar[d,dash,sloped,"\sim"]\ar[rd,phantom,"\circlearrowright"{xshift=-2}]&\alpha(\Gamma)\bs T^1\bb{H}^2\ar[d,dash,sloped,"\sim"]\ar[r,phantom,"\ni"]&\alpha(\Gamma)g_*w\\
&T^1(\Gamma\bs\bb{H}^2)\ar[r,"\wh{f}"']&T^1(\alpha(\Gamma)\bs\bb{H}^2). 
\end{tikzcd}
\end{equation*}
Then $(Df)(\Gamma)=\alpha(\Gamma)$. Now let $\wt{\alpha}$ be defined by 
\begin{equation}\label{pigggdfa}
\begin{tikzcd}
\pi_1(\Gamma\bs G,\Gamma)\ar[r,"(Df)_*"]\ar[d,dash,sloped,"\sim"]\ar[rd,phantom,"\circlearrowright"{xshift=1,yshift=-1}]&\pi_1(\alpha(\Gamma)\bs G,\alpha(\Gamma))\ar[d,dash,sloped,"\sim"]\\
\pi_1(\wt{\Gamma}\bs\wt{G},\wt{\Gamma})\ar[r]\ar[d,dash,sloped,"\sim"]\ar[rd,phantom,"\circlearrowright"{xshift=-6}]&\pi_1(\wt{\alpha(\Gamma)}\bs\wt{G},\wt{\alpha(\Gamma)})\ar[d,dash,sloped,"\sim"]\\
\wt{\Gamma}\ar[r,"\wt{\alpha}"']&\wt{\alpha(\Gamma)}\ar[r,phantom,"\subset"]&[-40pt]\wt{G}. 
\end{tikzcd}
\end{equation}
Hence $\wt{\alpha}\in\mca{H}(\wt{\Gamma},\wt{G})$. 

\begin{lem}\label{well}
The map 
\begin{align*}
\sigma\colon\mca{H}(\Gamma,G)&\to\mca{H}(\wt{\Gamma},\wt{G})\\
\alpha&\mapsto\wt{\alpha}
\end{align*}
is well-defined. 
\end{lem}

\begin{proof}
Proof will be given in Section \ref{well s}. 
\end{proof}

Let $\iota\colon\Gamma\to G$, $\wt{\iota}\colon\wt{\Gamma}\to\wt{G}$ be the inclusions: 
\begin{equation*}
\begin{tikzcd}
\wt{\Gamma}\ar[r,"\wt{\iota}"]\ar[d]\ar[rd,phantom,"\circlearrowright"]&\wt{G}\ar[d]\\
\Gamma\ar[r,"\iota"']&G. 
\end{tikzcd}
\end{equation*}
Then it is easy to see $\sigma(\iota)=\wt{\iota}$. 

\begin{lem}\label{section}
We have $\pi\sigma=\id$, ie $\sigma$ is a section of $\pi$: 
\begin{equation*}
\begin{tikzcd}
\mca{H}(\wt{\Gamma},\wt{G})\ar[d,"\pi"']\\
\mca{H}(\Gamma,G). \ar[u,bend right,"\sigma"']
\end{tikzcd}
\end{equation*}
\end{lem}

\begin{proof}
Proof will be given in Section \ref{section s}. 
\end{proof}

\begin{lem}\label{g eq}
$\sigma$ is $(\Aut(G)\simeq\Aut(\wt{G}))$-equivariant. 
\end{lem}

\begin{proof}
Proof will be given in Section \ref{g eq s}. 
\end{proof}

For any $\wt{h}\in\wt{G}$, let $h=p\wt{h}\in G$. Then 
\begin{equation*}
\begin{tikzcd}
\wt{\Gamma}\ar[r,"\wt{\iota}"]\ar[d]\ar[rd,phantom,"\circlearrowright"]&\wt{G}\ar[r,"\wt{h}\ \cdot\ \wt{h}^{-1}"]\ar[d]\ar[rd,phantom,"\circlearrowright"]&\wt{G}\ar[d]\\
\Gamma\ar[r,"\iota"']&G\ar[r,"h\ \cdot\ h^{-1}"']&G, 
\end{tikzcd}
\end{equation*}
so we have $\sigma(h\iota h^{-1})=\wt{h}\ \wt{\iota}\ \wt{h}^{-1}$ by Lemma \ref{g eq}. 

For any $\wt{\alpha}\in\Aut(\wt{\Gamma})$, there exists a unique $\alpha\in\Aut(\Gamma)$ such that 
\begin{equation*}
\begin{tikzcd}
Z\ar[r,"\sim"]\ar[d]\ar[rd,phantom,"\circlearrowright"]&Z\ar[d]\\
\wt{\Gamma}\ar[r,"\wt{\alpha}"]\ar[d]\ar[rd,phantom,"\circlearrowright"]&\wt{\Gamma}\ar[d]\\
\Gamma\ar[r,"\alpha"']&\Gamma. 
\end{tikzcd}
\end{equation*}
This defines a homomorphism 
\begin{align*}
\pi_0\colon\Aut(\wt{\Gamma})&\to\Aut(\Gamma)\\
\wt{\alpha}&\mapsto\alpha
\end{align*}
and it is a restriction of $\pi$: 
\begin{equation*}
\begin{tikzcd}[row sep=small]
\mca{H}(\wt{\Gamma},\wt{G})\ar[r,"\pi"]\ar[rd,phantom,"\circlearrowright"]&\mca{H}(\Gamma,G)\\
\Aut(\wt{\Gamma})\ar[r,"\pi_0"']\ar[u,phantom,sloped,"\subset"]&\Aut(\Gamma). \ar[u,phantom,sloped,"\subset"]
\end{tikzcd}
\end{equation*}
Since 
\begin{equation*}
\begin{tikzcd}
\wt{\Gamma}\ar[r,"\wt{\gamma}\ \cdot\ \wt{\gamma}^{-1}"]\ar[d]\ar[rd,phantom,"\circlearrowright"]&\wt{\Gamma}\ar[d]\\
\Gamma\ar[r,"p\wt{\gamma}\ \cdot\ p\wt{\gamma}^{-1}"']&\Gamma, 
\end{tikzcd}
\end{equation*}
$\pi_0$ descends to a homomorphism $\ol{\pi_0}$: 
\begin{equation*}
\begin{tikzcd}
\Aut(\wt{\Gamma})\ar[r,"\pi_0"]\ar[d]\ar[rd,phantom,"\circlearrowright"]&\Aut(\Gamma)\ar[d]\\
\Out(\wt{\Gamma})\ar[r,"\ol{\pi_0}"']&\Out(\Gamma). 
\end{tikzcd}
\end{equation*}
Let $\sigma_0$ be the restriction of $\sigma$: 
\begin{equation*}
\begin{tikzcd}[row sep=small]
\mca{H}(\Gamma,G)\ar[r,"\sigma"]\ar[rd,phantom,"\circlearrowright"]&\mca{H}(\wt{\Gamma},\wt{G})\\
\Aut(\Gamma)\ar[r,"\sigma_0"']\ar[u,phantom,sloped,"\subset"]&\Aut(\wt{\Gamma}). \ar[u,phantom,sloped,"\subset"]
\end{tikzcd}
\end{equation*}
Then $\sigma_0$ is a section of $\pi_0$: 
\begin{equation*}
\begin{tikzcd}
\Aut(\wt{\Gamma})\ar[d,"\pi_0"']\\
\Aut(\Gamma). \ar[u,bend right,"\sigma_0"']
\end{tikzcd}
\end{equation*}

\begin{lem}\label{gamma eq}
$\sigma$ is $(\Aut(\Gamma)\to\Aut(\wt{\Gamma}))$-equivariant. 
\end{lem}

\begin{proof}
Proof will be given in Section \ref{gamma eq s}. 
\end{proof}

In summary we have 
\begin{equation*}
\begin{tikzcd}[column sep=tiny]
\Aut(\wt{G})\ar[r,phantom,"\curvearrowright"]\ar[d,dash,sloped,"\sim"]&\mca{H}(\wt{\Gamma},\wt{G})\ar[r,phantom,"\curvearrowleft"]\ar[d,"\pi"']&\Aut(\wt{\Gamma})\ar[d,"\pi_0"']\\
\Aut(G)\ar[r,phantom,"\curvearrowright"]&\mca{H}(\Gamma,G)\ar[r,phantom,"\curvearrowleft"]\ar[u,bend right,"\sigma"']&\Aut(\Gamma). \ar[u,bend right,"\sigma_0"']
\end{tikzcd}
\end{equation*}

\begin{cor}
$\sigma_{0}$ is a homomorphism and induces a homomorphism $\ol{\sigma_0}$ which is a section of $\ol{\pi_0}$: 
\begin{equation*}
\begin{tikzcd}
\Out(\wt{\Gamma})\ar[d,"\ol{\pi_0}"']\\
\Out(\Gamma). \ar[u,bend right,"\ol{\sigma_0}"']
\end{tikzcd}
\end{equation*}
\end{cor}

\begin{proof}
For any $\wt{\gamma}\in\wt{\Gamma}$, let $\gamma=p\wt{\gamma}\in\Gamma$. Then we have 
\begin{equation*}
\sigma_0(\gamma\ \cdot\ \gamma^{-1})=\sigma(\gamma\iota\gamma^{-1})=\wt{\gamma}\ \wt{\iota}\ \wt{\gamma}^{-1}=\wt{\gamma}\ \cdot\ \wt{\gamma}^{-1}. 
\end{equation*}
Hence it induces a homomorphism $\ol{\sigma_0}$: 
\begin{equation*}
\begin{tikzcd}
\Aut(\Gamma)\ar[r,"\sigma_0"]\ar[d]\ar[rd,phantom,"\circlearrowright"{xshift=2,yshift=-1}]&\Aut(\wt{\Gamma})\ar[d]\\
\Out(\Gamma)\ar[r,"\ol{\sigma_0}"']&\Out(\wt{\Gamma}). 
\end{tikzcd}
\end{equation*}
\end{proof}

\subsubsection{Well-definedness of $\sigma$ (proof of Lemma \ref{well})}\label{well s}
Take another point $y^\prime\in\bb{H}^2$ and a diffeomorphism $f^\prime\colon(\Gamma\bs\bb{H}^2,\ol{i})\to(\alpha(\Gamma)\bs\bb{H}^2,\overline{y^\prime})$ such that 
\begin{equation*}
\begin{tikzcd}
\pi_1(\Gamma\bs\bb{H}^2,\ol{i})\ar[r,"f^\prime_*"]\ar[d,"\sim"{sloped},"\iota_i"']\ar[rd,phantom,"\circlearrowright"{xshift=-3}]&\pi_1(\alpha(\Gamma)\bs\bb{H}^2,\ol{y^\prime})\ar[d,"\sim"{sloped},"\iota_{y^\prime}"{xshift=7}]\\
\Gamma\ar[r,"\alpha"']&\alpha(\Gamma). 
\end{tikzcd}
\end{equation*}
Take a continuous curve $c\colon I\to\bb{H}^2$ such that $c(0)=y$, $c(1)=y^\prime$ and let $\ol{c}\colon I\to\alpha(\Gamma)\bs\bb{H}^2$ be its projection. 

\begin{claim}
We have 
\begin{equation*}
\begin{tikzcd}[row sep=tiny]
&[-10pt]&[-10pt]&[-10pt]\alpha(\Gamma)\ar[dddd,"\id"]\\
&&\pi_1(\alpha(\Gamma)\bs\bb{H}^2,\ol{y})\ar[dd,"\ol{c}\ \cdot\ \ol{c}^{-1}"]\ar[ru,dash,sloped,"\sim"]\\
\Gamma\ar[r,dash,"\sim"]\ar[rrruu,bend left,"\alpha"]\ar[rrruu,phantom,bend left=15,"\circlearrowright"]\ar[rrrdd,bend right,"\alpha"']\ar[rrrdd,phantom,bend right=15,"\circlearrowright"]&\pi_1(\Gamma\bs\bb{H}^2,\ol{i})\ar[ru,"f_*"]\ar[rd,"f^\prime_*"']&\ \ar[r,phantom,"\circlearrowright"pos=.6]&\ \\
&&\pi_1(\alpha(\Gamma)\bs\bb{H}^2,\ol{y^\prime})\ar[rd,dash,sloped,"\sim"]\\
&&&\alpha(\Gamma). 
\end{tikzcd}
\end{equation*}
\end{claim}

\begin{proof}
Take $\delta\in\alpha(\Gamma)$. Then the commutativity of the right square follows from the following figure. 
\begin{equation*}
\begin{tikzpicture}[every label/.append style={font=\scriptsize},matrix of math nodes,decoration={markings,mark=at position0.5with{\arrow{>}}}]
\clip(-.5,-2)rectangle(3.2,1.7);
\node(1)[label=below:y^\prime]{};
\node(2)[right=5em of 1,label=below:y]{};
\node(3)[above=2em of 2,label=above:\delta y]{};
\node(4)[above=2em of 1,label=above:\delta y^\prime]{};
\node(5)[below=of 1,label=below:\ol{y^\prime}]{};
\node(6)[below=of 2,label=below:\ol{y}]{};
\node(7)at(6){};
\draw[postaction={decorate}](2.center)to[out=190,in=-10]node[above,scale=.75]{c(t)}(1.center);
\draw[postaction={decorate}](3.center)to[out=-80,in=70]node[right,scale=.75]{c_\delta(t)}(2.center);
\draw[postaction={decorate}](3.center)to[out=190,in=-10]node[above,scale=.75]{\delta c(t)}(4.center);
\draw[postaction={decorate}](6.center)to[out=190,in=-10]node[above,scale=.75]{\ol{c}(t)}(5.center);
\draw[postaction={decorate}](6.center)..controls+(1,1)and+(1,-1)..(7.center)node[pos=.5,right,scale=.75]{\delta};
\foreach\x in{1,...,6}\filldraw(\x)circle(1pt);
\end{tikzpicture}\qedhere
\end{equation*}
\end{proof}

Hence 
\begin{equation}\label{pighipi1}
\begin{tikzcd}[row sep=tiny]
&\pi_1(\alpha(\Gamma)\bs\bb{H}^2,\ol{y})\ar[dd,"\ol{c}\ \cdot\ \ol{c}^{-1}"]\\
\pi_1(\Gamma\bs\bb{H}^2,\ol{i})\ar[ru,"f_*"]\ar[rd,"f^\prime_*"']\ar[r,phantom,"\circlearrowright"{xshift=6,yshift=-1}]&\ \\
&\pi_1(\alpha(\Gamma)\bs\bb{H}^2,\ol{y^\prime}). 
\end{tikzcd}
\end{equation}

\begin{lem}\label{cwpipi}
Let $(X,x_0)$ be a CW pair, where $x_0\in X$ and $X$ is connected, $Y$ be an aspherical topological space and $y$, $y^\prime\in Y$. Let $f\colon(X,x_0)\to(Y,y)$, $f^\prime\colon(X,x_0)\to(Y,y^\prime)$ be continuous maps, $c\colon I\to Y$ be a continuous curve such that $c(0)=y$, $c(1)=y^\prime$. Assume 
\begin{equation*}
\begin{tikzcd}[row sep=tiny]
&\pi_1(Y,y)\ar[dd,"c\ \cdot\ c^{-1}"]\\
\pi_1(X,x_0)\ar[ru,"f_*"]\ar[rd,"f^\prime_*"']\ar[r,phantom,"\circlearrowright"{xshift=5,yshift=-1}]&\ \\
&\pi_1(Y,y^\prime). 
\end{tikzcd}
\end{equation*}
Then there exists a homotopy $h_t$ between $f$ and $f^\prime$ such that the two curves $h_t(x_0)$ and $c(t)$ are homotopic relative to endpoints. 
\end{lem}

\begin{proof}
Since $(X,x_0)$ is a CW pair, it has the homotopy extension property. So there exists a continuous map $H\colon X\times I\to Y$ such that: 
\begin{itemize}
\setlength\itemsep{0em}
\item $H(\cdot,0)=f$
\item $H(x_0,\cdot)=c$. 
\end{itemize}
\begin{equation*}
\begin{tikzpicture}[every label/.append style={font=\scriptsize},matrix of math nodes,decoration={markings,mark=at position0.5with{\arrow{>}}}]
\def\a{2}
\def\b{1}
\node(1)at(0,0)[label=left:0]{};
\node(2)at(\a,0)[label=right:f]{};
\node(3)at(\a,\a){};
\node(4)at(0,\a)[label=left:1]{};
\node(5)at(\b,0)[label=below:x_0]{};
\node(6)at(\b,\a){};
\node(7)at(\a+1,\a/2){};
\node(8)at(\a+2,\a/2){};
\node(9)at(\a+3,\a/2)[label=below:y^\prime]{};
\node(10)at(\a+4,\a/2)[label=below:y]{};
\draw(1.center)--(2.center)--(3.center)--(4.center)--cycle;
\draw[postaction={decorate}](5.center)--node[label=right:c]{}(6.center);
\draw[->](7)--node[label=above:H]{}(8);
\draw[postaction={decorate}](10.center)--node[label=above:c(t)]{}(9.center);
\foreach\x in{9,10}\filldraw(\x)circle(1pt);
\end{tikzpicture}
\end{equation*}
We have 
\begin{equation*}
\begin{tikzcd}[row sep=tiny]
&\pi_1(Y,y)\ar[dd,"c\ \cdot\ c^{-1}"]\\
\pi_1(X,x_0)\ar[ru,"f_*"]\ar[rd,"H(\cdot{,}1)_*"']\ar[r,phantom,"\circlearrowright"{xshift=5,yshift=-1}]&\ \\
&\pi_1(Y,y^\prime). 
\end{tikzcd}
\end{equation*}
Hence $f^\prime_*=H(\cdot,1)_*$. Since $Y$ is aspherical, $H(\cdot,1)$ is homotopic to $f^\prime$ relative to $x_0$. Let $H^\prime\colon X\times I\to Y$ be such a homotopy. Then the concatenation of $H$ and $H^\prime$ gives a desired homotopy between $f$ and $f^\prime$. 
\end{proof}

\begin{lem}\label{sigconorihthe}
Let $\Sigma$ be a connected orientable closed surface of genus at least $2$ and $x\in\Sigma$ be a basepoint. Let $\theta\in\Diff(\Sigma)$ be such that $\theta(x)=x$ and $\theta$ is based $C^0$ homotopic to $\id$. Then there exists a based $C^0$ homotopy $H\colon\Sigma\times I\to\Sigma$ between $\theta$ and $\id$ such that: 
\begin{itemize}
\setlength\itemsep{0em}
\item $H(\cdot,t)\in\Diff(\Sigma)$ for all $t\in I$
\item the map 
\begin{align*}
T\Sigma\times I&\to T\Sigma\\
(v,t)&\mapsto H(\cdot,t)_*v
\end{align*}
is continuous. 
\end{itemize}
\end{lem}

\begin{proof}
Let 
\begin{align*}
\Homeo(\Sigma)_1&=\left\{\theta^\prime\in\Homeo(\Sigma)\ \middle|\ \text{$\theta^\prime$ is $C^0$ homotopic to $\id$}\right\}, \\
\Homeo(\Sigma,x)_1&=\left\{\theta^\prime\in\Homeo(\Sigma)_1\ \middle|\ \theta^\prime(x)=x\right\}, \\
\Diff(\Sigma)_1&=\left\{\theta^\prime\in\Diff(\Sigma)\ \middle|\ \text{$\theta^\prime$ is $C^0$ homotopic to $\id$}\right\}, \\
\Diff(\Sigma,x)_1&=\left\{\theta^\prime\in\Diff(\Sigma)_1\ \middle|\ \theta^\prime(x)=x\right\}. 
\end{align*}
Equip $\Homeo(\Sigma)$ with the compact-open topology and $\Diff(\Sigma)$ with the weak $C^\infty$ topology. Then $\Homeo(\Sigma)$ and $\Diff(\Sigma)$ are topological groups. 

\begin{claim}
With respect to the action $\Homeo(\Sigma)_1\curvearrowleft\Homeo(\Sigma,x)_1$ by right multiplication, 
\begin{align*}
\Homeo(\Sigma,x)_1\hookrightarrow\Homeo(\Sigma)_1&\xrightarrow{p}\Sigma\\
\theta&\mapsto\theta(x)
\end{align*}
is a principal $\Homeo(\Sigma,x)_1$ bundle. 
\end{claim}

\begin{proof}
We show the local triviality. There exist an open neighborhood $U$ of $x$ in $\Sigma$ and a continuous section $\sigma\colon U\to p^{-1}(U)$ of $p$ such that $\sigma(x)=\id$ (use the exponential map for a Riemannian metric of $\Sigma$ and a bump function). 

Let $y\in\Sigma$. There exists $\theta\in\Homeo(\Sigma)_1$ such that $\theta(x)=y$. Then $\theta(U)$ is an open neighborhood of $y$ in $\Sigma$ and let 
\begin{align*}
\sigma_y\colon\theta(U)&\to p^{-1}(\theta(U))\\
y^\prime&\mapsto\theta\sigma(\theta^{-1}(y^\prime)). 
\end{align*}
Then $\sigma_y$ is a continuous local section of $p$ and we have a local trivialization 
\begin{align*}
p^{-1}\theta(U)&\simeq\theta(U)\times\Homeo(\Sigma,x)_1\\
\sigma_y(y^\prime)\theta^\prime&\mapsfrom(y^\prime,\theta^\prime)\\
\theta^{\prime\prime}&\mapsto(\theta^{\prime\prime}(x),\sigma_y(\theta^{\prime\prime}(x))^{-1}\theta^{\prime\prime}). 
\end{align*}
(Note that this is a homeomorphism.) 
\end{proof}

Similarly $\Diff(\Sigma,x)_1\to\Diff(\Sigma)_1\to\Sigma$ is a principal $\Diff(\Sigma,x)_1$ bundle. We have 
\begin{equation*}
\begin{tikzcd}
\Diff(\Sigma,x)_1\ar[r,hook]\ar[d]&\Homeo(\Sigma,x)_1\ar[d]\\
\Diff(\Sigma)_1\ar[r,hook]\ar[d]&\Homeo(\Sigma)_1\ar[d]\\
\Sigma\ar[r,equal]&\Sigma. 
\end{tikzcd}
\end{equation*}

Let $\Homeo(\Sigma)_0$ be the identity component of $\Homeo(\Sigma)$. By Theorem 5.1 of \cite{Hamstrom}, $\pi_i\Homeo(\Sigma)_0=0$ for all $i\geq0$. In particular $\Homeo(\Sigma)_0$ is path connected, hence it coincides with the identity path component of $\Homeo(\Sigma)$. The identity path component of $\Homeo(\Sigma)$ coincides with $\Homeo(\Sigma)_1$ by \cite{Baer_1927} and \cite{Baer_1928} (or 6.4 Theorem of \cite{Ep}). Hence $\Homeo(\Sigma)_0=\Homeo(\Sigma)_1$ and $\pi_i\Homeo(\Sigma)_1=0$ for all $i\geq0$. 

On the other hand $\Diff(\Sigma)_1$ is contractible by 4 in the first corollary of \cite{EarleEells2} (or (c) of Theorem 1 in \cite{EarleEells1}). 

By the long exact sequences of homotopy groups and the five lemma, the inclusion $\Diff(\Sigma,x)_1\hookrightarrow\Homeo(\Sigma,x)_1$ induces 
\begin{equation}\label{pidiffsimp}
\pi_0\Diff(\Sigma,x)_1\simeq\pi_0\Homeo(\Sigma,x)_1. 
\end{equation}

Since $\theta$ is based $C^0$ homotopic to $\id$, $\theta$ is based $C^0$ homotopic to $\id$ through maps in $\Homeo(\Sigma)$ by 6.4 Theorem of \cite{Ep}. So $\theta$ is in the path component of $\Homeo(\Sigma,x)_1$ containing $\id$. By \eqref{pidiffsimp}, $\theta$ is in the path component of $\Diff(\Sigma,x)_1$ containing $\id$. A continuous path connecting $\theta$ and $\id$ in $\Diff(\Sigma,x)_1$ gives a desired homotopy. 
\end{proof}

\begin{cor}\label{sigmaconoriclH}
Let $\Sigma$, $\Sigma^\prime$ be connected orientable closed surfaces of genus at least $2$ and $x\in\Sigma$, $x^\prime\in\Sigma^\prime$ be basepoints. Let $\theta$, $\theta^\prime\colon\Sigma\to\Sigma^\prime$ be diffeomorphisms such that $\theta(x)=x^\prime$, $\theta^\prime(x)=x^\prime$ and $\theta$ is based $C^0$ homotopic to $\theta^\prime$. Then there exists a based $C^0$ homotopy $H\colon\Sigma\times I\to\Sigma^\prime$ between $\theta$ and $\theta^\prime$ such that: 
\begin{itemize}
\setlength\itemsep{0em}
\item $H(\cdot,t)$ is a diffeomorphism for all $t\in I$
\item the map 
\begin{align*}
T\Sigma\times I&\to T\Sigma^\prime\\
(v,t)&\mapsto H(\cdot,t)_*v
\end{align*}
is continuous. 
\end{itemize}
\end{cor}

\begin{proof}
Apply Lemma \ref{sigconorihthe} to $(\theta^\prime)^{-1}\theta\in\Diff(\Sigma)$. 
\end{proof}

By Diagram \eqref{pighipi1} and Lemma \ref{cwpipi}, there exists a $C^0$ homotopy $h_t$ between $f$ and $f^\prime$ such that $h_t(\ol{i})$ is homotopic to $\ol{c}(t)$ relative to endpoints. 

\begin{claim}
There exists a $C^0$ homotopy $f_t$ between $f$ and $f^\prime$ such that: 
\begin{itemize}
\setlength\itemsep{0em}
\item $f_t\colon\Gamma\bs\bb{H}^2\to\alpha(\Gamma)\bs\bb{H}^2$ is a diffeomorphism for each $t\in I$
\item the map 
\begin{align*}
T^1(\Gamma\bs\bb{H}^2)\times I&\to T^1(\alpha(\Gamma)\bs\bb{H}^2)\\
(v,t)&\mapsto\wh{f_t}(v)
\end{align*}
is continuous
\item $f_t(\ol{i})$ is homotopic to $h_t(\ol{i})$ relative to endpoints. 
\end{itemize}
\end{claim}

\begin{proof}
There exists a $C^\infty$ homotopy $h^\prime_t$ between $f^\prime$ and a diffeomorphism $g\colon\Gamma\bs\bb{H}^2\to\alpha(\Gamma)\bs\bb{H}^2$ such that: 
\begin{itemize}
\setlength\itemsep{0em}
\item $h^\prime_t$ is a diffeomorphism for each $t\in I$
\item the curve 
\begin{align*}
I&\to\alpha(\Gamma)\bs\bb{H}^2\\
t&\mapsto
\begin{cases}
h_{2t}(\ol{i})&0\leq t\leq\frac{1}{2}\\
h^\prime_{2t-1}(\ol{i})&\frac{1}{2}\leq t\leq1
\end{cases}
\end{align*}
is $C^0$ homotopic to the constant curve $\ol{y}$ relative to endpoints. 
\end{itemize}
(Thus $h_t(\ol{i})$ is $C^0$ homotopic to $h^\prime_{1-t}(\ol{i})$ relative to endpoints.) By the homotopy extension property there exists a $C^0$ homotopy $h^{\prime\prime}_t$ between $f$ and $g$ such that $h^{\prime\prime}_t(\ol{i})=\ol{y}$ for all $t\in I$. By Corollary \ref{sigmaconoriclH}, there exists a $C^0$ homotopy $h^{\prime\prime\prime}_t$ between $f$ and $g$ such that: 
\begin{itemize}
\setlength\itemsep{0em}
\item $h^{\prime\prime\prime}_t(\ol{i})=\ol{y}$
\item $h^{\prime\prime\prime}_t$ is a diffeomorphism for all $t\in I$
\item the map 
\begin{align*}
T^1(\Gamma\bs\bb{H}^2)\times I&\to T^1(\alpha(\Gamma)\bs\bb{H}^2)\\
(v,t)&\mapsto\wh{h^{\prime\prime\prime}_t}(v)
\end{align*}
is continuous. 
\end{itemize}
Define a $C^0$ homotopy $f_t$ between $f$ and $f^\prime$ by 
\begin{equation*}
f_t=
\begin{cases}
h^{\prime\prime\prime}_{2t}&0\leq t\leq\frac{1}{2}\\
h^\prime_{2-2t}&\frac{1}{2}\leq t\leq1. 
\end{cases}
\end{equation*}
Then $f_t(\ol{i})$ is homotopic to $h^\prime_{1-t}(\ol{i})$, hence is homotopic to $h_t(\ol{i})$ both relative to endpoints. 
\end{proof}

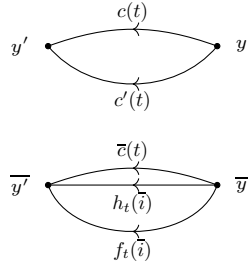
\begin{equation*}
\begin{tikzpicture}[every label/.append style={font=\scriptsize},matrix of math nodes,decoration={markings,mark=at position0.5with{\arrow{>}}}]
\def\a{1.6}
\node(1)[label=left:y^\prime]{};
\node(2)[right=2 of 1,label=right:y]{};
\node(3)[below=\a of 1,label=left:\ol{y^\prime}]{};
\node(4)[below=\a of 2,label=right:\ol{y}]{};
\draw[postaction={decorate}](2.center)to[out=160,in=20]node[above,scale=.75]{c(t)}(1.center);
\draw[postaction={decorate}](2.center)to[out=230,in=-50]node[below,scale=.75]{c^\prime(t)}(1.center);
\draw[postaction={decorate}](4.center)to[out=160,in=20]node[above,scale=.75]{\ol{c}(t)}(3.center);
\draw[postaction={decorate}](4.center)--node[below,scale=.75]{h_t(\ol{i})}(3.center);
\draw[postaction={decorate}](4.center)to[out=250,in=-70]node[below,scale=.75]{f_t(\ol{i})}(3.center);
\foreach\x in{1,...,4}\filldraw(\x)circle(1pt);
\end{tikzpicture}
\end{equation*}
By the homotopy lifting property applied to the homotopy between $\ol{c}(t)$ and $f_t(\ol{i})$, and a lift $c(t)$ of $\ol{c}(t)$, there exists a continuous map $c^\prime\colon I\to\bb{H}^2$ such that: 
\begin{itemize}
\setlength\itemsep{0em}
\item $\ol{c^\prime}(t)=f_t(\ol{i})$ for all $t\in I$
\item $c^\prime(0)=y$, $c^\prime(1)=y^\prime$. 
\end{itemize}
There exists a unique $w_t\in T_{c^\prime(t)}^1\bb{H}^2$ such that $\ol{w_t}=\wh{f_t}(\ol{u})$. Note that $w_t$ depends continuously on $t$. We have 
\begin{equation*}
\begin{tikzcd}
\Gamma g\ar[r,phantom,"\in"]\ar[d,mapsto]&[-20pt]\Gamma\bs G\ar[r,"Df_t"]\ar[d,dash,sloped,"\sim"]\ar[rd,phantom,"\circlearrowright"]&\alpha(\Gamma)\bs G\ar[d,dash,sloped,"\sim"]\ar[r,phantom,"\ni"]&[-20pt]\alpha(\Gamma)g\ar[d,mapsto]\\
\Gamma g_*u\ar[r,phantom,"\in"]&\Gamma\bs T^1\bb{H}^2\ar[r]\ar[d,dash,sloped,"\sim"]\ar[rd,phantom,"\circlearrowright"]&\alpha(\Gamma)\bs T^1\bb{H}^2\ar[d,dash,sloped,"\sim"]\ar[r,phantom,"\ni"]&\alpha(\Gamma)g_*w_t\\
&T^1(\Gamma\bs\bb{H}^2)\ar[r,"\wh{f_t}"']&T^1(\alpha(\Gamma)\bs\bb{H}^2)
\end{tikzcd}
\end{equation*}
and $(Df_t)(\Gamma)=\alpha(\Gamma)$. So $Df_t$ is a based $C^0$ homotopy between $Df$ and $Df^\prime$. Hence 
\begin{equation*}
(Df)_*=(Df^\prime)_*\colon\pi_1(\Gamma\bs G,\Gamma)\to\pi_1(\alpha(\Gamma)\bs G,\alpha(\Gamma))
\end{equation*}
and $\wt{\alpha}$ is well-defined (see Diagram \eqref{pigggdfa}).

\subsubsection{$\sigma$ is a section of $\pi$ (proof of Lemma \ref{section})}\label{section s}
We have 
\begin{equation*}
\begin{tikzcd}
\wt{\Gamma}\wt{g}\ar[r,phantom,"\in"]\ar[ddd,mapsto]&[-20pt]\wt{\Gamma}\bs\wt{G}\ar[r,"Df"]\ar[d,dash,sloped,"\sim"]\ar[rd,phantom,"\circlearrowright"]&\wt{\alpha(\Gamma)}\bs\wt{G}\ar[d,dash,sloped,"\sim"]\ar[r,phantom,"\ni"]&[-20pt]\wt{\alpha(\Gamma)}\wt{g}\ar[ddd,mapsto]\\
&\Gamma\bs G\ar[r]\ar[d,dash,sloped,"\sim"]\ar[rd,phantom,"\circlearrowright"]&\alpha(\Gamma)\bs G\ar[d,dash,sloped,"\sim"]\\
&T^1(\Gamma\bs\bb{H}^2)\ar[r,"\wh{f}"]\ar[d]\ar[rd,phantom,"\circlearrowright"]&T^1(\alpha(\Gamma)\bs\bb{H}^2)\ar[d]\\
\Gamma(p\wt{g})i\ar[r,phantom,"\in"]&\Gamma\bs\bb{H}^2\ar[r,"f"']&\alpha(\Gamma)\bs\bb{H}^2\ar[r,phantom,"\ni"]&\alpha(\Gamma)(p\wt{g})y. 
\end{tikzcd}
\end{equation*}
This shows the commutativity of the inner rectangle in the following diagram: 
\begin{equation}\label{gpgggp1}
\begin{tikzcd}
\wt{\Gamma}\ar[rrr,"\wt{\alpha}"]\ar[rd,dash,sloped,"\sim"]\ar[rrr,phantom,bend right=11,"\circlearrowright"]\ar[ddd,"p"']
&[-10pt]&&[-17pt]\wt{\alpha(\Gamma)}\ar[ddd,"p"]
\\[-10pt]
&\pi_1(\wt{\Gamma}\bs\wt{G},\wt{\Gamma})\ar[r,"(Df)_*"]\ar[d]\ar[rd,phantom,"\circlearrowright"]&\pi_1(\wt{\alpha(\Gamma)}\bs\wt{G},\wt{\alpha(\Gamma)})\ar[d]\ar[ru,dash,sloped,"\sim"]\\
&\pi_1(\Gamma\bs\bb{H}^2,\ol{i})\ar[r,"f_*"']\ar[ld,sloped,"\sim","\iota_i"']&\pi_1(\alpha(\Gamma)\bs\bb{H}^2,\ol{y})\ar[rd,sloped,"\sim","\iota_y"']\\[-5pt]
\Gamma\ar[rrr,"\alpha"']\ar[rrr,phantom,bend left=11,"\circlearrowright"]&&&\alpha(\Gamma). 
\end{tikzcd}
\end{equation}
The commutativity of the left trapezoid in \eqref{gpgggp1} follows from the following diagram and a figure. 
\begin{equation*}
\begin{tikzcd}
\wt{g}\ar[r,phantom,"\in"]\ar[d,mapsto]&[-20pt]\wt{G}\ar[r]\ar[d]\ar[rd,phantom,"\circlearrowright"]&\wt{\Gamma}\bs\wt{G}\ar[r,phantom,"\ni"]\ar[d]&[-20pt]\wt{\Gamma}\wt{g}\ar[d,mapsto]\\
(p\wt{g})i\ar[r,phantom,"\in"]&\bb{H}^2\ar[r]&\Gamma\bs\bb{H}^2\ar[r,phantom,"\ni"]&\Gamma(p\wt{g})i
\end{tikzcd}
\end{equation*}
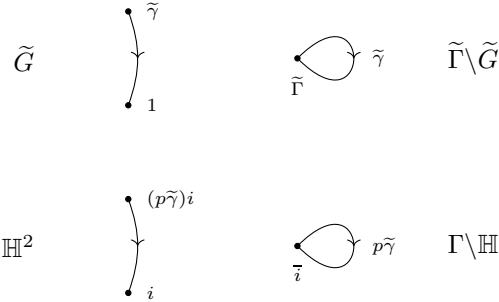
\begin{equation*}
\begin{tikzpicture}[every label/.append style={font=\scriptsize},matrix of math nodes,decoration={markings,mark=at position0.5with{\arrow{>}}}]
\node(1)[label=right:1]{};
\node(2)[above=of 1,label=right:\wt{\gamma}]{};
\path(1)--node(3){}(2);
\node(4)[right=2 of 3,label=below:\wt{\Gamma}]{};
\node(5)at(4){};
\node[left=of 3]{\wt{G}};
\draw[postaction={decorate}](2.center)to[out=-70,in=70](1.center);
\draw[postaction={decorate}](4.center)..controls+(1,1)and+(1,-1)..(5.center)node(6)[pos=.5,label=right:\wt{\gamma}]{};
\node[right=of 6]{\wt{\Gamma}\bs\wt{G}};

\node(7)[below=of 1,label=right:(p\wt{\gamma})i]{};
\node(8)[below=of 7,label=right:i]{};
\path(7)--node(9){}(8);
\node(10)[right=2 of 9,label=below:\ol{i}]{};
\node(11)at(10){};
\node[left=of 9]{\bb{H}^2};
\draw[postaction={decorate}](7.center)to[out=-70,in=70](8.center);
\draw[postaction={decorate}](10.center)..controls+(1,1)and+(1,-1)..(11.center)node(12)[pos=.5,label=right:p\wt{\gamma}]{};
\node[right=of 12]{\Gamma\bs\bb{H}^2};
\foreach\x in{1,2,4,7,8,10}\filldraw(\x)circle(1pt);
\end{tikzpicture}
\end{equation*}
The commutativity of the right trapezoid in \eqref{gpgggp1} follows similarly. Hence the commutativity of the outer rectangle in \eqref{gpgggp1} follows, ie $\pi(\wt{\alpha})=\alpha$.

\subsubsection{$(\Aut(G)\simeq\Aut(\wt{G}))$-equivariance of $\sigma$ (proof of Lemma \ref{g eq})}\label{g eq s}
Let 
\begin{equation*}
\PGL(2,\bb{R})=\GL(2,\bb{R})\left/\left\{
\begin{pmatrix}
a&\\
&a
\end{pmatrix}
\ \middle|\ a\in\bb{R}\setminus\{0\}\right\}\right.. 
\end{equation*}
The inclusion $\SL(2,\bb{R})\subset\GL(2,\bb{R})$ induces an inclusion $\PSL(2,\bb{R})\subset\PGL(2,\bb{R})$. We have a determinant homomorphism: 
\begin{equation*}
\begin{tikzcd}
\GL(2,\bb{R})\ar[r,"\det"]\ar[d]\ar[rd,phantom,"\circlearrowright"]&\bb{R}^\times\ar[d]\\
\PGL(2,\bb{R})\ar[r,"\det"']&\{\pm\}. 
\end{tikzcd}
\end{equation*}
Then $\PSL(2,\bb{R})=\left\{g\in\PGL(2,\bb{R})\ \middle|\ \det g>0\right\}$. 

For $g=
\begin{pmatrix}
a&b\\
c&d
\end{pmatrix}
\in\PGL(2,\bb{R})$, define a map $g\colon\bb{H}^2\to\bb{H}^2$ by 
\begin{equation*}
gz=
\begin{cases}
\frac{az+b}{cz+d}&\det g>0\\
\frac{a\ol{z}+b}{c\ol{z}+d}&\det g<0. 
\end{cases}
\end{equation*}
%
%
It is known that this gives an isomorphism $\PGL(2,\bb{R})\simeq\Isom(\bb{H}^2)$. We also have 
\begin{align*}
\PGL(2,\bb{R})&\simeq\Aut(G)\\
g&\mapsto g\cdot g^{-1}. 
\end{align*}

Let $\alpha\in\mca{H}(\Gamma,G)$, $h\in\PGL(2,\bb{R})$ and 
\begin{alignat*}{2}
\PGL(2,\bb{R})&\simeq\Aut(G)&&\simeq\Aut(\wt{G})\\
h&\mapsto\Phi&&\mapsto\wt{\Phi}. 
\end{alignat*}
$h\colon\bb{H}^2\to\bb{H}^2$ induces $\ol{h}\colon\alpha(\Gamma)\bs\bb{H}^2\to h\alpha(\Gamma)h^{-1}\bs\bb{H}^2$. Let $y\in\bb{H}^2$ and $f\colon(\Gamma\bs\bb{H}^2,\ol{i})\to(\alpha(\Gamma)\bs\bb{H}^2,\ol{y})$ be a diffeomorphism such that 
\begin{equation*}
\begin{tikzcd}
\pi_1(\Gamma\bs\bb{H}^2,\ol{i})\ar[r,"f_*"]\ar[d,"\sim"{sloped},"\iota_i"']\ar[rd,phantom,"\circlearrowright"{xshift=-3}]&\pi_1(\alpha(\Gamma)\bs\bb{H}^2,\ol{y})\ar[d,"\iota_y","\sim"'{sloped}]\\
\Gamma\ar[r,"\alpha"']&\alpha(\Gamma). 
\end{tikzcd}
\end{equation*}
Let $f^\prime$ be defined by 
\begin{equation}\label{ghahhahh}
\begin{tikzcd}[row sep=tiny]
&\alpha(\Gamma)\bs\bb{H}^2\ar[dd,"\ol{h}"]\\
\Gamma\bs\bb{H}^2\ar[ru,"f"]\ar[rd,"f^\prime"']\ar[r,phantom,"\circlearrowright"{xshift=10}]&\ \\
&h\alpha(\Gamma)h^{-1}\bs\bb{H}^2. 
\end{tikzcd}
\end{equation}

\begin{claim}
We have 
\begin{equation*}
\begin{tikzcd}
\pi_1(\alpha(\Gamma)\bs\bb{H}^2,\ol{y})\ar[r,"\sim","\iota_y"{yshift=7}]\ar[d,"\ol{h}_*"']\ar[rd,phantom,"\circlearrowright"{xshift=10}]&\alpha(\Gamma)\ar[d,"h\ \cdot\ h^{-1}"]\\
\pi_1(h\alpha(\Gamma)h^{-1}\bs\bb{H}^2,\overline{hy})\ar[r,"\sim","\iota_{hy}"']&h\alpha(\Gamma)h^{-1}. 
\end{tikzcd}
\end{equation*}
\end{claim}

\begin{proof}
Let $\delta\in\alpha(\Gamma)$. The claim follows from the following figure. 
\begin{equation*}
\begin{tikzpicture}[every label/.append style={font=\scriptsize},matrix of math nodes,decoration={markings,mark=at position0.5with{\arrow{>}}}]
\clip(-.8,-1.8)rectangle(6.7,1.6);
\node(1)[label=left:y]{};
\node(2)[above=of 1,label=left:\delta y]{};
\node(3)[below=of 1,label=left:\ol{y}]{};
\node(4)at(3){};
\draw[postaction={decorate}](2.center)to[out=-50,in=50]node[label=right:c_\delta(t)]{}(1.center);
\draw[postaction={decorate}](3.center)..controls+(1,1)and+(1,-1)..(4.center)node[pos=.5,label=right:\ol{c_\delta}(t)]{};

\node(5)[right=3.5 of 1,label=right:hy]{};
\node(6)[above=of 5,label=right:h\delta y\ {=}\ (h\delta h^{-1})hy]{};
\node(7)[below=of 5,label=left:\ol{h}\ol{y}\ {=}\ \ol{hy}]{};
\node(8)at(7){};
\draw[postaction={decorate}](6.center)to[out=-50,in=50]node[label=right:hc_\delta(t)]{}(5.center);
\draw[postaction={decorate}](7.center)..controls+(1,1)and+(1,-1)..(8.center)node[pos=.5,label=right:\ol{hc_\delta}(t)\ {=}\ \ol{h}\ol{c_\delta}(t)]{};
\foreach\x in{1,2,3,5,6,7}\filldraw(\x)circle(1pt);
\end{tikzpicture}\qedhere
\end{equation*}
\end{proof}

So $f^\prime\colon(\Gamma\bs\bb{H}^2,\ol{i})\to(h\alpha(\Gamma)h^{-1}\bs\bb{H}^2,\ol{hy})$ is a diffeomorphism such that 
\begin{equation*}
\begin{tikzcd}
\pi_1(\Gamma\bs\bb{H}^2,\ol{i})\ar[r,"f^\prime_*"]\ar[d,"\sim"{sloped},"\iota_i"']\ar[rd,phantom,"\circlearrowright"]&\pi_1(h\alpha(\Gamma)h^{-1}\bs\bb{H}^2,\ol{hy})\ar[d,"\iota_{hy}","\sim"'{sloped}]\\
\Gamma\ar[r,"h\alpha h^{-1}"']&h\alpha(\Gamma)h^{-1}. 
\end{tikzcd}
\end{equation*}

\begin{claim}
We have 
\begin{equation*}
\begin{tikzcd}[row sep=tiny]
&[-10pt]&[-10pt]&[-11pt]\alpha(\Gamma)\bs G\ar[r,phantom,"\ni"]\ar[dddd,"\ol{\Phi}"]&[-20pt]\alpha(\Gamma)g\ar[dddd,mapsto]\\
&&T^1(\alpha(\Gamma)\bs\bb{H}^2)\ar[dd,"(\ol{h})_*"]\ar[ru,dash,sloped,"\sim"]\\
\Gamma\bs G\ar[r,dash,"\sim"]\ar[rrruu,bend left,"Df"]\ar[rrruu,phantom,bend left=15,"\circlearrowright"]\ar[rrrdd,bend right,"Df^\prime"']\ar[rrrdd,phantom,bend right=15,"\circlearrowright"]&T^1(\Gamma\bs\bb{H}^2)\ar[ru,"\wh{f}"]\ar[rd,"\wh{f^\prime}"']\ar[r,phantom,"\circlearrowright"]&\ \ar[r,phantom,"\circlearrowright"]&\ \\
&&T^1(h\alpha(\Gamma)h^{-1}\bs\bb{H}^2)\ar[rd,dash,sloped,"\sim"]\\
&&&h\alpha(\Gamma)h^{-1}\bs G\ar[r,phantom,"\ni"]&h\alpha(\Gamma)h^{-1}hgh^{-1}. 
\end{tikzcd}
\end{equation*}
\end{claim}

\begin{proof}
The commutativity of the inner triangle follows from \eqref{ghahhahh}. We show the commutativity of the right trapezoid. Applying $(\ol{h})_*$ to $\wh{f}(\ol{u})=\ol{w}$, we get $\wh{f^\prime}(\ol{u})=\ol{h_*w}$. So starting at $\alpha(\Gamma)\bs G$ and going counterclockwise, 
\begin{align*}
\alpha(\Gamma)g&\mapsto\alpha(\Gamma)g_*w\mapsto h\alpha(\Gamma)h^{-1}h_*g_*w=h\alpha(\Gamma)h^{-1}(hgh^{-1})_*h_*w\\
&\mapsto h\alpha(\Gamma)h^{-1}hgh^{-1}. \qedhere
\end{align*}
\end{proof}

Recall that we have 
\begin{equation*}
\begin{tikzcd}
\wt{G}\ar[r,"\wt{\Phi}"]\ar[d,"p"']\ar[rd,phantom,"\circlearrowright"{xshift=1}]&\wt{G}\ar[d,"p"]\\
G\ar[r,"\Phi"']&G. 
\end{tikzcd}
\end{equation*}
So $p^{-1}(\Phi\alpha(\Gamma))=\wt{\Phi}p^{-1}(\alpha(\Gamma))$, ie $\wt{\Phi\alpha(\Gamma)}=\wt{\Phi}\wt{\alpha(\Gamma)}$, and 
\begin{equation}\label{aggaggpppa}
\begin{tikzcd}
\alpha(\Gamma)\bs G\ar[d,"\ol{\Phi}"']\ar[rd,phantom,"\circlearrowright"{xshift=3,yshift=-2}]&\wt{\alpha(\Gamma)}\bs\wt{G}\ar[l,"\sim"']\ar[d,"\ol{\wt{\Phi}}"]\\
\Phi\alpha(\Gamma)\bs G&\wt{\Phi\alpha(\Gamma)}\bs\wt{G}. \ar[l,"\sim"']
\end{tikzcd}
\end{equation}

\begin{claim}
We have 
\begin{equation*}
\begin{tikzcd}
\pi_1(\alpha(\Gamma)\bs G,\alpha(\Gamma))\ar[d,"(\ol{\Phi})_*"']\ar[rd,phantom,"\circlearrowright"{xshift=3,yshift=-2}]&\pi_1(\wt{\alpha(\Gamma)}\bs\wt{G},\wt{\alpha(\Gamma)})\ar[l,"\sim"']\ar[r,"\sim"]\ar[d,"(\ol{\wt{\Phi}})_*"']\ar[rd,phantom,"\circlearrowright"]&\wt{\alpha(\Gamma)}\ar[d,"\wt{\Phi}"]\\
\pi_1(\Phi\alpha(\Gamma)\bs G,\Phi\alpha(\Gamma))&\pi_1(\wt{\Phi\alpha(\Gamma)}\bs\wt{G},\wt{\Phi\alpha(\Gamma)})\ar[l,"\sim"']\ar[r,"\sim"]&\wt{\Phi\alpha(\Gamma)}. 
\end{tikzcd}
\end{equation*}
\end{claim}

\begin{proof}
The commutativity of left rectangle follows from \eqref{aggaggpppa}. Let $\delta\in\wt{\alpha(\Gamma)}$. The commutativity of the right rectangle follows from the following figure. 
\begin{equation*}
\begin{tikzpicture}[every label/.append style={font=\scriptsize},matrix of math nodes,decoration={markings,mark=at position0.5with{\arrow{>}}}]
\clip(-3,-2)rectangle(8.2,1.6);
\node(1)[label=left:1]{};
\node(2)[above=of 1,label=left:\delta]{};
\node(3)[below=of 1,label=left:\wt{\alpha(\Gamma)}]{};
\node(4)at(3){};
\draw[postaction={decorate}](2.center)to[out=-50,in=50](1.center);
\draw[postaction={decorate}](3.center)..controls+(1,1)and+(1,-1)..(4.center)node[pos=.5,label=right:\delta]{};

\node(5)[right=4 of 1,label=left:1]{};
\node(6)[above=of 5,label=left:\wt{\Phi}(\delta)]{};
\node(7)[below=of 5,label=left:\wt{\Phi\alpha(\Gamma)}]{};
\node(8)at(7){};
\draw[postaction={decorate}](6.center)to[out=-50,in=50](5.center);
\draw[postaction={decorate}](7.center)..controls+(1,1)and+(1,-1)..(8.center)node[pos=.5,label=right:\wt{\Phi}(\delta)]{};
\foreach\x in{1,2,3,5,6,7}\filldraw(\x)circle(1pt);

\path(1)--node(9){}(2);
\node[left=1.5 of 9]{\wt{G}};
\node[left=1.5 of 3]{\wt{\alpha(\Gamma)}\bs\wt{G}};
\path(5)--node(10){}(6);
\node[right=2 of 10]{\wt{G}};
\node[right=2 of 7]{\wt{\Phi\alpha(\Gamma)}\bs\wt{G}};

\path(9)--(10)node(11)[pos=.35]{}node[pos=.55](12){};
\draw[->](11)--node[label=above:\wt{\Phi}]{}(12);
\path(3)--(7)node(13)[pos=.35]{}node[pos=.55](14){};
\draw[->](13)--node[label=below:\ol{\wt{\Phi}}]{}(14);
\end{tikzpicture}\qedhere
\end{equation*}
\end{proof}

The two claims above show the commutativity of the following diagram: 
\begin{equation*}
\begin{tikzcd}[row sep=tiny]
&[-10pt]&[-10pt]&[-10pt]\wt{\alpha(\Gamma)}\ar[dddd,"\wt{\Phi}"]\\
&&\pi_1(\alpha(\Gamma)\bs G,\alpha(\Gamma))\ar[dd,"(\ol{\Phi})_*"]\ar[ru,dash,sloped,"\sim"]\\
\wt{\Gamma}\ar[r,dash,"\sim"]\ar[rrruu,bend left,"\wt{\alpha}"]\ar[rrruu,phantom,bend left=15,"\circlearrowright"]\ar[rrrdd,bend right,"\wt{\Phi\alpha}"']\ar[rrrdd,phantom,bend right=15,"\circlearrowright"]&\pi_1(\Gamma\bs G,\Gamma)\ar[ru,"(Df)_*"]\ar[rd,"(Df^\prime)_*"']\ar[r,phantom,"\circlearrowright"{xshift=6}]&\ \ar[r,phantom,"\circlearrowright"{xshift=3}]&\ \\
&&\pi_1(\Phi\alpha(\Gamma)\bs G,\Phi\alpha(\Gamma))\ar[rd,dash,sloped,"\sim"]\\
&&&\wt{\Phi\alpha(\Gamma)}. 
\end{tikzcd}
\end{equation*}
Therefore $\wt{\Phi\alpha}=\wt{\Phi}\wt{\alpha}$, which is the equivariance.

\subsubsection{$(\Aut(\Gamma)\to\Aut(\wt{\Gamma}))$-equivariance of $\sigma$ (proof of Lemma \ref{gamma eq})}\label{gamma eq s}
Let $\alpha\in\mca{H}(\Gamma,G)$ and $\beta\in\Aut(\Gamma)$. There exists a diffeomorphism $g\colon(\Gamma\bs\bb{H}^2,\ol{i})\to(\Gamma\bs\bb{H}^2,\ol{i})$ such that 
\begin{equation*}
\begin{tikzcd}
\pi_1(\Gamma\bs\bb{H}^2,\ol{i})\ar[r,"g_*"]\ar[d,"\sim"{sloped},"\iota_i"']\ar[rd,phantom,"\circlearrowright"{xshift=-7,yshift=2}]&\pi_1(\Gamma\bs\bb{H}^2,\ol{i})\ar[d,"\sim"{sloped},"\iota_i"']\\
\Gamma\ar[r,"\beta"']&\Gamma
\end{tikzcd}
\end{equation*}
and $(g_*)_{\ol{i}}=\id\colon T_i^1(\Gamma\bs\bb{H}^2)\to T_i^1(\Gamma\bs\bb{H}^2)$. Let $y\in\bb{H}^2$. There exists a diffeomorphism $f\colon(\Gamma\bs\bb{H}^2,\ol{i})\to(\alpha(\Gamma)\bs\bb{H}^{2},\ol{y})$ such that 
\begin{equation*}
\begin{tikzcd}
\pi_1(\Gamma\bs\bb{H}^2,\ol{i})\ar[r,"f_*"]\ar[d,dash,sloped,"\sim"]\ar[rd,phantom,"\circlearrowright"{xshift=-5}]&\pi_1(\alpha(\Gamma)\bs\bb{H}^2,\ol{y})\ar[d,dash,sloped,"\sim"]\\
\Gamma\ar[r,"\alpha"']&\alpha(\Gamma). 
\end{tikzcd}
\end{equation*}
Then $fg\colon(\Gamma\bs\bb{H}^2,\ol{i})\to(\alpha(\Gamma)\bs\bb{H}^2,\ol{y})$ is a diffeomorphism such that 
\begin{equation*}
\begin{tikzcd}
\pi_1(\Gamma\bs\bb{H}^2,\ol{i})\ar[r,"(fg)_*"]\ar[d,dash,"\sim"{sloped}]\ar[rd,phantom,"\circlearrowright"{xshift=-5}]&\pi_1(\alpha(\Gamma)\bs\bb{H}^2,\ol{y})\ar[d,dash,"\sim"{sloped}]\\
\Gamma\ar[r,"\alpha\beta"']&\alpha(\Gamma). 
\end{tikzcd}
\end{equation*}
We have $\wh{fg}=\wh{f}\wh{g}$ and $\wh{g}(\ol{u})=\ol{u}$, where $u=(i,i)\in T_i^1\bb{H}^2$. Recall that we have a unique $w\in T_y^1\bb{H}^2$ such that $\wh{f}(\ol{u})=\ol{w}$. Hence $\wh{fg}(\ol{u})=\ol{w}$. We have 
\begin{equation*}
\begin{tikzcd}
\Gamma g\ar[r,phantom,"\in"]\ar[d,mapsto]&[-20pt]\Gamma\bs G\ar[r,"Dg"]\ar[d,dash,sloped,"\sim"]\ar[rd,phantom,"\circlearrowright"]&\Gamma\bs G\ar[d,dash,sloped,"\sim"]\ar[r,phantom,"\ni"]&[-20pt]\Gamma g\ar[d,mapsto]\\
\Gamma g_*u\ar[r,phantom,"\in"]&T^1(\Gamma\bs\bb{H}^2)\ar[r,"\wh{g}"']&T^1(\Gamma\bs\bb{H}^2)\ar[r,phantom,"\ni"]&\Gamma g_*u, 
\end{tikzcd}
\end{equation*}
\begin{equation*}
\begin{tikzcd}
\Gamma g\ar[r,phantom,"\in"]\ar[d,mapsto]&[-20pt]\Gamma\bs G\ar[r,"Df"]\ar[d,dash,sloped,"\sim"]\ar[rd,phantom,"\circlearrowright"{yshift=-2}]&\alpha(\Gamma)\bs G\ar[d,dash,sloped,"\sim"]\ar[r,phantom,"\ni"]&[-20pt]\alpha(\Gamma)g\ar[d,mapsto]\\
\Gamma g_*u\ar[r,phantom,"\in"]&T^1(\Gamma\bs\bb{H}^2)\ar[r,"\wh{f}"']&T^1(\alpha(\Gamma)\bs\bb{H}^2)\ar[r,phantom,"\ni"]&\alpha(\Gamma)g_*w, 
\end{tikzcd}
\end{equation*}
\begin{equation*}
\begin{tikzcd}
\Gamma g\ar[r,phantom,"\in"]\ar[d,mapsto]&[-20pt]\Gamma\bs G\ar[r,"D(fg)"]\ar[d,dash,sloped,"\sim"]\ar[rd,phantom,"\circlearrowright"{xshift=2,yshift=-2}]&\alpha(\Gamma)\bs G\ar[d,dash,sloped,"\sim"]\ar[r,phantom,"\ni"]&[-20pt]\alpha(\Gamma)g\ar[d,mapsto]\\
\Gamma g_*u\ar[r,phantom,"\in"]&T^1(\Gamma\bs\bb{H}^2)\ar[r,"\wh{fg}"']&T^1(\alpha(\Gamma)\bs\bb{H}^2)\ar[r,phantom,"\ni"]&\alpha(\Gamma)g_*w. 
\end{tikzcd}
\end{equation*}
Thus $D(fg)=(Df)(Dg)$. We have 
\begin{equation*}
\begin{tikzcd}[row sep=tiny]
&[-10pt]&[-10pt]&[-10pt]\wt{\Gamma}\ar[dddd,"\wt{\alpha}"]\\
&&\pi_1(\Gamma\bs G,\Gamma)\ar[dd,"(Df)_*"]\ar[ru,dash,sloped,"\sim"]\\
\wt{\Gamma}\ar[r,dash,"\sim"]\ar[rrruu,bend left,"\wt{\beta}"]\ar[rrruu,phantom,bend left=15,"\circlearrowright"]\ar[rrrdd,bend right,"\wt{\alpha\beta}"']\ar[rrrdd,phantom,bend right=15,"\circlearrowright"]&\pi_1(\Gamma\bs G,\Gamma)\ar[ru,"(Dg)_*"]\ar[rd,"(D(fg))_*"']\ar[r,phantom,"\circlearrowright"{xshift=4}]&\ \ar[r,phantom,"\circlearrowright"{xshift=5}]&\ \\
&&\pi_1(\alpha(\Gamma)\bs G,\alpha(\Gamma))\ar[rd,dash,sloped,"\sim"]\\
&&&\wt{\alpha(\Gamma)}
\end{tikzcd}
\end{equation*}
Therefore $\wt{\alpha\beta}=\wt{\alpha}\wt{\beta}$. This is the equivariance.

\subsection{The structure of $\Aut(\wt{G})\bs\mca{H}(\wt{\Gamma},\wt{G})$}
\begin{lem}\label{egzez1}
For any $e\in\Hom(\wt{\Gamma},Z)$, we have $e(Z)=\{1\}$. 
\end{lem}

\begin{proof}
As in the proof of Lemma \ref{descends}, let 
\begin{equation*}
\Gamma=\left\langle a_1,\ldots,a_g,b_1,\ldots,b_g\ \middle|\ [a_1,b_1]\cdots[a_g,b_g]=1\right\rangle, 
\end{equation*}
$\wt{a}_i$, $\wt{b}_i\in\wt{\Gamma}$ be such that $p\wt{a}_i=a_i$, $p\wt{b}_i=b_i$, and $Z=\langle c\rangle$. Then $[\wt{a}_1,\wt{b}_1]\cdots[\wt{a}_g,\wt{b}_g]=c^{\pm(2-2g)}$. Since $Z\simeq\bb{Z}$, we must have $e(c)=1$. 
\end{proof}

\begin{lem}\label{ejeggauj}
For $e\in\Hom(\wt{\Gamma},Z)$, let 
\begin{align*}
j(e)\colon\wt{\Gamma}&\to\wt{\Gamma}\\
\wt{\gamma}&\mapsto\wt{\gamma}e(\wt{\gamma}). 
\end{align*}
Then $j(e)\in\Aut(\wt{\Gamma})$ and $j\colon\Hom(\wt{\Gamma},Z)\to\Aut(\wt{\Gamma})$ is an injective homomorphism. 
\end{lem}

\begin{proof}
$j(e)$ is a homomorphism. By Lemma \ref{egzez1}, we have $j(e)(\wt{\gamma})=\wt{\gamma}$ for any $\wt{\gamma}\in Z$. Hence 
\begin{equation*}
\begin{tikzcd}
1\ar[r]&Z\ar[r]\ar[d,"\id"]&\wt{\Gamma}\ar[r]\ar[d,"j(e)"]&\Gamma\ar[r]\ar[d,"\id"]&1\\
1\ar[r]&Z\ar[r]&\wt{\Gamma}\ar[r]&\Gamma\ar[r]&1
\end{tikzcd}
\end{equation*}
is commutative. It follows that $j(e)\in\Aut(\wt{\Gamma})$. For any $e$, $e^\prime\in Z$ and $\wt{\gamma}\in\wt{\Gamma}$, 
\begin{align*}
j(e)j(e^\prime)(\wt{\gamma})&=j(e)(\wt{\gamma}e^\prime(\wt{\gamma}))=\wt{\gamma}e^\prime(\wt{\gamma})e(\wt{\gamma})e(e^\prime(\wt{\gamma}))\\
&=\wt{\gamma}(ee^\prime)(\wt{\gamma})=j(ee^\prime)(\wt{\gamma}). 
\end{align*}
Hence $j$ is a homomorphism. Let $e\in\Hom(\wt{\Gamma},Z)$ be such that $j(e)=\id$. Then $\wt{\gamma}e(\wt{\gamma})=\wt{\gamma}$ for any $\wt{\gamma}\in\wt{\Gamma}$, ie $e(\wt{\gamma})=1$. Hence $e=1$ and $j$ is injective. 
\end{proof}

We have commuting actions $\Aut(\wt{G})\curvearrowright\mca{H}(\wt{\Gamma},\wt{G})\curvearrowleft\Aut(\wt{\Gamma})$ by composition. Hence the homomorphism $j$ induces commuting actions $\Aut(\wt{G})\curvearrowright\mca{H}(\wt{\Gamma},\wt{G})\curvearrowleft\Hom(\wt{\Gamma},Z)$. Explicitly for $\wt{\alpha}\in\mca{H}(\wt{\Gamma},\wt{G})$ and $e\in\Hom(\wt{\Gamma},Z)$, the right action is defined by $\wt{\alpha}*e:=\wt{\alpha}\circ j(e)$. Hence for $\wt{\gamma}\in\wt{\Gamma}$, $(\wt{\alpha}*e)(\wt{\gamma})=\wt{\alpha}(\wt{\gamma} e(\wt{\gamma}))=\wt{\alpha}(\wt{\gamma})\wt{\alpha}(e(\wt{\gamma}))$. Note that $\wt{\alpha}(e(\wt{\gamma}))\in Z$, so $\wt{\alpha}*e\equiv\wt{\alpha}$ mod $Z$ and $\pi(\wt{\alpha}*e)=\pi(\wt{\alpha})$. 

\begin{lem}
The action 
\begin{equation*}
\begin{tikzcd}[column sep=tiny]
\mca{H}(\wt{\Gamma},\wt{G})\ar[r,phantom,"\curvearrowleft"]\ar[d,"\pi"]&\Hom(\wt{\Gamma},Z)\\
\mca{H}(\Gamma,G)
\end{tikzcd}
\end{equation*}
is simply transitive on each fiber of $\pi$. 
\end{lem}

\begin{proof}
Let $\wt{\alpha}$, $\wt{\alpha}^\prime\in\mca{H}(\wt{\Gamma},\wt{G})$ be such that $\pi(\wt{\alpha})=\pi(\wt{\alpha}^\prime)$. Since $\wt{\alpha}\colon Z\xrightarrow{\sim}Z$ by Lemma \ref{descends}, there exists a unique $e\colon\wt{\Gamma}\to Z$ such that $\wt{\alpha}^\prime(\wt{\gamma})=\wt{\alpha}(\wt{\gamma})\wt{\alpha}(e(\wt{\gamma}))=\wt{\alpha}(\wt{\gamma}e(\wt{\gamma}))$ for all $\wt{\gamma}\in\wt{\Gamma}$. For any $\wt{\gamma}$, $\wt{\gamma}^\prime\in\wt{\Gamma}$, we have 
\begin{equation*}
\wt{\alpha}(\wt{\gamma}\wt{\gamma}^\prime e(\wt{\gamma}\wt{\gamma}^\prime))=\wt{\alpha}^\prime(\wt{\gamma}\wt{\gamma}^\prime)=\wt{\alpha}^\prime(\wt{\gamma}) \wt{\alpha}^\prime(\wt{\gamma}^\prime)=\wt{\alpha}(\wt{\gamma}e(\wt{\gamma}))\wt{\alpha}(\wt{\gamma}^\prime e(\wt{\gamma}^\prime))=\wt{\alpha}(\wt{\gamma}\wt{\gamma}^\prime e(\wt{\gamma})e(\wt{\gamma}^\prime)), 
\end{equation*}
hence $e\left(\wt{\gamma}\wt{\gamma}^\prime\right)=e(\wt{\gamma})e\left(\wt{\gamma}^\prime\right)$. So $e\in\Hom(\wt{\Gamma},Z)$ and $\wt{\alpha}^\prime=\wt{\alpha}*e$. 

Let $\wt{\alpha}\in\mca{H}(\wt{\Gamma},\wt{G})$, $e\in\Hom(\wt{\Gamma},Z)$ be such that $\wt{\alpha}*e=\wt{\alpha}$. Then for any $\wt{\gamma}\in\wt{\Gamma}$, we have $\wt{\alpha}(\wt{\gamma})\wt{\alpha}(e(\wt{\gamma}))=\wt{\alpha}(\wt{\gamma})$, hence $e(\wt{\gamma})=1$ and $e=1$. 
\end{proof}

We have $\Aut(\wt{\Gamma})\subset\mca{H}(\wt{\Gamma},\wt{G})$. The subset $\Aut(\wt{\Gamma})$ is $\Aut(\wt{\Gamma})$-invariant, hence $\Hom(\wt{\Gamma},Z)$-invariant. 

\begin{cor}
The action 
\begin{equation*}
\begin{tikzcd}[column sep=tiny]
\Aut(\wt{\Gamma})\ar[r,phantom,"\curvearrowleft"]\ar[d,"\pi_0"]&\Hom(\wt{\Gamma},Z)\\
\Aut(\Gamma)
\end{tikzcd}
\end{equation*}
is simply transitive on each fiber of $\pi$. 
\end{cor}

\begin{proof}
This follows from the diagram 
\begin{equation*}
\begin{tikzcd}[column sep=tiny]
\Aut(\wt{\Gamma})\ar[r,phantom,"\subset"]\ar[d,"\pi_0"]&\mca{H}(\wt{\Gamma},\wt{G})\ar[r,phantom,"\curvearrowleft"]\ar[d,"\pi"]&\Hom(\wt{\Gamma},Z)\\
\Aut(\Gamma)\ar[r,phantom,"\subset"]&\mca{H}(\Gamma,G). 
\end{tikzcd}
\end{equation*}
\end{proof}

\begin{lem}
The action 
\begin{equation*}
\begin{tikzcd}[column sep=tiny]
\Aut(\wt{G})\bs\mca{H}(\wt{\Gamma},\wt{G})\ar[r,phantom,"\curvearrowleft"]\ar[d,"\ol{\pi}"]&\Hom(\wt{\Gamma},Z)\\
\Aut(G)\bs\mca{H}(\Gamma,G)
\end{tikzcd}
\end{equation*}
is simply transitive on each fiber of $\ol{\pi}$. 
\end{lem}

\begin{proof}
The action preserves each fiber of $\ol{\pi}$. Let $\wt{\alpha}$, $\wt{\alpha}^\prime\in\mca{H}(\wt{\Gamma},\wt{G})$ be such that there exists $\Phi\in\Aut(G)$ with $\alpha^\prime=\Phi\circ\alpha$, where $\alpha:=\pi(\wt{\alpha})$, $\alpha^\prime:=\pi(\wt{\alpha}^\prime)\in\mca{H}(\Gamma,G)$. Since $\pi(\wt{\Phi}\circ\wt{\alpha})=\pi(\wt{\alpha}^\prime)$, there exists a unique $e\in\Hom(\wt{\Gamma},Z)$ such that $\wt{\alpha}^\prime=(\wt{\Phi}\circ\wt{\alpha})*e=\wt{\Phi}(\wt{\alpha}*e)$. 

Let $\wt{\alpha}\in\mca{H}(\wt{\Gamma},\wt{G})$, $e\in\Hom(\wt{\Gamma},Z)$ be such that there exists $\wt{\Phi}\in\Aut(\wt{G})$ with $\wt{\Phi}\circ(\wt{\alpha}*e)=\wt{\alpha}$. For any $\wt{\gamma}\in\wt{\Gamma}$, we have 
\begin{equation}\label{pagaegag}
\wt{\Phi}(\wt{\alpha}(\wt{\gamma})\wt{\alpha}(e(\wt{\gamma})))=\wt{\alpha}(\wt{\gamma}). 
\end{equation}
Let $\ul{G}=\wt{G}/Z^2\simeq\SL(2,\bb{R})$ and $\ul{\Gamma}=\wt{\Gamma}/Z^2$. Then (since $\wt{\alpha}(Z)\subset Z$) there exist unique $\ul{\Phi}$, $\ul{\alpha}$, $\ul{e}$ such that 
\begin{equation*}
\begin{tikzcd}
\wt{G}\ar[r,"\wt{\Phi}"]\ar[d]\ar[rd,phantom,"\circlearrowright"]&\wt{G}\ar[d]\\
\ul{G}\ar[r,"\ul{\Phi}"']&\ul{G}, 
\end{tikzcd}\qquad
\begin{tikzcd}
\wt{\Gamma}\ar[r,"\wt{\alpha}"]\ar[d]\ar[rd,phantom,"\circlearrowright"{xshift=1,yshift=-1}]&\wt{G}\ar[d]\\
\ul{\Gamma}\ar[r,"\ul{\alpha}"']&\ul{G}, 
\end{tikzcd}\qquad
\begin{tikzcd}
\wt{\Gamma}\ar[r,"e"]\ar[d]\ar[rd,phantom,"\circlearrowright"{xshift=2,yshift=-2}]&Z\ar[d]\\
\ul{\Gamma}\ar[r,"\ul{e}"']&Z/Z^2. 
\end{tikzcd}
\end{equation*}
Since $\wt{\alpha}\colon Z\xrightarrow{\sim}Z$ by Lemma \ref{descends}, $\ul{\alpha}$ is injective. Since $\wt{\alpha}(\wt{\Gamma})\supset Z\supset Z^2$ by the same lemma, $\ul{\alpha}(\ul{\Gamma})$ is discrete. Since there is a surjective map $\wt{\alpha}(\wt{\Gamma})\bs\wt{G}\to\ul{\alpha}(\ul{\Gamma})\bs\ul{G}$, $\ul{\alpha}(\ul{\Gamma})$ is cocompact. 

By \eqref{pagaegag}, 
we have $\ul{\Phi}(\ul{\alpha}(\ul{\gamma})\ul{\alpha}(\ul{e}(\ul{\gamma})))=\ul{\alpha}(\ul{\gamma})$ for all $\ul{\gamma}\in\ul{\Gamma}$. Let $\ul{\Gamma}^\prime=\ker\ul{e}\subset\ul{\Gamma}$. Then $\ul{\alpha}(\ul{\Gamma}^\prime)$ is a cocompact lattice in $\ul{G}$ and $\ul{\Phi}(\ul{\alpha}(\ul{\gamma}))=\ul{\alpha}(\ul{\gamma})$ for all $\ul{\gamma}\in\ul{\Gamma}^\prime$. Hence $\ul{\Phi}=\id$ by the Borel density theorem. So $\wt{\Phi}=\id$, $\wt{\alpha}*e=\wt{\alpha}$ and $e=1$. 
\end{proof}

We have a split short exact sequence 
\begin{equation*}
\begin{tikzcd}
1\ar[r]&\Hom(\wt{\Gamma},Z)\ar[r,"j"]&\Aut(\wt{\Gamma})\ar[r,"\pi_0"']&\Aut(\Gamma)\ar[r]\ar[l,bend right=18,"\sigma_0"']&1
\end{tikzcd}
\end{equation*}
by Lemma \ref{ejeggauj}. We get 
\begin{equation}\label{homgzautga}
\begin{tikzcd}
1\ar[r]&\Hom(\wt{\Gamma},Z)\ar[r,"j"]\ar[d,equal]\ar[rd,phantom,"\circlearrowright"{xshift=2,yshift=-1}]&\Aut(\wt{\Gamma})\ar[r]\ar[d]\ar[rd,phantom,"\circlearrowright"]&\Aut(\Gamma)\ar[r]\ar[d]\ar[l,bend right=15]&1\\
1\ar[r]&\Hom(\wt{\Gamma},Z)\ar[r]&\Out(\wt{\Gamma})\ar[r]&\Out(\Gamma)\ar[r]\ar[l,bend left=15]&1. 
\end{tikzcd}
\end{equation}

\begin{lem}
The bottom row of \eqref{homgzautga} is a split short exact sequence. 
\end{lem}

\begin{proof}
Exactness at $\Hom(\wt{\Gamma},Z)$. Let $e\in\Hom(\wt{\Gamma},Z)$ and assume that there exists $\delta\in\wt{\Gamma}$ such that $\wt{\gamma}e(\wt{\gamma})=\delta\wt{\gamma}\delta^{-1}$ for all $\wt{\gamma}\in\wt{\Gamma}$. Then $p\delta$ is in the center of $\Gamma$. Since $G$ has the trivial center, $p\delta=1$, ie $\delta\in Z$. Hence for any $\wt{\gamma}\in\wt{\Gamma}$, $\wt{\gamma}e(\wt{\gamma})=\wt{\gamma}$. Thus $e(\wt{\gamma})=1$ and $e=1$. 

Exactness at $\Out(\wt{\Gamma})$. Let $\wt{\alpha}\in\Aut(\wt{\Gamma})$ and assume that there exists $\delta\in\Gamma$ such that 
\begin{equation*}
\begin{tikzcd}
\wt{\Gamma}\ar[r,"\wt{\alpha}"]\ar[d]\ar[rd,phantom,"\circlearrowright"]&\wt{\Gamma}\ar[d]\\
\Gamma\ar[r,"\delta\ \cdot\ \delta^{-1}"']&\Gamma. 
\end{tikzcd}
\end{equation*}
Take $\wt{\delta}\in\wt{\Gamma}$ such that $p\wt{\delta}=\delta$. Then for any $\wt{\gamma}\in\wt{\Gamma}$, we have 
\begin{equation*}
p\left(\wt{\delta}^{-1}\wt{\alpha}(\wt{\gamma})\wt{\delta}\right)=\delta^{-1}\delta(p\wt{\gamma})\delta^{-1}\delta=p\wt{\gamma}. 
\end{equation*}
Hence there exists a unique $e(\wt{\gamma})\in Z$ such that $\wt{\delta}^{-1}\wt{\alpha}(\wt{\gamma})\wt{\delta}=\wt{\gamma}e(\wt{\gamma})$. We have $e\in\Hom(\wt{\Gamma},Z)$ and $\wt{\delta}^{-1}\wt{\alpha}\wt{\delta}=j(e)$. 
\end{proof}

\begin{thm}\label{authauthhom}
We have 
\begin{align*}
\Aut(\wt{G})\bs\mca{H}(\wt{\Gamma},\wt{G})&\simeq(\Aut(G)\bs\mca{H}(\Gamma,G))\times\Hom(\wt{\Gamma},Z)\\
(\sigma[\alpha])*e&\mapsfrom([\alpha],e). 
\end{align*}
\end{thm}

Note that $\Hom(\wt{\Gamma},Z)\simeq\Hom(\Gamma,Z)\simeq Z^{2g}$ and $\Aut(G)\bs\mca{H}(\Gamma,G)$ is the Teichm\"{u}ller space of $\Gamma$. 

The action $\Aut(\wt{G})\bs\mca{H}(\wt{\Gamma},\wt{G})\curvearrowleft\Aut(\wt{\Gamma})$ induces an action $\Aut(\wt{G})\bs\mca{H}(\wt{\Gamma},\wt{G})\curvearrowleft\Out(\wt{\Gamma})$. The situation is summarized in the following diagram: 
\begin{equation*}
\begin{tikzcd}[column sep=tiny]
&1\ar[d]\\[-7pt]
&\Hom(\wt{\Gamma},Z)\ar[d]\\[-7pt]
\Aut(\wt{G})\bs\mca{H}(\wt{\Gamma},\wt{G})\ar[r,phantom,"\curvearrowleft"]\ar[d,"\ol{\pi}"']&\Out(\wt{\Gamma})\ar[d,"\ol{\pi}_0"']\\
\Aut(G)\bs\mca{H}(\Gamma,G)\ar[r,phantom,"\curvearrowleft"]\ar[u,bend right,"\ol{\sigma}"']&\Out(\Gamma)\ar[u,bend right,"\ol{\sigma}_0"']\ar[d]\\[-7pt]
&1. 
\end{tikzcd}
\end{equation*}

\subsection{Computation of the mapping class group}
Recall from Section \ref{mpcgs} that 
\begin{align*}
\MCG_o(\rho_\Gamma)&=\Aut((\Gamma\bs G)\rtimes_{\rho_\Gamma}G)/\Aut((\Gamma\bs G)\rtimes_{\rho_\Gamma}G)_{o,1}\\
&=\Aut((\Gamma\bs G)\rtimes_{\rho_\Gamma}G)/{\sim}_o
\end{align*}
and we have an action $T_o(\rho_\Gamma)\curvearrowleft\MCG_o(\rho_\Gamma)$. We prove the following theorem in this section. 

\begin{thm}\label{tormcgo}
We have 
\begin{equation*}
\begin{tikzcd}
T_o(\rho_\Gamma)\ar[r,phantom,"\curvearrowleft"]\ar[d,dash,sloped,"\sim"]&[-23pt]\MCG_o(\rho_\Gamma)\ar[d,dash,sloped,"\sim"]\\
\Aut(\wt{G})\bs\mca{H}(\wt{\Gamma},\wt{G})\ar[r,phantom,"\curvearrowleft"]&\Out(\wt{\Gamma}). 
\end{tikzcd}
\end{equation*}
The left vertical map (see Theorem \ref{traghg}) is equivariant with respect to the actions. 
\end{thm}

\subsubsection{A proposition}
To prove Theorem \ref{tormcgo}, we first prove the following proposition. The groups and homomorphisms will be defined in the proof. 

\begin{prop}\label{gggzgzaut}
We have a diagram 
\begin{equation*}
\begin{tikzcd}[column sep=10]
1\ar[r]&\Gamma\ar[r]\ar[d,equal]\ar[rd,phantom,"\circlearrowright"{xshift=3,yshift=-2}]&\Aut_0((\Gamma\bs G)\rtimes G)/{\sim}_{o,0}\ar[r]\ar[d,dash,"\sim"{sloped}]\ar[rd,phantom,"\circlearrowright"{xshift=-2,yshift=0}]&(\Aut_0((\Gamma\bs G)\rtimes G)/{\sim}_{o,0})/\Gamma\ar[r]\ar[d,dash,"\sim"{sloped}]&1\\
1\ar[r]&\Gamma\ar[r]\ar[d,dash,"\sim"{sloped}]\ar[rd,phantom,"\circlearrowright"{xshift=5,yshift=-3}]&\Diff_0(G)_\Gamma/{\sim}_{o,0}\ar[r]\ar[d,dash,"\sim"{sloped}]\ar[rd,phantom,"\circlearrowright"{xshift=2,yshift=-2}]&(\Diff_0(G)_\Gamma/{\sim}_{o,0})/\Gamma\ar[r]\ar[d,dash,"\sim"{sloped}]&1\\
1\ar[r]&\wt{\Gamma}/Z\ar[r]\ar[d,equal]\ar[rd,phantom,"\circlearrowright"{xshift=0,yshift=-2}]&\Diff_0(\wt{G})_{\wt{\Gamma}}/{\sim}_{o,0}\ar[r]\ar[d,dash,"\sim"{sloped}]\ar[rd,phantom,"\circlearrowright"{xshift=-7,yshift=0}]&(\Diff_0(\wt{G})_{\wt{\Gamma}}/{\sim}_{o,0})/\wt{\Gamma}\ar[r]\ar[d,dash,"\sim"{sloped}]&1\\
1\ar[r]&\wt{\Gamma}/Z\ar[r]&\Aut(\wt{\Gamma})\ar[r]&\Out(\wt{\Gamma})\ar[r]&1
\end{tikzcd}
\end{equation*}
of groups and homomorphisms, where each row is exact. 
\end{prop}

We divide the proof of Proposition \ref{gggzgzaut} into the following three parts.

\subsubsection*{The diagram consisting of the third and fourth rows}
The exact sequence 
\begin{align*}
\wt{\Gamma}&\to\Aut(\wt{\Gamma})\to\Out(\wt{\Gamma})\to1\\
\wt{\gamma}&\mapsto\wt{\gamma}\cdot\wt{\gamma}^{-1}
\end{align*}
gives the fourth row of the diagram in Proposition \ref{gggzgzaut}. It is exact since $Z$ coincides with the center of $\wt{\Gamma}$. 

Recall from Corollary \ref{tblgcgsgxx} (or Definition \ref{gmxdiffgxth}) that 
\begin{equation*}
\Diff(\wt{G})_\wt{\Gamma}=\left\{\wt{\theta}\in\Diff(\wt{G})\ \middle|\ \text{$\wt{\theta}(\wt{\gamma}\wt{g})=\wt{\theta}(\wt{\gamma})\wt{\theta}(\wt{g})$ for all $\wt{\gamma}\in\wt{\Gamma}$, $\wt{g}\in\wt{G}$}\right\}. 
\end{equation*}
Let 
\begin{equation*}
\Diff_0(\wt{G})_\wt{\Gamma}=\left\{\wt{\theta}\in\Diff(\wt{G})_\wt{\Gamma}\ \middle|\ \wt{\theta}(\wt{\Gamma})=\wt{\Gamma}\right\}. 
\end{equation*}
Then $\Diff_0(\wt{G})_\wt{\Gamma}$ is a group. 

We have an equivalence relation $\sim $ on $\Diff(\wt{G})_\wt{\Gamma}$ defined in Definition \ref{gmxththhh}. 

Recall from Definition \ref{hnhnyychy} that 
\begin{equation*}
C^\infty(\wt{G},\wt{G})_{\wt{\Gamma},\wt{\Gamma}}=\left\{\wt{\theta}\in C^\infty(\wt{G},\wt{G})\ \middle|\ 
\begin{gathered}
\wt{\theta}(\wt{\Gamma})\subset\wt{\Gamma}\\
\text{$\wt{\theta}(\wt{\gamma}\wt{g})=\wt{\theta}(\wt{\gamma})\wt{\theta}(\wt{g})$ for all $\wt{\gamma}\in\wt{\Gamma}$, $\wt{g}\in\wt{G}$}
\end{gathered}
\right\}
\end{equation*}
and an equivalence relation $\sim_{o,0}$ on $C^\infty(\wt{G},\wt{G})_{\wt{\Gamma},\wt{\Gamma}}$. 

Since $\Diff_0(\wt{G})_\wt{\Gamma}\subset C^\infty(\wt{G},\wt{G})_{\wt{\Gamma},\wt{\Gamma}}$, the group $\Diff_0(\wt{G})_\wt{\Gamma}$ has two equivalence relations $\sim$ and $\sim_{o,0}$, but they coincide by Definition \ref{gmxththhh}. Hence the map $\Diff_0(\wt{G})_\wt{\Gamma}/{\sim}_{o,0}\to\Diff(\wt{G})_\wt{\Gamma}/{\sim}$ induced by the inclusion is injective. $\Diff_0(\wt{G})_\wt{\Gamma}/{\sim}_{o,0}$ is a group. 

The bijection \eqref{dffwgsimgg} restricts to an isomorphism of groups: 
\begin{equation*}
\begin{tikzcd}[row sep=small,column sep=small]
\Diff(\wt{G})_{\wt{\Gamma}}/{\sim}\ar[r,dash,"\sim"]\ar[d,phantom,sloped,"\supset"]&\mca{H}(\wt{\Gamma},\wt{G})\ar[d,phantom,sloped,"\supset"]\\
\Diff_0(\wt{G})_{\wt{\Gamma}}/{\sim}_{o,0}\ar[r,dash,"\sim"]&\Aut(\wt{\Gamma}). 
\end{tikzcd}
\end{equation*}

The map (note that $\Diff(\wt{G})_\wt{\Gamma}$ is not a group) 
\begin{align*}
\wt{G}&\to\Diff(\wt{G})_\wt{\Gamma}\\
\wt{g}&\mapsto\wt{g}\ \cdot\ \wt{g}^{-1}
\end{align*}
restricts to a homomorphism $\wt{\Gamma}\to\Diff_0(\wt{G})_\wt{\Gamma}$. 

We have 
\begin{equation*}
\begin{tikzcd}
\wt{\Gamma}\ar[r]\ar[d,equal]\ar[rd,phantom,"\circlearrowright"{xshift=5,yshift=-3}]&\Diff_0(\wt{G})_{\wt{\Gamma}}/{\sim}_{o,0}\ar[d,dash,"\sim"{sloped}]\\
\wt{\Gamma}\ar[r]&\Aut(\wt{\Gamma}). 
\end{tikzcd}
\end{equation*}
Hence the image of $\wt{\Gamma}$ in $\Diff_0(\wt{G})_{\wt{\Gamma}}/{\sim}_{o,0}$ is a normal subgroup. This induces 
\begin{equation*}
\begin{tikzcd}
1\ar[r]&\wt{\Gamma}/Z\ar[r]\ar[d,equal]\ar[rd,phantom,"\circlearrowright"{xshift=5,yshift=-3}]&\Diff_0(\wt{G})_{\wt{\Gamma}}/{\sim}_{o,0}\ar[r]\ar[d,dash,"\sim"{sloped}]\ar[rd,phantom,"\circlearrowright"{xshift=-3,yshift=0}]&(\Diff_0(\wt{G})_{\wt{\Gamma}}/{\sim}_{o,0})/\wt{\Gamma}\ar[r]\ar[d,dash,"\sim"{sloped}]&1\\
1\ar[r]&\wt{\Gamma}/Z\ar[r]&\Aut(\wt{\Gamma})\ar[r]&\Out(\wt{\Gamma})\ar[r]&1, 
\end{tikzcd}
\end{equation*}
which gives the third and fourth rows of the diagram in Proposition \ref{gggzgzaut}. Since $Z$ coincides with the center of $\wt{\Gamma}$, the fourth row of the diagram is exact.

\subsubsection*{The diagram consisting of the second and third rows}
The bijection in Lemma \ref{diffggdiffwgwg} restricts to an isomorphism of groups: 
\begin{equation*}
\begin{tikzcd}[row sep=small,column sep=small]
\Diff(G)_\Gamma\ar[r,dash,"\sim"]\ar[d,phantom,sloped,"\supset"]&\Diff(\wt{G})_{\wt{\Gamma}}\ar[d,phantom,sloped,"\supset"]\\
\Diff_0(G)_\Gamma\ar[r,dash,"\sim"]&\Diff_0(\wt{G})_{\wt{\Gamma}}, 
\end{tikzcd}
\end{equation*}
where $\Diff_0(G)_\Gamma$ is defined similarly. Hence 
\begin{equation*}
\begin{tikzcd}[row sep=small,column sep=small]
\Diff(G)_\Gamma/{\sim}\ar[r,dash,"\sim"]\ar[d,phantom,sloped,"\supset"]&\Diff(\wt{G})_{\wt{\Gamma}}/{\sim}\ar[d,phantom,sloped,"\supset"]\\
\Diff_0(G)_\Gamma/{\sim}_{o,0}\ar[r,dash,"\sim"]&\Diff_0(\wt{G})_{\wt{\Gamma}}/{\sim}_{o,0}
\end{tikzcd}
\end{equation*}
by Lemma \ref{diffsimdiffsim}. We have 
\begin{equation*}
\begin{tikzcd}
\Gamma\ar[r]\ar[rd,phantom,"\circlearrowright"{xshift=3,yshift=-3}]&\Diff_0(G)_\Gamma\ar[d,dash,"\sim"{sloped}]\\
\wt{\Gamma}\ar[r]\ar[u,"p"]&\Diff_0(\wt{G})_{\wt{\Gamma}}. 
\end{tikzcd}
\end{equation*}
This induces 
\begin{equation*}
\begin{tikzcd}
1\ar[r]&\Gamma\ar[r]\ar[rd,phantom,"\circlearrowright"{xshift=3,yshift=-3}]&\Diff_0(G)_\Gamma/{\sim}_{o,0}\ar[r]\ar[d,dash,"\sim"{sloped}]\ar[rd,phantom,"\circlearrowright"{xshift=2,yshift=-1}]&(\Diff_0(G)_\Gamma/{\sim}_{o,0})/\Gamma\ar[r]\ar[d,dash,"\sim"{sloped}]&1\\
1\ar[r]&\wt{\Gamma}/Z\ar[r]\ar[u,"p","\sim"'{sloped}]&\Diff_0(\wt{G})_{\wt{\Gamma}}/{\sim}_{o,0}\ar[r]&(\Diff_0(\wt{G})_{\wt{\Gamma}}/{\sim}_{o,0})/\wt{\Gamma}\ar[r]&1, 
\end{tikzcd}
\end{equation*}
which is the second and third rows of the diagram in Proposition \ref{gggzgzaut}. (The exactness at $\Gamma$ follows from the exactness at $\wt{\Gamma}/Z$.)

\subsubsection*{The diagram consisting of the first and second rows}
We see $(\Gamma\bs G)\rtimes G\rightrightarrows\Gamma\bs G$ as a based Lie groupoid $((\Gamma\bs G)\rtimes G\rightrightarrows\Gamma\bs G,\Gamma)$. By Corollary \ref{homhomtimessim}, we have 
\begin{align*}
\End_0((\Gamma\bs G)\rtimes G)&\simeq C^\infty(G,G)_{\Gamma,\Gamma}\\
\varphi&\mapsto\theta_\varphi\\
\varphi_\theta&\mapsfrom\theta, 
\end{align*}
where $\theta_\varphi(g)=\varphi_{\rho_\Gamma}\varphi(\Gamma,g)$, $\varphi_\theta(\Gamma x,g)=(\Gamma\theta(x),\theta(x)^{-1}\theta(xg))$, and $\sim_{o,0}$ on the left corresponds to $\sim_{o,0}$ on the right. Let 
\begin{equation*}
\Aut_0((\Gamma\bs G)\rtimes G)=\left\{\varphi\in\Aut((\Gamma\bs G)\rtimes G)\ \middle|\ F_\varphi(\Gamma)=\Gamma\right\}. 
\end{equation*}

\begin{lem}
We have a group isomorphism 
\begin{align*}
\Aut_0((\Gamma\bs G)\rtimes G)&\simeq\Diff_0(G)_\Gamma\\
\varphi&\mapsto\theta_\varphi\\
\varphi_\theta&\mapsfrom\theta. 
\end{align*}
\end{lem}

\begin{proof}
We have 
\begin{equation*}
\Aut_0((\Gamma\bs G)\rtimes G)=\left\{\varphi\in\End_0((\Gamma\bs G)\rtimes G)\ \middle|\ 
\begin{gathered}
\text{there is $\varphi^\prime\in\End_0((\Gamma\bs G)\rtimes G)$}\\
\text{such that $\varphi\varphi^\prime=\id$, $\varphi^\prime\varphi=\id$}
\end{gathered}
\right\}
\end{equation*}
and 
\begin{equation*}
\Diff_0(G)_\Gamma=\left\{\theta\in C^\infty(G,G)_{\Gamma,\Gamma}\ \middle|\ 
\begin{gathered}
\text{there exists $\theta^\prime\in C^\infty(G,G)_{\Gamma,\Gamma}$}\\
\text{such that $\theta\theta^\prime=\id$, $\theta^\prime\theta=\id$}
\end{gathered}
\right\}. 
\end{equation*}
Hence the lemma follows from Lemma \ref{compos}. 
\end{proof}

Hence it induces a group isomorphism 
\begin{equation*}
\Aut_0((\Gamma\bs G)\rtimes G)/{\sim}_{o,0}\simeq\Diff_0(G)_\Gamma/{\sim}_{o,0}. 
\end{equation*}

We have a homomorphism 
\begin{align*}
G&\to\Aut((\Gamma\bs G)\rtimes G)\\
g&\mapsto\varphi_g
\end{align*}
defined by 
\begin{equation*}
\varphi_g(\Gamma x,g^\prime)=(\Gamma xg^{-1},gg^\prime g^{-1}). 
\end{equation*}
It restricts to a homomorphism $\Gamma\to\Aut_0((\Gamma\bs G)\rtimes G)$. 

We have 
\begin{equation*}
\begin{tikzcd}
\Gamma\ar[r]\ar[d,equal]\ar[rd,phantom,"\circlearrowright"{xshift=5,yshift=-3}]&\Aut_0((\Gamma\bs G)\rtimes G)\ar[d,dash,"\sim"{sloped}]\\
\Gamma\ar[r]&\Diff_0(G)_\Gamma. 
\end{tikzcd}
\end{equation*}
This induces 
\begin{equation*}
\begin{tikzcd}[column sep=15]
1\ar[r]&\Gamma\ar[r]\ar[d,equal]\ar[rd,phantom,"\circlearrowright"{xshift=5,yshift=-3}]&\Aut_0((\Gamma\bs G)\rtimes G)/{\sim}_{o,0}\ar[r]\ar[d,dash,"\sim"{sloped}]\ar[rd,phantom,"\circlearrowright"{xshift=0,yshift=0}]&(\Aut_0((\Gamma\bs G)\rtimes G)/{\sim}_{o,0})/\Gamma\ar[r]\ar[d,dash,"\sim"{sloped}]&1\\
1\ar[r]&\Gamma\ar[r]&\Diff_0(G)_\Gamma/{\sim}_{o,0}\ar[r]&(\Diff_0(G)_\Gamma/{\sim}_{o,0})/\Gamma\ar[r]&1, 
\end{tikzcd}
\end{equation*}
which is the first and second row of the diagram in Proposition \ref{gggzgzaut}. The exactness of the first row follows from that of the second. This completes the proof of Proposition \ref{gggzgzaut}. 

\begin{prop}
We have 
\begin{equation*}
\begin{tikzcd}[row sep=13]
T_o(\rho_\Gamma)\ar[r,phantom,"\curvearrowleft"]\ar[d,dash,sloped,"\sim"]&[-31]\Aut_0((\Gamma\bs G){\rtimes}G)/{\sim}_{o,0}\ar[d,dash,sloped,"\sim"]\\
\Aut(G)\bs(\Diff(G)_\Gamma/{\sim})\ar[r,phantom,"\curvearrowleft"]\ar[d,dash,sloped,"\sim"]&\Diff_0(G)_\Gamma/{\sim}_{o,0}\ar[d,dash,sloped,"\sim"]\\
\Aut(\wt{G})\bs(\Diff(\wt{G})_{\wt{\Gamma}}/{\sim})\ar[r,phantom,"\curvearrowleft"]\ar[d,dash,sloped,"\sim"]&\Diff_0(\wt{G})_{\wt{\Gamma}}/{\sim}_{o,0}\ar[d,dash,sloped,"\sim"]\\
\Aut(\wt{G})\bs\mca{H}(\wt{\Gamma},\wt{G})\ar[r,phantom,"\curvearrowleft"]&\Aut(\wt{\Gamma}). 
\end{tikzcd}
\end{equation*}
The left vertical maps are equivariant with respect to the actions. 
\end{prop}

\begin{proof}
This can be seen easily from the definitions. 
\end{proof}

\begin{cor}
We have 
\begin{equation*}
\begin{tikzcd}[row sep=13]
T_o(\rho_\Gamma)\ar[r,phantom,"\curvearrowleft"]\ar[d,dash,sloped,"\sim"]&[-31](\Aut_0((\Gamma\bs G){\rtimes}G)/{\sim}_{o,0})/\Gamma\ar[d,dash,sloped,"\sim"]\\
\Aut(G)\bs(\Diff(G)_\Gamma/{\sim})\ar[r,phantom,"\curvearrowleft"]\ar[d,dash,sloped,"\sim"]&(\Diff_0(G)_\Gamma/{\sim}_{o,0})/\Gamma\ar[d,dash,sloped,"\sim"]\\
\Aut(\wt{G})\bs(\Diff(\wt{G})_{\wt{\Gamma}}/{\sim})\ar[r,phantom,"\curvearrowleft"]\ar[d,dash,sloped,"\sim"]&(\Diff_0(\wt{G})_{\wt{\Gamma}}/{\sim}_{o,0})/\wt{\Gamma}\ar[d,dash,sloped,"\sim"]\\
\Aut(\wt{G})\bs\mca{H}(\wt{\Gamma},\wt{G})\ar[r,phantom,"\curvearrowleft"]&\Out(\wt{\Gamma}). 
\end{tikzcd}
\end{equation*}
The left vertical maps are equivariant with respect to the actions. 
\end{cor}

\subsubsection{One more proposition}\label{onemore}
In this section we prove the following proposition, from which Theorem \ref{tormcgo} follows. 

\begin{prop}
We have an isomorphism 
\begin{equation*}
(\Aut_0((\Gamma\bs G)\rtimes G)/{\sim}_{o,0})/\Gamma\simeq\Aut((\Gamma\bs G)\rtimes G)/{\sim}_o=\MCG_o(\rho_\Gamma). 
\end{equation*}
\end{prop}

\begin{lem}
The homomorphism $\Aut_0((\Gamma\bs G)\rtimes G)/{\sim}_{o,0}\to\Aut((\Gamma\bs G)\rtimes G)/{\sim}_o$ induced by the inclusion is surjective. 
\end{lem}

\begin{proof}
Let $\varphi\in\Aut((\Gamma\bs G)\rtimes G)$. Let $g\in G$ be such that $F_\varphi(\Gamma)=\Gamma g$. Then $F_{\varphi_g\varphi}(\Gamma)=\Gamma$. So $\varphi_g\varphi\in\Aut_0((\Gamma\bs G)\rtimes G)$. Since $G$ is connected, we have $\varphi_g\sim_o\id$. Hence $\varphi_g\varphi\sim_o\varphi$. 
\end{proof}

The homomorphism in the lemma induces a surjective homomorphism 
\begin{equation}\label{aut0ggggaut}
(\Aut_0((\Gamma\bs G)\rtimes G)/{\sim}_{o,0})/\Gamma\to\Aut((\Gamma\bs G)\rtimes G)/{\sim}_o. 
\end{equation}
We will show that it is injective. 

Let $\lbrack\Gamma\bs G,\Gamma\bs G\rbrack$ denote the set of homotopy classes of continuous maps $\Gamma\bs G\to\Gamma\bs G$ and $\lbrack\Gamma\bs G,\Gamma\bs G\rbrack_0$ denote the set of based homotopy classes of continuous maps $(\Gamma\bs G,\Gamma)\to(\Gamma\bs G,\Gamma)$. The map $\lbrack\Gamma\bs G,\Gamma\bs G\rbrack_0\to\lbrack\Gamma\bs G,\Gamma\bs G\rbrack$ induced by the inclusion is surjective. 

Recall that there is an action $\lbrack\Gamma\bs G,\Gamma\bs G\rbrack_0\curvearrowleft\pi_1(\Gamma\bs G,\Gamma)$ defined by $[f][c]=[f_0]$ for $[f]\in\lbrack\Gamma\bs G,\Gamma\bs G\rbrack_0$ and $[c]\in\pi_1(\Gamma\bs G,\Gamma)$, where $f_t\colon\Gamma\bs G\to\Gamma\bs G$ is a continuous homotopy such that $f_1=f$ and $f_t(\Gamma)=c(t)$ for all $t\in I$, which exists by the homotopy extension property. Note that this is a right action by our convention of the product of fundamental groups. We have 
\begin{equation*}
\begin{tikzcd}
\lbrack\Gamma\bs G,\Gamma\bs G\rbrack_0\ar[r]\ar[d]\ar[rd,phantom,"\circlearrowright",pos=.1]&\lbrack\Gamma\bs G,\Gamma\bs G\rbrack\\
\lbrack\Gamma\bs G,\Gamma\bs G\rbrack_0/\pi_1(\Gamma\bs G,\Gamma). \ar[ru,sloped,"\sim"]&\ 
\end{tikzcd}
\end{equation*}
Since $\wt{\Gamma}\bs\wt{G}\simeq\Gamma\bs G$, we have an isomorphism $\Phi\colon\wt{\Gamma}\xrightarrow{\iota_1^{-1}}\pi_1(\wt{\Gamma}\bs\wt{G},\wt{\Gamma})\simeq\pi_1(\Gamma\bs G,\Gamma)$. 

Define $\Aut_0((\Gamma\bs G)\rtimes G)\curvearrowleft\Gamma$ by $(\varphi,\gamma)\mapsto\varphi_\gamma^{-1}\varphi$. It induces an action $\Aut_0((\Gamma\bs G)\rtimes G)/{\sim}_{o,0}\curvearrowleft\Gamma$, hence $\Aut_0((\Gamma\bs G)\rtimes G)/{\sim}_{o,0}\curvearrowleft\wt{\Gamma}$ through the projection $p\colon\wt{\Gamma}\to\Gamma$. 

\begin{lem}
The left vertical map in 
\begin{equation*}
\begin{tikzcd}
\Aut_0((\Gamma\bs G)\rtimes G)/{\sim}_{o,0}\ar[r,phantom,"\curvearrowleft"]\ar[d,hook,"b"']&[-35pt]\wt{\Gamma}\ar[d,"\sim"{sloped},"\Phi"{xshift=5}]\\
\lbrack\Gamma\bs G,\Gamma\bs G\rbrack_0\ar[r,phantom,"\curvearrowleft"]&\pi_1(\Gamma\bs G,\Gamma)
\end{tikzcd}
\end{equation*}
is equivariant with respect to the actions, where $b$ is the map induced by the map taking base maps. 
\end{lem}

\begin{proof}
The injectivity of $b$ follows from Theorem 6.29 in \cite{Lee} (the existence of a continuous (based) homotopy between $C^\infty$ maps implies the existence of a $C^\infty$ (based) homotopy between them), Proposition \ref{homcon}, Proposition \ref{connuni} and Lemma \ref{hnhnphipoh}. 

We prove the equivariance. Let $\wt{\gamma}\in\wt{\Gamma}$ and $\varphi\in\Aut_0((\Gamma\bs G)\rtimes G)$. Let $\gamma=p\wt{\gamma}\in\Gamma$. We have 
\begin{equation*}
\begin{tikzcd}
\wt{G}\ar[r,"p"]\ar[d]\ar[rd,phantom,"\circlearrowright"{xshift=2}]&G\ar[d]\\
\wt{\Gamma}\bs\wt{G}\ar[r,dash,"\sim"]&\Gamma\bs G. 
\end{tikzcd}
\end{equation*}
Let $\wt{c}\colon I\to\wt{G}$ be a continuous map such that $\wt{c}(0)=\wt{\gamma}$ and $\wt{c}(1)=1$. Let $\ol{\wt{c}}$, $c$ and $\ol{c}$ be the images of $\wt{c}$ in $\wt{\Gamma}\bs\wt{G}$, $G$ and $\Gamma\bs G$. 
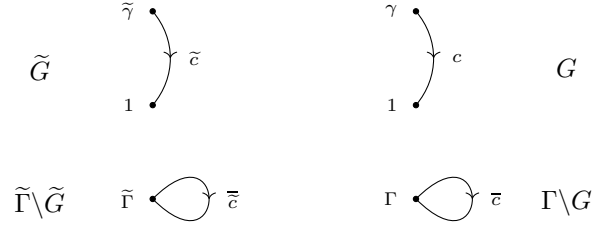
\begin{equation*}
\begin{tikzpicture}[every label/.append style={font=\scriptsize},matrix of math nodes,decoration={markings,mark=at position0.5with{\arrow{>}}}]
\clip(-2,-1.8)rectangle(6,1.6);
\node(1)at(-1.5,.5){\wt{G}};
\node(2)at(5.5,.5){G};
\node(3)at(-1.5,-1.3){\wt{\Gamma}\bs\wt{G}};
\node(4)at(5.5,-1.3){\Gamma\bs G};
\foreach\a/\b/\c/\d/\e/\f in{1/0/\wt{\gamma}/\wt{c}/\wt{\Gamma}/\ol{\wt{c}},2/3.5/\gamma/c/\Gamma/\ol{c}}{
\node(a\a)[xshift=\b cm,label=left:1]at(0,0){};
\node(b\a)[above=1cm of a\a,label=left:\c]{};
\node(c\a)[below=1cm of a\a,label=left:\e]{};
\node(d\a)at(c\a){};
\draw[postaction={decorate}](b\a.center)to[out=-50,in=50]node[label=right:\d]{}(a\a.center);
\draw[postaction={decorate}](c\a.center)..controls+(1,1)and+(1,-1)..(d\a.center)node[pos=.5,label=right:\f]{};
\foreach\x in{a,b,c}\filldraw(\x\a)circle(1pt);
}
\end{tikzpicture}
\end{equation*}
Then $\iota_1[\ol{\wt{c}}]=\wt{\gamma}$, so $\Phi(\wt{\gamma})=[\ol{c}]$. To show the equivariance, we compute the two compositions: 
\begin{align*}
[\varphi]&\mapsto[\varphi_{\gamma}^{-1}\varphi]\mapsto[R_\gamma F_\varphi], \\
[\varphi]&\mapsto[F_\varphi]\mapsto[F_\varphi][\ol{c}], 
\end{align*}
where $R_\gamma\colon\Gamma\bs G\to\Gamma\bs G$ is the right multiplication by $\gamma$. The equality $[R_\gamma F_\varphi]=[F_\varphi][\ol{c}]$ in $\lbrack\Gamma\bs G,\Gamma\bs G\rbrack_0$ can be seen from the following figure. 
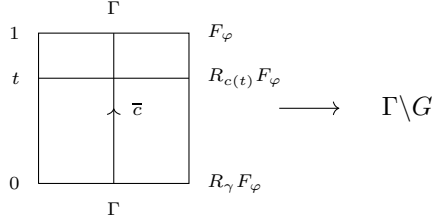
\begin{equation*}
\begin{tikzpicture}[every label/.append style={font=\scriptsize},matrix of math nodes,decoration={markings,mark=at position0.5with{\arrow{>}}}]
\def\a{2};
\def\b{1};
\def\c{1.4};
\node(1)at(0,0)[label=left:0]{};
\node(2)at(\a,0)[label=right:R_\gamma F_\varphi]{};
\node(3)at(\a,\a)[label=right:F_\varphi]{};
\node(4)at(0,\a)[label=left:1]{};
\node(5)at(\b,0)[label=below:\Gamma]{};
\node(6)at(\b,\a)[label=above:\Gamma]{};
\node(7)at(0,\c)[label=left:t]{};
\node(8)at(\a,\c)[label=right:R_{c(t)}F_\varphi]{};
\draw(1.center)--(2.center)--(3.center)--(4.center)--(1.center);
\draw[postaction={decorate}](5.center)--node[label=right:\ol{c}]{}(6.center);
\draw(7.center)--(8.center);
\draw[->](\a+1.2,\a/2)--(\a+2,\a/2);
\node at(\a+2.9,\a/2){\Gamma\bs G};
\end{tikzpicture}\qedhere
\end{equation*}
\end{proof}

To show that \eqref{aut0ggggaut} is injective, let $\varphi$, $\varphi^\prime\in\Aut_0((\Gamma\bs G)\rtimes G)$ be such that $\varphi\sim_o\varphi^\prime$. Since $[F_\varphi]=[F_{\varphi^\prime}]$ in $\lbrack\Gamma\bs G,\Gamma\bs G\rbrack$, there exists $[c]\in\pi_1(\Gamma\bs G,\Gamma)$ such that $[F_\varphi][c]=[F_{\varphi^\prime}]$ in $\lbrack\Gamma\bs G,\Gamma\bs G\rbrack_0$. Let $\wt{\gamma}=\Phi^{-1}[c]\in\wt{\Gamma}$. Then 
\begin{equation*}
b([\varphi]\wt{\gamma})=[F_\varphi][c]=[F_{\varphi^\prime}]=b[\varphi^\prime]. 
\end{equation*}
Hence $[\varphi^\prime]=[\varphi]\wt{\gamma}=[\varphi_{p\wt{\gamma}}^{-1}\varphi]$.

\subsubsection{The Teichm\"{u}ller space and the mapping class group of a surface}
Let $K=\PSO(2)$ and $\Sigma=\Gamma\bs G/K$. Then $\Sigma$ is a closed connected oriented hyperbolic surface of genus $g\geq2$. 

\begin{dfn}\label{mhssmhsfxx}
We say $(X,\phi)$ is a \emph{$\Sigma$-marked hyperbolic surface} if $X$ is a hyperbolic surface and $\phi\colon\Sigma\to X$ is a diffeomorphism. Let 
\begin{equation*}
\MHS(\Sigma)=\left\{\text{$\Sigma$-marked hyperbolic surfaces}\right\}. 
\end{equation*}
For $(X,\phi)$, $(X^\prime,\phi^\prime)\in\MHS(\Sigma)$, we write $(X,\phi)\sim(X^\prime,\phi^\prime)$ if there exists an isometry $\Phi\colon X\to X^\prime$ such that $\Phi\phi$ is homotopic to $\phi^\prime$. Then $\sim$ is an equivalence relation on $\MHS(\Sigma)$. The set $T(\Sigma)=\MHS(\Sigma)/{\sim}$ is called the \emph{Teichm\"{u}ller space} of $\Sigma$. 
\end{dfn}

\begin{prop}\label{bijtsigauth}
We have a bijection $T(\Sigma)\simeq\Aut(G)\bs\mca{H}(\Gamma,G)$. 
\end{prop}

\begin{proof}
Here we only construct a map $T(\Sigma)\to\Aut(G)\bs\mca{H}(\Gamma,G)$ and omit the proof of its bijectivity. See Proposition 10.2 of \cite{FMar}. Let $(X,\phi)\in\MHS(\Sigma)$. Let $(\wt{X},\wt{x})\to(X,\phi(\Gamma K))$ be a universal cover and $f\colon\wt{X}\to\bb{H}^2$ be an isometry. Let $\alpha(X,\phi)$ be the composition 
\begin{equation*}
\begin{tikzcd}[column sep=15]
\Gamma\ar[r,"\sim","\iota_K^{-1}"{yshift=7}]&\pi_1(\Sigma,\Gamma K)\ar[r,"\sim","\phi_*"{yshift=7}]&\pi_1(X,\phi(\Gamma K))\ar[r,hook,"\iota_\wt{x}"{yshift=3}]&\Isom_+(\wt{X})\ar[r,"\sim","f\ \cdot\ f^{-1}"{xshift=3,yshift=5}]&\Isom_+(\bb{H}^2)\ar[r,dash,"\sim"]&[-7]G, 
\end{tikzcd}
\end{equation*}
where $\iota_K$ (resp. $\iota_\wt{x}$) is defined with respect to the covering $(G/K,K)\to(\Sigma,\Gamma K)$ (resp. $(\wt{X},\wt{x})\to(X,\phi(\Gamma K))$). Then $[\alpha(X,\phi)]\in\Aut(G)\bs\mca{H}(\Gamma,G)$ is an element independent of the choices of $(\wt{X},\wt{x})\to(X,\phi(\Gamma K))$ and $f$. 

If $(X,\phi)$, $(X^\prime,\phi^\prime)\in\MHS(\Sigma)$ are such that $(X,\phi)\sim(X^\prime,\phi^\prime)$, we have a commutative diagram 
\begin{equation*}
\begin{tikzcd}[column sep=17]
&\pi_1(X,\phi(\Gamma K))\ar[d,"\Phi_*"]\ar[r,hook,"\iota_\wt{x}"]&\Isom_+(\wt{X})\ar[d,"\wt{\Phi}\ \cdot\ \wt{\Phi}^
{-1}"]\ar[r,dash,"\sim"]&\Isom_+(\bb{H}^2)\ar[r,dash,"\sim"]\ar[d,dash,"\sim"{sloped}]&G\ar[d,dash,"\sim"{sloped}]\\
\pi_1(\Sigma,\Gamma K)\ar[ru,"\phi_*"]\ar[rd,"\phi^\prime_*"']&\pi_1(X^\prime,\Phi\phi(\Gamma K))\ar[d,"c\ \cdot\ c^{-1}"]\ar[r,hook,"\iota_{\wt{\Phi}\wt{x}}"]&\Isom_+(\wt{X^\prime})\ar[r,dash,"\sim"]\ar[d,"\id"]&\Isom_+(\bb{H}^2)\ar[r,dash,"\sim"]\ar[d,"\id"]&G\ar[d,"\id"]\\
&\pi_1(X^\prime,\phi^\prime(\Gamma K))\ar[r,hook,"\iota_\wt{x^\prime}"]&\Isom_+(\wt{X^\prime})\ar[r,dash,"\sim"]&\Isom_+(\bb{H}^2)\ar[r,dash,"\sim"]&G, 
\end{tikzcd}
\end{equation*}
where $c\colon I\to X^\prime$ is a path obtained by evaluating $\Gamma K$ to a homotopy between $\Phi\phi$ and $\phi^\prime$, $(\wt{X^\prime},\wt{x^\prime})\to(X^\prime,\phi^\prime(\Gamma K))$ is a universal covering and $\wt{\Phi}\colon\wt{X}\to\wt{X^\prime}$ is a lift of $\Phi$. Hence the map $\MHS(\Sigma)\to\Aut(G)\bs\mca{H}(\Gamma,G)$ factors through $\MHS(\Sigma)\to T(\Sigma)$. 
\end{proof}

\begin{dfn}
The group $\MCG^\pm(\Sigma)=\pi_0\Homeo(\Sigma)$ is called the \emph{extended mapping class group} of $\Sigma$. 
\end{dfn}

\begin{thm}[Dehn--Nielsen--Baer]
We have a group isomorphism $\MCG^\pm(\Sigma)\simeq\Out(\Gamma)$. 
\end{thm}

\begin{proof}
We only define the map $\MCG^\pm(\Sigma)\to\Out(\Gamma)$ here and omit the proof of its bijectivity. See Theorem 8.1 of \cite{FMar}. Let $\psi\in\Homeo(\Sigma)$ such that $\psi(\Gamma K)=\Gamma K$. Let 
\begin{equation*}
\beta(\psi)\colon\Gamma\xrightarrow{\iota_K^{-1}}\pi_1(\Sigma,\Gamma K)\xrightarrow{\psi_*}\pi_1(\Sigma,\Gamma K)\xrightarrow{\iota_K}\Gamma. 
\end{equation*}
Then $\beta(\psi)\in\Aut(\Gamma)$. 

If $\psi$, $\psi^\prime\in\Homeo(\Sigma)$ are such that $\psi(\Gamma K)=\Gamma K$, $\psi^\prime(\Gamma K)=\Gamma K$ and $[\psi]=[\psi^\prime]$ in $\MCG^\pm(\Sigma)$, we have 
\begin{equation*}
\begin{tikzcd}[row sep=5]
&\pi_1(\Sigma,\Gamma K)\ar[dd,"c\ \cdot\ c^{-1}"{name=U}]\\
\pi_1(\Sigma,\Gamma K)\ar[ru,"\psi_*"]\ar[rd,"\psi^\prime_*"']\\
&\pi_1(\Sigma,\Gamma K), 
\ar[from=2-1,to=U,phantom,"\circlearrowright"]
\end{tikzcd}
\end{equation*}
where $c\colon I\to\Sigma$ is the loop obtained by evaluating $\Gamma K$ to a homotopy between $\psi$ an $\psi^\prime$. Hence the map 
\begin{align*}
\left\{\psi\in\Homeo(\Sigma)\ \middle|\ \psi(\Gamma K)=\Gamma K\right\}&\to\Out(\Gamma)\\
\psi&\mapsto[\beta(\psi)]
\end{align*}
induces a map $\MCG^\pm(\Sigma)\to\Out(\Gamma)$. 
\end{proof}

We have an action $T(\Sigma)\curvearrowleft\MCG^\pm(\Sigma)$ defined by $[X,\phi][\psi]=[X,\phi\psi]$, where $\psi\colon\Sigma\to\Sigma$ is taken to be a diffeomorphism. 

\begin{prop}
We have 
\begin{equation*}
\begin{tikzcd}
T(\Sigma)\ar[r,phantom,"\curvearrowleft"]\ar[d,dash,sloped,"\sim"]&[-19pt]\MCG^\pm(\Sigma)\ar[d,dash,sloped,"\sim"]\\
\Aut(G)\bs\mca{H}(\Gamma,G)\ar[r,phantom,"\curvearrowleft"]&\Out(\Gamma), 
\end{tikzcd}
\end{equation*}
where the left vertical map is equivariant with respect to the actions. 
\end{prop}

\begin{proof}
Let $[X,\phi]\in T(\Sigma)$, $[\psi]\in\MCG^\pm(\Sigma)$. We may assume $\psi$ is a diffeomorphism and $\psi(\Gamma K)=\Gamma K$. Then $\alpha(X,\phi\psi)=\alpha(X,\phi)\beta(\psi)$. This proves the equivariance. 
\end{proof}

\subsubsection{Summary}
What we have proved in this section can be summarized in the following diagram: 
\begin{equation*}
\begin{tikzcd}[row sep=10]
&[-20]1\ar[dd]&[-40]&[-15]\\
\ \\
\Hom(\wt{\Gamma},Z)\ar[dd]&\Hom(\wt{\Gamma},Z)\ar[dd]\\
\ \\
\Aut(\wt{G})\bs\mca{H}(\wt{\Gamma},\wt{G})\ar[r,phantom,"\curvearrowleft"]\ar[dd,"\ol{\pi}"]&\Out(\wt{\Gamma})\ar[rrd,dash,"\sim"{sloped}]\ar[dd,"\ol{\pi}_0"{yshift=-3}]\\
&&T_o(\rho_\Gamma)\ar[r,phantom,"\curvearrowleft"]&\MCG_o(\rho_\Gamma)\\
\Aut(G)\bs\mca{H}(\Gamma,G)\ar[r,phantom,"\curvearrowleft"]\ar[uu,bend left,"\ol{\sigma}"]&\Out(\Gamma)\ar[rrd,dash,"\sim"{sloped}]\ar[dd]\ar[uu,bend left,"\ol{\sigma}_0"]\\
&&T(\Sigma)\ar[r,phantom,"\curvearrowleft"]&\MCG^\pm(\Sigma)\\
&1. 
\ar[from=5-1,to=6-3,dash,"\sim"{sloped},crossing over]\ar[from=7-1,to=8-3,dash,"\sim"{sloped},crossing over]
\end{tikzcd}
\end{equation*}

\bibliographystyle{amsalpha}
\bibliography{maruhashi}
\end{document}